\documentclass[reqno, 12pt]{amsart}
\pdfoutput=1
\makeatletter
\let\origsection=\section \def\section{\@ifstar{\origsection*}{\mysection}} 
\def\mysection{\@startsection{section}{1}\z@{.7\linespacing\@plus\linespacing}{.5\linespacing}{\normalfont\scshape\centering\S}}
\makeatother        

\usepackage{amsmath,amssymb,amsthm}
\usepackage{mathrsfs}
\usepackage{mathabx}\changenotsign
\usepackage{wasysym}
\usepackage{dsfont}
 
\usepackage{xcolor}
\usepackage[backref]{hyperref}
\hypersetup{
    colorlinks,
    linkcolor={red!60!black},
    citecolor={green!60!black},
    urlcolor={blue!60!black}
}

\usepackage[open,openlevel=2,atend]{bookmark}

\usepackage[abbrev,msc-links,backrefs]{amsrefs} 
\usepackage{doi}

\renewcommand{\PrintDOI}[1]{\doi{#1}}

\usepackage[T1]{fontenc}
\usepackage{lmodern}
\usepackage[babel]{microtype}
\usepackage[english]{babel}

\usepackage{geometry}
\numberwithin{equation}{section}
\numberwithin{figure}{section}

\usepackage{scalerel}

\makeatletter
\newcommand{\overrighharpoonup}[1]{\ThisStyle{ \vbox {\m@th\ialign{##\crcr
 \rightharpoonupfill \crcr
 \noalign{\kern-\p@\nointerlineskip}
 $\hfil\SavedStyle#1\hfil$\crcr}}}}

\def\rightharpoonupfill{$\SavedStyle\m@th\mkern+0.8mu\cleaders\hbox{$\shortbar\mkern-4mu$}\hfill\rightharpoonuptip\mkern+0.8mu$}

\def\rightharpoonuptip{ \raisebox{\z@}[2pt][1pt]{\scalebox{0.55}{$\SavedStyle\rightharpoonup$}}}

\def\shortbar{ \smash{\scalebox{0.55}{$\SavedStyle\relbar$}}}
\makeatother

\let\seq=\overrighharpoonup

\DeclareRobustCommand{\harpp}[1]{\seq{#1}}

\usepackage{accents}

\usepackage{enumitem}
\def\rmlabel{\upshape({\itshape \roman*\,})}

\def\alabel{\upshape({\itshape \alph*\,})}

\def\nlabel{\upshape({\itshape \arabic*\,})}

\makeatletter
\def\greek#1{\expandafter\@greek\csname c@#1\endcsname}
\def\Greek#1{\expandafter\@Greek\csname c@#1\endcsname}
\def\@greek#1{\ifcase#1
	\or $\alpha$\or $\beta$\or $\gamma$\or $\delta$\or $\epsilon$\or $\zeta$\or $\eta$\or 
	$\theta$\or $\iota$\or $\kappa$\or $\lambda$	\or $\mu$\or $\nu$\or $\xi$\or $o$\or 
	$\pi$\or $\rho$\or $\sigma$\or $\tau$\or $\upsilon$\or $\phi$\or $\chi$\or $\psi$\or 
	$\omega$\fi}
\def\@Greek#1{\ifcase#1
	\or $\mathrm{A}$\or $\mathrm{B}$\or $\Gamma$\or $\Delta$\or $\mathrm{E}$\or $\mathrm{Z}$\or 
	$\mathrm{H}$\or $\Theta$\or $\mathrm{I}$\or $\mathrm{K}$\or $\Lambda$\or $\mathrm{M}$\or 
	$\mathrm{N}$\or $\Xi$\or $\mathrm{O}$\or $\Pi$\or $\mathrm{P}$\or $\Sigma$\or 
	$\mathrm{T}$\or $\mathrm{Y}$\or $\Phi$\or $\mathrm{X}$\or $\Psi$\or $\Omega$\fi}

\AddEnumerateCounter{\greek}{\@greek}{24}
\AddEnumerateCounter{\Greek}{\@Greek}{12}

\makeatother

\def\glabel{\upshape({\itshape \greek*})}

\let\polishlcross=\l
\def\l{\ifmmode\ell\else\polishlcross\fi}

\let\sm=\setminus

\makeatletter
\def\moverlay{\mathpalette\mov@rlay}
\def\mov@rlay#1#2{\leavevmode\vtop{   \baselineskip\z@skip \lineskiplimit-\maxdimen
   \ialign{\hfil$\m@th#1##$\hfil\cr#2\crcr}}}
\newcommand{\charfusion}[3][\mathord]{
    #1{\ifx#1\mathop\vphantom{#2}\fi
        \mathpalette\mov@rlay{#2\cr#3}
      }
    \ifx#1\mathop\expandafter\displaylimits\fi}
\makeatother

\newcommand{\dcup}{\charfusion[\mathbin]{\cup}{\cdot}}
\newcommand{\bigdcup}{\charfusion[\mathop]{\bigcup}{\cdot}}

\DeclareFontFamily{U}  {MnSymbolC}{}
\DeclareSymbolFont{MnSyC}         {U}  {MnSymbolC}{m}{n}
\DeclareFontShape{U}{MnSymbolC}{m}{n}{
    <-6>  MnSymbolC5
   <6-7>  MnSymbolC6
   <7-8>  MnSymbolC7
   <8-9>  MnSymbolC8
   <9-10> MnSymbolC9
  <10-12> MnSymbolC10
  <12->   MnSymbolC12}{}
\DeclareMathSymbol{\powerset}{\mathord}{MnSyC}{180}
\DeclareMathSymbol{\leftY}{\mathord}{MnSyC}{42}

\DeclareSymbolFont{symbolsC}{U}{txsyc}{m}{n}
\SetSymbolFont{symbolsC}{bold}{U}{txsyc}{bx}{n}
\DeclareFontSubstitution{U}{txsyc}{m}{n}
\DeclareMathSymbol{\strictif}{\mathrel}{symbolsC}{74}

\let\epsilon=\varepsilon

\let\rho=\varrho
\let\wh=\widehat
\let\wt=\widetilde

\def\NN{{\mathds N}}
\def\ZZ{{\mathds Z}}

\newcommand{\ccA}{\mathscr{A}}
\newcommand{\ccB}{\mathscr{B}}
\newcommand{\ccC}{\mathscr{C}}
\newcommand{\ccD}{\mathscr{D}}
\newcommand{\ccE}{\mathscr{E}}
\newcommand{\ccF}{\mathscr{F}}
\newcommand{\ccG}{\mathscr{G}}
\newcommand{\ccH}{\mathscr{H}}
\newcommand{\ccM}{\mathscr{M}}
\newcommand{\ccN}{\mathscr{N}}
\newcommand{\ccP}{\mathscr{P}}
\newcommand{\ccQ}{\mathscr{Q}}
\newcommand{\ccS}{\mathscr{S}}
\newcommand{\ccW}{\mathscr{W}}
\newcommand{\ccX}{\mathscr{X}}
\newcommand{\ccY}{\mathscr{Y}}
\newcommand{\ccZ}{\mathscr{Z}}

\newcommand{\gM}{\mathfrak{M}}

\theoremstyle{plain}
\newtheorem{thm}{Theorem}[section]
\newtheorem{lemma}[thm]{Lemma}
\newtheorem{prop}[thm]{Proposition}
\newtheorem{prob}[thm]{Problem}

\newtheorem{cor}[thm]{Corollary}
\newtheorem{fact}[thm]{Fact}
\newtheorem{clm}[thm]{Claim}
\newtheorem{summary}[thm]{Summary}

\theoremstyle{definition}
\newtheorem{remark}[thm]{Remark}
\newtheorem{dfn}[thm]{Definition}
\newtheorem{example}[thm]{Example}

\def\gth{\mathrm{girth}}
\def\ggth{\mathfrak{girth}}
\def\Gth{\mathrm{Girth}}
\def\GTH{\mathfrak{Girth}}

\def\HJ{\mathrm{HJ}}
\def\CPL{\mathrm{CPL}}
\def\PC{\mathrm{PC}}
\def\Rms{\mathrm{Rms}}
\def\Ext{\mathrm{Ext}}
\def\nni{\mathrm{n.n.i.}}
\def\pt{\mathrm{pt}}
\def\ff{\mathrm{fpt}}
\def\len#1{\vert #1\vert}
\def\ord#1{\mathrm{ord}(#1)}

\let\ups=\varUpsilon
\let\phi=\varphi

\newcommand{\str}{\scalebox{1.3}[1.3]{{$\blacktriangleleft$}}}
\newcommand{\Str}{\hskip.4em \str \hskip.4em}

\newcommand{\karo}{\scalebox{1.2}[1.2]{{$\diamondsuit$}}}

\let\theta=\vartheta
\let\lra=\longrightarrow
\let\vn=\varnothing

\usepackage{yfonts}

\usepackage{tikz}

\usetikzlibrary{calc,decorations.pathreplacing,decorations.pathmorphing, calligraphy}
\usetikzlibrary{arrows,decorations.pathreplacing}
\usetikzlibrary {arrows.meta} 
 
\pgfdeclarelayer{background}
\pgfdeclarelayer{foreground}
\pgfdeclarelayer{front}
\pgfsetlayers{background,main,foreground,front}

\usepackage{multicol}
\usepackage{subcaption}
\newcommand{\conc}{\charfusion[\mathbin]{+}{\times}}

\makeindex

\newcommand{\redge}[8]{
	
	\ifx\relax#5\relax
	\def\qoffs{0pt}
	\else
	\def\qoffs{#5}
	\fi
	
	\def\rhedge{
		($#1+#4!\qoffs!-90:#2-#4$) -- 
		($#2+#1!\qoffs!-90:#3-#1$) -- 
		($#3+#2!\qoffs!-90:#4-#2$) -- 
		($#4+#3!\qoffs!-90:#1-#3$) -- cycle}

	\coordinate (12) at ($#1!\qoffs!90:#2$);
	\coordinate (14) at ($#1!\qoffs!-90:#4$);
	\coordinate (23) at ($#2!\qoffs!90:#3$);
	\coordinate (21) at ($#2!\qoffs!-90:#1$);
	\coordinate (34) at ($#3!\qoffs!90:#4$);
	\coordinate (32) at ($#3!\qoffs!-90:#2$);
	\coordinate (41) at ($#4!\qoffs!90:#1$);
	\coordinate (43) at ($#4!\qoffs!-90:#3$);
	
	\def\nrhedge{
		(14) let \p1=($(14)-#1$), \p2=($(12)-#1$) in 
		arc[start angle={atan2(\y1,\x1)}, delta angle={atan2(\y2,\x2)-atan2(\y1,\x1)-360*(atan2(\y2,\x2)-atan2(\y1,\x1)>0)}, x radius=\qoffs, y radius=\qoffs] --
		(21) let \p1=($(21)-#2$), \p2=($(23)-#2$) in 
		arc[start angle={atan2(\y1,\x1)}, delta angle={atan2(\y2,\x2)-atan2(\y1,\x1)-360*(atan2(\y2,\x2)-atan2(\y1,\x1)>0)}, x radius=\qoffs, y radius=\qoffs] --
		(32) let \p1=($(32)-#3$), \p2=($(34)-#3$) in 
		arc[start angle={atan2(\y1,\x1)}, delta angle={atan2(\y2,\x2)-atan2(\y1,\x1)-360*(atan2(\y2,\x2)-atan2(\y1,\x1)>0)}, x radius=\qoffs, y radius=\qoffs] --
		(43) let \p1=($(43)-#4$), \p2=($(41)-#4$) in 
		arc[start angle={atan2(\y1,\x1)}, delta angle={atan2(\y2,\x2)-atan2(\y1,\x1)-360*(atan2(\y2,\x2)-atan2(\y1,\x1)>0)}, x radius=\qoffs, y radius=\qoffs] --
		cycle}
	
	\ifx\relax#6\relax
	\def\rlwidth{1pt}
	\else
	\def\rlwidth{#6}
	\fi
	
	\ifx\relax#8\relax
	\fill \nrhedge;
	\else
	\fill[#8]\nrhedge;
	\fi
	
	\ifx\relax#7\relax
	\draw[line width=\rlwidth,rounded corners=\qoffs]\nrhedge;
	\else
	\draw[line width=\rlwidth,#7]\nrhedge;
	\fi
}

\begin{document}
\title[The Girth Ramsey Theorem]{The Girth Ramsey Theorem}

\author[Christian~Reiher]{Christian Reiher}
\address{Fachbereich Mathematik, Universit\"at Hamburg, Hamburg, Germany}
\email{Christian.Reiher@uni-hamburg.de}

\author[Vojt\v{e}ch~R\"{o}dl]{Vojt\v{e}ch R\"{o}dl}
\address{Department of Mathematics and Computer Science, Emory University, Atlanta, USA}
\email{rodl@mathcs.emory.edu}
\thanks{The second author is supported by NSF grant DMS 2300347.}


\begin{abstract}
Given a hypergraph $F$ and a number of colours $r$, there exists 
a hypergraph~$H$ of the same girth satisfying $H\lra (F)_r$.  
Moreover, for every linear hypergraph $F$ there exists a Ramsey 
hypergraph $H$ that locally looks like a forest of copies of $F$.  
\end{abstract}

\maketitle

\section{Introduction}
\label{sec:intro}
\subsection{Colouring vertices} 
\label{subsec:colver}

We commence with the well known fact, due to Erd\H{o}s
and Hajnal~\cite{EH66}, that for every $k\ge 2$ there exist $k$-uniform hypergraphs 
whose girth and chromatic number are simultaneously 
arbitrarily large (see also~\cites{AKRWZ16, Lov68, Ecken, Erd59}).

Recall that the {\it chromatic number} of a hypergraph $H$ is the least natural number 
$\chi(H)$ such that there exists a colouring of the vertices of $H$ using $\chi(H)$ 
colours with the property that no edge of $H$ is monochromatic. This is a Ramsey 
theoretic invariant of $H$, for a lower bound estimate of the form $\chi(H) > r$
can equivalently be expressed by the partition relation 
\begin{equation} \label{eq:11a}
	H \lra (e)^v_r\,,
\end{equation}
where the superscripted $v$ on the right side means that the objects receiving 
colours are {\it vertices} and the letter $e$ enclosed in parentheses indicates that 
the object we want to find monochromatically is an {\it edge}. 

The absence of cycles can equivalently be described in terms of forests.
Let us call a set~$N$ of edges a {\it forest} if there exists an 
enumeration $N=\{e_1, \ldots, e_{|N|}\}$ such that for 
every $j\in [2, |N|]$ the set $\bigl(\bigcup_{i<j}e_i\bigr)\cap e_j$
is either empty or it consists of a single vertex.
Now the aforementioned result on hypergraphs with high chromatic number 
and large girth reformulates as follows. 
 
\begin{thm}\label{thm:7338}
	For every $k\ge 2$ and all $r, n\in \NN$ there exists a $k$-uniform hypergraph $H$ 
	with $H \lra (e)^v_r$ having the property that any set consisting of at most $n$ edges 
	of $H$ forms a forest. \qed 
\end{thm} 

This result is optimal in the sense that for every forest $W$ there is some number
of colours $r$ such that every hypergraph $H$ satisfying $H \lra (e)^v_r$ contains
a copy of $W$. 
From a Ramsey theoretic perspective, it is equally natural to investigate 
the problem that instead of a monochromatic edge one wishes to enforce a monochromatic
copy of a given hypergraph $F$. For any two hypergraphs $F$ and $H$ we write
$\binom{H}{F}$ for the set of all {\it induced} subhypergraphs of $H$ isomorphic 
to $F$.
Given $\ccH\subseteq \binom{H}{F}$ and $r\in \NN$ the partition relation 
\begin{equation} \label{eq:11b}
	\ccH\lra (F)^v_r
\end{equation}
is defined to hold if for every colouring of the vertices of $H$ with $r$ colours there
exists a monochromatic copy $F_\star\in \ccH$. The existence of a system $\ccH$
having this property for given~$F$ and~$r$ is easily established by starting 
with a linear $v(F)$-uniform hypergraph whose chromatic number exceeds $r$,
and inserting copies of $F$ into its edges. Pursuing this argument further one 
arrives at the following result (see~\cite{NeRo76}). 

\begin{thm}\label{thm:1535}
	For every $k$-uniform hypergraph $F$ and all $r, n\in \NN$ there exists 
	a hypergraph~$H$ together with a system of copies $\ccH\subseteq \binom{H}{F}$
	such that 
	\begin{enumerate}[label=\rmlabel]
		\item $\ccH\lra (F)^v_r$ and    
  		\item for every $\ccN\subseteq \ccH$ with $|\ccN|\le n$ there exists an 
			enumeration $\ccN=\{F_1, \ldots, F_{|\ccN|}\}$ with the property that for 
			every $j\in [2, |\ccN|]$ the sets $\bigcup_{i<j}V(F_i)$ and $V(F_j)$
			have at most one vertex in common. \qed
	\end{enumerate}
\end{thm}

Again this result is optimal in the sense that it describes all configurations of 
copies of~$F$ that need to be present in systems $\ccH$ satisfying $\ccH\lra (F)^v_r$
for large $r$. For a precise statement along these lines, we refer to the recent 
work of Daskin, Hoshen, Krivelevich, and Zhukovskii~\cite{DHKZ}.

\subsection{Colouring edges}
\label{subsec:coledges}

An entirely new level of difficulty emerges when edges rather than vertices
are the entities subject to colouration. Erd\H{o}s asked more than forty years
ago whether for every hypergraph $F$ and every number of colours $r$ there exists 
a hypergraph~$H$ with 
\[
	H\lra (F)_r\,,
\]
where now we are colouring edges\footnote[1]{Henceforth all colourings are colourings 
of edges and attempting to keep the notation simple we refrain from adding a 
superscripted ``$e$'' on the right side of our partition relations.} 
and the desired monochromatic occurrence of $F$ is still supposed to be induced.  
This problem was first solved in the $2$-uniform 
case~\cites{Deuber75, EHP75, Rodl73, Rodl76} and later in full 
generality~\cites{AH78, NeRo1}. In \S\ref{sssec:irt} we describe a simple proof of 
this {\it induced Ramsey theorem} from~\cite{NeRo5}. As a matter of fact, the 
articles~\cites{AH78, NeRo1} show much stronger results allowing subhypergraphs 
and not only edges to be coloured. Moreover,~\cite{NeRo1} proves 
that one can demand $H$ to have certain 
additional properties, provided that $F$ has these properties as well. 
For instance, let~$G$ be a hypergraph any pair of whose vertices 
is covered by an edge. Now if~$F$ has no subhypergraph isomorphic to $G$, 
then we can insist that the Ramsey hypergraph $H$ likewise does not have such a 
subhypergraph.
This result allows to preserve the clique number of~$F$, but not the girth. 
Let us recall that girth is defined as follows.     

\begin{dfn} \label{dfn:girth}
	Given a hypergraph $H=(V, E)$ and an integer $n\ge 2$ we say that a cyclic sequence
		\begin{equation}\label{eq:1c}
		e_1v_1\ldots e_nv_n
	\end{equation}
		is an {\it $n$-cycle} in $H$ provided 
		\index{$n$-cycle}
	\begin{enumerate}[label=\upshape{($C\arabic*$)}]
		\item\label{it:C2} the edges $e_1, \ldots, e_n\in E$ are distinct;
		\item\label{it:C1} the vertices $v_1, \ldots, v_n\in V$ are distinct;
		\item\label{it:C3} and $v_i \in e_i\cap e_{i+1}$ for each $i\in\ZZ/n\ZZ$.
	\end{enumerate}	 
	Moreover, for an integer $g\ge 2$ we write $\gth(H)>g$ if for no~$n\in [2, g]$
	there is an $n$-cycle contained in $H$.
\end{dfn} 

In particular, $\gth(H)>2$ means that $H$ is {\it linear} in the sense that 
no two of its edges intersect in more than one vertex. We can now state the 
simplest form of the girth Ramsey theorem, which is among the main results 
of this article. 

\begin{thm} \label{thm:grth1} 
	Given an integer $g\ge 2$, a hypergraph $F$ with $\gth(F)>g$, and a natural 
	number $r$, there exists a hypergraph $H$ with $\gth(H) > g$ and 
 	 	\begin{equation} \label{eq:1}
 		H\longrightarrow (F)_r\,.
 	\end{equation}
\end{thm}

Roughly speaking this result gives us quite a lot of local control over the 
Ramsey hypergraph $H$ of $F$. Ultimately, one would like to analyse
the local structure of Ramsey hypergraphs in the same way as Theorem~\ref{thm:7338} 
describes the local structure of hypergraphs with large chromatic number. 
Here we solve this problem for linear hypergraphs $F$ (and therefore, in particular, 
for all graphs).

\begin{dfn}\label{dfn:1606}
	Let $F$ be a linear hypergraph. We call a set $\ccN$ of hypergraphs isomorphic 
	to $F$ a {\it forest of copies of $F$} \index{forest of copies} if there exists 
	an enumeration 
	$\ccN=\bigl\{F_1, \ldots, F_{|\ccN|}\bigr\}$
	such that for every $j\in [2, |\ccN|]$ the 
	set $z_j=\bigl(\bigcup_{i<j}V(F_i)\bigr)\cap V(F_j)$ satisfies  
		\begin{enumerate}[label=\rmlabel]
		\item\label{it:1606i} either $|z_j|\le 1$
		\item\label{it:1606ii} or $z_j\in \bigl(\bigcup_{i<j}E(F_i)\bigr)\cap E(F_j)$.
	\end{enumerate}
	We denote the union of a forest of copies $\ccN$ 
	by $\bigcup\ccN$; explicitly, this is the hypergraph
	with vertex set $\bigcup_{F_\star\in \ccN}V(F_\star)$ and edge 
	set $\bigcup_{F_\star\in \ccN}E(F_\star)$. 
	A hypergraph $G$ is said to be a {\it partial $F$-forest} \index{partial 
	$F$-forest} if it is an induced subhypergraph of $\bigcup\ccN$ for some 
	forest $\ccN$ of copies of~$F$.
\end{dfn}

It is not difficult to see that every such forest of copies of $F$ needs to be 
contained in every Ramsey hypergraph of $F$ with sufficiently many colours. 
Another main result of this work states that, conversely, we can build
Ramsey hypergraphs that locally look like forests of copies of $F$. 

\begin{thm}\label{cor:19}
	For every linear hypergraph $F$ and all $r, n\in \NN$ there exists a linear
	hypergraph~$H$ with $H\lra (F)_r$ such that every set $X\subseteq V(H)$
	whose size it at most~$n$ induces a partial $F$-forest in $H$. 
\end{thm}

An oversimplified way of looking at the construction of $H$ is the following:
We start with an extremely large set $\ccH$ of mutually disjoint copies of $F$. Then 
we perform many steps of glueing some copies together along edges, with the aim of
obtaining the desired hypergraph $H$. Now, on the one hand, we need to glue quite a 
lot in order to ensure the Ramsey property. On the other hand, locally we are not 
allowed to glue too much because we want to exclude short cycles of copies of $F$. 
In any case, $H$ is constructed together with a system of copies $\ccH\subseteq\binom HF$ such that $\ccH\lra (F)_r$ and we may wonder whether we can insist that all 
small subsets of $\ccH$ should be forests of copies of $F$. 

Before answering this question we need to draw attention to a somewhat bizarre 
difference between the notion of forests of copies of $F$ and the standard forests 
of edges considered in~\S\ref{subsec:colver}. 
It is well known that if a set of edges forms a forest, then so does 
each of its subsets. But, as the following counterexample demonstrates, the analogous 
statement for forests of copies fails (see Figure~\ref{fig:11}). 

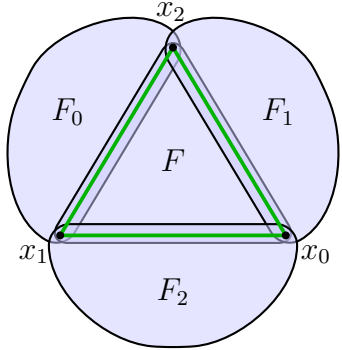
\begin{figure}[ht]
	\centering

	\begin{tikzpicture}[scale=.5]
	
	\coordinate (x1) at (-3, 0);
	\coordinate (x0) at (3,0);
	\coordinate (x2) at (0,5);

\fill [blue!20!white, rounded corners = 8pt, opacity=.5] (.4,5.1) to[out=140,in=60] (-4,4) to[out=-110, in=150] (-2.9,-.4) --cycle;

\draw [thick, rounded corners = 8pt] (.4,5.1) to[out=140,in=60] (-4,4) to[out=-110, in=150] (-2.9,-.4) --cycle;

\fill [blue!20!white, rounded corners = 8pt, opacity=.5] (0,5.4) -- (-3.36,-.2) -- (3.36,-.2)--cycle;

\draw [thick, rounded corners = 8pt] (0,5.4) -- (-3.36,-.2) -- (3.36,-.2)--cycle;

\fill [blue!20!white, rounded corners = 8pt, opacity=.5] (-.4,5.1) to[out=40,in=120] (4,4) to[out=-70, in=30] (2.9,-.4) --cycle;

\draw [thick, rounded corners = 8pt] (-.4,5.1) to[out=40,in=120] (4,4) to[out=-70, in=30] (2.9,-.4) --cycle;

\fill [blue!20!white, rounded corners = 8pt, opacity=.5] (-3.3,.3) to[out=-90,in=175] (0,-3) to[out=5, in=-90] (3.3,.3) --cycle;

\draw [thick, rounded corners = 8pt] (-3.3,.3) to[out=-90,in=175] (0,-3) to[out=5, in=-90] (3.3,.3) --cycle;

		\draw [green!70!black, line width=.5mm] (x0) -- (x1) -- (x2) -- (x0);
	
	\foreach \i in {0,1,2}{
		\fill (x\i) circle (3pt);}
	
	\node at (0,2) {$F$};
	\node at (-2.8,3.3) {$F_0$};
	\node at (2.8,3.3) {$F_1$};
	\node at (0,-1.5) {$F_2$};
	
	\node at (-.05,5.95) {$x_2$};
	\node at (-3.7,-.5) {$x_1$};
	\node at (3.8,-.5) {$x_0$};
	
	\end{tikzpicture}

	\caption{A subforest $\{F_0, F_1, F_2\}$ that fails to be a forest.}
	\label{fig:11} 
\end{figure} 
Let $F$ be a graph containing a triangle~$x_0x_1x_2$. 
For every index $i\in\ZZ/3\ZZ$ let $F_i$ be a graph isomorphic 
to $F$ having the edge~$x_{i+1}x_{i+2}$ but nothing else in common with $F$. 
Suppose that except for these intersections the copies in $\ccN=\{F, F_0, F_1, F_2\}$ 
are mutually disjoint. This enumeration exemplifies that~$\ccN$ is 
a forest of copies. However its subset $\ccN^-=\ccN\sm\{F\}$ fails 
to be such a forest. For instance, for the enumeration $\ccN^-=\{F_0, F_1, F_2\}$ the 
set $z_2=\bigl(V(F_0)\cup V(F_1)\bigr)\cap V(F_2)=\{x_0, x_1\}$ is certainly 
not in case~\ref{it:1606i} and, as it 
fails to be an edge of $F_0$ or $F_1$, it does not satisfy~\ref{it:1606ii} either.
By symmetry a similar problem arises when one enumerates~$\ccN^-$ in any other 
way. 

Summarising this discussion, we have seen that being a forest of copies is 
not preserved under taking subsets. This phenomenon explains the r\^{o}le of $\ccX$
in the most general version of the girth Ramsey theorem that follows.   

\begin{thm} \label{thm:1522}
	Given a linear hypergraph $F$ and $r, n\in \NN$ there exists a linear hypergraph~$H$
	together with a system of copies $\ccH\subseteq\binom{H}{F}$ 
	satisfying not only $\ccH\lra (F)_r$ but also the following statement:
	For every $\ccN\subseteq \ccH$ with $|\ccN|\in [2, n]$ there exists 
	a set $\ccX\subseteq \ccH$ such that $|\ccX|\le |\ccN|-2$ and $\ccN\cup\ccX$
	is a forest of copies.
\end{thm}

Let us emphasise again that allowing such a set $\ccX$ is necessary. For instance, 
if $F$ is a triangle and $r$ is large, then $\ccH$ needs to have a subset $\ccN$ 
consisting of five triangles arranged ``cyclically'' (see Figure~\ref{fig:12a}). 
Now $\ccN$ itself fails to be a forest of triangles, but by triangulating the pentagon 
one can create a set $\ccX$ of three further triangles such that $\ccN\cup\ccX$ is a 
forest of triangles and, hence, unavoidable in $\ccH$ (see Figure~\ref{fig:12b}). 

\begin{figure}[ht]
	\centering

		\begin{subfigure}[b]{0.4\textwidth}
			\centering
	\begin{tikzpicture}[scale=.8]
	
	\def\r{1.8cm};
	\coordinate (x0) at (142:\r);
	\coordinate (x1) at (214: \r);
	\coordinate (x2) at (286: \r);
	\coordinate (x3) at (358: \r);
	\coordinate (x4) at (70: \r);
	\coordinate (x5) at (106:3.5cm);
	\coordinate (x6) at (178:3.5cm);
	\coordinate (x7) at (250:3.5cm);
	\coordinate (x8) at (322:3.5cm);
	\coordinate (x9) at (34:3.5cm);

	\draw [thick] (x0)--(x1) --(x2) --(x3)--(x4) -- cycle;
	\draw [thick] (x4) -- (x5) -- (x0) -- (x6) -- (x1) -- (x7) -- (x2) -- (x8) -- (x3) -- (x9) -- cycle;

	\end{tikzpicture}

	\caption{A cycle of triangles\\ \phantom{(b) Adding further triangles creates a }}
	\label{fig:12a} 

	\end{subfigure}
	\hfill    
	\begin{subfigure}[b]{0.4\textwidth}
		\centering
		
			\begin{tikzpicture}[scale=.8]
			
			\def\r{1.8cm};
			\coordinate (x0) at (142:\r);
			\coordinate (x1) at (214:\r);
			\coordinate (x2) at (286: \r);
			\coordinate (x3) at (358: \r);
			\coordinate (x4) at (70: \r);
			\coordinate (x5) at (106:3.5cm);
			\coordinate (x6) at (178:3.5cm);
			\coordinate (x7) at (250:3.5cm);
			\coordinate (x8) at (322:3.5cm);
			\coordinate (x9) at (34:3.5cm);

			\draw [thick] (x0)--(x1) --(x2) --(x3)--(x4) -- cycle;
			\draw [thick] (x4) -- (x5) -- (x0) -- (x6) -- (x1) -- (x7) -- (x2) -- (x8) -- (x3) -- (x9) -- cycle;
			\draw [thick] (x3) -- (x0) -- (x2);

			\end{tikzpicture}
				\caption{Adding further triangles creates a forest}
				\label{fig:12b} 
				\end{subfigure}    
		\caption{The necessity of $\ccX$ in Theorem~\ref{thm:1522}}	\label{fig:12}
		\vspace{-1em}
	\end{figure} 
In general one needs $|\ccN|-2$ triangles 
for triangulating an $|\ccN|$-gon and, hence, the bound $|\ccX|\le |\ccN|-2$ is optimal
whenever $F$ contains a triangle. If $\gth(F)>g\ge 2$, then the upper bound on $|\ccX|$ can be 
improved to $|\ccX|\le (|\ccN|-2)/(g-1)$ (see Theorem~\ref{thm:6643} below).

We conclude this introduction by discussing some partial results towards the girth 
Ramsey theorem that have been 
obtained over the years. First, for general $k$-uniform hypergraphs 
Theorem~\ref{thm:1522} has been 
proved by Ne{\v{s}}et{\v{r}}il and R\"odl~\cite{NeRo4} for $n=2$ and their 
approach yields Theorem~\ref{thm:grth1} for $g=3$ as well. 
\index{Ne\v{s}et\v{r}il}

Most other partial results deal with the case $k=2$, i.e., with graphs. The main result 
of~\cite{NeRo4} asserts that Theorem~\ref{thm:grth1} holds for $k=2$ and $g=4$ and, 
as Ne{\v{s}}et{\v{r}}il and R\"odl point out, by means of a more elaborate version of 
their argument one can treat every $g\le 7$. However, the new difficulties arising for $g=8$
are fairly overwhelming. In general, it seems that even cycles cause more trouble 
than odd cycles and, in fact, an {\it odd-girth} version of Theorem~\ref{thm:grth1} was 
obtained in~\cite{NR79a}.

For $k=2$ and arbitrary girth R\"odl and Ruci\'nski~\cite{JAMS} proved 
probabilistically that Theorem~\ref{thm:grth1} holds for $F=C_{g+1}$, 
thus answering a question of Erd\H{o}s~\cite{Erd75}. 
Hypergraph extensions of this result follow from the work of Friedgut, 
R\"odl, and Schacht~\cite{FRS10}, and of Conlon and Gowers~\cite{CG16}.
The random graph approach was also used by Haxell, Kohayakawa, and 
\L uczak~\cite{HKL} in order to determine the smallest number of edges 
that a Ramsey graph for~$C_{g+1}$ can have.  

It appears, however, that the usual probabilistic model $G(n, p)$ is not 
suitable for proving any version of the girth Ramsey theorem for arbitrary 
graphs or hypergraphs $F$. We shall thus return to the explicit constructions 
that were developed in the early days of this area. 

\subsection*{Organisation} The next section offers an informal discussion 
of some aspects of the proof of the girth Ramsey theorem. It ends with an 
annotated table of contents, that we hope to be helpful for navigating 
through this article. 
From a logical point of view, this section can be skipped entirely. 
The remaining eleven sections, on the other hand, are, with the exception of a
small number of subsections, necessary for our argument and the discussion in 
Section~\ref{sec:overview} is intended to shed some light on the r\^{o}le they 
will play in due course. 
These exceptional omittable subsections are typically added at the end of 
some sections and have intentionally the same title ``orientation''.  
They deal with summaries of where we 
currently are, where we want to go, and how we plan to arrive there.   \section{Overview}
\label{sec:overview}

It is quite hard to summarise in a few pages how the girth Ramsey theorem is proved, but 
we would like to use this section for pointing to some problems one naturally encounters 
when thinking about it, and to some ideas we use
for solving them. Throughout this informal discussion, we assume some familiarity with 
the partite construction method invented long ago by Ne\v{s}et\v{r}il and the second 
author (see Section~\ref{subsec:PC} for a thorough introduction to this topic).
Among all the partial results we mentioned in Section~\ref{sec:intro} the perhaps deepest one
concerns graphs without four-cycles.

\begin{thm}[R\"{o}dl and Ne\v{s}et\v{r}il]\label{thm:g=4}
	For every graph $F$ with $\gth(F)>4$ and every integer $r\geq 2$ there exists a 
	graph $H$ with $\gth(H)>4$ and $H\lra(F)_r$.
\end{thm}

The proof utilises a strong form of the following fact: for every linear hypergraph $M$ 
(of arbitrary uniformity $k$) and every number of colours $r$ there is a linear 
hypergraph $N$ such that $N\lra (M)_r$. 
In other words, one appeals to the case $g=2$ (and arbitrary $k$) of Theorem~\ref{thm:grth1}.
This suggests that if one wants to continue along those lines, the following firm
decisions ought to be made:
\begin{enumerate}
	\item[$\bullet$] The proof proceeds by induction on $g$.
	\item[$\bullet$] Even if ultimately one should only care about the graph case, 
		one still needs to treat all values of $k$ at the same time. 
\end{enumerate}  
Thus the ``smallest open case'' before this work was the following. 

\begin{prob}
	Extend Theorem~\ref{thm:g=4} to $3$-uniform hypergraphs.
\end{prob}

The solution to this problem is roughly as complicated as the proof 
of Theorem~\ref{thm:grth1} itself and we would like to focus on another aspect
of our proof strategy based on induction on~$g$ first. Suppose 
we have already understood everything about $g=99$ and that we now want to handle
a graph $F$ with $\gth(F)>100$. So we need to construct a graph~$H$ with $\gth(H)>100$
possessing a set of copies $\ccH\subseteq\binom HF$ satisfying e.g.\ $\ccH\lra (F)_2$.
Imagine that some four of the copies in $\ccH$ can be arranged cyclically such that 
any two consecutive copies share a vertex but nothing more 
(see Figure~\ref{fig:N1A}). 
If for each $i\in\ZZ/4\ZZ$ the distance from $q_{i-1}$ to $q_{i}$ within the copy $F_i$
was at most $20$, then we could build a cycle in $H$ through $q_1$, $q_2$, $q_3$,
and $q_4$ whose length would be at most $4\times 20=80$, contrary to $\gth(H)>100$.
We will therefore devote some effort into ensuring that the system $\ccH$ we are about 
to construct contains no four-cycles of copies as in Figure~\ref{fig:N1A}.
As a matter of fact, the existence of Ramsey systems $\ccH$ without these cycles
can be proved as soon as $\gth(F)>4$ and we shall obtain this together with our 
treatment of the case $g=4$. In other words, the strength of the statement we shall 
actually prove by induction on $g$ is somewhere between Theorem~\ref{thm:grth1}
and Theorem~\ref{thm:1522}.  

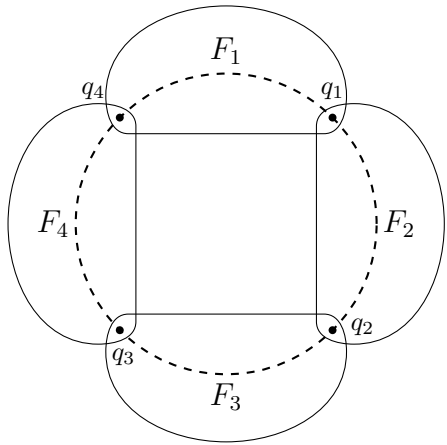
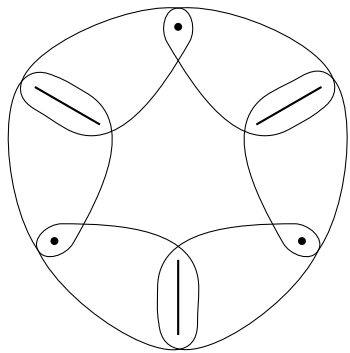
\begin{figure}[ht]
	\centering	
	
		\begin{subfigure}[b]{0.45\textwidth}
		\centering
	
			\begin{tikzpicture}[scale=1]

	\foreach \i in {1,...,4}{
	\coordinate (q\i) at (\i*90-45:2cm);
	\fill (q\i) circle (1.5pt);
}

	\draw [dashed, thick] (0,0) circle (2cm);

\draw  (-1.6,1.7) [out=-90, in=180] to (-1.3,1.2)--(1.3,1.2) [out=0, in = -90] to (1.6, 1.7) [out=90, in = 0] to (0,2.9) [out=180, in = 90] to (-1.6,1.7);

\draw [rotate=90] (-1.6,1.7) [out=-90, in=180] to (-1.3,1.2)--(1.3,1.2) [out=0, in = -90] to (1.6, 1.7) [out=90, in = 0] to (0,2.9) [out=180, in = 90] to (-1.6,1.7);

\draw [rotate=180] (-1.6,1.7) [out=-90, in=180] to (-1.3,1.2)--(1.3,1.2) [out=0, in = -90] to (1.6, 1.7) [out=90, in = 0] to (0,2.9) [out=180, in = 90] to (-1.6,1.7);

\draw [rotate=-90] (-1.6,1.7) [out=-90, in=180] to (-1.3,1.2)--(1.3,1.2) [out=0, in = -90] to (1.6, 1.7) [out=90, in = 0] to (0,2.9) [out=180, in = 90] to (-1.6,1.7);
	
	\node at ($(q1)+(0,.35)$) {\footnotesize $q_1$};
	\node at ($(q2)+(-.35,.35)$) {\footnotesize $q_4$};
	\node at ($(q3)+(.05,-.35)$) {\footnotesize $q_3$};
	\node at ($(q4)+(.4,.05)$) {\footnotesize $q_2$};
	
	\node at (90:2.3cm) {$F_1$};
	\node at (180:2.3cm) {$F_4$};
	\node at (-90:2.3cm) {$F_3$};
	\node at (0:2.3cm) {$F_2$};
			
			\end{tikzpicture}
		 \caption{A four-cycle of copies}
		\label{fig:N1A} 		
	\end{subfigure}
\hfill    
\begin{subfigure}[b]{0.5\textwidth}
\centering
			\begin{tikzpicture}[scale=1]

			\foreach \i in {0,1,2}{
				\coordinate (a\i) at (\i*120+30:1.2cm);
				\coordinate (b\i) at (\i*120+30:2.2cm);
				\coordinate (c\i) at (\i*120-30:1.9);
			}

		\foreach \i in {0,1,2}{
		\draw [rotate=120*\i](0,2.15) [out =180, in= 120] to (-.15,1.7)[out=-60, in=210] to (1.5, .55) [out=30, in=220] to (1.9, .8) [out=35, in=5] to (0, 2.15);
	
		\draw [rotate=120*\i] (0,2.15) [out =0, in= 60] to (.15,1.7) [out=240, in=-30] to (-1.5, .55) [out=150, in=-30] to (-1.9, .8) [out=150, in=175] to (0, 2.15);
	}		

				\foreach \i in {0,1,2}{
				\draw [thick](a\i)--(b\i);
				\fill (c\i) circle (1.5pt);
			}

		\end{tikzpicture}
	\caption{A six-cycle with alternating connectors}
	\label{fig:N1B} 		
\end{subfigure}

				\caption{Two cycles of copies.}
				\label{fig:N1}

	\end{figure} 
For now we just want to say that due to the Ramsey property some copies in the 
systems $\ccH\subseteq \binom HF$
we plan to exhibit need to intersect in entire edges and not just in mere vertices.  
Thus we can also form cycles of copies of a type slightly more general than what we saw 
in Figure~\ref{fig:N1A}. 
That is, we have to allow the so-called {\it connectors} between consecutive 
copies to be either vertices or edges (see Figure~\ref{fig:N1B}). 
An important idea in our analysis of 
cycles of copies, which seems to be new, is that the difficulty of ``excluding'' such cycles 
does not only depend on their {\it length} (i.e., the number of copies they contain), 
but also on the number of times vertex connectors and edge connectors alternate: 
Having many such alternations will turn out to be helpful. 
Thus the cycle drawn in Figure~\ref{fig:N1B} is the ``easiest'' possible cycle of 
length~$6$ and we shall deal with it before approaching the cycles of length~$4$ 
depicted in Figure~\ref{fig:N1A}.
In fact, for general linear hypergraphs cycles of six copies with alternating connectors are
the most complex cycles that could have been handled with existing methods (even though 
apparently this has never been noticed before), while cycles with four vertex connectors 
require some genuinely new ideas. We shall discuss cycles of copies further in 
Section~\ref{subsec:AG}. In particular, we will introduce there our notion of the 
$\Gth$ of a system of copies (spelled with a capital $G$), 
which takes alternations of connectors into account. 
The ensuing Section~\ref{subsec:PCAG} elaborates on the fact that
partite constructions sometimes increase the $\Gth$ of our Ramsey systems of copies. 

\centerline{$*\,\, *\,\, * $}

Let us return to Theorem~\ref{thm:g=4}. The partite construction method essentially
reduces its proof to the bipartite case. Thus the main ingredient is a partite lemma preserving the girth assumption, i.e., a statement of the following form:
\begin{equation}
  \tag{$\star$}\label{eq:PLG4}
  \parbox{\dimexpr\linewidth-6em}{    \strut
    \it
    For every bipartite graph $B$ with $\gth(B)>4$ and every integer $r\geq 2$ 
	there is a bipartite graph $B_\star$ with $\gth(B_\star)>4$ and $B_\star\lra(B)_r$.    \strut
  }
\end{equation}
It is the proof of this statement where the Ramsey theorem for linear hypergraphs
is required. Denoting the vertex classes of $B$ by $X$ and $Y$ we may assume that 
all vertices in~$X$ have the same degree $d\ge 2$. (If this is not the case already,
we attach some pendant edges to the vertices in $X$). Now $B$ can be viewed as a
union of edge-disjoint stars $K_{1, d}$ whose centres are in $X$. This so-called 
{\it star decomposition} of $B$ gives rise to a $d$-uniform hypergraph $F$ on $Y$ whose 
edges correspond to the non-central vertices of those stars, i.e.,
\begin{equation}\label{eq:stdec}
	F=(Y,\{N(x)\colon x\in X\})\,.
\end{equation}
The assumed absence of $4$-cycles in $B$ translates into the linearity of $F$.

It does not help much to apply the Ramsey theorem for linear hypergraphs directly 
to~$F$ itself. Rather, one first constructs an auxiliary, linear $k$-uniform hypergraph $G$, 
where $k=(d-1)r+1$ has the property that for every $r$-colouring of a $k$-set
there is a monochromatic $d$-set (Schubfachprinzip). Moreover, $G$ is constructed together
with some linear order of its vertex set and to render this notationally visible 
we shall write~$G_<$ instead of $G$. We now apply an ordered version of the induction 
hypothesis\footnote{Keeping track of vertex orderings will often be important in the sequel;
but these orderings never complicate the proofs, so we do not treat them carefully in this
outline.} to~$G_<$ with $r^k$ colours, thus obtaining a linear, ordered, $k$-uniform 
hypergraph $H_<$ such that $(H_<)\lra (G_<)_{r^k}$. Going back to bipartite graphs 
we take for every edge $e$ of $H$ a new vertex $v_e$ and join it to all members of $e$.
This yields a bipartite graph $B_\star$ with vertex 
classes $\{v_e\colon e\in E(H_<)\}$ and~$V(H_<)$. Since $H_<$ is linear, $B_\star$ 
is $C_4$-free and it turns out that~$B_\star\lra (B)_r$ can be guaranteed by an 
appropriate choice of~$G_<$. 
Roughly, this is because every colouring $\gamma\colon E(B_\star)\lra [r]$
associates with every edge $e\in E(H_<)$ one of $r^k$ possible colour patterns, namely 
the $k$-tuple consisting of the colours of the $k$ edges from $v_e$ to $e$. 
Owing to the construction of $H_<$ there is some copy $G^\star_<$ of $G_<$ all of whose 
edges receive the same colour pattern. By our choice of $k$ this colour pattern contains 
some colour $\rho_\star\in [r]$ at least $d$ times and an appropriate construction 
of $G_<$ ensures that the copy $G^\star_<$ we have just found contains a monochromatic 
copy of $F$ (whose colour is $\rho_\star$). This is all we want to 
say about the proof of the partite lemma at this juncture. An abstract version of the
construction which leads us from $B$ via $F$, $G_<$, and $H_<$ to $B^\star$, called the 
{\it extension process}, will be discussed in Section~\ref{subsec:EP}.

\centerline{$*\,\, *\,\, * $}

There is one further subtlety in the proof of Theorem~\ref{thm:g=4} we would like 
to emphasise here. The question is why or under what circumstances a partite construction 
based on the above partite lemma does not create four-cycles. The worry one might 
have is whether forbidden cycles can arise when amalgamating graphs 
(previous ``pictures'') over a bipartite graph (as in the partite construction). 
For instance in Figure~\ref{fig:N2} we see two copies $B_1$, $B_2$
of the bipartite graph we subjected to the partite lemma, and two vertices $x$, $y$
belonging to both of them. If at the previous stage of the construction $x$ and $y$ 
had common neighbours $u$, $v$, then altogether a four-cycle $x-u-y-v$ arises.  

\begin{figure}[ht]
	\centering

			\begin{tikzpicture}[scale=1]
			
			\coordinate (x) at (-.5, 0);
			\coordinate (y) at (.5,0);
			\coordinate (z) at (0, 1.5);
			\coordinate (u) at (-2.5, 2.5);
			\coordinate (v) at (2.5,2.5);
			
			\foreach \i in {x,y,z,u,v}{
			\fill (\i) circle (1.5pt);}
		
			\draw [blue, thick]  (-3.5,.07)[out=95, in=-105] to (-3.36,1.43)--(.48,1.43) [out=-75, in = 85] to (.68,.07)--cycle;
			
			\fill [blue!40, opacity=.5]  (-3.5,.05)[out=95, in=-105] to (-3.36,1.45)--(.48,1.45) [out=-75, in = 85] to (.68,.05)--cycle;
			
			\draw [blue, thick]  (3.5,.1)[out=85, in=-75] to (3.37,1.4)--(-.48,1.4) [out=-105, in = 95] to (-.68,.1)--cycle;
			
			\fill [blue!40, opacity=.5]  (3.5,.1)[out=85, in=-75] to (3.36,1.4)--(-.48,1.4) [out=-105, in = 95] to (-.68,.1)--cycle;

		\node [below] at (x) {$x$};
		\node [below] at (y) {$y$};
		\node [above] at (z) {$z$};
		\node [above] at (u) {$u$};
		\node [above] at (v){$v$};
		
		\draw [thick](-4.5,0) -- (4.5,0);
			\draw [thick](-4.5,1.5) -- (4.5,1.5);
			\draw (x)--(u)--(y)--(v)--(x);
			
			\draw [dashed] (x)--(z)--(y);
			
			\draw [red!80!black, thick]  (-1.5, 3.5) [out=180, in=170] to (-2.7, -1) [out = -10, in = 190] to (-.1, -1) [out = 10, in = 0] to (-1.5, 3.5);
			
				\draw [red!80!black, thick]  (1.5, 3.5) [out=180, in=170] to (.1, -1) [out = -10, in = 190] to (2.7, -1) [out = 10, in = 0] to (1.5, 3.5);
			
			\node [blue] at (-4,.8) {$B_1$};
				\node [blue] at (4,.8) {$B_2$};
			
		\end{tikzpicture}

				\caption{Four-cycles in amalgamations}
				\label{fig:N2}

	\end{figure}
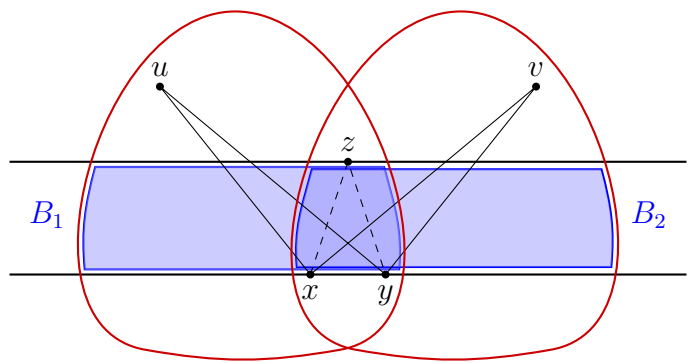 
This problem was addressed in~\cite{NeRo4} by working with a very strong form of the 
Ramsey theorem for linear hypergraphs, which has the effect that we may assume 
$B_1$, $B_2$ to intersect in a star. More precisely, the partite lemma actually
comes together with a system of copies $\ccB\subseteq \binom{B_\star}{B}$ such 
that $\ccB\lra (B)_r$ and any two distinct copies $B_1$, $B_2$ are either disjoint, 
or they intersect in a single vertex, or they intersect in a star. 
Thus the situation in Figure~\ref{fig:N2} requires the existence of a common neighbour $z$
of $x$, $y$ belonging to the intersection of $B_1$ and $B_2$. 
If $z\ne u$ then $x-u-y-z$ is a four-cycle in the left copy of the previous picture,
which is absurd. Similarly, if $z=u$, then the cycle $x-u-y-v$ is entirely contained 
in the right copy of the previous picture.   
There are some other potential cases of four-cycles after the amalgamation that one 
needs to exclude before declaring Theorem~\ref{thm:g=4} to be proved, but we shall 
not go into such details here. 

\centerline{$*\,\,*\,\,*$}

When attempting to extend these ideas to $3$-uniform hypergraphs
one actually does not need a partite lemma applicable to {\it all}
tripartite $3$-uniform 
hypergraphs without $4$-cycles. This is due to the fact that when executing the 
decisive partite construction one can ensure that all $3$-partite hypergraphs one has to 
handle are highly structured. Indeed, the $3$-partite hypergraphs the initial picture 
is composed of are just matchings. Moreover, using standard techniques one can ensure 
that the next picture is composed of $3$-partite hypergraphs built from many copies 
of such matchings by gluing them together along single vertices belonging to the same 
vertex class (see Figure~\ref{fig:N3A}). 
When one proceeds to the next picture, many copies of such hypergraphs get glued 
together along single vertices sitting on another common vertex class 
(see Figure~\ref{fig:N3B}), and so it goes on. Later we shall call partite
hypergraphs that can arise in this manner {\it trains} and we shall prove a partite
lemma for trains by induction on their {\it height}, i.e., the number of times the 
gluing process needs to be iterated in their formation. 

\begin{figure}[ht]
	\centering	
	
		\begin{subfigure}[b]{0.39\textwidth}
		\centering
	
			\begin{tikzpicture}[scale=.55]
	
\draw (-4.5,1) -- (4.5,1);
\draw (-4.5,0)--(4.5,0);
\draw (-4.5,-1)--(4.5,-1);
\phantom{\draw [ultra thick](-4,-1.35)--(3,1.35);}

\draw [red!80!black, thick ] (0,1)--(0,-1);
\draw [red!80!black, thick ] (1,1)--(1,-1);
\draw [red!80!black, thick ] (-1,1)--(-1,-1);

\draw [red!80!black, thick ] (-1,1)--(-2,-1);
\draw [red!80!black, thick ] (-2,1)--(-3,-1);
\draw [red!80!black, thick ] (-3,1)--(-4,-1);

\draw [red!80!black, thick ] (1,1)--(2,-1);
\draw [red!80!black, thick ] (2,1)--(3,-1);
\draw [red!80!black, thick ] (3,1)--(4,-1);

\draw [rounded corners, green!40!black, ultra thick] (-1.2, 1.15)--(-1.2,-1.15)--(1.2,-1.15)--(1.2,1.15)--cycle;

\draw [rounded corners, green!40!black,ultra thick] (-3.15, 1.15)--(-4.25,-1.15)--(-1.85,-1.15)--(-0.75,1.15)--cycle;

\draw [rounded corners, green!40!black, ultra thick] (3.15, 1.15)--(4.25,-1.15)--(1.85,-1.15)--(0.75,1.15)--cycle;

			\end{tikzpicture}
		 \caption{}
		\label{fig:N3A} 		
	\end{subfigure}
\hfill    
\begin{subfigure}[b]{0.6\textwidth}
\centering
			\begin{tikzpicture}[scale=.55]

			\draw (-8.5,1) -- (8.5,1);
			\draw (-8.5,0)--(8.5,0);
			\draw (-8.5,-1)--(8.5,-1);
			
	\draw [red!80!black, thick ] (.5,1)--(.5,-1);
	\draw [red!80!black, thick ] (-.5,1)--(-.5,-1);		
	\draw [red!80!black, thick ] (-.5,1)--(-1.5,-1);
	\draw [red!80!black, thick ] (-1.5,1)--(-2.5,-1);		
	\draw [red!80!black, thick ] (.5,1)--(1.5,-1);
	\draw [red!80!black, thick ] (1.5,1)--(2.5,-1);
			
\draw [shift=({5,0}), red!80!black, thick ] (.5,1)--(.5,-1);
\draw [shift=({5,0}), red!80!black, thick ] (-.5,1)--(-.5,-1);		
\draw [shift=({5,0}), red!80!black, thick ] (-.5,1)--(-1.5,-1);
\draw [shift=({5,0}), red!80!black, thick ] (-1.5,1)--(-2.5,-1);		
\draw [shift=({5,0}), red!80!black, thick ] (.5,1)--(1.5,-1);
\draw [shift=({5,0}), red!80!black, thick ] (1.5,1)--(2.5,-1);

\draw [shift=({-5,0}), red!80!black, thick ] (.5,1)--(.5,-1);
\draw [shift=({-5,0}), red!80!black, thick ] (-.5,1)--(-.5,-1);		
\draw [shift=({-5,0}), red!80!black, thick ] (-.5,1)--(-1.5,-1);
\draw [shift=({-5,0}), red!80!black, thick ] (-1.5,1)--(-2.5,-1);		
\draw [shift=({-5,0}), red!80!black, thick ] (.5,1)--(1.5,-1);
\draw [shift=({-5,0}), red!80!black, thick ] (1.5,1)--(2.5,-1);

		\draw [rounded corners, green!40!black, ultra thick] (-.75, 1.15)--(-.75,-1.15)--(.75,-1.15)--(.75,1.15)--cycle;
		\draw [rounded corners, green!40!black,ultra thick] (-1.75, 1.15)--(-2.85,-1.15)--(-1.2,-1.15)--(-0.25,1.15)--cycle;		
		\draw [rounded corners, green!40!black, ultra thick] (1.75, 1.15)--(2.85,-1.15)--(1.2,-1.15)--(0.25,1.15)--cycle;
			
		\draw [shift=({5,0}), rounded corners, green!40!black, ultra thick] (-.75, 1.15)--(-.75,-1.15)--(.75,-1.15)--(.75,1.15)--cycle;
		\draw [shift=({5,0}), rounded corners, green!40!black,ultra thick] (-1.75, 1.15)--(-2.85,-1.15)--(-1.2,-1.15)--(-0.25,1.15)--cycle;		
		\draw [shift=({5,0}), rounded corners, green!40!black, ultra thick] (1.75, 1.15)--(2.85,-1.15)--(1.2,-1.15)--(0.25,1.15)--cycle;
			
		\draw [shift=({-5,0}), rounded corners, green!40!black, ultra thick] (-.75, 1.15)--(-.75,-1.15)--(.75,-1.15)--(.75,1.15)--cycle;
		\draw [shift=({-5,0}), rounded corners, green!40!black,ultra thick] (-1.75, 1.15)--(-2.85,-1.15)--(-1.2,-1.15)--(-0.25,1.15)--cycle;		
		\draw [shift=({-5,0}), rounded corners, green!40!black, ultra thick] (1.75, 1.15)--(2.85,-1.15)--(1.2,-1.15)--(0.25,1.15)--cycle;	
			
			\draw [rounded corners, violet!80!blue, ultra thick] (-1.9,1.35)--(-3.2,-1.35)--(3.2,-1.35)--(1.9,1.35)--cycle;
			\draw [shift=({5,0}), rounded corners, violet!80!blue, ultra thick] (-1.9,1.35)--(-3.2,-1.35)--(3.2,-1.35)--(1.9,1.35)--cycle;
			\draw [shift=({-5,0}), rounded corners, violet!80!blue, ultra thick] (-1.9,1.35)--(-3.2,-1.35)--(3.2,-1.35)--(1.9,1.35)--cycle;
			
		\end{tikzpicture}
	\caption{}
	\label{fig:N3B} 		
\end{subfigure}

				\caption{Two 3-uniform trains}
				\label{fig:N3}

	\end{figure}
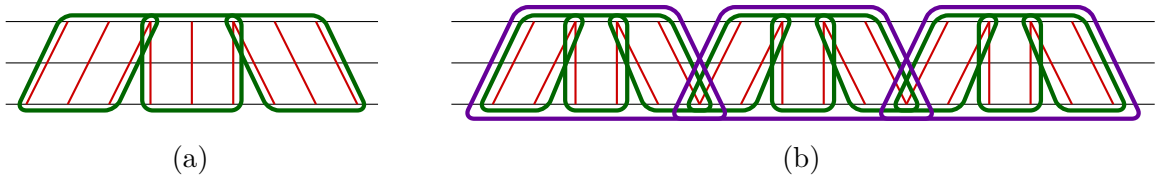 
For reasons of comparison it should perhaps be mentioned that in the $2$-uniform 
case the distinction between bipartite graphs with train structure and general $C_4$-free
bipartite graphs does not exist. This is because the star decomposition we used earlier
(recall~\eqref{eq:stdec})
exemplifies that all $C_4$-free bipartite graphs are trains. 
In the $3$-uniform case, however, the difference is quite pronounced and ``most'' tripartite
hypergraphs without four-cycles fail to admit any train structure. 
   
Dealing with trains gets quickly somewhat technical, but 
fortunately one can analyse a great portion 
of the general case by just looking at much simpler structures, 
which we call {\it pretrains}.
These are hypergraphs equipped with an equivalence relation on their set of edges.
The corresponding equivalence classes will be called the {\it wagons} of the pretrain.
For instance, in the proof of the partite lemma for $C_4$-free bipartite graphs one can 
regard the given bipartite graph $F$ as a pretrain, whose wagons are stars 
with centres in $X$. 
   
It turns out that pretrains are the natural structures for describing the 
extension process the proof of the partite lemma~\eqref{eq:PLG4} was based on. 
The intersections of copies in stars we saw earlier will generalise to intersections 
in entire wagons. It will be a nontrivial, yet important problem to improve this to 
copies intersecting at most in edges. Of course one hopes to accomplish this by means 
of a further partite construction and, therefore, we shall study the behaviour of 
pretrains in such constructions in Section~\ref{subsec:PCEP}. 

In an attempt to point to something essential we are still missing at this 
moment we would like to formulate a concrete test problem. We shall write pretrains 
in the form~$(F, \equiv^F)$, where $F$ is a hypergraph and $\equiv^F$ refers to an 
equivalence relation on $E(F)$. We say that $(F, \equiv^F)$ is an {\it induced 
subpretrain} of another pretrain $(H, \equiv^H)$ if $F$ is an induced subhypergraph 
of $H$ and, moreover, any two edges of $F$ are $\equiv^F$-equivalent if and only if they 
are $\equiv^H$-equivalent. Given a pretrain $(F, \equiv^F)$ and a number of colours $r$
we will be interested in constructing a pretrain $(H, \equiv^H)$ and a 
system $\ccH$ of copies of $(F, \equiv^F)$ in $(H, \equiv^H)$ such that, 
in an obvious sense, the partition relation $\ccH\lra (F, \equiv^F)$ holds.
There are several known methods for obtaining such systems $\ccH$ (see e.g., 
Lemma~\ref{lem:0059} or Proposition~\ref{prop:ofpt}).  

But now assume, in addition, that the given pretrain $(F, \equiv^F)$ is {\it linear} in the 
sense that 
\begin{enumerate}
	\item[$\bullet$] the hypergraph $F$ is linear 
	\item[$\bullet$] and any two wagons intersect in at most one vertex. 
\end{enumerate}
We would like to find $(H, \equiv^H)$ and $\ccH$ as above such that, moreover, 
\begin{enumerate}
	\item[$\bullet$] the pretrain $(H, \equiv^H)$ is linear in the same sense 
	\item[$\bullet$] and any two copies in $\ccH$ intersect at most in an edge. 
\end{enumerate}

Partite constructions are very good at meeting the second requirement, but they appear 
to have difficulties to maintain the linearity of the pretrains. In fact, the only way 
we know for achieving the linearity of $(H, \equiv^H)$ is to use the extension process 
instead, but this process is incapable of producing systems of copies intersecting in 
less than a wagon. 

This conundrum will be resolved by the introduction of a new concept, 
called German $\GTH$. It turns out that this generalises the $\Gth$ of 
systems of copies we studied earlier (before delving into pretrains) and 
interacts very well with partite constructions. $\GTH$ itself will be introduced 
and studied in Section~\ref{subsec:EPAG}. Two of its most essential properties 
will then be demonstrated in Section~\ref{sec:1912}. First, the extension process 
translates ordinary $\Gth$ properties of hypergraph constructions into the 
corresponding $\GTH$ properties of pretrain constructions. Second, $\GTH$ 
properties are indestructible by partite constructions. 

We will then have all the ingredients necessary for proving a strong form
of Theorem~\ref{thm:grth1} in Sections~\ref{sec:trains}\,--\,\ref{subsec:0136}, 
though it is still technically challenging to combine all these pieces together
(cf.\ Theorem~\ref{thm:6653}). 
Deducing Theorem~\ref{thm:1522} and Theorem~\ref{cor:19} from this result will 
then be comparatively routine (see Section~\ref{sec:paradise}). 

\medskip

\centerline{\sc Annotated table of contents} 

\medskip

\centerline{\hyperref[subsec:PC]{\bf Section 3. The partite construction method}}

\noindent\hyperref[sssec:HJ]{\bf \S3.1 --- Partite lemmata} 

\smallskip

\hfill
\begin{minipage}{0.9\textwidth}
	[We introduce the Hales-Jewett construction $\HJ_r(F)=(H, \ccH)$, 
	where $F$, $H$ are $k$-partite {$k$-uniform} hypergraphs and $\ccH\lra (F)_r$.]
\end{minipage}

\smallskip

\noindent\hyperref[sssec:pict]{\bf \S3.2 --- Pictures}

	\smallskip

	\hfill
	\begin{minipage}{0.9\textwidth}
		[We define pictures $(\Pi, \ccP, \psi_\Pi)$ over systems $(G, \ccG)$.
		Then we explain picture zero and partite amalgamations.]
	\end{minipage}
	
	\smallskip

\noindent\hyperref[sssec:irt]{\bf \S3.3 --- The induced Ramsey theorem}

	\smallskip

	\hfill
	\begin{minipage}{0.9\textwidth}
		[We illustrate the partite construction method by proving the induced Ramsey
		theorem for hypergraphs and introduce the notation $\PC(\Phi, \Xi)$.]
	\end{minipage}
	
	\smallskip

\noindent\hyperref[sssec:vindya]{\bf \S3.4 --- Strong inducedness}

	\smallskip

	\hfill
	\begin{minipage}{0.9\textwidth}
		[We define strong inducedness and show that the partite lemma $\HJ$ delivers
		strongly induced copies. Moreover, we define systems with clean intersections 
		and prove that the clean partite lemma $\CPL=\PC(\HJ, \HJ)$ as well as the 
		construction $\Omega^{(2)}=\PC(\Rms, \CPL)$ produce systems of strongly induced 
		copies with clean intersections.]
	\end{minipage}
	
	\smallskip

\noindent\hyperref[sssec:ord]{\bf \S3.5 --- Ordered constructions}

	\smallskip

	\hfill
	\begin{minipage}{0.9\textwidth}
		[If $\Phi$ is an ordered Ramsey construction and $\Xi$ is a partite lemma, 
		then $\PC(\Phi, \Xi)$ is ordered.]
	\end{minipage}
	
	\smallskip

\noindent\hyperref[sssec:fpart]{\bf \S3.6 --- $f$-partite hypergraphs}

	\smallskip

	\hfill
	\begin{minipage}{0.9\textwidth}
		[Similarly, if $\Phi$ is $f$-partite, then so is $\PC(\Phi, \Xi)$.]
	\end{minipage}
	
	\smallskip

\noindent\hyperref[sssec:lin]{\bf \S3.7 --- Linearity}

	\smallskip

	\hfill
	\begin{minipage}{0.9\textwidth}
		[The partite lemmata $\HJ$, $\CPL$, and the construction $\Omega^{(2)}$ are linear.
		We also give two sufficient conditions for $\PC(\Phi, \Xi)$ to be linear.]
	\end{minipage}

\smallskip

\noindent\hyperref[subsec:2346]{\bf \S3.8 --- $A$-intersecting hypergraphs}

	\smallskip

	\hfill
	\begin{minipage}{0.9\textwidth}
		[An $f$-partite hypergraph is $A$-intersecting for a subset $A$ of its index set,
		if the intersections of distinct edges of $F$ are contained in vertex classes 
		whose indices are in $A$. The constructions $\HJ$, $\CPL$, and $\Omega^{(2)}$
		preserve this property.]
	\end{minipage}

	\bigskip\goodbreak

\centerline{\hyperref[subsec:AG]{\bf Section 4. Girth considerations}}

\noindent\hyperref[sssec:GSS]{\bf \S4.1 --- Set systems and girth}

	\smallskip

	\hfill
	\begin{minipage}{0.9\textwidth}
		[We recapitulate standard facts on the classical girth concept.]
	\end{minipage}

	\smallskip

\noindent\hyperref[sssec:coc]{\bf \S4.2 --- Cycles of copies}

	\smallskip

	\hfill
	\begin{minipage}{0.9\textwidth}
		[We introduce several central notions: cycles of copies, tidiness, master copies,
		and, finally, the $\Gth$ of a linear system of hypergraphs.]
	\end{minipage}

	\smallskip

\noindent\hyperref[sssec:157]{\bf \S4.3 --- Semitidiness}

	\smallskip

	\hfill
	\begin{minipage}{0.9\textwidth}
		[We study a relaxation of tidiness leading to an equivalent 
		definition of $\Gth$.]
	\end{minipage}

	\smallskip

\noindent\hyperref[sssec:9902]{\bf \S4.4 --- Orientation}

	\smallskip

	\hfill
	\begin{minipage}{0.9\textwidth}
		[We promise the existence of a Ramsey construction $\Omega^{(g)}$ applicable 
		to ordered $f$-partite hypergraphs $F$ with $\gth(F)>g$ that generates 
		systems $(H, \ccH)$ with $\Gth(H, \ccH^+)>g$.]
	\end{minipage}

	\bigskip
	
\centerline{\hyperref[subsec:PCAG]{\bf Section 5. Girth in partite constructions}}

\noindent\hyperref[sssec:1510]{\bf \S5.1 --- From $(g, g)$ to $g$}

	\smallskip

	\hfill
	\begin{minipage}{0.9\textwidth}
		[We study constructions applicable to hypergraphs $F$ with $\gth(F)>g$. 
		If a partite lemma yields linear systems $(H, \ccH)$ with $\Gth(H, \ccH^+)>(g, g)$,
		then by means of a partite construction we can improve $(g, g)$ to $g$.]
	\end{minipage}
	
	\smallskip
	
\noindent\hyperref[sssec:1514]{\bf \S5.2 --- From $g-1$ to $(g, g)$}

	\smallskip

	\hfill
	\begin{minipage}{0.9\textwidth}
		[Similarly, if for $g\ge 3$ a partite lemma yields systems $(H, \ccH)$ 
		with $\gth(H)>g$ and ${\Gth(H, \ccH^+)>g-1}$, 
		then we can strengthen $g-1$ to $(g, g)$.]
	\end{minipage}
	
	\bigskip
	
\centerline{\hyperref[subsec:EP]{\bf Section 6. The extension process}}

\noindent\hyperref[sssec:1128]{\bf \S6.1 --- Pretrains}

	\smallskip

	\hfill
	\begin{minipage}{0.9\textwidth}
		[We introduce systems of pretrains.]
	\end{minipage}
	
	\smallskip
	
\noindent\hyperref[sssec:1448]{\bf \S6.2 --- Extensions}

	\smallskip

	\hfill
	\begin{minipage}{0.9\textwidth}
		[We define extensions of pretrains and the notation $(F, \equiv^F)\ltimes (X, W)$.]
	\end{minipage}
	
	\smallskip
	
\noindent\hyperref[{sssec:lp}]{\bf \S6.3 --- Linear pretrains}

	\smallskip

	\hfill
	\begin{minipage}{0.9\textwidth}
		[Given linear ordered constructions $\Phi$, $\Psi$ for hypergraphs we 
		define the construction 
		$\Ext(\Phi, \Psi)$, which is applicable to linear ordered pretrains.]
	\end{minipage}

	\bigskip
	
\centerline{\hyperref[subsec:PCEP]{\bf Section 7. Pretrains in partite constructions}}

\noindent\hyperref[sssec:311]{\bf \S7.1 --- A partite lemma for pretrains}

	\smallskip

	\hfill
	\begin{minipage}{0.9\textwidth}
		[The partite lemma $\HJ$ applies to pretrains.]
	\end{minipage}
	
	\smallskip
	
\noindent\hyperref[sssec:2340]{\bf \S7.2 --- Pretrain pictures}

	\smallskip

	\hfill
	\begin{minipage}{0.9\textwidth}
		[We introduce pretrain pictures $(\Pi, \equiv, \ccP, \psi)$.]
	\end{minipage}
	
	\smallskip\goodbreak
	
\noindent\hyperref[sssec:0021]{\bf \S7.3 --- Partite amalgamations}

	\smallskip

	\hfill
	\begin{minipage}{0.9\textwidth}
		[We discuss partite amalgamations of pretrain pictures.]
	\end{minipage}
	
	\smallskip
	
\noindent\hyperref[sssec:314]{\bf \S7.4 --- Proof of Proposition~\ref{prop:ofpt}}

	\smallskip

	\hfill
	\begin{minipage}{0.9\textwidth}
		[If $\Phi$ is a Ramsey construction for hypergraphs and $\Xi$ denotes a 
		partite lemma for pretrains, 
		then $\PC(\Phi, \Xi)$ applies to pretrains.]
	\end{minipage}
	
	\smallskip
	
\noindent\hyperref[subsec:7o]{\bf \S7.5 --- Orientation}

	\smallskip

	\hfill
	\begin{minipage}{0.9\textwidth}
		[We ask for a construction that given a linear pretrain and a number of colours
		yields a linear Ramsey system of pretrains with clean intersections.]
	\end{minipage}	
	
\bigskip
			
\centerline{\hyperref[subsec:EPAG]{\bf Section 8. Basic properties
of \texorpdfstring{$\GTH$}{Girth}}}

\noindent\hyperref[sssec:1657]{\bf \S8.1 --- Basic concepts}

	\smallskip

	\hfill
	\begin{minipage}{0.9\textwidth}
		[We introduce the paraphernalia of German $\GTH$: big cycles, acceptability, 
		and supreme copies.]
	\end{minipage}
	
	\smallskip
	
\noindent\hyperref[sssec:1758]{\bf \S8.2 --- Special cases}

	\smallskip

	\hfill
	\begin{minipage}{0.9\textwidth}
		[We characterise $\GTH$ when there are only edge copies. If the wagons of $\equiv$ 
		are single edges, then $\GTH$ is essentially the same as $\Gth$.]
	\end{minipage}
	
	\smallskip
	
\noindent\hyperref[sssec:1815]{\bf \S8.3 --- Two further facts}

	\smallskip

	\hfill
	\begin{minipage}{0.9\textwidth}
		[We generalise some results on $\Gth$ to $\GTH$.]
	\end{minipage}
	
	\bigskip

\centerline{\hyperref[{sec:1912}]{\bf Section 9. $\GTH$ in constructions}}

\noindent\hyperref[subsec:ExtL]{\bf \S9.1 --- The extension lemma}

	\smallskip

	\hfill
	\begin{minipage}{0.9\textwidth}
		[If $\Phi$ yields strongly induced copies, then $\Gth$ properties of $\Psi$
		translate into $\GTH$ properties of $\Ext(\Phi, \Psi)$.]
	\end{minipage}
	
	\smallskip
	
\noindent\hyperref[subsec:GTHpres]{\bf \S9.2 --- $\GTH$ preservation}

	\smallskip

	\hfill
	\begin{minipage}{0.9\textwidth}
		[If $\Phi$ yields strongly induced copies, then $\PC(\Phi, \Xi)$ inherits $\GTH$ 
		properties from the partite lemma $\Xi$.]
	\end{minipage}

\smallskip
	
\noindent\hyperref[subsec:9o]{\bf \S9.3 --- Orientation}

	\smallskip

	\hfill
	\begin{minipage}{0.9\textwidth}
		[We solve the problem from~\S\ref{subsec:7o} using German $\GTH$.]
	\end{minipage}	
		
	\bigskip

\centerline{\hyperref[{sec:trains}]{\bf Section 10. Trains}}

\noindent\hyperref[sssec:3133]{\bf \S10.1 --- Quasitrains and parameters}

	\smallskip

	\hfill
	\begin{minipage}{0.9\textwidth}
		[We define quasitrains (hypergraphs equipped with a nested sequence of 
		equivalence relations on their sets of edges) and trains.]
	\end{minipage}
	
	\smallskip
	
\noindent\hyperref[sssec:3135]{\bf \S10.2 --- More German girth}

	\smallskip

	\hfill
	\begin{minipage}{0.9\textwidth}
		[We look at $\ggth$ and $\GTH$ notions applicable 
		to (systems of) quasitrains.]
	\end{minipage}
	
	\smallskip
			
\noindent\hyperref[sssec:Karo]{\bf \S10.3 --- Diamonds}

	\smallskip

	\hfill
	\begin{minipage}{0.9\textwidth}
		[We formulate a very strong Ramsey theoretic principle for trains
		and announce some implications, thus 
		explaining the inductive structure of the proof of the girth Ramsey theorem.]
	\end{minipage}
	
	\bigskip
		
\centerline{\hyperref[{sec:TUD}]{\bf Section 11. Trains in the extension process}}

\noindent\hyperref[sssec:3136]{\bf \S11.1 --- Extensions of trains}

	\smallskip

	\hfill
	\begin{minipage}{0.9\textwidth}
		[Preparing a generalisation of the extension process to quasitrains
		we define and study so-called $1$-extensions of quasitrains.]	
	\end{minipage}
	
	\smallskip

\noindent\hyperref[subsec:0135]{\bf \S11.2 --- A generalised extension lemma}

\smallskip

\hfill
\begin{minipage}{0.9\textwidth}
[Assuming $\karo_{\seq{g}}$ we describe a construction $\ups$ applicable to
trains of height $|\seq{g}|+1$.] 
\end{minipage} 

\bigskip

\centerline{\hyperref[{subsec:0136}]{\bf Section 12. Trains in partite constructions}}

\noindent\hyperref[sssec:4137]{\bf \S12.1 --- Quasitrain constructions}

	\smallskip

	\hfill
	\begin{minipage}{0.9\textwidth}
		[If $\Phi$ is a Ramsey construction for hypergraphs and $\Xi$ denotes a 
		partite lemma for quasitrains, then $\PC(\Phi, \Xi)$ is a Ramsey construction 
		for quasitrains.]
	\end{minipage}
	
\smallskip
	
\noindent\hyperref[subsec:tipc]{\bf \S12.2 --- Train amalgamations}

	\smallskip

	\hfill
	\begin{minipage}{0.9\textwidth}
		[When we want to handle trains (as opposed to mere quasitrains), we will 
		always start with a train construction $\Phi$. 
		We prove a general amalgamation lemma for train pictures.]
	\end{minipage}	
	
	\smallskip

\noindent\hyperref[subsec:napl]{\bf \S12.3 --- Amenable partite lemmata}

\smallskip

\hfill
\begin{minipage}{0.9\textwidth}
[We decompose the problem of proving $\karo_{\seq{g}}$ into two simpler tasks.] 
\end{minipage}

\smallskip

\noindent\hyperref[subsec:ngr]{\bf \S12.4 --- Girth resurrection}

\smallskip

\hfill
\begin{minipage}{0.9\textwidth}
[Continuing~\S\ref{subsec:0135} we show that $\PC(\ups, \PC(\ups, \ups))$ 
exemplifies~$\karo_{(2)\circ\seq{g}}$.] 
\end{minipage}

\smallskip

\noindent\hyperref[sssec:1746]{\bf \S12.5 --- Revisability}

\smallskip

\hfill
\begin{minipage}{0.9\textwidth}
[We convert $\karo_{(g)^m\circ\seq{g}}$ into a partite lemma applicable to certain
trains $\seq{F}$ with $\ggth(\seq{F})>(g+1)\circ\seq{g}$ that we call $m$-revisable.] 
\end{minipage} 

\smallskip

\noindent\hyperref[sssec:1748]{\bf \S12.6 --- Train constituents}

\smallskip

\hfill
\begin{minipage}{0.9\textwidth}
[Performing a diagonal partite construction we complete the proof of 
Proposition~\ref{prop:0139}.] 
\end{minipage} 

\bigskip

\centerline{\hyperref[sec:paradise]{\bf Section 13. Paradise}}

\noindent\hyperref[subsec:6142]{\bf \S13.1 --- Forests of copies}

\smallskip

\hfill
\begin{minipage}{0.9\textwidth}
[We relate $\Gth$ to forests of copies by showing that $\Gth(H, \ccN^+)>|\ccN|$ 
implies~$\ccN$ to be a forest of copies (see Lemma~\ref{lem:6822}).] 
\end{minipage} 

\smallskip

\noindent\hyperref[subsec:2142]{\bf \S13.2 --- The final partite construction}

\smallskip

\hfill
\begin{minipage}{0.9\textwidth}
[We prove a generalisation of Theorem~\ref{thm:1522}, stating that if 
$\gth(F)>g$, then the upper bound on the number $|\ccX|$ of additional copies can be 
improved to $\frac{|\ccN|-2}{g-1}$.
Finally, we deduce Theorem~\ref{cor:19} (on the local structure of Ramsey hypergraphs).] 
\end{minipage}  \section{The partite construction method} 
\label{subsec:PC}

A common point of departure for many results in structural Ramsey theory is 
Ramsey's theorem~\cite{Ramsey30} \index{Ramsey's theorem} asserting that given 
natural numbers $k$, $r$, and~$m$ there exists a natural number $M$ for which 
the partition relation 
\[
	M\longrightarrow (m)^k_r
\]
holds. Let us recall that this means: no matter how the edges of a complete 
$k$-uniform hypergraph $K^{(k)}_M$ on $M$ vertices get coloured with $r$ colours,
there will always exist a monochromatic $K^{(k)}_m$. Thus if a $k$-uniform 
hypergraph $F$ and a number of colours $r$ are given, we may apply Ramsey's 
theorem to $m=v(F)$ and obtain a large clique $H=K^{(k)}_M$ arrowing $F$
in the sense that the collection $\ccH=\binom{H}{F}_\nni$ of subhypergraphs of~$H$ 
isomorphic to $F$ has the partition property
\[
	  \ccH\lra(F)_r\,,
\]
where ``$\nni$'' abbreviates ``not necessarily induced''. 
The defect of non-induced copies can be remedied by an
iterative amalgamation method called the {\it partite construction}. Introduced by 
Ne\v{s}et\v{r}il and R{\"o}dl in~\cite{NeRo3a}, this method 
has since then been utilised in a variety of different contexts 
(see e.g.~\cites{BNRR, HN19, LR06, NeRo1, NeRo4}). 
\index{Ne\v{s}et\v{r}il}
The present work heavily relies on the partite construction method as well and 
we shall therefore provide an overview in the remainder of this section.  

\subsection{Partite lemmata} 
\label{sssec:HJ}

The $k$-partite $k$-uniform hypergraphs $F$ occurring in this article will always be
accompanied by $k$-element {\it index sets} $I$ and by fixed partitions 
\[
	V(F)=\bigdcup_{i\in I} V_i(F)
\]
of their vertex sets such that every 
edge $e\in E(F)$ intersects each {\it vertex class}~$ V_i(F)$ in exactly one vertex. 
The most natural 
choice for $I$ is, of course, the set~$[k]$. Later on, however, when we shall be
executing partite constructions, we encounter $k$-partite hypergraphs whose vertex 
classes are already indexed 
differently and in order to avoid excessive relabelling it seems advisable to 
allow general index sets $I$ from the very beginning.

Given two $k$-partite $k$-uniform hypergraphs $F$ and $H$ with the same index set $I$ 
we say that $F$ is a {\it partite subhypergraph} of $H$ if $F$ is a subhypergraph 
of $H$ in the ordinary sense and, moreover, $V_i(F)\subseteq V_i(H)$ holds for 
every $i\in I$. The collection of all {\it partite copies} of~$F$ in~$H$ is denoted 
by $\binom{H}{F}^\pt$.

Suppose now that a $k$-partite $k$-uniform hypergraph $F$ with index set $I$ as well as 
a number of colours $r$ are given. A {\it partite lemma} \index{partite lemma}
is a construction delivering a $k$-partite $k$-uniform hypergraph $H$ with the same index 
set $I$ as well as a system of copies $\ccH\subseteq\binom{H}{F}^\pt$ 
such that $\ccH\longrightarrow (F)_r$.

For the existence of such hypergraphs $H$ we could simply refer to 
the literature. But on several later occasions we need to check that a particular 
construction based on the Hales-Jewett theorem~\cites{HJ63} 
\index{Hales-Jewett theorem} has certain 
additional properties (see~\cite{Sh329} for a beautiful alternative 
proof of the Hales-Jewett theorem).
For this reason, we briefly sketch the so-called {\it Hales-Jewett construction}. 
\index{Hales-Jewett construction}

In the degenerate case, where $F$ has no edges, we simply set $H=F$ and $\ccH=\{F\}$. 
Otherwise we appeal to the Hales-Jewett theorem and obtain a natural number $n$  
such that every $r$-colouring of $E(F)^n$ contains a monochromatic 
combinatorial line. 
The vertex classes of $H$ are defined by $V_i(H)=V_i(F)^n$ 
for every $i\in I$. 
The set~$E(H)$ is constructed together with a {\it canonical bijection} 
$\lambda\colon E(F)^n\longrightarrow E(H)$. 
Given an $n$-tuple $(e_1, \ldots, e_n)$ of edges of~$F$, we write $x_{\nu i}$ for the 
common vertex of $e_\nu$ and~$V_i(F)$, where~$\nu\in [n]$ and $i\in I$, and we define 
the edge $e=\lambda(e_1, \ldots, e_n)$ by demanding for every~$i\in I$ that $e$ 
intersects~$V_i(H)$ in its vertex $(x_{1i}, \ldots, x_{ni})$.

It is well known (and follows from Lemma~\ref{lem:hj-str} below) that every 
combinatorial line $L\subseteq E(F)^n$
gives rise to an induced subhypergraph $F_L$ of $H$ which is isomorphic to $F$ 
and satisfies $E(F_L)=\{\lambda(\seq{e})\colon \seq{e}\in L\}$. Therefore 
\[
	\ccH=\bigl\{
			F_L\colon L \text{ is a combinatorial line in } E(F)^n
		\bigr\}
\]
is a system of copies of $F$ in $H$ and, owing to our choice of $n$ involving 
the Hales-Jewett theorem, it has the desired property 
\[
	\ccH\longrightarrow (F)_r\,.
\]

To facilitate later references to this construction we set $\HJ_r(F)=(H, \ccH)$, 
where $\HJ$ abbreviates ``Hales-Jewett''.

\subsection{Pictures} 
\label{sssec:pict}

Suppose that $F$ and $G$ are two $k$-uniform hypergraphs and 
that $\ccG\subseteq\binom{G}{F}_\nni$ is a system of subhypergraphs of $G$, which 
are isomorphic to $F$ but not necessarily induced. A {\it picture over $(G, \ccG)$}
\index{picture}
is a triple $(\Pi, \ccP, \psi)$ consisting of a $k$-uniform hypergraph $\Pi$, 
a system~$\ccP\subseteq \binom{\Pi}{F}$ of induced copies of $F$ in $\Pi$, 
and a hypergraph homomorphism $\psi\colon \Pi\longrightarrow G$ mapping the 
copies in $\ccP$ onto copies in $\ccG$ (see Figure~\ref{fig:24}). 
Occasionally we shall encounter pictures carrying an additional 
structure compatible with additional structure that might be present on $G$. 

\begin{figure}[ht]
	\centering
	
	\begin{tikzpicture}[scale=.9]

	\draw (0,3.2) -- (0,-3.2);
	\node at (-.7, 3.2) {\large $G$};
	
	\draw (2.5,2)--(6.5,2);
	
	\draw [rounded corners=20] (2,3) rectangle (7,-3);
	
	\draw [<-, thick] (.5,0)--(1.8,0);
	
	\node at (1.2, .3) {$\psi$};
	
	\node at (4.5, 2.3) {$\psi^{-1}(x)$};
	
	\node at (7.5,3) {\large $\Pi$};
	
	\node at (6,.5) [blue!80!black] {\LARGE $\ccP$};
	
	\newcommand{\edge}[3]{
	\draw [rotate around={#3:(#1,#2)}] (#1,#2) ellipse (.15 cm and .5cm);
	\fill [rotate around={#3:(#1,#2)}, color=blue!50!white, opacity=.2] (#1,#2) ellipse (.15 cm and .5cm);
	}
	
	\edge{4}{.5}{-40}
	\edge{3.5}{.5}{35}
	\edge{3.75}{-.2}{0}
	\edge{3.55}{1.15}{-40}
	
	\edge{3.5}{-1.8}{-35}
	\edge{4}{-1.8}{35}
	\edge{4.5}{-1.8}{-35}
	
	\foreach \y in {1.15, .5, -.15, -1.8}{
	\draw [thick] (0,\y) ellipse (.15 cm and .45cm);
	}

	\draw [decorate,
	decoration = {brace, mirror}] (-.4,1.6)--(-.4,-2.25);
	\fill (0,2) circle (2pt);
	\node at (-.3, 2) {$x$};
	
	\node at (-.7,-.3) {$\ccG$};
	
	\end{tikzpicture}

	\caption{Picture $(\Pi, \ccP, \psi)$ over $(G, \ccG)$.}
	\label{fig:24} 
\end{figure} 
Given a picture $(\Pi, \ccP, \psi)$ we call the preimages of the vertices of $G$
under $\psi$ {\it music lines} \index{music line} and put 
\[
	V_x=V_x(\Pi)=\{v\in V(\Pi)\colon \psi(v)=x\}
\]
for every $x\in V(G)$.
Observe that $\{V_x\colon x\in V(G)\}$ partitions $V(\Pi)$, while the edges of~$\Pi$
cross each music line at most once. When visualising a picture we usually draw the music 
lines {\it horizontally} \index{horizontal} and we imagine $(G, \ccG)$ to be drawn 
{\it vertically} \index{vertical} next to 
the picture. Associated with every edge $e\in E(G)$ we have its {\it constituent}~$\Pi^e$,
\index{constituent} which is the $k$-uniform $k$-partite hypergraph with index set $e$ 
and vertex partition $\{V_x\colon x\in e\}$ whose set of edges is defined by 
\[
	E(\Pi^e)=\{f\in E(\Pi)\colon \psi[f]=e\}\,.
\]

A partite construction over $(G, \ccG)$ always commences with the 
corresponding {\it picture zero} (see Figure~\ref{fig:21}). \index{picture zero}
This is a picture $(\Pi_0, \ccP_0, \psi_{0})$
having for each copy $F_\star\in \ccG$ a unique copy $\wh{F}_\star$ projected 
to $F_\star$ by $\psi_0$. It is required that 
\begin{enumerate}
	\item[$\bullet$] for distinct copies $F_\star, F_{\star\star}\in \ccG$ 
		the corresponding copies $\wh{F}_\star, \wh{F}_{\star\star}\in \ccP$
		be vertex-disjoint 
	\item[$\bullet$] and that every vertex or edge of $\Pi_0$ belong to 
		one of the copies 	$\wh{F}_\star\in \ccP$.
\end{enumerate}     
Evidently such a picture zero can always be constructed and, in fact, it is uniquely 
determined up to isomorphism; it has $|\ccG|\cdot |V(F)|$ vertices and $|\ccG|\cdot |E(F)|$
edges. 

\begin{figure}[ht]
	\centering
	
	\begin{tikzpicture}[scale=.8]
	
	\foreach \i in {1,..., 5}{
		\draw[thick] (0,\i)--(15,\i);
		\foreach \j in {1,...,20}{
			\def \k{\j*.71};
			\coordinate (x\j\i) at (\k,\i);
					}
	}

\foreach \i/\j/\k/\l/\o/\p in {1/4/2/5/2/3, 3/4/4/5/4/2, 5/4/6/5/6/1, 7/3/8/5/8/2, 
9/3/10/5/10/1, 11/2/12/5/12/1, 13/3/14/4/14/2, 15/3/16/4/16/1, 17/2/18/4/18/1, 
19/2/20/3/20/1}{ 
	\draw [green!50!black, thick] (x\i\j) -- (x\k\l) -- (x\o\p) -- cycle;
	\fill [green!50!black] (x\i\j) circle (2pt);
	\fill [green!50!black] (x\k\l) circle (2pt);
	\fill [green!50!black] (x\o\p) circle (2pt);
}	
	\end{tikzpicture}

	\caption{Picture zero over $(K_5, \ccG)$, where $\ccG = \binom{K_5}{K_3}$.}
	\label{fig:21} 
\end{figure} 
Now suppose that $\Pi^e$ is a constituent of some picture $(\Pi, \ccP, \psi_\Pi)$
and that, in addition, we have a $k$-partite $k$-uniform hypergraph $H$ with index set $e$ 
together with a system of induced
partite copies $\ccH\subseteq\binom{H}{\Pi^e}^\pt$. We may then define a new picture 
$(\Sigma, \ccQ, \psi_\Sigma)$ by means of the following {\it partite amalgamation} 
(see Figure~\ref{fig:22}). \index{partite amalgamation}
Start with $\Sigma^e=H$ and extend each copy $\Pi_\star^e\in \ccH$ to its own 
copy $\Pi_\star$ of $\Pi$. 
These extensions are to be performed as disjointly as possible, so for any two 
distinct copies $\Pi_\star^e, \Pi_{\star\star}^e\in \ccH$ it is demanded that 
their extensions 
$\Pi_\star$, $\Pi_{\star\star}$ satisfy
$V(\Pi_\star)\cap V(\Pi_{\star\star})=V(\Pi_\star^e)\cap V(\Pi_{\star\star}^e)$.
From each of these extended copies $\Pi_\star$ there is a canonical isomorphism to
the original hypergraph~$\Pi$, which allows us to pull~$\ccP$ and~$\psi_\Pi$ 
back onto~$\Pi_\star$.
Thereby we obtain a collection of pictures 
$\bigl\{(\Pi_\star, \ccP_\star, \psi_{\Pi_\star})\colon \Pi_\star^e\in \ccH\bigr\}$ 
each of which is isomorphic to $(\Pi, \ccP, \psi_\Pi)$. The new 
picture $(\Sigma, \ccQ, \psi_\Sigma)$ is defined to be their union with~$H$. Explicitly
\allowdisplaybreaks
\begin{alignat*}{2}
	V(\Sigma)&=\bigcup_{\Pi^e_\star\in \ccH} V(\Pi_\star)\cup V(H)\,, & \quad \quad\quad
	E(\Sigma)&=\bigcup_{\Pi^e_\star\in \ccH} E(\Pi_\star)\cup E(H)\,, \\ 
		\ccQ &= \bigcup_{\Pi^e_\star\in \ccH} \ccP_\star\,, 
	\qquad \text{and}  &
	\psi_\Sigma&=\bigcup_{\Pi^e_\star\in \ccH} \psi_{\Pi_\star}\cup \psi_H\,,
\end{alignat*}
where, as expected, $\psi_H\colon V(H)\lra e$ is the map sending for each $x\in e$
all vertices in~$V_x(H)$ to $x$.
One checks immediately that the copies in $\ccQ$ are induced, 
that $\psi_\Sigma$ is a hypergraph homomorphism from $\Sigma$ to $G$,
and, finally, that the copies in $\ccQ$ project appropriately. Therefore
$(\Sigma, \ccQ, \psi_\Sigma)$ is indeed a picture.  
The copies~$\Pi_\star$ of $\Pi$ obtained by extending the copies $\Pi^e_\star$ in $\ccH$
are called the {\it standard copies} of $\Pi$ in~$\Sigma$. 
\index{standard copy (in partite constructions)}
In the sequel we shall indicate the partite amalgamation just explained by 
writing 
\[
	(\Sigma, \ccQ, \psi_\Sigma)=(\Pi, \ccP, \psi_\Pi)\conc (H, \ccH)\,.
\]

\begin{figure}[ht]
	\centering
	
	\begin{tikzpicture}[scale=.8]
	
	\def\w{1.7};
	\def\h{4};
	
	\foreach \x in {-4.8,-1.6,1.6,4.8}{
	
		\draw[black,thick] (-\x,0) ellipse (\w cm and \h cm);
}
	\draw (-8,4.5) -- (-8,-4.5);
	
		\foreach \i/\j in {-6.37/-3.2, -3.2/-.005,.005/3.2, 3.2/6.37}{
		\fill [red!90!white, opacity = .2] (\j,1.5) [out =-82, in=82] to (\j,-1.5)--(\i,-1.5)[out =98, in=-98] to (\i,1.5)-- cycle;
	}
	
	\draw (-7.2,1.5) rectangle (7.2,-1.5);
	
	\foreach \i/\j in {-6.37/-3.2, -3.2/-.005,.005/3.2, 3.2/6.37}{
	\draw [red!75!black, ultra thick] (\j,1.5) [out =-82, in=82] to (\j,-1.5)--(\i,-1.5)[out =98, in=-98] to (\i,1.5)-- cycle; 
}

	\foreach \i in {-.4,.4,-1.2,1.2}{
		\draw (-7,\i)--(7,\i);
		\fill (-8,\i) circle (3pt);
	}

	\draw[black,thick] (-8,0) ellipse (.3 cm and 1.6 cm);
	
	\node at (-8.5,.8) {$e$};
	\node at (-8,-5) {$(G, \ccG)$};

	\foreach \i in {-.4,.4,1.2,2,2.7}{	
		\foreach \j in {-5.5,4.2,-2.3,1}{
			\fill (\j,\i) [blue!80!black] circle (3pt);
			\draw [blue!80!black, thick] (\j,-.4)--(\j,2.7);
		}
	}

\foreach \i in {-.4,.4,-1.2,-2,-2.7}{	
	\foreach \j in {-4.2,-1,5.5,2.3}{
		\fill (\j,\i) [blue!80!black] circle (3pt);
		\draw [blue!80!black, thick] (\j,.4)--(\j,-2.7);
	}
}

\foreach \i in {-2.7,2.7, -2,2}{
	\draw (-5.9,\i) -- (-3.8,\i);
	\draw (-2.7,\i) -- (-.6,\i); 
	\draw (5.9,\i) -- (3.8,\i);
	\draw (2.7,\i) -- (.6,\i); 
}

\node at (-.8,5.1) [blue!75!black]{Copies of $F$};

\node at (7.3,-2){$(H,\ccH)$};

\foreach \i/\j in {-5.3/-2.2, 4/.5} 
	\draw [blue!80!black, ultra thick, <-] (\i,3) -- (\j,4.8);
\foreach \i/\j in {-2.2/-1.1, .8/-.5} 
	\draw [blue!80!black, ultra thick, <-] (\i,3) -- (\j,4.8);

	\node at (0,-5.3) {Standard copies of $(\Pi, \ccP, \psi_\Pi)$};
	
	\foreach \i/\j in {-4/-3, 4/3} 
	\draw [ultra thick, <-] (\i,-3.8) -- (\j,-4.8);
	\foreach \i/\j in {-1.2/-1, 1.2/1} 
	\draw [ultra thick, <-] (\i,-4) -- (\j,-4.8);

	\end{tikzpicture}

	\caption{Partite amalgamation $(\Pi, \ccP, \psi_\Pi)\conc (H, \ccH)$}
	\label{fig:22} 
\end{figure} 
\subsection{The induced Ramsey theorem}
\label{sssec:irt}

To perform a partite construction, one needs a Ramsey theoretic result that gets 
applied vertically, as well as a partite lemma that one iteratively utilises horizontally. 
As an illustration of the method, we brief\-ly describe one of the proofs of 
the {\it induced Ramsey theorem} from~\cite{NeRo5}. \index{induced Ramsey theorem}

Given a $k$-uniform 
hypergraph $F$ as well as a number of colours $r$ we intend to construct a hypergraph
$H$ as well as a system of induced copies $\ccH\subseteq \binom{H}{F}$ such that 
$\ccH\longrightarrow (F)_r$. We begin by taking any pair $(G, \ccG)$ consisting 
of a $k$-uniform hypergraph $G$ and a system of not necessarily induced copies
$\ccG\subseteq\binom{G}{F}_\nni$ satisfying 
\begin{equation} \label{eq:213a}
	\ccG\longrightarrow (F)_r\,.
\end{equation}

For instance, by Ramsey's theorem we could take $G$ to be a sufficiently large clique 
and set $\ccG=\binom{G}{F}_\nni$. The precise choice of~$(G, \ccG)$ is, however, 
quite irrelevant and all that matters is the partition relation~\eqref{eq:213a}.       

Let $\{e(1), \ldots, e(N)\}$ enumerate the edges of $G$ in any order and 
perform the following recursive construction of pictures 
$(\Pi_\alpha, \ccP_\alpha, \psi_\alpha)_{0\le \alpha\le N}$ over $(G, \ccG)$. 
Begin with picture zero $(\Pi_0, \ccP_0, \psi_0)$.
If for some positive $\alpha\le N$ the picture 
$(\Pi_{\alpha-1}, \ccP_{\alpha-1}, \psi_{\alpha-1})$ has just been constructed, 
we apply any partite lemma to its constituent $\Pi_{\alpha-1}^{e(\alpha)}$, 
thus getting a pair~$(H_\alpha, \ccH_\alpha)$ with
\begin{equation} \label{eq:213b}
	\ccH_\alpha\subseteq \binom{H_\alpha}{\Pi^{e(\alpha)}_{\alpha-1}}^\pt
	\quad \text{ and } \quad
	\ccH_\alpha\longrightarrow \bigl(\Pi_{\alpha-1}^{e(\alpha)}\bigr)_r\,.
\end{equation}
For instance, at this moment we may just employ the Hales-Jewett construction 
from~\S\ref{sssec:HJ} and set
\[
	(H_\alpha, \ccH_\alpha)=\HJ_r\bigl(\Pi_{\alpha-1}^{e(\alpha)}\bigr)\,,
\]
but any other choice validating~\eqref{eq:213b} is equally fine.
Having selected $(H_\alpha, \ccH_\alpha)$ we define the next picture by
\[
	(\Pi_\alpha, \ccP_\alpha, \psi_\alpha)
	=
	(\Pi_{\alpha-1}, \ccP_{\alpha-1},\psi_{\alpha-1})
	\conc
	(H_\alpha, \ccH_\alpha)\,.
\]

It is known that the final picture $(\Pi_N, \ccP_N, \psi_N)$ has the property
\begin{equation}\label{eq:5544}
	\ccP_N\longrightarrow (F)_r\,,
\end{equation}
i.e., that the system $(\Pi_N, \ccP_N)$ is as requested by the induced 
Ramsey theorem. Indeed, an easy proof by induction on $\alpha$ using~\eqref{eq:213b}
in the induction step shows that if for $0\le \alpha\le N$ the edges of $\Pi_\alpha$
get coloured with $r$ colours, then there is a copy $(\Pi^\star_0, \ccP^\star_0)$ 
of picture zero with the property that for every $\beta\in [\alpha]$ the constituent 
of $\Pi^\star_0$ 
over $e(\beta)$ is monochromatic. In particular, if one colours the edges of the final 
picture with $r$ colours, then there exists a copy of picture zero all of whose 
constituents are monochromatic. The colours of these constituents project to an edge 
colouring of the vertical hypergraph~$G$ and~\eqref{eq:213a} leads to a monochromatic copy of $F$ in $\ccP_N$. Thereby~\eqref{eq:5544} is proved,

As the proof of the girth Ramsey theorem involves a great number of nested 
applications of the partite construction method, it will safe a considerable 
amount of space to establish an appropriate terminology.

Recall that in \S\ref{sssec:HJ}
we denoted the Hales-Jewett construction proving the partite lemma by $\HJ$. 
Similarly, if we appeal to Ramsey's theorem to obtain a system 
$(G, \ccG)$ with ${\ccG\longrightarrow (F)_r}$ whose copies not necessarily induced,  
then we shall write 
\[
	\Rms_r(F)=(G, \ccG)\,. 
\]
Explicitly, this means that $G$ is a 
sufficiently large clique whose size depends on $v(F)$ and~$r$, 
and that $\ccG=\binom{G}{F}_\nni$. So both $\Rms$ and $\HJ$ are examples 
of Ramsey theoretic {\it constructions}. 

Pairs $(H, \ccH)$ consisting of a hypergraph $H$ and a set $\ccH$ of 
subhyergraphs of $H$ will be called {\it systems of hypergraphs}
\index{system of hypergraphs}
and the members of $\ccH$ will be referred to as copies. 
It is neither required that the copies 
of a system be mutually isomorphic nor that they be induced. 

In general, if a {\it construction} $\Phi$ is applied to a hypergraph $F$ and a number 
of colours~$r$, then it delivers a system of hypergraphs $\Phi_r(F)=(H, \ccH)$ 
with ${\ccH\subseteq\binom{H}{F}_\nni}$ and $\ccH\longrightarrow (F)_r$. 
\index{construction}
Not every construction is applicable to every hypergraph. For instance,~$\HJ_r(F)$
is only defined if $F$ is a $k$-partite $k$-uniform hypergraph for some $k\ge 2$.
It will be convenient to call any construction having this property and which, 
moreover, delivers systems of induced partite copies, a {\it partite 
lemma}.   
\index{partite lemma}

The partite construction method is, strictly speaking, not a construction in the 
sense of the previous paragraph, but rather an {\it operation} capable of producing 
a new construction from two given ones. 
More precisely, for a construction~$\Phi$ and a partite lemma $\Xi$ the construction 
$\PC(\Phi, \Xi)$ is defined as follows. Given a hypergraph~$F$ as well a number
of colours $r$, one sets $(G, \ccG)=\Phi_r(F)$ and generates a sequence of pictures
$(\Pi_\alpha, \ccP_\alpha, \psi_\alpha)_{0\le \alpha\le N}$ over $(G, \ccG)$ in 
the same way as above but ensuring the partition relation~\eqref{eq:213b}
by means of the partite lemma $\Xi$, i.e., by setting 
\[
	(H_\alpha, \ccH_\alpha)=\Xi_r\bigl(\Pi_{\alpha-1}^{e(\alpha)}\bigr)\,.
\]
Finally one defines 
\[
	\PC(\Phi, \Xi)_r(F)=(\Pi_N, \ccP_N)\,,
\]
where $N=e(G)$ denotes the index of the final picture. For instance, the construction $\PC(\Rms, \HJ)$
produces systems of hypergraphs that verify the induced Ramsey theorem.  

Let us emphasise at this moment that it is not always easy to foresee
the class of hypergraphs a construction of the form $\PC(\Phi, \Xi)$ 
is defined on. For instance, let $\Lambda$ be a {\it linear partite lemma}, 
i.e., a partite lemma that is only applicable to linear $k$-partite $k$-uniform
hypergraphs and delivers linear systems of $k$-partite $k$-uniform hypergraphs. 
\index{linear partite lemma}
One might then come up with the idea of using the construction $\PC(\Rms, \Lambda)$ 
in order to obtain a Ramsey theorem for linear hypergraphs, i.e., the case $g=2$
of Theorem~\ref{thm:grth1}. The main problem with this approach is that at this 
level of generality it is not completely clear whether the constituents of our pictures 
stay linear throughout the construction. There are two known ways of addressing 
this difficulty, one of which is explained in \S\ref{sssec:lin} later. 

\subsection{Strong inducedness} 
\label{sssec:vindya}
 
Let us explore some properties of constructions derivable from~$\HJ$ and $\Rms$
by means of the partite construction method. 
 
\begin{dfn}
	Given a hypergraph $G$ and a subhypergraph $F$ of $G$ we say that $F$ is 
	{\it strongly induced} \index{strongly induced}
	in $G$ and write $F\Str G$ if for every edge $e\in E(G)$
	there exists an edge $f\in E(F)$ with $e\cap V(F)\subseteq f$.
\end{dfn}

Evidently every strongly induced subhypergraph is, in particular, induced. 
One also checks easily that every hypergraph is a strongly induced subhypergraph 
of itself and that passing to a strongly induced subhypergraph is 
transitive, i.e., that $F\Str G\Str H$ implies $F\Str H$. 

\begin{lemma} \label{lem:hj-str}
	If $F$ is a $k$-partite $k$-uniform hypergraph, $r\in\NN$, 
	and $\HJ_r(F)=(H, \ccH)$, then every $F_\star\in\ccH$ is strongly induced in $H$. 
\end{lemma}

\begin{proof}
	In the trivial case $E(F)=\vn$ we have $H=F=F_\star$ and the result is clear. 
	Now suppose $E(F)\ne\vn$ and let $n$ be the Hales-Jewett exponent entering
	the construction of~$(H, \ccH)$. As usual, $F$ and $H$ have index set $I$, 
	and $\lambda$ denotes the canonical bijection
	from $E(F)^n$ onto $E(H)$.  
	
	Let $\eta\colon E(F)\longrightarrow E(F)^n$ be the combinatorial
	embedding with $E(F_\star)=(\lambda\circ\eta)[E(F)]$. This means that there
	are a partition $[n]=C\dcup M$ of the set of {\it coordinates} into a set~$C$ of
	{\it constant coordinates} and a nonempty set $M$ of {\it moving coordinates} 
	as well as a map $\wt{\eta}\colon C\lra E(F)$ such that for every $f\in E(F)$
	and every $\nu\in[n]$ the $\nu^{\mathrm{th}}$ coordinate of $\eta(f)$ is the edge
		\begin{enumerate}
		\item[$\bullet$] $\wt{\eta}(\nu)$ if $\nu\in C$
		\item[$\bullet$] and $f$ if $\nu\in M$.
	\end{enumerate}
	\index{constant coordinate}
	\index{moving coordinate}
		
	Associated with $\eta$ we have bijections $\eta_i\colon V_i(F)\lra V_i(F_\star)$
	for all indices $i\in I$. Explicitly, if $x\in V_i(F)$, 
	then $\eta_i(x)=(x_1, \ldots, x_n)$ satisfies $x_\nu=x$ for all $\nu\in M$, while
	for $\nu\in C$ the vertex $x_\nu$ is, independently of $x$, the vertex of 
	$\wt{\eta}(\nu)$ belonging to $V_i(F)$. Observe that if $x\in V_i(F)$
	and $f\in E(F)$ are incident, then so are $\eta_i(x)$ and $(\lambda\circ\eta)(f)$.
	
	Now let $e=\lambda(f_1, \ldots, f_n)$ be an arbitrary edge of $H$. We are to 
	exhibit an edge~$e'$ of~$F_\star$ with $e\cap V(F_\star)\subseteq e'$. To this 
	end we fix a moving coordinate $\nu(\star)\in M$ and 
	define $e'=(\lambda\circ\eta)(f_{\nu(\star)})$. For any $i\in I$ let 
	$\seq{x}$ be a vertex in $e\cap V_i(F_\star)$. We need to prove $\seq{x}\in e'$. 
	Owing to $\seq{x}\in V_i(F_\star)$ there exists a vertex $x\in V_i(F)$ with 
	$\seq{x}=\eta_i(x)$. Looking at the projection of $\seq{x}\in e$ 
	to the $\nu(\star)^{\mathrm {th}}$ coordinate
	we see $x\in f_{\nu(\star)}$ and thus we have indeed
	$\seq{x}=\eta_i(x)\in (\lambda\circ\eta)(f_{\nu(\star)})=e'$.
\end{proof}

The possible intersections of copies in $(H, \ccH)=\HJ_r(F)$ are
somewhat complicated (and were explicitly described in~\cite{NeRo4}). 
We circumvent this issue by moving on to the construction 
$\CPL=\PC(\HJ, \HJ)$ called the {\it clean partite lemma} (see Corollary~\ref{cor:CPL}).
\index{clean partite lemma} 

\begin{dfn}
	A system of copies $(H, \ccH)$ is said to have {\it clean intersections}
	if for any two distinct copies $F_\star, F_{\star\star}\in \ccH$ there 
	exist edges $e_\star\in E(F_\star)$, $e_{\star\star}\in E(F_{\star\star})$
	with \index{clean intersection}
		\[
		V(F_\star)\cap V(F_{\star\star})=e_\star\cap e_{\star\star}\,.
	\]
\end{dfn}

We proceed with a general result saying that copies with clean intersections 
arise automatically when one performs a partite construction  
employing strongly induced copies vertically. (For historical reasons we refer 
to~\cite{BNRR}*{Lemma~2.12}). 

 \begin{lemma} \label{lem:cleancap}
	If $\Phi$ is a Ramsey construction delivering strongly induced copies,
	then for any partite lemma $\Xi$ the construction $\PC(\Phi, \Xi)$
	delivers systems of strongly induced copies whose intersections are clean.  
\end{lemma}

\begin{proof}
	Let a hypergraph $F$ as well as a number of colours $r$ be given and 
	set~$\Phi_r(F)=(G, \ccG)$. Due to the hypothesis on $\Phi$ the copies in $\ccG$
	are strongly induced. When performing the partite construction $\PC(\Phi, \Xi)_r(F)$ 
	we eventually reach a last picture $(\Pi_N, \ccP_N, \psi_N)$. 
	
	Let us show first that every copy $F_\star\in \ccP_N$ is strongly induced in $\Pi_N$. 
	Given an arbitrary edge $e\in E(\Pi_N)$ we need to find an edge 
	$f_\star\in E(F_\star)$ with 
		\begin{equation} \label{eq:214a}
		V(F_\star)\cap e\subseteq f_\star\,.
	\end{equation}
		The projection $\psi_N$ sends $F_\star$ and $e$ to a copy $F'\in\ccG$ and an 
	edge $e'\in E(G)$. As a consequence of $F'\Str G$ there exists an edge $f'\in E(F')$ 
	with $V(F')\cap e'\subseteq f'$. 
	Now the edge $f_\star\in E(F_\star)$ corresponding to $f'$ satisfies~\eqref{eq:214a},
	since $\psi_N$ establishes a bijection between~$V(F_\star)$ and~$V(F')$.  
	
	It remains to show that the copies in $\ccP_N$ have clean intersections. 
	To this end we argue by {\it induction along the partite construction}. 
	\index{induction along partite construction}
	That is we prove 
	inductively that each of the successively constructed pictures has a 
	system of copies with clean intersections. This is clear for picture zero,
	for disjoint copies have an empty intersection which is, in particular, clean.
	
	For the induction step it suffices to show that if $(G, \ccG)$ is a system 
	with strongly induced copies, $(\Pi, \ccP, \psi_\Pi)$ is a picture over $(G, \ccG)$
	whose copies have clean intersections, $e\in E(G)$, $\Xi_r(\Pi^e)=(H, \ccH)$, and 
	$(\Sigma, \ccQ, \psi_\Sigma)=(\Pi, \ccP, \psi_\Pi)\conc (H, \ccH)$, then 
	the copies in $\ccQ$ have clean intersections as well. 
	
	Consider any two distinct copies 
	$F_\star, F_{\star\star}\in \ccQ$. Recall that each of them belongs to a 
	standard copy of $\Pi$ in $\Sigma$. Let $\Pi_\star$ and $\Pi_{\star\star}$ 
	be such standard copies. If they are equal, then the induction hypothesis 
	shows that the intersection of $F_\star$ and $F_{\star\star}$ is clean.
	
	So suppose $\Pi_\star\ne \Pi_{\star\star}$ from now on. The projection 
	$\psi_\Sigma$ sends $F_\star$ and $F_{\star\star}$ to certain copies 
	$F'$ and $F''$ in $\ccG$. Exploiting that these copies are strongly 
	induced we obtain edges $e'\in E(F')$ and $e''\in E(F'')$
 	satisfying $V(F')\cap e\subseteq e'$ and $V(F'')\cap e\subseteq e''$.
	
	Let $e_\star\in E(F_\star)$ be the inverse image of $e'$ under the projection 
	from $F_\star$ to $F'$ and define $e_{\star\star}\in E(F_{\star\star})$
	similarly with respect to $F''$. Owing to 
		\[
		V(F_\star)\cap V(F_{\star\star})
		\subseteq 
		V(\Pi_\star)\cap V(\Pi_{\star\star})
		\subseteq 
		V(\Pi^e)
	\]
		we have $V(F_\star)\cap V(F_{\star\star})=e_\star\cap e_{\star\star}$,
	meaning that $e_\star$ and $e_{\star\star}$ witness that the intersection of $F_\star$
	and $F_{\star\star}$ is clean. 
\end{proof}

\begin{cor} \label{cor:CPL}
	The clean partite lemma $\CPL=\PC(\HJ, \HJ)$ delivers systems of strongly induced
	copies with clean intersections. 
\end{cor}

\begin{proof}
	Let a $k$-partite $k$-uniform hypergraph $F$ with index set $I$ and a number
	of colours $r$ be given. Recall that in order to construct $\CPL_r(F)$ we 
	first need to construct $(G, \ccG)=\HJ_r(F)$ and then we need to construct
	a sequence of pictures over $(G, \ccG)$ using the partite lemma~$\HJ_r(\cdot)$
	in every step. Let us denote the final picture by $(\Pi, \ccP, \psi_\Pi)$.
	
	By Lemma~\ref{lem:hj-str} the copies of $F$ in $\ccG$ are strongly induced in $G$
	and thus, by Lemma~\ref{lem:cleancap}, the copies in $\ccP$ are strongly induced and
	have clean intersections. 
	It remains to explain why~$\CPL$ can be regarded as 
	a partite lemma. Recall from \S\ref{sssec:HJ} that $G$ is again a $k$-partite 
	$k$-uniform hypergraph with index set $I$. This means that the map 
	$\psi_G\colon V(G)\lra I$ with $\psi_G^{-1}(i)=V_i(G)$ for every $i\in I$ is 
	a hypergraph homomorphism from $G$ to $I^+=(I, \{I\})$. Consequently 
	$\psi_G\circ\psi_{\Pi}\colon V(\Pi)\lra I$ is a hypergraph homomorphism 
	from $\Pi$ to $I^+$, which shows that $\Pi$ is indeed of the desired form. 
	Since the copies in $\ccP$ project via $\psi_\Pi$ to copies 
	in $\ccG\subseteq \binom{G}{F}^\pt$,
	it is also clear that, in an obvious sense, $\ccP\subseteq \binom{\Pi}{F}^\pt$.     
\end{proof}

Now it is natural to wonder whether we can gain anything by cleaning further.
For instance, one may consider the partite lemma $\CPL^{(2)}=\PC(\CPL, \CPL)$
and ask whether it has any desirable properties going beyond clean intersections.
An affirmative answer to this question is obtained in Section~\ref{subsec:PCAG} 
below. However, even the higher iterates such
as $\CPL^{(3)}=\PC(\CPL^{(2)}, \CPL^{(2)})$ are insufficient for proving the girth
Ramsey theorem due to the simple reason that when applied to the graph $C_6$ they 
always yield graphs containing a $4$-cycle.  

Our next result asserts that clean intersections delivered by a partite lemma are 
indestructible under further applications of the partite construction method. 
The Propositions~\ref{prop:girthclean} and~\ref{prop:2235} proved later vastly 
generalise this fact.  

\begin{lemma} \label{lem:clean-preserve}
	If $\Phi$ denotes an arbitrary Ramsey construction and $\Xi$ is a partite lemma 
	producing systems of strongly induced copies with clean intersections, 
	then $\PC(\Phi, \Xi)$ has the same properties.  
\end{lemma}

\begin{proof}
	We argue by induction along the partite construction. It is plain that the copies 
	in picture zero are strongly induced and that their intersections are clean. For 
	the induction step we assume that $(G, \ccG)$ is a system of not necessarily induced 
	hypergraphs, that $(\Pi, \ccP, \psi_\Pi)$ is a picture over $(G, \ccG)$ whose copies 
	are strongly induced and have clean intersections, that $(H, \ccH)$ is a partite
	system of hypergraphs with strongly induced copies whose intersections are clean, 
	and finally that $(\Sigma, \ccQ, \psi_\Sigma)=(\Pi, \ccP, \psi_\Pi)\conc (H, \ccH)$.
	We need to show that the copies in $\ccQ$ are strongly induced and that their 
	intersections are clean.
	
	Beginning with the latter task we show first that the intersection of any two distinct
	copies $F_\star, F_{\star\star}\in\ccQ$ is clean. Let $\Pi_\star$ 
	and $\Pi_{\star\star}$ be standard copies of $\Pi$ containing 
	$F_\star$ and~$F_{\star\star}$, respectively. 
	If those copies coincide we may appeal
	to the induction hypothesis, and thus it suffices to treat the 
	case $\Pi_\star\ne\Pi_{\star\star}$. 
	Since $V(\Pi_\star)\cap V(\Pi_{\star\star})\subseteq V(H)$
	and owing to the fact that in $\ccH$ the intersections are clean, there exist 
	edges $e_\star\in E(\Pi_\star)\cap E(H)$ 
	and $e_{\star\star}\in E(\Pi_{\star\star})\cap E(H)$ with 
		\begin{equation} \label{eq:214b}
		V(F_\star)\cap V(F_{\star\star}) 
		\subseteq 
		V(\Pi_\star)\cap V(\Pi_{\star\star}) 
		=  
		e_\star\cap e_{\star\star}\,.
	\end{equation} 
		As the copies of $(\Pi_\star, \ccP_\star)$ and $(\Pi_{\star\star}, \ccP_{\star\star})$
	are strongly induced, there exist 
	edges $e'\in E(F_\star)$ and $e''\in E(F_{\star\star})$ with 
		\begin{equation} \label{eq:214c}
		V(F_\star)\cap e_\star \subseteq e'
		\quad \text{ and } \quad 
 		V(F_{\star\star})\cap e_{\star\star}\subseteq e'' \,.
	\end{equation} 
		Now we have 
		\[
		V(F_\star)\cap V(F_{\star\star}) 
		\overset{\eqref{eq:214b}}{=}
		\bigl(V(F_\star)\cap e_\star\bigr) 
		\cap 
		\bigl(V(F_{\star\star})\cap e_{\star\star}\bigr)
		\overset{\eqref{eq:214c}}{\subseteq } 
		e'\cap e''
	\]
		and, consequently, the edges $e'$ and $e''$ exemplify that the intersection 
	of $F_\star$ and $F_{\star\star}$ is clean. 
	
	It remains to show that every $F_\star\in\ccQ$ is strongly induced in $\Sigma$. 
	That is, given an edge $e\in E(\Sigma)$ we need to prove that there exists an 
	edge $e'\in E(F_\star)$ with $V(F_\star)\cap e\subseteq e'$. Again let~$\Pi_\star$
	be a standard copy containing $F_\star$. 
	If there exists a standard copy $\Pi_{\star\star}$ containing~$e$ we may simply repeat 
	the above argument ignoring $e''$. If no such standard copy exists, then necessarily
	$e\in E(H)$. Since the copy in $\ccH$ extended by $\Pi_\star$ is strongly induced 
	in $H$, there exists an edge $e_0\in E(\Pi_\star)$ with $V(\Pi_\star)\cap e\subseteq e_0$.
	Moreover, $F_\star\Str \Pi_\star$ leads to an edge $e'\in E(F_\star)$ with
	$V(F_\star)\cap e_0\subseteq e'$. Now 
		\[
			V(F_\star)\cap e 
			= 
			V(F_\star)\cap V(\Pi_\star)\cap e 
			\subseteq 
			V(F_\star)\cap e_0
			\subseteq 
			e'
	\]
		shows that $e'$ has the desired property.  
\end{proof}

As an example, we may consider the construction $\Omega^{(2)}=\PC(\Rms, \CPL)$, which 
given a hypergraph $F$ and a number of colours $r$ delivers a 
system $\Omega^{(2)}_r(F)=(H, \ccH)$ of strongly induced copies with clean intersections
such that $\ccH\lra (F)_r$. (The construction~$\Omega^{(2)}$ occurs implicitly
in~\cite{BNRR}). 
\index{$\Omega^{(2)}$}

\subsection{Ordered constructions} 
\label{sssec:ord}

Suppose that $F_<$ is an {\it ordered hypergraph}, i.e., a hypergraph
on whose set of vertices a linear ordering is imposed. 
\index{ordered hypergraph}
Now for any number of colours~$r$ we want to construct an ordered 
hypergraph $H_<$ with 
\begin{equation}\label{eq:5555}
	H_< \lra(F_<)_r
\end{equation}
which means that the monochromatic induced copy of $F_<$ needs to occur with the 
correct ordering. If we were to omit the demand that the ordered monochromatic
copy needs to be induced, we could just use Ramsey's theorem, thus getting 
a pair $\Rms_r(F_<)=(G_<, \ccG)$ consisting of a sufficiently large ordered clique $G_<$
and the system $\ccG=\binom{G_<}{F_<}_\nni$ of all not necessarily induced ordered 
subhypergraphs of $G_<$ isomorphic to $F_<$. 

In order to obtain $H_<$ as in~\eqref{eq:5555} we run the partite construction 
$\PC(\Rms, \HJ)_r(F)$, thus getting a final picture $(\Pi, \ccP, \psi)$
over $(G_<, \ccG)$, which is known to satisfy $\ccP \lra (F)_r$ in the sense of 
(unordered) hypergraphs. Recall that the copies in $\ccP$ are induced and project 
via $\psi$ onto copies in~$\ccG$,
i.e., to copies that are ordered correctly in the vertical world.  
Thus if $\strictif$ denotes any linear ordering on $V(\Pi)$ with the property
\begin{equation} \label{eq:morast}
	\forall x, y\in V(\Pi) \,\,\, 
		\bigl[ \psi(x)<\psi(y) 
		\,\,\, \Longrightarrow \,\,\, 
		x \strictif y \bigr]\,,
\end{equation}
then the copies in $\ccP$ become ordered copies of $F_<$ in $\Pi_\strictif$,
meaning that $\Pi_\strictif$ is as desired. 

Let us observe that this argument generalises as follows: If $\Phi$ denotes 
any {\it ordered} Ramsey construction, and $\Xi$ is a partite lemma,
then $\PC(\Phi, \Xi)$ is again an ordered construction. That is, if for some 
ordered hypergraph $F_<$ and $r\in \NN$ the system $\PC(\Phi, \Xi)_r(F)=(H, \ccH)$
is defined, then we may endow $V(H)$ with a linear ordering $\strictif$ as 
in~\eqref{eq:morast}, and $\ccH\subseteq \binom{H_\strictif}{F_<}$ will hold 
automatically. In other words, $\PC(\Phi, \Xi)$ becomes an ordered construction 
by setting $\PC(\Phi, \Xi)_r(F_<)=(H_\strictif, \ccH)$.

Thus we can likewise consider $\Omega^{(2)}=\PC(\Rms, \CPL)$ to be an ordered 
construction. It still has the benefits discussed in~\S\ref{sssec:vindya}, i.e.,
it delivers strongly induced copies with clean intersections. 

\subsection{\texorpdfstring{$f$}{\it f}-partite hypergraphs}
\label{sssec:fpart}  

The following concept interpolates between $k$-partite $k$-uni\-form hypergraphs
and general $k$-uniform hypergraphs. 

\begin{dfn}
	Let $f\colon I\lra\NN$ be a function from a finite index set $I$ to
	the positive integers such that $k=\sum_{i\in I} f(i)$ is at least $2$. 
	\index{$f$-partite hypergraph}
	An {\it $f$-partite hypergraph} is a $k$-uniform hypergraph $F$
	together with a distinguished partition $V(F)=\bigdcup_{i\in I} V_i(F)$ of its vertex
	set satisfying 
		\begin{equation} \label{eq:216a}
		|e\cap V_i(F)|=f(i) 
		\quad \text{ for all }
			e\in E(F) \text{ and } i\in I\,.
	\end{equation}
	\end{dfn}

For example, if $|I|=k$ and $f$ is the constant function whose value is always $1$, 
then an $f$-partite hypergraph is the same as a $k$-partite $k$-uniform hypergraph. On the other 
end of the spectrum there is the possibility that $|I|=1$ and $f$ attains the 
value $k$, in which case $f$-partite hypergraphs are the same as ordinary $k$-uniform
hypergraphs. 

Let us return to the case of a general function $f\colon I\lra\NN$. 
Extending the terminology from~\S\ref{sssec:HJ} we say that an $f$-partite hypergraph $F$ 
is an {\it $f$-partite subhypergraph} of an $f$-partite hypergraph $H$ if $F$ is a
subhypergraph of $H$ and $V_i(F)\subseteq V_i(H)$ holds for all~$i\in I$. Moreover,
for two $f$-partite hypergraphs $F$ and $H$ the symbol $\binom{H}{F}^\ff$
refers to the set of all $f$-partite subhypergraphs of $H$ that are isomorphic to $F$
in the $f$-partite sense.  
 
For $m\in \NN$ we let $K^f_m$ denote the {\it complete $f$-partite hypergraph} having 
for every index~$i\in I$ a vertex class $V_i$ of size~$m$ and having all edges compatible 
with~\eqref{eq:216a}. It is well known that given $m, r\in \NN$ there exists a natural 
number $M$ which is so large that no matter how the edges of $K^f_M$ get coloured 
with~$r$ colours there will always exist a monochromatic copy of~$K^f_m$. 
Indeed, for $|I|=1$ this is just Ramsey's original theorem and the general version 
of this so-called {\it product Ramsey theorem} is easily established by induction on $|I|$ 
(see e.g.,~\cite{Promel}*{Theorem 5.1.5}).
\index{product Ramsey theorem}

As usual, this result yields a Ramsey theorem for $f$-partite hypergraphs with 
non-induced copies. Extending our earlier notation we denote the corresponding 
construction by $\Rms$ as well. So explicitly for an $f$-partite hypergraph $F$
and a number of colours $r$ the formula $\Rms_r(F)=(G, \ccG)$ indicates that $G$ is 
a sufficiently large complete $f$-partite hypergraph, that 
$\ccG=\binom{G}{F}_\nni^\ff$ is the collection
of all $f$-partite subhypergraphs of $G$ isomorphic to $F$ and, finally, 
that $\ccG\lra(F)_r$. 

We say that a {\it construction $\Phi$ is $f$-partite} if applied to an $f$-partite 
hypergraph $F$ it delivers an {\it $f$-partite system of hypergraphs}, i.e., a system
$(H, \ccH)$ consisting of an $f$-partite hypergraph $H$ and a 
set $\ccH\subseteq\binom{H}{F}^\ff_\nni$ of $f$-partite 
subhypergraphs of $H$. For instance the version of $\Rms$ described in the previous 
paragraph is an $f$-partite construction. 

Given an $f$-partite construction $\Phi$ as well as a partite lemma $\Xi$, we can utilise 
the partite construction method and form 
$\Theta=\PC(\Phi, \Xi)$. It is not hard to see that the construction~$\Theta$ is 
again $f$-partite. (A similar argument occurred in the proof of Corollary~\ref{cor:CPL}.)
 
We may now regard $\Omega^{(2)}=\PC(\Rms, \CPL)$ as an $f$-partite construction, 
which yields a rather strong form of the induced Ramsey theorem for $f$-partite hypergraphs. 
As explained in~\S\ref{sssec:ord} we may actually apply this construction to ordered hypergraphs.

\begin{prop}\label{prop:1738}
	Given an ordered $f$-partite hypergraph $F_<$ and a number of colours $r$ 
	the ordered $f$-partite system $\Omega^{(2)}_r(F_<)=(H_<, \ccH)$ 
	satisfies $\ccH\lra(F_<)_r$, the copies in $\ccH$ are strongly induced, and 
	the intersections of copies in $\ccH$ are clean. \qed 
	\index{$\Omega^{(2)}$}
\end{prop}

\subsection{Linearity}
\label{sssec:lin}

Clearly, if $F$ is a non-trivial linear $k$-uniform hypergraph for some $k\ge 3$, 
and $r\ge 2$ is a number of colours, then $\Rms_r(F)$ fails to be linear. For the 
purposes of girth Ramsey theory it is important to know that the other constructions 
we have encountered so far behave better in this regard. We begin this discussion 
with a special case of a result in~\cite{NeRo4}.

\begin{lemma}\label{lem:1819}
	If $F$ is a linear $k$-partite $k$-uniform hypergraph, $r\in\NN$ is a number of colours, 
	and $\HJ_r(F)=(H, \ccH)$, then $H$ is a linear hypergraph as well. 
\end{lemma}

\begin{proof}
	The degenerate case $E(F)=\vn$ being clear we suppose $E(F)\ne\vn$ from now on. 
	Let~$n$ be the Hales-Jewett exponent the construction of $H$ is based on and
	let ${\lambda\colon E(F)^n\lra E(H)}$ denote the canonical bijection.
	For any two distinct edges 
		\[
		e=\lambda(e_1, \ldots, e_n) 
		\quad \text{ and } \quad 
		e'=\lambda(e'_1, \ldots, e'_n)
	\]
		of $H$ we are to prove $|e\cap e'|\le 1$.
	To this end we take an index $\nu\in [n]$ with $e_\nu\ne e'_\nu$. If $e_\nu$
	and~$e'_\nu$ are disjoint, then so are $e$ and $e'$. Otherwise, the linearity 
	of $F$ discloses that $e_\nu$ and~$e'_\nu$ have a unique vertex $x$ in common. If $i$
	denotes the index with $x\in V_i(F)$, then all vertices that $e$ and $e'$ have 
	in common belong to $V_i(H)$, whence $|e\cap e'|\le |e\cap V_i(H)|=1$.
\end{proof}   

In the sequel, a {\it linear construction} will be a construction $\Phi$, which,
when applied to a linear hypergraph $F$ and a number of colours $r$, yields 
a {\it linear system} $\Phi_r(F)=(H, \ccH)$, i.e., a system of hypergraphs whose 
underlying hypergraph $H$ is linear.
\index{linear construction}
\index{linear system (of hypergraphs)}
\index{linear partite lemma}
So in other words the previous lemma asserts that $\HJ$ is a linear partite lemma. 
A picture $(\Pi, \ccP, \psi)$ is said to be {\it linear} if its underlying 
hypergraph $\Pi$ is linear. In the linear
case strong inducedness can be characterised as follows. 

\begin{fact} \label{fact:1836}
	A subhypergraph $F$ of a linear hypergraph $H$ is strongly induced if and only if 
	it has the following three properties.
	\begin{enumerate}[label=\rmlabel]
		\item\label{it:1836a} If an edge $e$ of $H$ intersects $F$ in at least two vertices, 
		then it belongs to $F$.
		\item\label{it:1836b} If an edge $e$ of $H$ has only a single vertex $x$ with $F$ 
		in common, then $x$ is non-isolated in $F$.
		\item\label{it:1836c} If $F$ has no edges, then neither does $H$. \hfill $\Box$
\end{enumerate}
\end{fact} 

Clearly, of these conditions~\ref{it:1836a} is the most important one. 
The next question we would like to address is why the clean partite lemma $\CPL$ 
is linear. 

\begin{lemma}\label{lem:2207}
	If $\Phi$ and $\Xi$ are linear, then so is $\PC(\Phi, \Xi)$.
\end{lemma}

\begin{proof}
	We prove inductively that all pictures encountered in the partite construction
	are linear. There is no problem with picture zero and, therefore, it suffices to 
	establish the following statement. 
		\begin{quotation}
		\it
		If $(\Pi, \ccP, \psi_\Pi)$ is a linear picture over a linear system $(G, \ccG)$, 
		$e\in E(G)$, and the copies of the linear system $(H, \ccH)$ are isomorphic to 
		the constituent $\Pi^e$, then the picture  
		\[	
			(\Sigma, \ccQ, \psi_\Sigma) = (\Pi, \ccP, \psi_\Pi) \conc (H, \ccH)
		\]
		is linear again. 
	\end{quotation}
		
	Given any two distinct edges $f', f''\in E(\Sigma)$ we are to
	prove $|f'\cap f''|\le 1$. 
	If their projections $\psi_\Sigma[f']$ and $\psi_\Sigma[f'']$ are distinct, 
	this follows from the fact that $G$ is linear. So we may assume 
	that $f=\psi_\Sigma[f']=\psi_\Sigma[f'']$ holds for some $f\in E(G)$. In the 
	special case $e=f$ we may appeal to the linearity of $H$, so suppose $f\ne e$
	from now on. Now $f', f''\not\in E(H)$ implies that there are unique standard 
	copies $\Pi'$ and $\Pi''$ in $\Sigma$ containing~$f'$ and~$f''$, respectively. 
	
	If $\Pi'=\Pi''$, then the linearity of $\Pi$ leads to the desired conclusion, 
	so it remains to consider the case $\Pi'\ne \Pi''$. Given any two vertices
	$x, y\in f'\cap f''$ we need to prove $x=y$. Since the standard copies $\Pi'$ 
	and~$\Pi''$
	were constructed to be as disjoint as possible, we 
	have 
		\[
		f'\cap f''\subseteq V(\Pi')\cap V(\Pi'')\subseteq V(H)\,,
	\]
		whence $\psi_\Sigma(x), \psi_\Sigma(y)\in e\cap f$. Owing to the linearity of~$H$ 
	this yields $\psi_\Sigma(x)=\psi_\Sigma(y)$. In other words, $x$ and $y$ are on 
	the same music line.
	But~$f'$ intersects this music line only once and, consequently, we have 
	indeed $x=y$.    
\end{proof}

\begin{cor}\label{cor:2217}
	The clean partite lemma $\CPL$ is linear.
	\index{clean partite lemma}
\end{cor}

\begin{proof}
	By Lemma~\ref{lem:1819} the Hales-Jewett construction $\HJ$ is linear; so 
	Lemma~\ref{lem:2207} tells us that $\CPL=\PC(\HJ, \HJ)$ is linear as well. 
\end{proof}

Recall that by Corollary~\ref{cor:CPL} the partite lemma $\CPL$ delivers systems whose copies have 
clean intersections. We observe that if $(H, \ccH)$ is a {\it linear system}, 
then the copies in $\ccH$
have clean intersections if and only if for any two distinct 
copies $F_\star, F_{\star\star}\in\ccH$ the following three statements hold.
\begin{enumerate}[label=\rmlabel]
	\item\label{it:1519a} If $|V(F_\star)\cap V(F_{\star\star})|\ge 2$, then there exists an 
		edge $e\in E(F_\star)\cap E(F_{\star\star})$ with $V(F_\star)\cap V(F_{\star\star})=e$
	\item\label{it:1519b} If $V(F_\star)\cap V(F_{\star\star})$ consists of a single 
		vertex, then this vertex is non-isolated in $F_\star$ and $F_{\star\star}$. 
	\item\label{it:1519c} $E(F_\star), E(F_{\star\star})\ne\vn$.  	
\end{enumerate}
Again, the main property of relevance is~\ref{it:1519a}.

\begin{lemma}\label{lem:1527}
	If $\Omega$ is an arbitrary Ramsey construction and $\Xi$ denotes a linear partite
	lemma delivering systems with strongly induced copies whose intersections are 
	clean, then $\PC(\Omega, \Xi)$ is a linear construction. 
\end{lemma}

\begin{proof}
	Arguing by induction along the partite construction it suffices to prove
	the following picturesque statement.
	\begin{quotation}
		\it
		If $(\Pi, \ccP, \psi_\Pi)$ is a linear picture over a system of 
		hypergraphs $(G, \ccG)$ and $(H, \ccH)$ is a linear $k$-partite $k$-uniform 
		system whose copies are strongly induced and have clean intersections, then 
		the picture 
				\[
			(\Sigma, \ccQ, \psi_\Sigma)=(\Pi, \ccP, \psi_\Pi)\conc (H, \ccH)
		\]
				is linear as well. 
	\end{quotation}
	
	Let us emphasise that while we are assuming here that the 
	hypergraphs $\Pi$ and $H$ are linear, it is allowed that the vertical 
	projection $G$ fails to be linear.
	Assume for the sake of contradiction that $\Sigma$ is nonlinear. Choose a pair 
	of distinct edges $f_\star, f_{\star\star}\in E(\Sigma)$ 
	whose intersection $t=f_\star\cap f_{\star\star}$
	satisfies $|t|\ge 2$ and such that, subject to 
	this, $\Lambda=|E(H)\cap \{f_\star, f_{\star\star}\}|$ is maximal.

	Suppose first that $\Lambda=0$, i.e., that $f_\star, f_{\star\star}\not\in E(H)$. 
	Denote the unique standard copies containing $f_\star$ and $f_{\star\star}$
	by $\Pi_\star$ and $\Pi_{\star\star}$, respectively. 
	Since $\Pi$ is linear, we have $\Pi_\star\ne\Pi_{\star\star}$.  
	Let $\Pi^e_\star, \Pi^e_{\star\star}\in \ccH$ be the copies of the 
	constituent $\Pi^e$ extended by $\Pi_\star$ and $\Pi_{\star\star}$. 
	Recall that these copies are linear and that their intersection is clean.
	Together with $t\subseteq V(\Pi^e_\star)\cap V(\Pi^e_{\star\star})$
	this proves that they have an edge $f$ in common. But now $t\subseteq f_\star\cap f$
	and $f\in E(H)$ show that the pair $\{f_\star, f\}$ contradicts the 
	maximality of $\Lambda$. 
	
	Let us deal with the case $\Lambda=1$ next. By symmetry we may suppose that    
	$f_\star\not\in E(H)$ and $f_{\star\star}\in E(H)$. Define the standard 
	copy $\Pi_\star$ and $\Pi^e_\star\in\ccH$ as in the foregoing paragraph. 
	Due to $t\subseteq V(\Pi^e_\star)\cap f_{\star\star}$ the fact that $\Pi^e_\star$
	is strongly induced in $H$ shows that the edge $f_{\star\star}$ belongs 
	to $\Pi_\star$. Thus we get a contradiction to the linearity of $\Pi$.
	
	Altogether we have thereby proved $\Lambda=2$, i.e., that necessarily 
	$f_\star, f_{\star\star}\in E(H)$. But this contradicts the linearity of $H$. 	    
\end{proof}

\begin{cor}\label{cor:0059}
	The construction $\Omega^{(2)}=\PC(\Rms, \CPL)$ is linear. 
	\index{$\Omega^{(2)}$}
\end{cor}

\begin{proof}
	The assumptions of Lemma~\ref{lem:1527} are satisfied by Corollary~\ref{cor:CPL}
	and Corollary~\ref{cor:2217}.
\end{proof}

\subsection{\texorpdfstring{$A$}{\it A}-intersecting hypergraphs}
\label{subsec:2346}

The definition of the train hypergraphs we shall study later 
(see Figure~\ref{fig:N3}) will contain a demand that certain kinds of edges 
are allowed to intersect in certain vertex classes only. As a very modest step 
into this direction we show in this subsection that the construction $\Omega^{(2)}$ preserves such a property. Later this result will contribute to the base case of our
main induction (see Lemma~\ref{lem:0120}).

\begin{dfn}\label{dfn:n38}
	Given a finite index set $I$ let $f\colon I\lra \NN$ be a function such 
	that $\sum_{i\in I}f(i)$ is at least $2$. Further, let $A$ be a subset
	of $I$ and let $F$ be an $f$-partite hypergraph.
	\index{$A$-intersecting hypergraph}
		\begin{enumerate}[label=\alabel]
		\item\label{it:n38a} We set $V_A(F)=\bigcup_{i\in A}V_i(F)$.
		\item\label{it:n38b} If $e\cap e'\subseteq V_A(F)$ holds for any two distinct 
			edges $e$, $e'$ of $F$ we say that $F$ is {\it $A$-intersecting}. 
	\end{enumerate}
	\end{dfn}

An $f$-partite Ramsey construction or a partite lemma $\Phi$ is said to  
be $A$-intersecting if whenever $\Phi_r(F)=(H, \ccH)$ and $F$ is $A$-intersecting 
for some subset $A$ of the relevant index set, then so is $H$. The next three 
lemmata show that $\HJ$, $\CPL$, and $\Omega^{(2)}$ have this property. 

\begin{lemma}\label{lem:n381}
	The Hales-Jewett construction $\HJ$ is $A$-intersecting. 
	\index{Hales-Jewett construction}
\end{lemma}

\begin{proof}
	Let $\HJ_r(F)=(H, \ccH)$ for some $A$-intersecting $k$-partite $k$-uniform 
	hypergraph~$F$ and for some number of colours~$r$. As usual, we denote 
	the implied Hales-Jewett exponent by~$n$ and we let $\lambda\colon E(F)^n\lra E(H)$
	be the canonical bijection. 
	
	Given any two distinct edges $e$, $e'$ of $H$ we are to 
	prove $e\cap e'\subseteq V_A(H)$. To this end we write 
	$e=\lambda(e_1, \dots, e_n)$ and $e'=\lambda(e'_1, \dots, e'_n)$
	with appropriate edges $e_1, \dots, e_n$ and $e'_1, \dots, e'_n$ of $F$. 
	Due to $e\ne e'$ there exists a coordinate direction $\nu\in [n]$ 
	such that $e_\nu\ne e'_\nu$. Since $F$ is $A$-intersecting, 
	we have $e_\nu\cap e'_\nu\subseteq V_A(F)$, and $e\cap e'\subseteq V_A(H)$
	follows. 
\end{proof} 

\begin{lemma}\label{lem:n382}
	The clean partite lemma $\CPL$ is $A$-intersecting.
	\index{clean partite lemma} 
\end{lemma}

\begin{proof}
	Let an $A$-intersecting $k$-partite $k$-uniform hypergraph $F$ and a number 
	of colours~$r$ be given. Due to $\CPL=\PC(\HJ, \HJ)$ every picture encountered 
	in the construction of~$\CPL_r(F)$ possesses a $k$-partite structure and we 
	shall prove inductively that all these pictures are $A$-intersecting. 
	As this is clear for picture zero and the vertical system $\HJ_r(F)$ 
	is $A$-intersecting, it thus suffices to prove the following statement. 
	
	\begin{quotation}
	\it 
	Let $G$ be an $A$-intersecting $k$-partite $k$-uniform hypergraph and suppose 
	that 
		\[
		(\Sigma, \ccQ, \psi_\Sigma)=(\Pi, \ccP, \psi_\Pi)\conc (H, \ccH)
	\]
		holds for two pictures $(\Sigma, \ccQ, \psi_\Sigma)$, $(\Pi, \ccP, \psi_\Pi)$
	over some system $(G, \ccG)$, 
	and for a $k$-partite $k$-uniform system $(H, \ccH)$. If $\Pi$ and $H$ are 
	$A$-intersecting, then so is $\Sigma$.   
	\end{quotation}
	
	To verify this we consider any two distinct edges $e'$, $e''$ of $\Sigma$. 
	If their projections to $G$ are distinct, then the assumption that $G$ be 
	$A$-intersecting yields $\psi_\Sigma(e')\cap \psi_\Sigma(e'')\subseteq V_A(G)$
	and the desired inclusion $e'\cap e''\subseteq V_A(\Sigma)$ follows. 
	
	So from now on we may assume $e_\star=\psi_\Sigma(e')=\psi_\Sigma(e'')$
	for some edge $e_\star\in E(G)$. If $e_\star$ coincides with the edge $e\in E(G)$
	over which the amalgamation happens, 
	then $e'\cap e''\subseteq V_A(H)\subseteq V_A(\Sigma)$ is a consequence 
	of $H$ being $A$-intersecting. 
	
	Thus we can assume $e\ne e_\star$ in the sequel, 
	which implies $e\cap e_\star\subseteq V_A(G)$. If $e'$ and $e''$ are in the 
	same standard copy of $\Pi$, we just need to appeal to $\Pi$ being $A$-intersecting
	and if those standard copies are distinct, then $e'\cap e''$
	is contained in $V(H)$ 
	and 
		\[
		\psi_\Sigma[e'\cap e'']\subseteq e\cap e_\star\subseteq V_A(G)
	\]
		leads again to $e'\cap e''\subseteq V_A(\Sigma)$.  
\end{proof}

\begin{lemma}\label{lem:n383}
	The construction $\Omega^{(2)}$ is $A$-intersecting. 
	\index{$\Omega^{(2)}$}
\end{lemma}

\begin{proof}
	Again we argue by induction along the partite construction. There is no problem 
	with picture zero. As $\Omega^{(2)}=\PC(\Rms, \CPL)$ utilises the clean partite 
	lemma, Corollary~\ref{cor:CPL} and Lemma~\ref{lem:n382} show that it suffices to 
	prove the following picturesque statement.
	\begin{quotation}
	\it
	Suppose that
		\[
		(\Sigma, \ccQ, \psi_\Sigma)=(\Pi, \ccP, \psi_\Pi)\conc (H, \ccH)
	\]
		holds for two pictures $(\Sigma, \ccQ, \psi_\Sigma)$, $(\Pi, \ccP, \psi_\Pi)$
	over an $f$-partite system $(G, \ccG)$ and for a $k$-partite $k$-uniform system 	
	$(H, \ccH)$ with strongly induced copies whose intersections are clean. 
	If $\Pi$ and $H$ are $A$-intersecting, then so is $\Sigma$.
	\end{quotation}
	
	Let $e'$ and $e''$ be two distinct edges of $\Sigma$. Each of them is either 
	in $E(H)$ or it is not, and by symmetry there are three possibilities to 
	consider. If both edges belong to $H$, 
	then $e'\cap e''\subseteq V_A(H)\subseteq V_A(\Sigma)$ is clear.
		
	Suppose next that $e'\not\in E(H)$ and $e''\in E(H)$. 
	Let~$\Pi_\star$ be the standard copy of~$\Pi$ containing~$e'$ and 
	let~$\Pi^e_\star\in\ccH$ be the copy extended by~$\Pi_\star$. 
	Due to $\Pi^e_\star\Str H$ there is an edge $e_\star\in E(\Pi^e_\star)$ such 
	that $V(\Pi^e_\star)\cap e''\subseteq e_\star$.
	In the special case $e''=e_\star$ both edges $e'$ and~$e''$ belong to the 
	same standard copy $\Pi_\star$ and 
	$e'\cap e''\subseteq V_A(\Pi_\star)\subseteq V_A(\Sigma)$ follows 
	from the assumption that $\Pi$ be $A$-intersecting. 
	Moreover, if $e''\ne e_\star$, then we have 
		\[
		e'\cap e''
		\subseteq 
		V(\Pi^e_\star)\cap e''
		\subseteq 
		e_\star\cap e''
		\subseteq 
		V_A(H)
		\subseteq 
		V_A(\Sigma)\,.
	\]
		
	It remains to deal with the case $e', e''\not\in E(H)$. 
	Let $\Pi_\star$, $\Pi_{\star\star}$ be the standard copies of~$\Pi$ containing 
	these two edges and let $\Pi^e_\star$, $\Pi^e_{\star\star}$ be the corresponding
	copies in $\ccH$. 
	If they coincide we just need to appeal to the fact that $\Pi$
	is $A$-intersecting, so we can henceforth assume $\Pi^e_\star\ne\Pi^e_{\star\star}$.
	Now the intersection of these two copies is clean and thus there 
	exists an edge $e_\star\in E(\Pi^e_\star)$ covering their intersection. 
	The discussion of the previous paragraph 
	shows $e_\star\cap e''\subseteq V_A(\Sigma)$,
	whence 
	$e'\cap e''\subseteq V(\Pi^e_\star)\cap V(\Pi^e_{\star\star})\cap e''
	\subseteq e_\star\cap e''\subseteq V_A(\Sigma)$.  
\end{proof} \section{Girth considerations} 
\label{subsec:AG}

This section begins by enumerating, for the sake of completeness, some 
folkloric statements related to the classical girth concept introduced in 
Definition~\ref{dfn:girth}. From~\S\ref{sssec:coc} onwards, however, we move 
on to new territory and study a concept of Girth applicable to linear systems 
of hypergraphs (see Definition~\ref{dfn:Girth} below). The notational difference 
between the two notions is that in the former case ``$\gth$'' is written with a 
lower case~``g'', whereas the capital ``G'' in ``$\Gth$'' indicates that we consider
the Girth of a linear system. 
 
\subsection{Set systems and girth}
\label{sssec:GSS}

By a {\it set system} we mean a pair $S=(V, E)$ consisting of a
set of {\it vertices}~$V$ and a collection $E\subseteq \powerset(V)$ of subsets of $V$ 
such that every {\it edge} $e\in E$ has at least two elements. Thus a hypergraph is a 
set system with the special property that its edges are of the same cardinality.
\index{set system}

Definition~\ref{dfn:girth} applies to 
set systems in place of hypergraphs as well and for reasons that will become apparent 
later we formulate the two facts that follow in this more general context. 
First, we study the effect of dropping condition~\ref{it:C2}.

\begin{fact}\label{fact:231a}
	Let $g\ge 2$ be an integer and let $S=(V, E)$ be a set system with $\gth(S)>g$.
	If for some integer $n\in [2, g+1]$ we have a sequence 
		\begin{equation*}
		e_1v_1\ldots e_nv_n
	\end{equation*}
		satisfying~\ref{it:C1},~\ref{it:C3}, and $e_1, \ldots, e_n\in E$, then
	\begin{enumerate}[label=\alabel]
		\item\label{it:1701a} either $e_1=\dots=e_n$
		\item\label{it:1701b} or $n=g+1$ and the edges $e_1, \ldots, e_n$ are distinct.
	\end{enumerate}
\end{fact}

\begin{proof}
	Otherwise let $\ccC=e_1v_1\ldots e_nv_n$ be a counterexample with $n$ minimum.
	We contend that there is an edge occurring at least twice in $\ccC$. If $n=g+1$
	this is immediate from the failure of~\ref{it:1701b}. If $n\in [2, g]$, then $\gth(S)>g$
	tells us that~$\ccC$ cannot be an $n$-cycle,
	which in turn means that~\ref{it:C2} fails, i.e., that $\ccC$ again contains 
	two equal edges. Now by cyclic symmetry we may suppose that there exists 
	an index $i\in [2, n]$ with $e_1=e_i$. 
	
	Our plan is to prove 
		\begin{equation}\label{eq:231a}
		e_1=\dots =e_{i-1}
		\quad \text{ and } \quad 
		e_i=\dots =e_n\,.
	\end{equation}	
		Together with our choice of $i$ this will show that alternative~\ref{it:1701a}
	holds, thus concluding the proof. 
	
	For reasons of symmetry it suffices to establish the first part of~\eqref{eq:231a}.
	Our claim is obvious for $i=2$ and in case $i\ge 3$ we can apply the minimality of $n$
	to the sequence $\ccD=e_1v_1\ldots e_{i-1}v_{i-1}$. As $\ccD$ contains $i-1\le n-1\le g$
	edges, only option~\ref{it:1701a} can apply to~$\ccD$, which has the desired consequence.
\end{proof}

Second, we immediately obtain the following well known ``transitivity property'' 
of girth that will assist us later when analysing the girth of 
trains (see Lemma~\ref{lem:0036} below). \index{transitivity of girth}

\vbox{
\begin{fact}\label{fact:girth-trans}
	Let an integer $g\ge 2$ and a set system $S=(V, E)$ with $\gth(S)>g$ 
	be given. If for every edge $e\in E$ we have a set system $F_e$ 
	with vertex set $e$ and $\gth(F_e)>g$, then the set system $T=(V, E')$ 
	defined by $E'=\bigcup_{e\in E}E(F_e)$ satisfies $\gth(T)>g$ as well. 
	
	Moreover, if $\gth(F_e)>g+1$ holds for every $e\in E$ 
	and $\ccC=f_1v_1\ldots f_{g+1}v_{g+1}$ is a $(g+1)$-cycle in $T$,
	then $\ccC$ contains at most one edge from every system $F_e$. 
\end{fact}
}

\begin{proof}
	Assume first that contrary to $\gth(T)>g$ we have for some $n\in [2, g]$ an $n$-cycle 
	\begin{equation*}		
		\ccC=f_1v_1\ldots f_nv_n
	\end{equation*}
	in	$T$. Owing to the linearity of $S$, there are uniquely determined 
	edges $e(1), \ldots, e(n)\in E$ such that $f_i\in E(F_{e(i)})$ holds
	for every $i\in\ZZ/n\ZZ$. Since $f_i\subseteq e(i)$, the cyclic
	sequence  
		\[
		\ccD=e(1)v_1\ldots e(n)v_n
	\]
	has the properties~\ref{it:C1},~\ref{it:C3} of an $n$-cycle in $S$. 
	Due to $\gth(S)>g\ge n$
	Fact~\ref{fact:231a} applied to $S$ and~$\ccD$ informs us that for some $e\in E$ we 
	have $e=e(1)=\dots =e(n)$.  
	In other words, the sequence~$\ccC$ is an $n$-cycle in $F_e$, contrary 
	to $\gth(F_e)>g$. We have thereby proved that $\gth(T)>g$.
	
	Now suppose moreover that $\gth(F_e)>g+1$ holds for every $e\in E$ 
	and that 
		\[
		\ccC=f_1v_1\ldots f_{g+1}v_{g+1}
	\]
		is a $(g+1)$-cycle in $T$.
	Choosing the edges $e(1), \ldots, e(g+1)\in E$ as before we arrive again at a cyclic 
	sequence $\ccD=e(1)v_1\ldots e(g+1)v_{g+1}$ and Fact~\ref{fact:231a} is still applicable.
	Its option~\ref{it:1701a} would lead to the same contradiction as before, 
	so~\ref{it:1701b} holds and we are done.
\end{proof}

\subsection{Cycles of copies}
\label{sssec:coc}

Suppose that $(H, \ccH)$ is a linear system. Intuitively, a cycle in~$\ccH$
is just a cyclic arrangement of copies any two consecutive ones of which 
are distinct but overlap. Here, ``overlapping'' means having at least a vertex 
and possibly even an edge in common. We are thus led to the following concept.  

\begin{dfn}\label{dfn:coc}
	A {\it cycle of copies} in a linear system $(H, \ccH)$ is a cyclic sequence
		\begin{equation}\label{eq:coc}
		\ccC=F_1q_1F_2q_2\ldots F_nq_n
	\end{equation}
		such that $n\ge 2$ and the following conditions hold.
	\index{cycle of copies} 
	\begin{enumerate}[label=($L\arabic*$)]
		\item\label{it:L1} The copies  $F_1, \ldots, F_n\in \ccH$ 
			satisfy $F_i\ne F_{i+1}$ for all $i\in\ZZ/n\ZZ$.
		\item\label{it:L2} The vertices and edges $q_1, \ldots, q_n\in V(H)\cup E(H)$ are distinct.
		\item\label{it:L3} If $i\in \ZZ/n\ZZ$ and $q_i$ is a vertex, 
			then $q_i\in V(F_i)\cap V(F_{i+1})$.
		\item\label{it:L4} If $i\in \ZZ/n\ZZ$ and $q_i$ is an edge, 
			then $q_i\in E(F_i)\cap E(F_{i+1})$.
	\end{enumerate}
	We say that $q_1, \ldots, q_n$ are the {\it connectors of $\ccC$}, 
	while $F_1, \ldots, F_n$  will be known as its {\it copies}. \index{connector}
\end{dfn}

Suppose now that $\ccC=F_1q_1F_2q_2\ldots F_nq_n$ is a cycle of copies. 
The number $n$ will be called the {\it length} of $\ccC$ and denoted by $n=\len{\ccC}$.
\index{length}
In most of our arguments, however, the length of a cycle of copies will 
only play a secondary r\^{o}le and a more central notion is that of its order, 
which we shall introduce next. 
\index{order}
To this end, we call an index~$i\in\ZZ/n\ZZ$
\begin{enumerate}
\item[$\bullet$] {\it pure} if either both of $q_{i-1}$ and $q_i$ are vertices or both 
	are edges and 
\item[$\bullet$] {\it mixed} if one of $q_{i-1}$ and $q_i$ is a vertex while the 
	other one is an edge. \index{mixed index}\index{pure index}
\end{enumerate}
Recall that by~\ref{it:L2} every $i\in\ZZ/n\ZZ$ is either pure or mixed. 
For parity reasons 
the number of mixed indices has to be even and, therefore, the quantity
\begin{equation}\label{eq:ord}
	\ord{\ccC}=\big|\{i\in\ZZ/n\ZZ\colon i \text{ is pure}\}\big|
		+\tfrac12 \big|\{i\in\ZZ/n\ZZ\colon i \text{ is mixed}\}\big|\,,
\end{equation}
called the {\it order} of $\ccC$, has to be an integer. 
Evidently, the length of a cycle of copies can deviate from its order at most 
by a factor of~$2$, i.e., we have $\len{\ccC}\in [\ord{\ccC}, 2\ord{\ccC}]$. 
Occasionally we shall need to take both the order and the length 
into account and in such situations it is convenient to set \index{$h(\ccC)$}
\begin{equation} \label{eq:232a}
	h(\ccC)=\bigl(\ord{\ccC}, \len{\ccC}\bigr)\in\NN^2\,.
\end{equation}
When relating two ordered pairs of natural numbers by an
inequality we always have the lexicographic ordering of~$\NN^2$ in mind.
This convention puts greater emphasis on the order than on the length, 
for in~\eqref{eq:232a} the order comes first. 
E.g., $h(\ccC)\le (g, n)$ means that either~$\ord{\ccC}<g$ 
or~$\ord{\ccC}=g\,  \&\,  \len{\ccC}\le n$. 

One needs to be careful when defining the Girth of a linear system $(H, \ccH)$
in terms of cycles of copies as introduced above. The reason for this is that 
one can take copies that are arranged like a tree and present them as a cycle of 
copies. It may be instructive to illustrate this point by means of two examples.

\begin{example}\label{exmp:939}
	Suppose that $e'$ and $e''$ are two edges of a copy $F_1$ that have a 
	vertex $x$	in common (see Figure~\ref{fig:N4A}). 
	Let $F_2$ be a further copy having with $F_1$ only the 
	edge~$e'$ in common and, similarly, let $F_3$ be a copy meeting $F_1$ only 
	in $e''$. It is now forced that the copies~$F_2$ and~$F_3$ overlap in $x$ and 
	for transparency we assume that they are otherwise disjoint. This situation 
	gives rise to the cycle of copies $\ccA=F_1e'F_2xF_3e''$ with $h(\ccA)=(2, 3)$.
	It should be clear, however, that such cycles are unavoidable in the Ramsey systems 
	we seek to construct and, therefore, that they should have no bearing on the 
	Girth of our systems. 
\end{example}

\begin{example}\label{exmp:938}
	Let $F_1$, $F_2$, and $F_3$ be three copies which have an edge $e$ in 
	common but are otherwise disjoint (see Figure~\ref{fig:N4B}). 
	If $x_1$, $x_2$, and $x_3$ are any three 
	distinct vertices of~$e$, then $\ccB=F_1x_1F_2x_2F_3x_3$ is a valid example 
	for a cycle of copies with $h(\ccB)=(3, 3)$  that is likewise unavoidable.
\end{example}

\usetikzlibrary {shadows}
\usetikzlibrary{shadows.blur}

\begin{figure}[ht]
\centering	
	
\begin{subfigure}[b]{0.49\textwidth}
\centering
	
\begin{tikzpicture}[scale=1]

\fill [rounded corners=30, blue!15!white, opacity = .5] (0,2.3)--(2,-1)--(-2,-1)--cycle;
\draw [rounded corners=30, blue!75!black, thick] (0,2.3)--(2,-1)--(-2,-1)--cycle;
\fill [rounded corners=30, blue!15!white, opacity = .5, shift=({-1.3,-.3})] (0,-1)--(2,2.3)--(-2,2.3)--cycle;
\draw [rounded corners=30, blue!75!black, thick, shift=({-1.3,-.3})] (0,-1)--(2,2.3)--(-2,2.3)--cycle;
\fill [rounded corners=30, blue!15!white, opacity = .5, shift=({1.3,-.3})] (0,-1)--(2,2.3)--(-2,2.3)--cycle;
\draw [rounded corners=30, blue!75!black, thick, shift=({1.3,-.3})] (0,-1)--(2,2.3)--(-2,2.3)--cycle;

\fill (0,1.4) circle (1.5pt);
\draw [green!70!black, thick] (-1.2,-.4)--(0,1.4)--(1.2,-.4);

\node at (0,1.6) {\footnotesize $x$};
\node at (-1.8,1.2) {$F_2$};
\node at (1.8,1.2) {$F_3$};
\node at (0,-.2) {$F_1$};
\node at (-.8,.6) {\footnotesize $e'$};
\node at (.8,.6) {\footnotesize $e''$};

			\end{tikzpicture}
		 \caption{$\ccA=F_1e' F_2 x F_3e''$}
		\label{fig:N4A} 		
	\end{subfigure}
\hfill    
\begin{subfigure}[b]{0.5\textwidth}
\centering
			\begin{tikzpicture}[scale=1]

				\fill [ opacity = .8,blue!20, thick, rounded corners=20, blur shadow={shadow blur steps=5, shadow xshift=.3em,
					shadow yshift=-.3em}] (1.3,2.9)-- (-2,0)-- (2,-.4)--cycle;

					\fill [rounded corners=20, white, blur shadow={shadow blur steps=5, shadow xshift=.5em,
					shadow yshift=-.3em}] (0,3.5)--(2,-.4)--(-2,0)--cycle;
			
				\fill [white, blur shadow={shadow blur steps=5, shadow xshift=.6em,
					shadow yshift=-.1em}, thick, rounded corners=20] (-2,2.6)--(2.1,-.4) --(-2,0)--cycle;		
				
			\draw [ blue, thick, rounded corners=20] (1.3,2.9)-- (-2,0)-- (2,-.4)--cycle;
			
			\fill [rounded corners=20, blue!20, opacity=.8, blur shadow={shadow blur steps=5, shadow xshift=.2em,
				shadow yshift=-.3em}] (0,3.5)--(2,-.4)--(-2,0)--cycle;

			\draw [blue, rounded corners=20, thick] (0,3.5)--(2,-.4)--(-2,0)--cycle;

			\fill [blue!20, blur shadow={shadow blur steps=5, shadow xshift=.4em,
			shadow yshift=-.1em, opacity = .8}, thick, rounded corners=20] (-2,2.6)--(2.1,-.4) --(-2,0)--cycle;
		
		\draw [blue, thick, rounded corners=20] (-2,2.6)--(2.1,-.4) --(-2,0)--cycle;
		
		\draw [black, thick, line width=2pt] (-1.3,.04) -- (1.3,-.23);
		
		\fill (-1,.02) circle (2pt);
		\fill (0,-.08) circle (2pt);
		\fill (1,-.18) circle (2pt);
		
		\node at (-1,-.4) {\tiny $x_1$};
		\node at (0,-.5) {\tiny $x_2$};
		\node at (1,-.6) {\tiny $x_3$};

\node at (-1.7, 2.7) {$F_1$};
\node at (0,3.5) {$F_2$};
\node at (1.3,2.9) {$F_3$};
						
		\end{tikzpicture}
	\caption{$\ccB=F_1x_1F_2x_2F_3x_3$}
	\label{fig:N4B} 		
\end{subfigure}  
		
\caption{Two unavoidable cycles of copies}
\label{fig:N4} 
\end{figure}
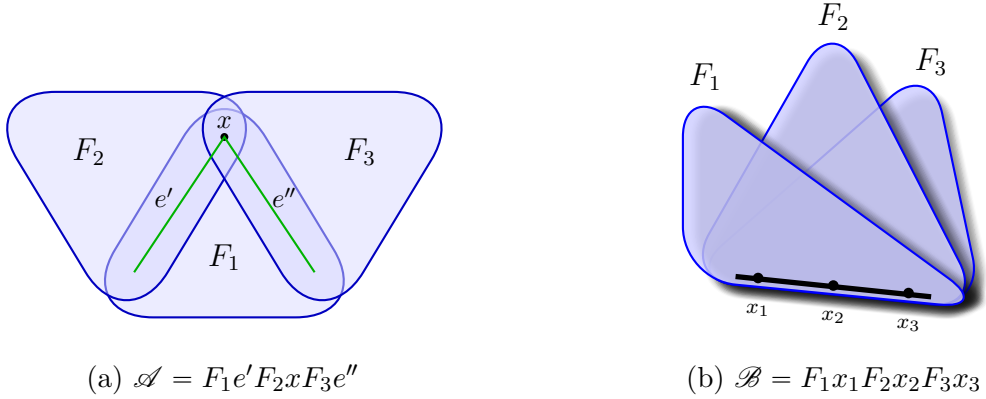 
These circumstances suggest to declare cycles such as $\ccA$ and $\ccB$ to be 
``untidy'' and to resolve that only tidy cycles of copies are allowed to affect 
the Girth of a system. Intuitively speaking, the reason why the above cycle $\ccA$ 
should not be tidy is that its connectors~$x$ and~$e'$ satisfy  
$x\in e'$. The untidiness of $\ccB$, on the other hand, derives from the 
existence of an edge $e$ containing too many vertex connectors. The next definition 
renders these ideas in a more precise form (see Figure~\ref{fig:23}).  

\begin{dfn} \label{dfn:952} 
	A cycle of copies $\ccC=F_1q_1\ldots F_nq_n$ in a linear system $(H, \ccH)$
	is said to be {\it tidy} if it has the following two properties.
	\begin{enumerate}[label=\upshape{($T\arabic*$)}]		
		\item\label{it:T1} There do not exist connectors $q_i$ and $q_j$ 
			with $q_i\in q_j$.
		\item\label{it:T2} For every edge $f\in E(H)$ the set 
						\[
				M(f)=\{i\in \ZZ/n\ZZ\colon q_i \text{ is a vertex belonging to } f\}
			\]
						can be covered by a set of the from $\{i(\star), i(\star)+1\}$,
			where $i(\star)\in\ZZ/n\ZZ$.
	\end{enumerate}
	\index{tidy}
\end{dfn}  

\begin{figure}[ht]
	\centering

		\begin{subfigure}[b]{0.4\textwidth}
			\centering
	\begin{tikzpicture}[scale=.9]
	
\coordinate (x1) at (-1,0);
\coordinate (x2) at (-1,1.5);
\coordinate (x3) at (1,0);
\coordinate (x4) at (1,1.5);
\coordinate (x5) at (-1.5,-1);

\fill [rounded corners = 8pt, blue!15!white, opacity = .5](1.85,-.45) to[out=180, in= -20] (.8,-.2) to[out = 110, in=-110] (.8,1.8)  to[out=20, in=160] (2.9,1.8) to[out=-70, in=70] (2.9,-.2) to[out=200, in=0] (1.85,-.45);
\draw [rounded corners = 8pt, blue!75!black, thick](1.85,-.45) to[out=180, in= -20] (.8,-.2) to[out = 110, in=-110] (.8,1.8)  to[out=20, in=160] (2.9,1.8) to[out=-70, in=70] (2.9,-.2) to[out=200, in=0] (1.85,-.45);

\fill [rounded corners = 8pt, blue!15!white, opacity = .5] (-1.3,-.3) -- (-1.3,1.6)  to[out=45, in=135] (1.3,1.6) -- (1.3,-.3) -- cycle;
\draw [rounded corners = 8pt, blue!75!black, thick] (-1.3,-.3) -- (-1.3,1.6)  to[out=45, in=135] (1.3,1.6) -- (1.3,-.3) -- cycle;

\fill [rounded corners = 8pt, blue!15!white, opacity = .5]   (1.15,.3) to[out=-70, in=50] (.7,-1.5) to[out=-150, in =-70] (-1.8,-.9)-- cycle;
\draw [rounded corners = 8pt, blue!75!black, thick]   (1.15,.3) to[out=-70, in=50] (.7,-1.5) to[out=-150, in =-70] (-1.8,-.9)-- cycle;

\fill [rounded corners = 8pt, blue!15!white, opacity = .5]   (-1.4,2)to [out = 160, in =50 ] (-2.6,1.5) -- (-2.9,.5) to[out=-90, in=180]  (-1.2,-1.2)  to[out=60, in =-90] (-.75,1.65) to[out=150, in =-20] (-1.4,2);
\draw [rounded corners = 8pt, blue!75!black, thick]   (-1.4,2)to [out = 160, in =50 ] (-2.6,1.5) -- (-2.9,.5) to[out=-90, in=180]  (-1.2,-1.2)  to[out=60, in =-90] (-.75,1.65) to[out=150, in =-20] (-1.4,2);

\draw [green!70!black, thick] (x1) -- (x2);
\draw [green!70!black, thick] (x3) -- (x4);

	\foreach \i in {1,...,5}{
		\fill (x\i) circle (2pt);}
	
\node at (0,1) {$F_2$};
\node at (-2,.6) {$F_1$};
\node at (2.1,.8) {$F_3$};
\node at (0,-1) {$F_4$};

	\end{tikzpicture}

	\caption{An untidy cycle violating~\ref{it:T1}}
	\label{fig:231a} 

	\end{subfigure}
	\hfill    
	\begin{subfigure}[b]{0.4\textwidth}
		\centering
		
			\begin{tikzpicture}[scale=.9]
			
		\coordinate (x1) at (0,0);
		\coordinate (x2) at (-1.5,0);
		\coordinate (x3) at (0,1.5);
		\coordinate (x4) at (0,-1.5);
		\coordinate (x5) at (2.5,1.5);
		\coordinate (x6) at (2.5,-1.5);
		
		\fill [rounded corners = 8pt, blue!15!white, opacity = .5](-.75,-.45) to[out=180, in= -20] (-1.8,-.2) to[out = 110, in=-110] (-1.8,1.7)  to[out=20, in=160] (.3,1.7) to[out=-70, in=70] (.3,-.2) to[out=200, in=0] (-.75,-.45);
		
		\draw [blue!75!black,rounded corners = 8pt, thick](-.75,-.45) to[out=180, in= -20] (-1.8,-.2) to[out = 110, in=-110] (-1.8,1.7)  to[out=20, in=160] (.3,1.7) to[out=-70, in=70] (.3,-.2) to[out=200, in=0] (-.75,-.45);
		
		\fill [rounded corners = 8pt, blue!15!white, opacity = .5](-.75,-1.95) to[out=180, in= -20] (-1.8,-1.7) to[out = 110, in=-110] (-1.8,.2)  to[out=20, in=160] (.3,.2) to[out=-70, in=70] (.3,-1.7) to[out=200, in=0] (-.75,-1.95);
		
			\draw [blue!75!black,rounded corners = 8pt, thick](-.75,-1.95) to[out=180, in= -20] (-1.8,-1.7) to[out = 110, in=-110] (-1.8,.2)  to[out=20, in=160] (.3,.2) to[out=-70, in=70] (.3,-1.7) to[out=200, in=0] (-.75,-1.95);

				\fill[blue!15!white, opacity = .5] (2.5,0) ellipse (.9cm and 1.9cm);
					\draw[blue!75!black,thick] (2.5,0) ellipse (.9cm and 1.9cm);

		\fill[blue!15!white, opacity = .5] (1.25,1.5) ellipse (1.6cm and .8cm);
		\draw[blue!75!black,thick] (1.25,1.5) ellipse (1.6cm and .8cm);
	
		\fill[blue!15!white, opacity = .5] (1.25,-1.5) ellipse (1.6cm and .8cm);
		\draw[blue!75!black,thick] (1.25,-1.5) ellipse (1.6cm and .8cm);
		
	\draw [green!70!black, thick] (x1) -- (x2);
		
		\foreach \i in {1,...,6}{
			\fill (x\i) circle (2pt);}
			
				\draw [thick, shorten <=-1cm,shorten >=-1cm] (x4) --(x5);
				
				\node at (3.4,2.4) {$f$};
			
			\end{tikzpicture}
			
				\caption{Same for~\ref{it:T2}}
				\label{fig:231b}

		\end{subfigure}    
		\caption{}	\label{fig:23}
		\vspace{-1em}
	\end{figure} 
Now it might be tempting to define a linear system to have large Girth if it contains 
no tidy cycles of copies of low order. However, there is one more phenomenon 
we did not take into account yet. 

\begin{example} \label{exmp:958}
	Return to the cycle of copies $\ccA$ considered in Example~\ref{exmp:939}.
	Take any vertices $y'\in e'$ and $y''\in e''$ distinct from $x$.
	Now $\ccA_\star=F_1y'F_2xF_3y''$ is a tidy cycle of 
	copies. Indeed,~\ref{it:T1} holds due to the absence of edge connectors 
	while the failure of~\ref{it:T2} would require the existence of an edge $f\in E(H)$
	containing all three of $x$, $y'$, and $y''$, which is absurd.    
\end{example}

In such a situation we shall say that $F_1$ is a master copy of $\ccA_\star$. The intuitive
reason for this is that everything of relevance happens within this copy.
Treating the edges $e'$ and~$e''$ for the moment as if they were copies, we can 
form the cycle $F_1y'e'xe''y''$, which lives completely in its master copy $F_1$. 
Extending $e'$ and $e''$ to the real copies $F_2$ and $F_3$ only conceals
this situation. 

Before making this precise in Definition~\ref{dfn:master} below we introduce  
a notational device for letting edges play the r\^{o}les of copies. 
With every edge $e\in E(H)$ we associate the subhypergraph $e^+=(e, \{e\})$ of $H$. 
It will be convenient to write $E^+(H)=\{e^+\colon e\in E(H)\}$.
Now $\bigl(H, E^+(H)\bigr)$ is already a legitimate linear system 
and it will be sensible to investigate its Girth (see Lemma~\ref{lem:Gth-gth} below). 

More interestingly, however, with every linear system $(H, \ccH)$ we associate the 
system $(H, \ccH^+)$ defined by $\ccH^+=\ccH\cup E^+(H)$. 
Whenever such systems occur, we call the members of $E^+(H)$ {\it edge copies},
while the other members of $\ccH^+$ will be referred to as {\it real copies}.
\index{edge copy}
\index{real copy}
When we want to direct attention to the fact that the linear systems 
we deal with are of the form $(H, \ccH^+)$, we call them {\it extended linear 
systems}. 
\index{extended linear system}
Let us now return to the question what master copies actually are.
 
\begin{dfn} \label{dfn:master}
	Given a cycle of copies $\ccC=F_1q_1\ldots F_nq_n$ in an extended 
	linear system $(H, \ccH^+)$ as well as a copy $F_\star\in\{F_1, \ldots, F_n\}$ 
	we call $F_\star$ a {\it master copy} of $\ccC$ (see Figure~\ref{fig:42}), 
	if there is a family 
		\[
		\bigl\{f_i\in E(F_\star)\colon i\in \ZZ/n\ZZ \text{ and } F_i\ne F_\star\bigr\}
	\]
		of edges such that the cyclic sequence $\ccD$ obtained from $\ccC$ upon replacing 
	every copy~$F_i\ne F_\star$ by the edge copy $f^+_i$ is again a cycle 
	of copies. When passing from $\ccC$ to $\ccD$ we say that the copies $F_i$ distinct 
	to $F_\star$ get {\it collapsed} to the edge copies $f_i^+$.
	\index{master copy}
	\index{collapse}
\end{dfn}  

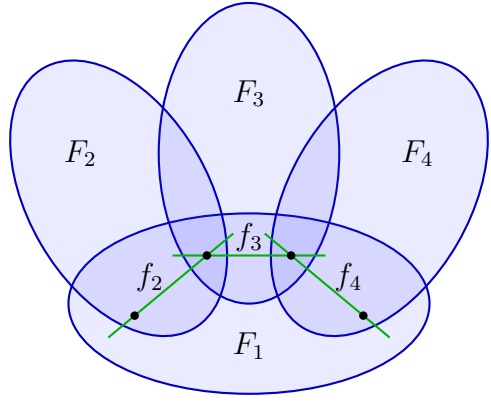
\begin{figure}[ht]
	\centering	
			\begin{tikzpicture}[scale=.8]
	
	\def\w{2.5};
	\def\h{1.5};
	
	\fill[blue!80!white, opacity = .1] (0,0) ellipse (3 cm and \h cm);
	\draw[blue!75!black,thick] (0,0) ellipse (3 cm and \h cm);
	
	\fill[blue!80!white, opacity = .1] (0,2.5) ellipse (\h cm and \w cm);
	\draw[blue!75!black,thick] (0,2.5) ellipse (\h cm and \w cm);
	
	\fill[blue!80!white, opacity = .1, rotate=30] (-1,2.6) ellipse (\h cm and \w cm);
	\draw[blue!75!black,thick, rotate = 30] (-1,2.6) ellipse (\h cm and \w cm);
	
	\fill[blue!80!white, opacity = .1, rotate = -30] (1,2.6) ellipse (\h cm and \w cm);
	\draw[blue!75!black,thick, rotate = -30] (1,2.6) ellipse (\h cm and \w cm);
	
	\coordinate (a) at (-1.9,-.2);
	\coordinate (b) at (-.7,.8);
	\coordinate (c) at (.7,.8);
	\coordinate (d) at (1.9,-.2);

	\draw [green!70!black, shorten <= -13pt, shorten >= -13pt, thick](a) -- (b);
	
	\draw [green!70!black, shorten <= -13pt, shorten >= -13pt, thick](b) -- (c);
	
	\draw [green!70!black, shorten <= -13pt, shorten >= -13pt, thick](c) -- (d);

		\foreach \i in {a,b,c,d}{
			\fill (\i) circle (2pt);}
						
	\node at (0,-.7) {$F_1$};
	\node at (0,3.5) {$F_3$};
	\node at (-2.8,2.5) {$F_2$};
	\node at (2.8,2.5) {$F_4$};
	
	\node at (0,1.1) {$f_3$};
	\node at (-1.65,.44) {$f_2$};
	\node at (1.65, .44) {$f_4$};

			\end{tikzpicture}
			
				\caption{Master copy $F_1$}
				\label{fig:42}

	\end{figure} 
Notice that a master copy $F_\star$ is allowed to appear multiple times on the cycle $\ccC$. 
If there are two or more occurrences of $F_\star$ in $\ccC$, then the collapsing does not 
change this situation.   
Now we finally reach our Girth concept applying to extended linear systems.
\index{$\Gth$}

\goodbreak

\begin{dfn}\label{dfn:Girth}
	If $(H, \ccH^+)$ is an extended linear system and $g, n\in\NN$, then 
	\begin{enumerate}[label=\alabel]
		\item $\Gth(H, \ccH^+)>(g, n)$ means that every tidy cycle of copies $\ccC$ 
			in~$\ccH^+$ satisfying $h(\ccC)\le (g, n)$ has a master copy 
		 \item and $\Gth(H, \ccH^+)>g$ means that every tidy cycle of copies $\ccC$ 
			in~$\ccH^+$ with $\ord{\ccC}\le g$ possesses a master copy.
	\end{enumerate}
\end{dfn}

Since the length of a cycle of copies is at most twice its order, 
$\Gth(H, \ccH^+)>g$ is equivalent to $\Gth(H, \ccH^+)>(g, 2g)$.
We proceed with two simple facts that follow immediately from our definitions. 

\begin{fact}\label{rem:1200}
	Master copies are always real copies. 
\end{fact}

\begin{proof}
	If $F_\star$ is a master copy of the cycle 
	$\ccC=F_1q_1\ldots F_nq_n$, then there needs to exist an index $i\in\ZZ/n\ZZ$ 
	with $F_i=F_\star$.
	Due to~\ref{it:L1} we have $F_{i+1}\ne F_\star$ and, therefore, $F_{i+1}$ is collapsible 
	to an edge copy $f_{i+1}^+$ with $f_{i+1}\in E(F_\star)$. 
	Now if $F_\star$ is an edge copy, 
	then its only edge is $f_{i+1}$ and the collapsed cycle 
	violates~\ref{it:L1}.
\end{proof}

\begin{fact}\label{rem:5027}
	If $(H, \ccH^+)$ is an extended linear system with $\Gth(H, \ccH^+)>(2, 2)$
	and $\ccC=F_1q_1\ldots F_nq_n$ denotes a cycle of copies in $(H, \ccH^+)$
	having a master copy $F_\star$, then the edges $f_i$ exemplifying this state
	this affair satisfy $f_i\in E(F_i)\cap E(F_\star)$. 
\end{fact}

\begin{proof}
	Since Definition~\ref{dfn:master} demands $f_i\in E(F_\star)$, we only have 
	to check $f_i\in E(F_i)$, which is clear in the special case $F_i=f_i^+$. 
	So we may assume
	$F_i\ne f_i^+$, whence $\ccD = F_i q_i f_i^+ q_{i-1}$
	is a cycle of copies. If one of the connectors $q_{i-1}$, $q_i$ is an edge, 
	then it has to be equal to $f_i$ and $f_i\in E(F_i)$ follows. If both of these
	connectors are vertices, then $\ccD$ is tidy, $h(\ccD) = (2, 2)$, and 
	by the previous fact $F_i$ is a master copy of $\ccD$. 
	Moreover, the edge copy $f_i^+$ can only be collapsed to itself.
\end{proof}

Next we check that for the system of edge copies Girth is essentially 
the same as ordinary girth.
  
\begin{lemma} \label{lem:Gth-gth}
	If $H$ is a linear hypergraph and $g\ge 2$, then 
		\[
		\gth(H) > g 
		\,\,\, \Longleftrightarrow \,\,\,
		\Gth\bigl(H, E^+(H)\bigr) >(g, g)
		\,\,\, \Longleftrightarrow \,\,\,
		\Gth\bigl(H, E^+(H)\cup\{H\}\bigr) >g\,.
	\]
	\end{lemma}

\begin{proof}
	Clearly the last condition implies the middle one. Next we assume the middle 
	condition and intend to derive the first. Suppose contrariwise that  
	$e_1x_1\ldots e_nx_n$ is an $n$-cycle in $H$ for some $n\in [2, g]$, chosen 
	in such a way that $n$ is as small as possible. Now $\ccC=e_1^+x_1\ldots e_n^+x_n$
	is a cycle of copies in $E^+(H)$ and the minimality of $n$ shows that $\ccC$ is tidy.
	Due to $h(\ccC)=(n, n)\le (g, g)$ there has to exist a master copy of $\ccC$,
	contrary to Fact~\ref{rem:1200}. 
	
	Thus it remains to prove that assuming $\gth(H) > g$ we can return 
	to 
		\[
		\Gth\bigl(H, E^+(H)\cup\{H\}\bigr)>g\,.
	\]
		Let $\ccC$ denote a tidy cycle of copies in $E^+(H)\cup\{H\}$ with $\ord{\ccC}\le g$.
	If $\ccC$ involves the copy~$H$, then~$H$ is a master copy of $\ccC$ and we are done.
	Now suppose that $\ccC$ is of the form $e_1^+q_1\ldots e_n^+q_n$.
	Notice that~$q_i$ cannot be an edge for any $i\in \ZZ/n\ZZ$, for then~\ref{it:L4} 
	would imply $e_i=q_i=e_{i+1}$, contrary to~\ref{it:L1}. 
	Thus all connectors of $\ccC$ are vertices, all indices are pure and, consequently,  
	we have $n=\ord{\ccC}\le g$. But now Fact~\ref{fact:231a} applied to   
   $e_1q_1\ldots e_nq_n$ yields $e_1=\dots=e_n$, which is absurd.
\end{proof}	

The lemma that follows relates $\Gth$ to concepts introduced earlier. 

\begin{lemma} \label{lem:Gth22}
	If a linear system $(H, \ccH)$ has strongly induced copies with clean intersections,
	then $\Gth(H, \ccH^+) > (2, 2)$.  
\end{lemma}

\begin{proof}
	Consider any cycle of two copies $\ccC=F_1q_1F_2q_2$ in $\ccH^+$.
	Due to~\ref{it:L2} and~\ref{it:L4} it cannot be the case that both connectors are edges.
	
	Suppose next that exactly one of them, say $q_1$, is an edge, while $q_2$ is a vertex. 
	If one of~$F_1$ or~$F_2$ is an edge copy, say~$F_1=f^+$, then $q_2\in f=q_1$
	shows that~$\ccC$ fails to be tidy. If both~$F_1$ and~$F_2$ are real copies, then
	they have the edge $q_1$ in common and, as their intersection is clean, they 
	have nothing else in common. Thus we have again $q_2\in q_1$ and $\ccC$ is not tidy. 
	
	Finally, we consider the case that both $q_1$ and $q_2$ are vertices. By the linearity 
	of $H$, at least one copy of~$\ccC$, say $F_1$, is a real copy. We shall prove
	that~$F_1$ is a master copy of~$\ccC$. If $F_2=f^+$ is an edge copy, then $F_1\Str H$ 
	yields $f\in E(F_1)$ and we just need to collapse $F_2$ to itself. If $F_2$ is a real
	copy, then the assumption that $F_1$ and $F_2$ intersect cleanly shows that these copies 
	have an edge $f$ in common. Now we can collapse $F_2$ to $f^+$.   
\end{proof} 

It may be observed that in the last case if both $F_1$ and $F_2$ are real copies, 
then both of them 
are master copies. Thus a cycle of copies of length $2$ can have two master copies and 
part~\ref{it:1522a} of the following lemma elaborates on this fact.  
We also include a part~\ref{it:1522b} stating that longer cycles of copies can 
have at most one master copy. This observation is never going to be used in the sequel, but 
it may help to explain why cycles consisting of two copies will sometimes require a special
treatment in later arguments. 
   
\begin{lemma}\label{lem:1522}
	Let $\ccC$ be a tidy cycle of copies in an extended linear 
	system $(H, \ccH^+)$ with $\Gth(H, \ccH^+)>(2, 2)$.
	\begin{enumerate}[label=\alabel]
		\item\label{it:1522a} If $\len{\ccC}=2$, then every real copy of $\ccC$ is a master copy. 
		\item\label{it:1522b} If $\len{\ccC}\ge 3$, then $\ccC$ has at 
			most one master copy.  
	\end{enumerate}
\end{lemma}   

\begin{proof}
	For the verification of~\ref{it:1522a} let $\ccC=F_1q_1F_2q_2$. 
	Due to ${\Gth(H, \ccH^+)>(2, 2)\ge h(\ccC)}$ we know that $\ccC$ has a master copy. 
	Suppose that $F_1$ is such a master copy and that $F_2$ is collapsible to $f^+$.
	We need to prove that if $F_2$ is a real copy, then $F_2$ is a master copy of $\ccC$
	as well. Since $F_2\ne f^+$, this can be seen by looking at the cycle $f^+q_1F_2q_2$.
	
	Proceeding with~\ref{it:1522b} we write $\ccC=F_1q_1\ldots F_nq_n$ and assume 
	contrariwise that both $F_1$ and $F_i$ are master copies of $\ccC$, where $i\in [2, n]$
	satisfies $F_1\ne F_i$.
	According to Fact~\ref{rem:5027} there are edges $f, f'\in E(F_1)\cap E(F_i)$ 
	such that the master copy $F_i$ allows to collapse $F_1$ to~$f^+$ 
	and, similarly, $F_1$ allows to collapse $F_i$ to $(f')^+$.
	If $f'\ne f$ then $F_1fF_if'$ is a tidy cycle of copies without a master copy, contrary
	to $\Gth(H, \ccH^+)>(2, 2)$. 
	
	Thus we have $f=f'$.
	Let us now look at the connectors $q_1$, $q_n$, $q_{i-1}$, and $q_i$. 
	Due to $n=\len{\ccC}\ge 3$
	their number is at least $3$. Moreover, the collapsibility of $F_1$ and $F_i$ to $f^+$
	shows that all vertices among them are in $f$, while all edges among them are equal to $f$.
	So if $f\in \{q_1, q_n, q_{i-1}, q_i\}$ we get a contradiction to~\ref{it:T1}, while 
	$f\not \in \{q_1, q_n, q_{i-1}, q_i\}$ causes the failure of~\ref{it:T2}. 
	This contradiction proves that $\ccC$ has indeed at most one master copy. 
\end{proof}

We conclude this subsection with two lemmata that will help us later to analyse 
the Girth of systems arising by partite constructions. 

\begin{lemma}\label{lem:1447}
	Let $\ccC$ be a tidy cycle of copies in an extended linear system $(H, \ccH^+)$.
	If 
		\[
		h(\ccC)<\Gth(H, \ccH^+)\,, 
	\]
		then all connectors of $\ccC$ are vertices. 
\end{lemma}	

\begin{proof}
	Let $\ccC=F_1q_1\ldots F_nq_n$ be such a tidy cycle of copies. 
	Definition~\ref{dfn:Girth} tells us that it has a master copy $F_\star$.
	Let $i\in\ZZ/n\ZZ$ be an index for which we want to prove that $q_i$ 
	is a vertex. By~\ref{it:L1} at least one of its neighbouring copies $F_i$
	and $F_{i+1}$ is distinct from $F_\star$ and by symmetry it suffices to 
	treat the case $F_{i}\ne F_\star$. Let $f_i$ be the edge $F_i$ is collapsible 
	to. Due to~\ref{it:L4} every edge among the connectors $q_i$, $q_{i-1}$ 
	is equal to $f_i$ and by~\ref{it:L3} every vertex among them belongs to $f_i$.
	So by~\ref{it:L2} it cannot be the case that both connectors are edges and~\ref{it:T1}
	tells us that $i$ cannot be mixed in $\ccC$. The only remaining case is that 
	both of $q_i$ and $q_{i-1}$ are vertices.    
\end{proof}

\begin{lemma} \label{lem:949}
	Let $\ccC=F_1q_1\ldots F_nq_n$ be a tidy cycle of copies of length $n\ge 3$ in 
	an extended linear system $(H, \ccH^+)$. Suppose that for some set of indices
	$K\subseteq \ZZ/n\ZZ$ we are given a family of edge copies $\{f_k^+\colon k\in K\}$.
	\begin{enumerate}[label=\alabel]
		\item\label{it:949a} If the cyclic sequence $\ccD$ obtained from $\ccC$ upon 
			replacing $F_k$ by $f^+_k$ for every $k\in K$ has the 
			properties~\ref{it:L2}\,--\,\ref{it:L4}, then~$\ccD$ satisfies~\ref{it:L1} 
			as well and, hence, it is a tidy cycle of copies. 
		\item\label{it:949b} Moreover, in this case every master copy of $\ccD$, if there exists 
			any, is a master copy of $\ccC$ as well. 
	\end{enumerate}
\end{lemma}  

\begin{proof}
	If part~\ref{it:949a} fails, then $\ccD$ has an edge copy $e^+$ occurring twice
	in consecutive positions. Due to $n\ge 3$, there are three connectors next to 
	these identical edge copies, say $q_i$, $q_{i+1}$, and $q_{i+2}$. As~\ref{it:T2}
	rules out the case $M(e)\supseteq \{i, i+1, i+2\}$, at least one of these connectors must 
	be an edge and by~\ref{it:L4} this edge can only be $e$. By~\ref{it:L2} the edge $e$
	occurs exactly once among $q_i$, $q_{i+1}$, and $q_{i+2}$	while the other two of these 
	connectors are vertices. By~\ref{it:L3} these vertices are in $e$, contrary
	to $\ccC$ satisfying~\ref{it:T1}. This concludes the proof of part~\ref{it:949a}.
	
	Moving on to part~\ref{it:949b} it is convenient to 
	write $\ccD=\wt{F}_1q_1\ldots \wt{F}_nq_n$, where 
		\[
		\wt{F}_i=\begin{cases}
						F_i &   \text{ if $i\not\in K$} \cr
						f_i^+ & \text{ if $i\in K$}
					\end{cases}
	\]
	for every $i\in\ZZ/n\ZZ$. 
	Suppose now that $F_\star\in\bigl\{\wt{F}_1, \ldots, \wt{F}_n\bigr\}$ 
	is a master copy of $\ccD$ and that the family of edges 
	$\ccF=\bigl\{e_i\in E(F_\star)\colon i\in\ZZ/n\ZZ \text{ and } \wt{F}_i\ne F_\star\bigr\}$ 
	exemplifies this. 
	Fact~\ref{rem:1200} shows that $F_\star$ is a real copy, 
	whence $F_\star\in \{F_1, \ldots, F_n\}$.  
	Thus the underlying index set of~$\ccF$ is a superset of $K$.
	By part~\ref{it:949a} the subfamily 
	$\bigl\{e_i\in E(F_\star)\colon i\in\ZZ/n\ZZ \text{ and } F_i\ne F_\star\bigr\}$
	of~$\ccF$ witnesses that $F_\star$ is indeed a master copy of~$\ccC$.
\end{proof}

\subsection{Semitidiness}
\label{sssec:157}

Consider a cycle of copies $\ccC=F_1xF_2eF_3q_3\ldots F_nq_n$ in an extended linear 
system $(H, \ccH^+)$, where the connector $x$ is a vertex, $e$ is an edge, and $x\in e$.
The last condition clearly violates~\ref{it:T1} and, therefore, it causes $\ccC$ to be 
untidy. So in case $\ord{\ccC}<\Gth(H, \ccH^+)$ Definition~\ref{dfn:Girth} does not 
tell us whether $\ccC$ has a master copy. Roughly speaking Lemma~\ref{lem:1342} below
asserts that if $\ccC$ is ``otherwise tidy'', then the existence of such a master
copy can nevertheless be proved. Any attempt to make this precise leads inevitably 
to the following concept. 

\begin{dfn}\label{dfn:1343}
	Let $\ccC=F_1q_1\ldots F_nq_n$ be a cycle of copies in an extended linear 
	system $(H, \ccH^+)$. Set 
		\[
		M(f)=\{i\in \ZZ/n\ZZ\colon q_i \text{ is a vertex belonging to } f\}
	\]
		for every edge $f\in E(H)$. 
	We say that $\ccC$ is {\it semitidy} if it has the following    
	two properties. 
	\index{semitidy}
		\begin{enumerate}[label=\upshape{($S\arabic*$)}]		
		\item\label{it:S1} If for some $i\in \ZZ/n\ZZ$ the connector $q_i$ is an edge, 
			then $M(q_i)\subseteq\{i-1, i+1\}$ and $|M(q_i)|\le 1$.
		\item\label{it:S2} For every edge $f\in E(H)$ that fails to be a connector of $\ccC$
			there exists an index $i(\star)\in \ZZ/n\ZZ$ 
			with $M(f)\subseteq \{i(\star), i(\star)+1\}$. 
	\end{enumerate}
\end{dfn}	

Observe that every tidy cycle is, in particular, semitidy. Conversely, if $\ccC$ is semitidy
and satisfies~\ref{it:S1} in the stronger form that $M(q_i)=\vn$ holds for every edge
connector $q_i$, then~$\ccC$ is tidy.
Now the result we have been alluding to is the following. 

\begin{lemma}\label{lem:1342}
	Given an extended linear system $(H, \ccH^+)$ and $g\in \NN$, we have 
		\[
		\Gth(H, \ccH^+)>g
	\]
		if and only if every semitidy cycle of copies $\ccC$ with $\ord{\ccC}\le g$ possesses 
	a master copy. 
\end{lemma} 
 
The proof of this lemma is essentially straightforward and, admittedly, more lengthy 
than illuminating. We would therefore like to say a few words on why we care to acquire 
this knowledge. 

In Section~\ref{subsec:EPAG} we shall study a generalisation of $\Gth$ applicable to 
systems of pretrains and it turns out that the characterisation of $\Gth$ provided
by Lemma~\ref{lem:1342} fits better into this context. For instance, consider 
a system of pretrains with the special property that every wagon consists of a 
single edge. It is then desirable that its $\GTH$ introduced later coincides 
with the $\Gth$ of the corresponding system of hypergraphs. It will turn out that owing to 
Lemma~\ref{lem:1342} this is indeed the case. 

Now even if one has some trust that the characterisation of $\Gth$ provided here 
is really going to be relevant in the future, one may still ask why we bother with the proof 
of Lemma~\ref{lem:1342} rather than declaring it to be the official definition of $\Gth$.
This has the simple reason that the arguments in Section~\ref{subsec:PCAG} 
dealing with $\Gth$ increments obtainable by means of the partite construction method 
become much more transparent if we only have to exhibit master copies of  
cycles that are actually tidy and not just semitidy.     

The proof of Lemma~\ref{lem:1342} proceeds by induction on $h(\ccC)$ and thus it requires 
a way to detect that another cycle $\ccD$ satisfies $h(\ccD)<h(\ccC)$. For this purpose
we shall always utilise the following observation. 

\begin{fact}\label{fact:h}
	Let $n>i\ge 2$. If 
		\[
		\ccC = F_1q_1 \ldots F_nq_n 
			\quad \text{ and } \quad
		\ccD = F_1q_1 \ldots F_iq_i 
	\]
		are cycles of copies, then $h(\ccD) < h(\ccC)$. 
\end{fact}

\begin{proof}
	Since $\len{\ccD} < \len{\ccC}$ is clear, it suffices to check 
	that $\ord{\ccD}\le \ord{\ccC}$.
	Setting 
		\[
		\eta^\ccC_j=
			\begin{cases}
				1  & \text{ if $j$ is pure in $\ccC$,} \cr
				1/2 & \text{ if $j$ is mixed in $\ccC$} 
			\end{cases}
	\]
		we have 
		\begin{equation} \label{eq:234a}
		\ord{\ccC}=\eta^\ccC_1+ \dots+ \eta^\ccC_n\,.
	\end{equation}
		
	Let the numbers $\eta^\ccD_1, \ldots, \eta^\ccD_i$ be defined similarly with respect 
	to $\ccD$. Since $\eta^\ccC_j=\eta^\ccD_j$ holds for every $j\in [2, i]$, we have 
		\begin{align*}
		\ord{\ccD}&=(\eta^\ccD_2+\dots +\eta^\ccD_{i})+\eta^\ccD_1 \\
		&\le
		(\eta^\ccC_2+\dots +\eta^\ccC_{i})+1 \\
		&\le
		(\eta^\ccC_2+\dots +\eta^\ccC_{i})+(\eta^\ccC_1+\eta^\ccC_n) \\
		&\le 
		\ord{\ccC}\,. \qedhere
	\end{align*}
	\end{proof}

\begin{proof}[Proof of Lemma~\ref{lem:1342}]
	As tidiness implies semitidiness, the backward implication is clear. 
	For the forward implication we suppose that $\ccC=F_1q_1\ldots F_nq_n$
	is a semitidy cycle of copies in an extended linear system $(H, \ccH^+)$
	with $\ord{\ccC}<\Gth(H, \ccH^+)$. We are to prove that $\ccC$ has a master copy. 
	Arguing by induction on $h(\ccC)$ we assume that every semitidy cycle of copies $\ccC'$
	in $(H, \ccH^+)$ with $h(\ccC')<h(\ccC)$ has a master copy. If $\ccC$ happens 
	to be tidy we just need to appeal to Definition~\ref{dfn:Girth}, so the interesting 
	case is that, despite being semitidy, $\ccC$ fails to be tidy. Up to symmetry this 
	is only possible if $q_1=x$ is a vertex, $q_2=e$ is an edge, and $x\in e$.
	
	Suppose first that $n=2$, i.e., that $\ccC=F_1xF_2e$.	 Now at least one of the 
	copies $F_1$, $F_2$ needs to be distinct from $e^+$. If, say, 
	$F_1\ne e^+$, then the cycle $F_1xe^+e$ exemplifies that $F_1$ is a master copy of $\ccC$. 
	
	So henceforth we may suppose that $n\ge 3$, whence $g\ge 2$. 
	Our goal is to prove that $F_3$ is a 
	master copy of $\ccC$. To this end we first check that 
		\begin{equation}\label{eq:1637}
		F_3 \text{ is a real copy. }
	\end{equation}     
		Otherwise~\ref{it:L4} implies $F_3=e^+$, so if $q_3$ is an edge we have $e=q_3$,
	contrary to~\ref{it:L2}, and if $q_3$ is a vertex we have $|M(e)|\ge 2$, 
	contrary to~\ref{it:S1}. This proves~\eqref{eq:1637}. 
	
	\smallskip
	
	{\it \hskip2em First Case. $F_1\ne F_3$}
	
	\smallskip
	 
   Now $\ccD=F_1xF_3q_3\ldots F_nq_n$ is a semitidy cycle of copies 
   with $h(\ccD)<h(\ccC)$, so by our induction hypothesis it has a master copy. 
   
   Assume first that $F_3$ fails to be such a master copy. Due to~\eqref{eq:1637}
   and Lemma~\ref{lem:1522}\ref{it:1522a} we know that $\ccD$ consists of at least 
   three copies, 
   whence $n\ge 4$. Moreover, using the true master copy of $\ccD$ we can collapse $F_3$
   to an appropriate edge copy $f^+$. Since we already have $x\in f$ and $\ccC$ is semitidy,
   it cannot be the case that $q_3$ is a vertex. But now only the case $q_3=f$ remains
   and $x\in f$ contradicts~\ref{it:S1}.   
   
   We have thereby proved that $F_3$ is a master copy of $\ccD$. This fact allows us to 
   collapse, with the exception of $F_2$, all copies of $\ccC$ that are distinct from $F_3$. 
   Provided we can finally collapse $F_2$ to $e^+$ this establishes that $F_3$ is 
   indeed a master copy of $\ccC$. The only reason why this last collapse could be 
   illegal is that it might cause the resulting ``cycle'' to violate~\ref{it:L1}. 
   Owing to~\eqref{eq:1637} this can only occur if we have just collapsed~$F_1$ to~$e^+$. 
   In this case $q_n$ cannot be a vertex, because this would 
   imply $\{1, n\}\subseteq M(e)$, contrary to~\ref{it:S1}. But~$q_n$ cannot be an 
   edge either, for then~\ref{it:L4} implies $q_n=e$, contrary to~\ref{it:L2}. 
   Altogether,~$F_3$ is the desired master copy of $\ccC$.
   	
	\smallskip
	
	{\it \hskip2em Second Case. $F_1=F_3$}
	
	\smallskip
	
	Notice that this is only possible if $n\ge 4$. It will be convenient to set $F=F_1=F_3$. 
	Now~$\ccC$ splits into the shorter cycles 
		\[
		\ccA=FxF_2e
		\quad \text{ and } \quad
		\ccB=Fq_3\ldots F_nq_n\,,
	\]
		and it suffices to prove that $F$ is a common master copy of $\ccA$ and $\ccB$.  
	By collapsing $F_2$ to~$e^+$ we see that $F$ is indeed a master copy of $\ccA$.
	
	Assume towards contradiction that $F$ fails to be a master copy of $\ccB$.
	Owing to~\eqref{eq:1637} and Lemma~\ref{lem:1522}\ref{it:1522a} we have $n\ge 5$. 
	Furthermore, Fact~\ref{fact:h} yields $h(\ccB)<h(\ccC)$ and, as one easily 
	checks, $\ccB$ is semitidy. So our induction hypothesis shows that $\ccB$ 
	has some master copy. 
	In particular, $F$ is collapsible to an appropriate edge copy $f^+$ in $\ccB$. 
	Every edge connector among $q_3$ and $q_n$ needs to be equal to $f$ and every 
	vertex connector among them needs to be incident with $f$. But in view of $n\ge 5$ 
	this yields a contradiction to the semitidiness of~$\ccC$.   
\end{proof}

\subsection{Orientation}\label{sssec:9902}

The notion of $\Gth$ allows us to formulate a strengthening of 
Theorem~\ref{thm:grth1} that will play a central r\^{o}le in our 
work.

\begin{quotation}
	For every integer $g\ge 2$ there exists a Ramsey construction 
	$\Omega^{(g)}$ that given an ordered $f$-partite hypergrsph $F$ 
	with $\gth(F)>g$ and a number of colours $r$ produces an ordered 
	$f$-partite system $\Omega^{(g)}_r(F)=(H, \ccH)$ satisfying
	$\ccH\lra (F)_r$ and $\Gth(H, \ccH^+)>g$.
\end{quotation}

For clarity we would like to point out that due to Lemma~\ref{lem:Gth-gth}
we have $\gth(H)>g$ in this situation, which shows that the existence
of $\Omega^{(g)}$ is indeed a much stronger result than Theorem~\ref{thm:grth1}.
In case one is willing to believe this statement without proof, one can now jump 
directly to Section~\ref{sec:paradise} and read why it implies Theorem~\ref{thm:1522} 
and Theorem~\ref{cor:19}---so all prerequisites 
for this deduction have been covered already. 

The construction $\Omega^{(2)}=\PC(\Rms, \CPL)$ has been presented
in the previous section, but we do not know yet that the systems $(H, \ccH)$ it 
delivers satisfy ${\Gth(H, \ccH^+)>2}$. We shall reach this knowledge in the next 
section when studying the $\Gth$ of pictures that arise in partite constructions
(see Corollary~\ref{cor:2201}).

In some sense, the existence of $\Omega^{(g)}$ is the statement we shall actually 
prove by induction on $g$, but the passage from $\Omega^{(g)}$ to $\Omega^{(g+1)}$
involves some further nested inductions whose intermediate stages deal with trains
(see Figure~\ref{fig:N3}). For a lack of an appropriate language we need to defer 
a more detailed description of our intended induction scheme to~\S\ref{sssec:Karo}.  \section{Girth in partite constructions}
\label{subsec:PCAG}

As we saw in Section~\ref{subsec:PC}, the partite construction method allows us to 
pass from strongly induced copies to copies with clean intersections 
(Lemma~\ref{lem:cleancap}) and clean intersections themselves are preserved under 
further applications of this method (Lemma~\ref{lem:clean-preserve}). Both facts can,
roughly speaking, be regarded as dealing with cycles consisting of two copies. The aim
of the present section is twofold. First, we want to improve the latter result by 
addressing cycles of order two, rather than cycles of length two. Second, we develop 
generalistions to arbitrary $\Gth$. 

\subsection{From \texorpdfstring{$(g, g)$}{(g, g)} to \texorpdfstring{$g$}{g}.}
\label{sssec:1510}
    
Let us return to the question raised in~\S\ref{sssec:vindya} whether there is anything 
to gain if one attempts to clean the clean partite lemma~$\CPL$ further. 
In the light of Lemma~\ref{lem:Gth22} we know that if $F$ denotes a linear $k$-partite 
{$k$-uniform} hypergraph,~$r$ signifies a number of colours, and $\CPL_r(F)=(H, \ccH)$,
then $\Gth(H, \ccH^+)>(2, 2)$. It turns out that if we run any partite
construction using $\CPL$ as our partite lemma, then the resulting final picture
$(\Pi, \ccP, \psi)$ has the stronger property $\Gth(\Pi, \ccP^+)> 2$.  
In fact we shall show the following more general result, whose special case $g=2$
corresponds to the aforementioned fact.  

\begin{prop}\label{prop:girthclean}
	Suppose for any $g\ge 2$ that $\Xi$ is a partite lemma 
	\begin{enumerate}
		\item[$\bullet$] applicable to $k$-partite $k$-uniform hypergraphs $B$ 
			with $\gth(B) >g$ 
		\item[$\bullet$] and delivering linear systems $\Xi_r(B)=(H, \ccH)$ 
		   with $\Gth(H, \ccH^+) > (g, g)$. 
	\end{enumerate}
	If $\Phi$ denotes any further Ramsey construction, then the construction 
	$\Theta=\PC(\Phi, \Xi)$ applies to all hypergraphs $F$ with $\gth(F) >g$ 
	and yields linear systems $\Theta_r(F)=(\Pi, \ccP)$ 
	satisfying 
		\[
		\Gth(\Pi, \ccP^+) > g\,.
	\]
	\end{prop}

The heart of the matter is, of course, the following picturesque statement that 
will allow us to carry an inductive proof along the partite construction. 

\begin{lemma} \label{lem:251}
	Suppose that $g\ge 2$, that $(G, \ccG)$ is an arbitrary system of $k$-uniform 
	hypergraphs, and that 
		\[
		(\Sigma, \ccQ, \psi_\Sigma)=(\Pi, \ccP, \psi_\Pi) \conc (H, \ccH)
	\]
		holds for two pictures $(\Pi, \ccP, \psi_\Pi)$ and $(\Sigma, \ccQ, \psi_\Sigma)$ 
	over $(G, \ccG)$ 
	as well as a linear $k$-partite system $(H, \ccH)$. If 
		\[
		\Gth(\Pi, \ccP^+) > g
		\quad \text{ and } \quad
	 	\Gth(H, \ccH^+) > (g, g)\,, 
	\]
		then $\Gth(\Sigma, \ccQ^+) > g$.
\end{lemma}

\begin{proof}
	Consider a tidy cycle of copies $\ccC=F_1q_1\ldots F_nq_n$ 
	in $(\Sigma, \ccQ^+)$ with ${\ord{\ccC}\le g}$. We are to prove that $\ccC$ has a 
	master copy. 
	
	By a {\it segment} we shall mean 
	a subsequence $I=F_iq_i\ldots F_j$ of $\ccC$ starting and ending with a 
	copy and for which either 
		\begin{enumerate}[label=\rmlabel]
		\item\label{it:1313i} there exists a standard copy $(\Pi_\star, \ccP_\star^+)$
			with $F_i, \ldots, F_j\in \ccP_\star^+$, 
		\item\label{it:1313ii} or $i=j$ and $F_i$ is an edge copy corresponding to
			an edge of $H$.
	\end{enumerate}  
		
	In the former case we call the copy $\Pi_\star^e\in\ccH$ extended by 
	the standard copy $\Pi_\star$ the {\it leader} of $I$ and in the latter
	case $F_i$ itself is considered to be its own {\it leader}. Observe that the leader 
	of a segment is always a member of $\ccH^+$. A {\it segmentation}
	of $\ccC$ is a sequence of the form
		\begin{equation}\label{eq:1338}
		\ccC=I_1r_1\ldots I_tr_t\,,
	\end{equation}
		where $I_1, \ldots, I_t$ are segments 
	and $\{r_1, \ldots, r_t\}\subseteq \{q_1, \ldots, q_n\}$.
	Such segmentations exist, for each of the copies $F_1, \ldots, F_n$ forms 
	a segment on its own. From now on we let the 
	segmentation~\eqref{eq:1338} be chosen in such a way that $t$ is minimal.
	
	Denote the leaders of $I_1, \ldots, I_t$ by $\Pi^e_1, \ldots, \Pi^e_t$, 
	respectively. If $\Pi^e_i=\Pi^e_{i+1}$ holds for some $i\in [t-1]$, then 
	$I_ir_iI_{i+1}$ is again a segment, contrary to the 
	minimality of $t$. Similarly, we may assume by cyclic symmetry that 
	in case $t\ge 2$ we have $\Pi_1^e\ne \Pi^e_t$. 
	
	If $t=1$, then the segment $I_1$ contains at least two copies from $\ccC$,
	wherefore $\Pi^e_1\in\ccH$. The standard copy 
	$\Pi_1$ extending $\Pi^e_1$ hosts the entire cycle $\ccC$ and 
	$\Gth\bigl(\Pi_1, \ccP_1^+\bigr) > g$ yields the desired master copy. 
	
	Now suppose $t\ge 2$ and, consequently, $\Pi^e_i\ne \Pi^e_{i+1}$
	for all $i\in\ZZ/t\ZZ$. It follows that the vertices and edges $r_1, \ldots, r_t$
	belong to $H$, whence
	\[
		\ccD=\Pi^e_1r_1\ldots \Pi^e_tr_t
	\]
		is a cycle of copies in $(H, \ccH^+)$. 
	
	\begin{clm}\label{clm:1826}
		The length of $\ccD$ is at most $g$.
	\end{clm}
	
	\begin{proof}
		Define the numbers $\eta_1, \ldots, \eta_n$ by
				\[
			\eta_i=
				\begin{cases}
					1  & \text{ if $i$ is pure in $\ccC$,} \cr
					1/2 & \text{ if $i$ is mixed in $\ccC$.} 
				\end{cases}
		\]
				For each $\tau\in\ZZ/t\ZZ$ 
		let~$\theta_\tau$ be the sum of all $\eta_k$ for which $F_k$ belongs
		to the segment $I_\tau$. Owing to 
				\[
			\theta_1+\dots+\theta_t 
			= 
			\eta_1+\dots+\eta_n 
			= 
			\ord{\ccC} \le g
		\]
				it suffices to show $\theta_1, \ldots, \theta_t\ge 1$. Fix an arbitrary index
		$\tau\in\ZZ/t\ZZ$. If the segment $I_\tau$ contains at least two copies 
		from~$\ccC$, then $\theta_\tau\ge 2/2=1$ is clear. It remains to deal with the 
		case that the segment $I_\tau$ consists of 
		a single copy, say $F_k$. Observe that $q_{k-1}=r_{\tau-1}$, $q_k=r_\tau$,
		and $\theta_\tau=\eta_k$. There is no problem if $k$ is pure in $\ccC$,
		so assume towards a contradiction that~$k$ is mixed. In other words, one of
		$q_{k-1}$ and $q_k$ is a vertex, the other one is an edge, and both belong 
		to $F_k$. Since the copies in $\ccQ^+$ cross the music lines of $\Sigma$ at
		most once, all this can only happen if the edge among $q_{k-1}$, $q_k$
		contains the vertex. But this option is ruled out by~\ref{it:T1}.	
	\end{proof}
	
	We infer $\ord{\ccD}\le g$, 
	and $h(\ccD)\le (g, g)<\Gth(H, \ccH^+)$. Moreover, the tidiness 
	of $\ccC$ implies that $\ccD$ is tidy as well. 
	Therefore 
		\begin{equation}\label{eq:1229}
		\text{$\ccD$ has a master copy,} 
	\end{equation}
		Lemma~\ref{lem:1447} discloses that the connectors $r_1, \ldots, r_t$ are 
		vertices, and by Fact~\ref{rem:1200} the copies in $\ccH$ need to be real. 
	
	\begin{clm}\label{clm:1230}
		We may assume that $n\ge t\ge 3$.
	\end{clm}
	
	\begin{proof}
		Let us first look at the case that $n=2$ and, consequently, $t=2$. 
		By symmetry and~\eqref{eq:1229} we may suppose that $\Pi^e_1$ 
		is a master copy of $\ccD$, i.e., that there 
		is an edge $f\in E(\Pi^e_1)$ such that $\ccE=\Pi^e_1r_1f^+r_2$ is a cycle of copies
		in $(H, \ccH^+)$. 
		Now we have $\ccC=F_1q_1F_2q_2$ and, if $F_2$ belongs to the segment with leader $\Pi^e_2$,
		then $F_1q_1f^+q_2$ is a cycle of copies witnessing that~$F_1$ is a master copy of $\ccC$.
		
		So henceforth we may assume that $n\ge 3$. If $t=2$, then $\ccD$ has a copy	 leading 
		a segment containing at least two copies of $\ccC$. Suppose that $\Pi^e_1$
		is such a copy and that~$\Pi_1$ denotes the standard copy extending it. 
		Now $\Pi^e_1$ needs to be a real copy of $\ccD$ and Lemma~\ref{lem:1522}\ref{it:1522a}
		reveals that, actually, $\Pi^e_1$ is a master copy of $\ccD$. So there is 
		some edge $f\in E(\Pi_1)$
		such that $\Pi_e^1r_1f^+r_2$ is a cycle of copies in $(H, \ccH^+)$. 
		Since $\ccC$ satisfies~\ref{it:T2} with respect to~$f$, the connectors $r_1$ 
		and $r_2$ occur in consecutive positions on $\ccC$ and by our choice of $\Pi^e_1$
		it follows that the segment $I_2$ contains a unique copy of $\ccC$. 
		Now by Lemma~\ref{lem:949}\ref{it:949a} we can collapse this copy of $\ccC$ to $f^+$, 
		thus obtaining a new cycle which lives completely in the standard copy $\Pi_1$.
		Since $\Gth(\Pi_1, \ccP_1^+)>g$, the new cycle has a master copy and by 
		Lemma~\ref{lem:949}\ref{it:949b} it follows that $\ccC$ has a master copy as well. 
\end{proof}
	
	Utilising~\eqref{eq:1229} we now pick a copy $\Pi^e_\star$ of $\ccD$ 
	and a family of edges 
		\[
		\bigl\{f_i\in E(\Pi_\star^e)\colon i\in \ZZ/t\ZZ\text{ and }\Pi^e_i\ne \Pi^e_\star\bigr\}
	\]
		exemplifying that $\Pi^e_\star$ is a master copy of $\ccD$.
	Denote the standard copy extending $\Pi^e_\star$ by~$\Pi_\star$.
		
	\begin{clm} \label{clm:1629}
		If $i\in \ZZ/t\ZZ$ is an index with $\Pi^e_i\ne \Pi^e_\star$, then the 
		segment $I_i$ consists of a single copy $\wh{F}_i$.
	\end{clm}
	
	\begin{proof}
		Recall that $\ccC$ satisfies~\ref{it:T2}. Applying this fact to the 
		edge $f_i$ we learn that the vertices $r_{i-1}$ and $r_i$ occur in 
		consecutive positions on $\ccC$. Due to $t\ge 3$ this proves the claim.	
	\end{proof}
	
	Consider the cyclic sequence $\ccE$ obtained from $\ccC$
	upon collapsing every copy $\wh{F}_i$ that Claim~\ref{clm:1629} delivers
	to the corresponding edge copy $f_i^+$. Since $\ccC$ is tidy and $n\ge 3$, 
	Lemma~\ref{lem:949}\ref{it:949a} shows that $\ccE$ is again a tidy cycle 
	of copies. As this cycle belongs entirely to $(\Pi_\star, \ccP_\star^+)$, it 
	needs to have a master copy and owing to Lemma~\ref{lem:949}\ref{it:949b}
	we infer that $\ccC$ has a master copy as well. 
\end{proof}

\begin{proof}[Proof of Proposition~\ref{prop:girthclean}]
	On first sight it may not be obvious whether $\Theta_r(F)$ is  
	defined for every hypergraph $F$ with $\gth(F) >g$. Given any such 
	hypergraph $F$ as well as a number of colours $r$ we 
	set $\Phi_r(F)=(G, \ccG)$ and construct the corresponding picture 
	zero $(\Pi_0, \ccP_0, \psi_0)$. Every copy $F_0\in \ccP_0$ is isomorphic 
	to $F$, which by Lemma~\ref{lem:Gth-gth} implies
		\[
		\Gth\bigl(F_0, E^+(F_0)\cup\{F_0\}\bigr)>g\,.
	\]
		Since the copies in $\ccP_0$ are disjoint, it follows that $\Gth(\Pi_0, \ccP_0^+)>g$.
	As usual we take an enumeration $E(G)=\{e(1), \ldots, e(N)\}$ and start running the
	partite construction. 
	
	Consider any positive $\alpha\le N$ for which 
	\begin{enumerate}
		\item[$\bullet$] the picture $(\Pi_{\alpha-1}, \ccP_{\alpha-1}, \psi_{\alpha-1})$ 
			exists 
		\item[$\bullet$] and satisfies $\Gth(\Pi_{\alpha-1}, \ccP_{\alpha-1}^+) > g$. 
	\end{enumerate}
	
	A further application of Lemma~\ref{lem:Gth-gth} shows, in particular, that
	$\gth\bigl(\Pi^{e(\alpha)}_{\alpha-1}\bigr)>g$. 
	Thus the constituent $\Pi^{e(\alpha)}_{\alpha-1}$ can be plugged into the partite 
	lemma $\Xi$ and we obtain a linear system $\Xi_r(\Pi^{e(\alpha)}_{\alpha-1}\bigr)=(H, \ccH)$ 
	with $\Gth(H, \ccH^+)>(g, g)$.
	By Lemma~\ref{lem:251} the new picture
		\[
		(\Pi_{\alpha}, \ccP_{\alpha}, \psi_{\alpha})
		=
		(\Pi_{\alpha-1}, \ccP_{\alpha-1}, \psi_{\alpha-1})\conc (H, \ccH)
	\]
		satisfies $\Gth(\Pi_{\alpha}, \ccP_{\alpha}^+) > g$ again.
	
	This completes an induction on $\alpha$. In the last step we learn that 
	the final picture $(\Pi_N, \ccP_N, \psi_N)$ is well defined and, as desired, 
	that $\Gth\bigl(\Pi_{N}, \ccP_{N}^+\bigr) > g$.  
\end{proof}

\begin{cor}\label{cor:2201}
	The ordered $f$-partite Ramsey construction $\Omega^{(2)}=\PC(\Rms, \CPL)$ is applicable 
	to linear hypergraphs 
	and delivers linear systems $(H, \ccH)$ with $\Gth(H, \ccH^+)>2$. 
	In particular, there is a partite lemma producing such systems.
	\index{$\Omega^{(2)}$}
\end{cor}

\begin{proof}
	By Corollary~\ref{cor:CPL} and Lemma~\ref{lem:Gth22} the clean partite lemma $\CPL$ 
	satisfies the assumptions of Proposition~\ref{prop:girthclean} for $g=2$.
	\index{clean partite lemma}
\end{proof}

\subsection{From \texorpdfstring{$g-1$}{g-1} to \texorpdfstring{$(g, g)$}{({\it g, g})}}
\label{sssec:1514}

Imagine that we have a partite lemma $\Xi^{(4)-}$ applicable to \mbox{$k$-partite}, 
$k$-uniform hypergraphs $F$ with $\gth(F)>4$ and delivering systems of hypergraphs 
$(H, \ccH^+)$ with $\gth(H)>4$ and $\Gth(H, \ccH^+)>3$. 
Can we then derive the case $g=4$ of Theorem~\ref{thm:grth1} by means of the partite 
construction method? If so, can we even build Ramsey systems without the four-cycles 
of copies shown in Figure~\ref{fig:N1A}? As the next result shows,
the answers so these questions become affirmative when we resolve to utilise 
strongly induced copies vertically.

\begin{prop} \label{prop:1737}
	Let $\Omega$ denote a linear Ramsey construction delivering systems of strongly 
	induced copies and let $g\ge 3$ be an integer.
	Suppose that $\Xi$ refers to a partite lemma 
		\begin{enumerate}
		\item[$\bullet$] applicable to $k$-partite $k$-uniform hypergraphs $B$ with $\gth(B)>g$
		\item[$\bullet$] yielding systems of copies $\Xi_r(B)=(H, \ccH)$ with 
						\[	
				\gth(H)>g
				\quad \text{ and } \quad 
				\Gth(H, \ccH^+)>g-1\,.
			\]
				\end{enumerate}
		If $\Theta=\PC(\Omega, \Xi)$, then for every hypergraph $F$ with $\gth(F)>g$ 
	and every $r\in \NN$ the system $\Theta_r(F)=(H, \ccH)$ exists and satisfies 
	$\Gth(H, \ccH^+)>(g, g)$. 
\end{prop} 

The corresponding picturesque statement reads as follows. 

\begin{lemma}\label{lem:1753}
	Suppose $g\ge 3$ and that
		\[
		(\Sigma, \ccQ, \psi_\Sigma)
		=
		(\Pi, \ccP, \psi_\Pi)
		\conc 
		(H, \ccH)
	\]
		holds for two pictures $(\Pi, \ccP, \psi_\Pi)$ and $(\Sigma, \ccQ, \psi_\Sigma)$
	over a linear system $(G, \ccG)$ with strongly induced copies and a $k$-partite 
	$k$-uniform system $(H, \ccH)$. If 
		\[
		\Gth(\Pi, \ccP^+)>(g, g)\,,
		\quad 
		\gth(H) > g\,,
		\quad \text{ and } \quad 
		\Gth(H, \ccH^+) > g-1\,,
	\]
		then $\Gth(\Sigma, \ccQ^+)>(g, g)$.
\end{lemma}

The deduction of Proposition~\ref{prop:1737} from Lemma~\ref{lem:1753} is very similar 
to the proof of Proposition~\ref{prop:girthclean} based on Lemma~\ref{lem:251} and we 
omit the details.\footnote[1]{Strictly speaking, Lemma~\ref{lem:251} and 
Lemma~\ref{lem:1753} are the only results of this section that will be quoted later.} 
During the proof of Lemma~\ref{lem:1753}
it is helpful to keep in mind that a cycle of copies $\ccC$ satisfies 
$h(\ccC)\le (g, g)$ if and only if either $h(\ccC)=(g, g)$ or $\ord{\ccC}\le g-1$. 
In the former case, $\ccC$ has length $g$ and all connectors of~$\ccC$ are of the 
same type. On the whole, the proofs of Lemma~\ref{lem:251} and Lemma~\ref{lem:1753} are very similar and, therefore, 
it suffices to indicate the appropriate changes.   

\begin{proof}[Proof of Lemma~\ref{lem:1753}]
	Consider a tidy cycle of copies $\ccC=F_1q_1\ldots F_nq_n$ in $(\Sigma, \ccQ^+)$
	with $h(\ccC)\le (g, g)$. We need to prove that $\ccC$ possesses a master copy. 
	Define segments of $\ccC$, their leaders, and segmentations of $\ccC$ in the same 
	way as in the proof of Lemma~\ref{lem:251}. Again let 
		\[
		\ccC=I_1r_1\ldots I_tr_t
	\]
	be a segmentation of $\ccC$ such that $t$ is minimal. 
	As we are done otherwise, we may suppose 
	$t\ge 2$ and that the leaders of any two consecutive segments are distinct. Let 
		\[
		\ccD=\Pi^e_1r_1\ldots \Pi^e_tr_t
	\]
		denote the associated cycle of leaders, which lives in $(H, \ccH^+)$. The point 
	of Claim~\ref{clm:1826} was to verify $\ord{\ccD} < \Gth(H, \ccH^+)$, which 
	requires a different reasoning now. 
	
	\begin{clm}\label{clm:1831}
		We have $\ord{\ccD}\le g-1$.
	\end{clm} 
	
	\begin{proof}
		Repeated applications of Fact~\ref{fact:h} show $h(\ccD)\le h(\ccC)$, so the 
		claim could only fail if $h(\ccD)=h(\ccC)=(g, g)$, which we assume from now on. 
		In particular, this entails that every segment $I_\tau$ with $\tau\in\ZZ/t\ZZ$
		consists of a single copy and that the connectors of $\ccC$ and~$\ccD$ are 
		the same. In other words, we have $n=t=g$ and $q_i=r_i$ for all $i\in\ZZ/n\ZZ$.
		
		If all these connectors are edges, then $F_1$ contains two distinct edges 
		$q_1$ and $q_n$ from~$H$, contrary to the fact that $F_1$ crosses every music 
		line of $\Sigma$ at most once. 
		
		It remains to consider the case that all connectors are vertices. Now for every 
		$i\in \ZZ/n\ZZ$ the copy $F_i$ projects via $\psi_\Sigma$ to a copy $\wt{F}_i$
		in $\ccG^+$. This copy contains the distinct vertices $\psi_\Sigma(q_{i-1})$
		and $\psi_\Sigma(q_{i})$. Owing to $\wt{F}_i\Str G$ this leads to an 
		edge $\wt{e}_i$ 
		with $\psi_\Sigma(q_{i-1}), \psi_\Sigma(q_{i})\in \wt{e}_i\in E(\wt{F}_i)$.  
		But $G$ is linear, so $\wt{e}_i$ is actually the edge $H$ is projected to 
		by $\psi_\Sigma$. By looking at the inverse of the projection map one obtains 
		an edge $e_i\in E(F_i)\cap E(H)$ with $q_{i-1}, q_i\in e_i$. In view of 
		Fact~\ref{fact:231a} the only possibility how the cyclic sequence 
		$e_1q_1\ldots e_nq_n$ can fail to violate the assumption $\gth(H)>g$ is that
		we have $e=e_1=\dots =e_n$ for some edge $e\in E(H)$. 
		But now $q_1, \ldots, q_n\in M(e)$ and $n=g\ge 3$ contradict the tidiness 
		of $\ccC$. 
	\end{proof} 
	
	Mutatis mutandis the rest of the proof carries over. 
\end{proof}   \section{The extension process}
\label{subsec:EP}

This section deals with a Ramsey theoretic construction pioneered by
Ne{\v{s}}et{\v{r}}il and R\"odl, who utilised it for 
proving the $2$-uniform case of Theorem~\ref{thm:grth1} for $g=4$ in~\cite{NeRo4}. 
\index{Ne\v{s}et\v{r}il}
The process is applied there in a rather concrete manner, exploiting that 
every $C_4$-free bipartite graph~$F$ is expressible as an edge-disjoint union of stars, 
any two of which intersect in at most one vertex. Such a decomposition of the set of 
edges contains, of course, the same information as an equivalence relation on~$E(F)$ 
that regards any two edges of $F$ to be equivalent if they pertain to the same star. 
This perspective leads us to the concept of a pretrain and thereby to a rather general 
abstract framework for discussing the extension process. 
  
\subsection{Pretrains}
\label{sssec:1128}

A {\it pretrain} is a pair $(H, \equiv)$ consisting of a 
hypergraph $H$ and an equivalence relation $\equiv$ on $E(H)$.
\index{pretrain}
If $H$ is an ordered hypergraph, we call~$(H, \equiv)$ 
an {\it ordered pretrain}. 
\index{ordered pretrain}
Similarly one defines the notion of an {\it $f$-partite pretrain}.
\index{$f$-partite pretrain}
A {\it wagon} of a pretrain $(H, \equiv)$ is a subhypergraph of $H$ without isolated 
vertices whose set of edges forms an equivalence class of $\equiv$.
\index{wagon}

\begin{dfn} \label{dfn:1129}
	Given two pretrains~$(F, \equiv^F)$ and~$(H, \equiv^H)$, the former is said to be 
	a {\it subpretrain} of the latter if 
	\begin{enumerate}[label=\rmlabel]
		\item\label{it:spt1} $F$ is a subhypergraph of $H$ 
		\item\label{it:spt2} and any two edges of~$F$ are equivalent with respect to~$\equiv^F$ if 
			and only if they are equivalent with respect to~$\equiv^H$.
	\end{enumerate}
	\index{subpretrain}
\end{dfn}		
		
Condition~\ref{it:spt1} is the only one that will occasionally receive further specifications.
For instance, it may happen that $F$ is an induced subhypergraph of $H$ and 
then~$(F, \equiv^F)$ is said to be an {\it induced subpretrain} of~$(H, \equiv^H)$.
\index{induced subpretrain} 
If~$(F, \equiv^F)$ and~$(H, \equiv^H)$ are ordered pretrains, then for $(F, \equiv^F)$ 
to be an {\it ordered subpretrain} of~$(H, \equiv^H)$ we require~\ref{it:spt1} to hold
in the stronger sense that $F$ is an ordered subhypergraph of $H$. 
\index{ordered subpretrain}
A similar modification of~\ref{it:spt1} is required for {\it $f$-partite subpretrains}.
\index{$f$-partite subpretrains}

The demand~\ref{it:spt2} tells us that $\equiv^F$ is entirely determined by $F$ 
and $\equiv^H$, and thus we shall frequently suppress the equivalence relation when 
talking about subpretrains. Notice that when passing from~$(H, \equiv^H)$ 
to~$(F, \equiv^F)$ every wagon $W_H$ of $H$ either {\it vanishes} in the sense 
that $E(W_H)\cap E(F)=\varnothing$, or it {\it contracts} to a wagon $W_F$ of $F$ 
with $E(W_F)=E(W_H)\cap E(F)$. 
\index{vanishing (of wagons)}
\index{contraction (of wagons)}
  
It should be clear that every pretrain is a subpretrain of itself and that 
the subpretrain relation is transitive. 
  
A {\it system of pretrains} is a triple $(H, \equiv, \ccH)$ consisting of a pretrain 
$(H, \equiv)$ and of a collection~$\ccH$ of subpretrains of $(H, \equiv)$. 
\index{system of pretrains}
Moreover $\binom{(H, \equiv^H)}{(F, \equiv^F)}$ refers to the collection of all 
induced subpretrains of $(H, \equiv^H)$ {\it isomorphic} to~$(F, \equiv^F)$, where 
the notion of {\it pretrain isomorphism} is assumed to be self-explanatory. 
The specifiers $\nni$, $\pt$, and $\ff$ may be 
attached to~$\binom{(H, \equiv^H)}{(F, \equiv^F)}$
in the usual way. For $\ccH\subseteq \binom{(H, \equiv^H)}{(F, \equiv^F)}$ 
and $r\in \NN$ the partition relation 
\[
	\ccH\lra (F, \equiv^F)_r
\]
means that no matter how the edges of $H$ are coloured with $r$ colours, there will 
always exist a copy $(F_\star, \equiv^{F_\star})\in\ccH$ whose edges are the same colour. 

To aid the readers orientation we remark that the induced Ramsey theorem does also hold 
for pretrains in place of hypergraphs and that, in fact, this result can be shown by 
means of the same partite construction (see Proposition~\ref{prop:ofpt} below).  

\subsection{Extensions}
\label{sssec:1448}

For a concise description of the extension process we require a little bit of preparation. 
First of all, suppose that we start with a pretrain $(F, \equiv^F)$ and enlarge its wagons
as disjointly as possible, thus obtaining a new pretrain $(H, \equiv^H)$. We then say 
that~$(H, \equiv^H)$ is an extension of $(F, \equiv^F)$. Let us say the same thing again 
in a more precise way.

\begin{dfn}\label{dfn:1811}
	Given a subpretrain $(F, \equiv^F)$ of a pretrain $(H, \equiv^H)$ we say 
	that $(H, \equiv^H)$ is an {\it extension} of $(F, \equiv^F)$ provided
	the following conditions hold.
	\index{extension (of pretrains)}
	\begin{enumerate}[label=\rmlabel]
		\item\label{it:1811a} The hypergraphs $F$ and $H$ have the same isolated 
			vertices.\footnote[1]{Here and at several other places that follow one 
			might get the 
			impression that we are overly pedantic in our treatment of isolated vertices.
			After all, for results such as Theorem~\ref{thm:grth1} it does not matter 
			whether~$F$ is allowed to have isolated vertices or not. (If necessary, one 
			could first remove the isolated vertices from~$F$, apply the theorem, and put 
			the isolated vertices back at the end.) Nevertheless, hypergraphs with 
			isolated vertices frequently arise in an auxiliary r\^{o}le throughout 
			our constructions. For instance the picture zero shown in Figure~\ref{fig:21},
			even though it has no isolated vertices in itself, consists of ten constituents
			each of which possesses six isolated vertices. Therefore, when starting with 
			this picture we occasionally have to apply a partite lemma to a bipartite 
			graph with isolated vertices.}
		\item\label{it:1811b} Every wagon $W$ of $H$ contracts to a wagon of $F$, i.e., 
			satisfies $E(W)\cap E(F)\ne\varnothing$.
		\item\label{it:1811c} If two distinct wagons $W_H^\star$ and $W_H^{\star\star}$ of $H$ 
			contract to the wagons $W_F^\star$ and $W_F^{\star\star}$ of~$F$, then 
			$V(W_H^\star)\cap V(W_H^{\star\star})=V(W_F^\star)\cap V(W_F^{\star\star})$.
	\end{enumerate} 
\end{dfn}

If $(H, \equiv^H)$ is an extension of $(F, \equiv^F)$, then  
$\binom{(H, \equiv^H)}{(F, \equiv^F)}$ can have 
more than one element. As it turns out to be useful later, 
we call $(F, \equiv^F)$ itself 
the {\it standard copy} of $(F, \equiv^F)$ in~$(H, \equiv^H)$. 
\index{standard copy (in a pretrain extension)}
If $(F, \equiv^F)$ and $(H, \equiv^H)$ are ordered and a 
pretrain $(H_\star, \equiv^{H_\star})$ is order-isomorphic to $(H, \equiv^H)$, then
it should be clear what we mean by the {\it standard copy} 
of~$(F, \equiv^F)$ in $(H_\star, \equiv^{H_\star})$. 

We proceed with an easy statement that follows directly from the definition
of extensions.   

\begin{lemma}\label{lem:5210}
	Suppose that the pretrain $(H, \equiv^H)$ is an extension of the 
	pretrain $(F, \equiv^F)$ and that the wagon $X$ of $(H, \equiv^H)$
	contracts to the wagon $W$ of $(F, \equiv^F)$. If $e\in E(X)$, then 
		\[
		e\cap V(W) = e\cap V(F)\,.
	\]
	\end{lemma}

\begin{proof}
	Since $V(W)\subseteq V(F)$, the left side is a subset of the right side. 
	For the converse direction we consider an arbitrary vertex $x\in e\cap V(F)$.
	Because of $x\in e$ and clause~\ref{it:1811a} of Definition~\ref{dfn:1811} we
	know that $x$ cannot be isolated in $F$ and, hence, there is a wagon $W'$ 
	of $(F, \equiv^F)$ with $x\in V(W')$. If $W=W'$ we are done immediately, 
	so suppose $W\ne W'$ from now on. Denoting the wagon of $(H, \equiv^H)$
	contracting to $W'$ by $X'$ we have 
		\[
		x\in V(X)\cap V(X') = V(W)\cap V(W')
	\]
		by Definition~\ref{dfn:1811}\ref{it:1811c} and thus, in particular, $x\in V(W)$.
\end{proof}

Due to the next fact we can often restrict our attention to pretains all of whose 
wagons are isomorphic to each other. 

\begin{fact}\label{exmp:2345}
	Every ordered pretrain $(F, \equiv^F)$ has an extension $(\wh{F}, \equiv^{\wh{F}})$ 
	all of whose wagons are order-isomorphic to the disjoint union of all wagons 
	of $(F, \equiv^F)$. \qed
\end{fact}

In this situation we say that $(\wh{F}, \equiv^{\wh{F}})$ arises from $(F, \equiv^F)$ 
by {\it wagon assimilation} (see Figure~\ref{fig:WA}).
\index{wagon assimilation}

When extending wagons we sometimes want to say in an exact manner how the old wagons
are supposed to ``sit'' in the new wagons and the definition that follows will help us 
to verbalise our intentions.  

\begin{dfn} \label{dfn:1500}
	We say that $(X, W)$ is an {\it ordered hypergraph pair} if $X$ is an ordered 
	hypergraph and $W$ is an ordered subhypergraph of $X$.
	\index{ordered hypergraph pair}
	Two ordered hypergraph pairs~$(X, W)$ and $(X', W')$ are called {\it isomorphic}
	if $X$ is isomorphic to $X'$ and the unique isomorphism from~$X$ to~$X'$ maps~$W$ 
	onto~$W'$. 
\end{dfn}

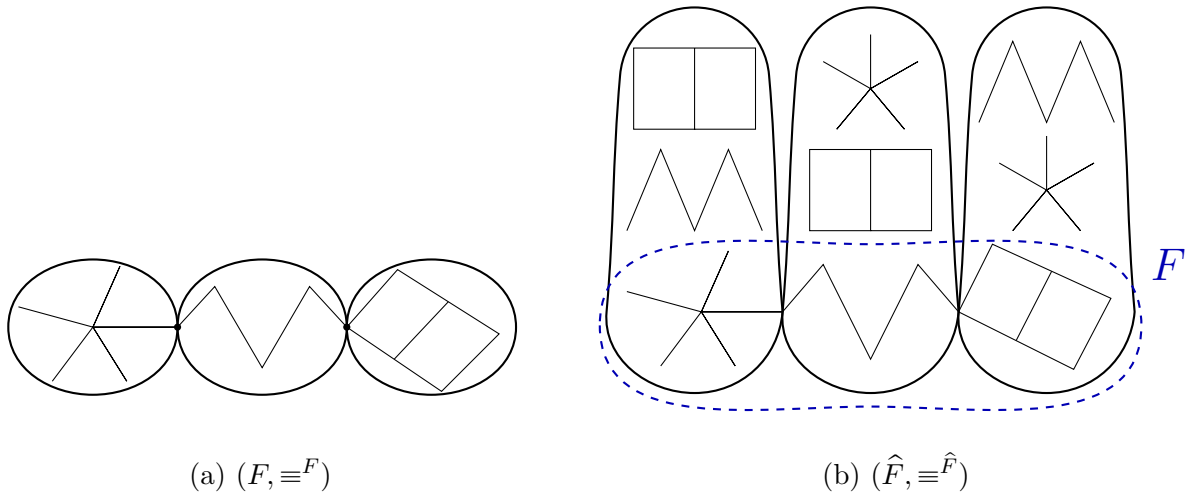
\begin{figure}[h]
	\centering	
	
		\begin{subfigure}[b]{0.45\textwidth}
		\centering
	
			\begin{tikzpicture}[scale=.9]

	\newcommand{\edge}[6]{\draw [color={#6},thick, rotate around={#5:(#1,#2)}] (#1-#3, #2) [out=-90, in= 180] to (#1,#2-#4)  [out=0, in=-90] to (#1+#3,#2) [out=90, in=0] to (#1,#2+#4)  [out=-180, in=90] to (#1-#3,#2);}

\def\w{1.25};
\def\h{1};

\edge{-2.5}{0}{\w}{\h}{0}{black};
\edge{2.5}{0}{\w}{\h}{0}{black};
\edge{0}{0}{\w}{\h}{0}{black};

\coordinate (v) at (-2.5,0);

\draw (-3.6,.3)--(v)--(-2.1,.9)--(v)--(-1.25,0)--(v)--(-2,-.8)--(v)--(-3.1,-.8);
\draw (-1.25,0)--(-.7,.6)--(0,-.6)--(.7,.6)--(1.25,0);
\draw (1.25,0)--(2,.85)--(3.5,-.1)--(2.65,-.95)--(1.25,0);
\draw (2.75,.375)--(1.95,-.475);

\fill (1.25,0) circle (1.5pt);
\fill (-1.25,0) circle (1.5pt);

\phantom{\fill (0,-1.6) circle (1pt);}

			\end{tikzpicture}
		 \caption{$(F,\equiv^F)$}
		\label{fig:WAa} 		
	\end{subfigure}
\hfill    
\begin{subfigure}[b]{0.5\textwidth}
\centering
			\begin{tikzpicture}[scale=.9]
				
				\coordinate (v) at (-2.5,0);
				
		\draw (-3.6,.3)--(v)--(-2.1,.9)--(v)--(-1.3,0)--(v)--(-2,-.8)--(v)--(-3.1,-.8);
		\draw (-1.3,0)--(-.7,.7)--(0,-.7)--(.7,.7)--(1.3,0);
		\draw (1.3,0)--(1.8,1)--(3.55,.2)--(3,-.85)--(1.3,0);
		\draw (2.675,.6)--(2.15,-.425);

		\draw[shift={(0,-.5)}] (0,4.6)--(0,3.8)--(.7,4.2)--(0,3.8)--(.5,3.2)--(0,3.8)--(-.5,3.2)--(0,3.8)--(-.7,4.2);
		
		\draw [shift={(2.6,-2)}] (0,4.6)--(0,3.8)--(.7,4.2)--(0,3.8)--(.5,3.2)--(0,3.8)--(-.5,3.2)--(0,3.8)--(-.7,4.2);
		
		\draw (0,1.2)--(.9,1.2)--(.9,2.4)--(-.9,2.4)--(-.9,1.2)--(0,1.2)--(0,2.4);
				
					\draw [shift={(-2.6,1.5)}] (0,1.2)--(.9,1.2)--(.9,2.4)--(-.9,2.4)--(-.9,1.2)--(0,1.2)--(0,2.4);
					
		\draw (-3.6,1.2)--(-3.1,2.4)--(-2.6,1.2)--(-2.1,2.4)--(-1.6,1.2);
		
		\draw [shift={(5.2,1.6)}] (-3.6,1.2)--(-3.1,2.4)--(-2.6,1.2)--(-2.1,2.4)--(-1.6,1.2);

			\def\w{3};
			\def\h{1.3};
			\coordinate (a) at (-1,1.5);
			\coordinate (b) at (2,1);
			\coordinate (c) at (2,-3);
			
			\draw [thick] (0,-1.2)[out=0, in=-95] to(1.3,0) [out=95,in=-85] to (1.1,3.5) [out=95,in=0] to(0,4.5) [out=180, in=85] to (-1.1,3.5) [out=-95, in=85] to (-1.3,0) [out=-95,in=180] to (0,-1.2);
			
				\draw [thick, shift={(-2.6,0)}] (0,-1.2)[out=0, in=-95] to(1.3,0) [out=95,in=-85] to (1.1,3.5) [out=95,in=0] to(0,4.5) [out=180, in=85] to (-1.1,3.5) [out=-95, in=85] to (-1.3,0) [out=-95,in=180] to (0,-1.2);
			
		\draw [thick, shift={(2.6,0)}] (0,-1.2)[out=0, in=-95] to(1.3,0) [out=95,in=-85] to (1.1,3.5) [out=95,in=0] to(0,4.5) [out=180, in=85] to (-1.1,3.5) [out=-95, in=85] to (-1.3,0) [out=-95,in=180] to (0,-1.2);
		
		\draw [blue!70!black, dashed, thick] (-4,-.2)[out=-90, in= 180] to (0, -1.4)[out=0, in=-90] to (4,-.2) [out=90, in=0] to (0,1)  [out=-180, in=90] to (-4,-.2);

			\node [blue!70!black]at (4.4,.7) {\Large $F$};

		\end{tikzpicture}
	\caption{$(\widehat F, \equiv^{\widehat{F}})$}
	\label{fig:WAb} 		
\end{subfigure}  
					
\caption{The pretrain $(\widehat{F},\equiv^{ \widehat{F}})$ arises from $(F,\equiv^F)$ 
by wagon assimilation. Notice that $F$ is disconnected from the rest of $\widehat F$.}
\label{fig:WA} 
\end{figure} 

We may now say what it means to extend all wagons of an ordered pretrain ``in the same way''.  
      
\begin{dfn}\label{dfn:1529}
	Let $(X, W)$ be an ordered hypergraph pair such that neither~$X$ nor~$W$ has 
	isolated vertices, and let $(F, \equiv^F)$ be an ordered pretrain all of whose 
	wagons are order-isomorphic to $W$.  
	\index{$(F, \equiv^F)\ltimes (X, W)$}
	We write $(F, \equiv^F)\ltimes (X, W)$ for the extension of $(F, \equiv^F)$
	having the following property: If $X_\star$ is a wagon 
	of $(F, \equiv^F)\ltimes (X, W)$ contracting to the wagon~$W_\star$ 
	of $(F, \equiv^F)$, then the ordered
	hypergraph pair $(X_\star, W_\star)$ is isomorphic to $(X, W)$ (see 
	Figure~\ref{fig:31}). 
\end{dfn}

\begin{figure}[h]
	\centering
	
		\def\h{.9}
		\begin{subfigure}[b]{0.3\textwidth}
			\centering
	\begin{tikzpicture}[scale=1]

\draw [ultra thick] (-.5, 0) [out=-90, in= 180] to (0,-\h)  [out=0, in=-90] to (.5,0);
\fill [white] (-.5, 0) [out=-90, in= 180] to (0,-\h)  [out=0, in=-90] to (.5,0);

\draw [ultra thick] (-.5, 0) [out=90, in= 180] to (0,\h)  [out=0, in=90] to (.5,0)-- cycle;
\fill [lime] (-.5, 0) [out=90, in= 180] to (0,\h)  [out=0, in=90] to (.5,0)-- cycle;

\node at (0,.4) {\large $W$};
\node at (-1.3,\h) {\large $X$};
	
	\end{tikzpicture}

	\caption{{The pair $(X,W)$}}
	\label{fig:31a} 

	\end{subfigure}
	\hfill    
	\begin{subfigure}[b]{0.3\textwidth}
		\centering
		
			\begin{tikzpicture}[scale=1]
			
	\foreach \x in {0,1.04,-1.04}{
	\def\a{\x-.5}
	\def\b{\x+.5}
			\draw [ultra thick] (\a,0) [out=90, in= 180] to (\x,\h)  [out=0, in=90] to (\b,0)-- cycle;
			\fill [lime] (\a, 0) [out=90, in= 180] to (\x,\h)  [out=0, in=90] to (\b,0)-- cycle;
}		

			\node at (-1,.4) {\large $W$};
		
			\phantom{\draw [ultra thick] (-.5, 0) [out=-90, in= 180] to (0,-\h)  [out=0, in=-90] to (.5,0);}
			
			\end{tikzpicture}
		
					\caption{The pretrain $(F, \equiv^F)$}
				\label{fig:31b}

		\end{subfigure}    
	\hfill    
	\begin{subfigure}[b]{0.3\textwidth}
		\centering
		
		\begin{tikzpicture}[scale=1]
			
				\foreach \x in {0,1.04,-1.04}{
				\def\a{\x-.5}
				\def\b{\x+.5}
				\draw [ultra thick] (\a, 0) [out=-90, in= 180] to (\x,-\h)  [out=0, in=-90] to (\b,0);
				\fill [white] (\a, 0) [out=-90, in= 180] to (\x,-\h)  [out=0, in=-90] to (\b,0);
				\draw [ultra thick] (\a,0) [out=90, in= 180] to (\x,\h)  [out=0, in=90] to (\b,0)-- cycle;
				\fill [lime] (\a, 0) [out=90, in= 180] to (\x,\h)  [out=0, in=90] to (\b,0)-- cycle;	
			}
			
		\end{tikzpicture}
		
		\caption{$(F, \equiv^F)\ltimes (X,W)$}
		\label{fig:31c} 		
	\end{subfigure}   
		\caption{Illustration of Definition~\ref{dfn:1529}}	\label{fig:31}
			\end{figure}
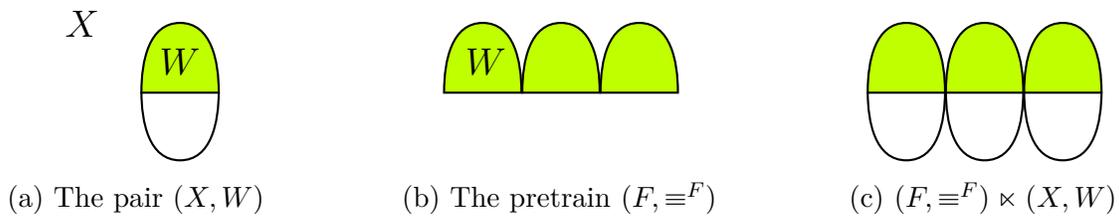 
Strictly speaking the ordered pretrain $(F, \equiv^F)\ltimes (X, W)$ has thereby 
not been defined in a unique manner, for there are no rules as to how two new vertices 
from distinct wagons compare under the order relation. 
The ambiguity that remains, however, has no bearing on later developments 
and, therefore, we shall ignore it in the sequel. 
That is we talk about~$(F, \equiv^F)\ltimes (X, W)$ as if it were uniquely determined. 
A later collapsing argument will hinge on a certain tameness property of 
this construction that we shall introduce next.

\begin{dfn}\label{dfn:n854}
	Suppose that the pretrain $(H, \equiv^H)$ is an extension of $(F, \equiv^F)$. 
	We say that this extension is {\it tame}
	if $F$ is strongly induced in $H$ and, moreover, the following holds:
	For every edge $e\in E(H)$ and every vertex $x\in V(F)\cap e$ there exists 
	an edge $e'$ such that $x\in e'\in E(F)$ and~$e\equiv^H e'$. 
	\index{tame extension}
\end{dfn} 

\begin{figure}[h]
\begin{center}	
\begin{tikzpicture}[scale=1]
			
	\def\h{1.5}  	\def\w{.8}   	
	\coordinate (x) at (-\w-.02,0);
	
	\def\nn{2*\w+.04}
	\def\mm{-2*\w-.04}
		\foreach \x in {0,\nn,\mm}{
		\def\a{\x-\w}
		\def\b{\x+\w}
		\draw [ultra thick] (\a, 0) [out=-90, in= 180] to (\x,-\h)  [out=0, in=-90] to (\b,0);
		\fill [white] (\a, 0) [out=-90, in= 180] to (\x,-\h)  [out=0, in=-90] to (\b,0);
		\draw [ultra thick] (\a,0) [out=90, in= 180] to (\x,\h)  [out=0, in=90] to (\b,0)-- cycle;
		\fill [lime] (\a, 0) [out=90, in= 180] to (\x,\h)  [out=0, in=90] to (\b,0)-- cycle;	
	}

\draw [red, thick] (-1.6, .8)--(x)--(-1.6,-.8);

\fill (x) circle (1.5pt);
\node at (-.65,.15) {$x$};
\node [red] at (-1.5,.4)  {$e'$};
\node [red] at (-1.5,-.5)  {$e$};

\end{tikzpicture}
\caption{A tame extension}
\end{center}
\end{figure}
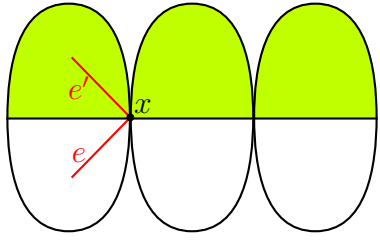

\begin{lemma}\label{lem:5757}
	Let $(\wh{F}, \equiv^{\wh{F}})$ be the ordered pretrain that arises 
	from $(F, \equiv^F)$ by wagon assimilation. 
	Suppose further that all wagons of $(\wh{F}, \equiv^{\wh{F}})$ 
	are isomorphic to $W$ and that $(X, W)$ is an ordered hypergraph pair without
	isolated vertices. 
	If $W$ is strongly induced in~$X$, then  
	$(\wh{F}, \equiv^{\wh{F}})\ltimes (X, W)$ is a tame extension of $(F, \equiv^F)$.
\end{lemma}

\begin{proof}
	We begin with the strong inducedness.  
	Given an edge $e$ of $(\wh{F}, \equiv^{\wh{F}})\ltimes (X, W)$
	we need to find an edge $e'$ of $F$ such that $V(F)\cap e\subseteq e'$.
	
	Let us first deal with the special case that $V(F)\cap e=\vn$. Now we can
	take an arbitrary edge $e'$ of $F$ and, hence, our claim could only fail 
	if $F$ has no edges. But then none of the hypergraphs $W$, $X$, 
	and $(\wh{F}, \equiv^{\wh{F}})\ltimes (X, W)$ can have any edges (recall that 
	$W\Str X$), meaning that there is no edge $e$ to consider. 
	
	So we may assume $V(F)\cap e\ne \vn$ from now on. 
	Let $X_\star$ be the wagon of $(\wh{F}, \equiv^{\wh{F}})\ltimes (X, W)$ 
	containing $e$ and denote its contraction to $(\wh{F}, \equiv^{\wh{F}})$ 
	by $W_\star$. Since the ordered hypergraph pairs $(X_\star, W_\star)$ 
	and $(X, W)$ are isomorphic, the assumption $W\Str X$ 
	implies $W_\star \Str X_\star$.
	Thus there exists an edge $e'\in E(W_\star)$ such 
	that $V(W_\star)\cap e\subseteq e'$.
	Owing to Lemma~\ref{lem:5210} we infer $V(\wh{F})\cap e\subseteq e'$. 
	In particular, $e'\in E(\wh{F})$ covers the nonempty intersection $V(F)\cap e$.
	Since $F$ is disconnected from the rest of $\wh{F}$ (cf.\ Figure~\ref{fig:WA})
	and $V(F)\cap e\ne\vn$, this entails $e'\in E(F)$. 
	So altogether $e'$ is as required, i.e.,~$F$ is 
	indeed strongly induced in $(\wh{F}, \equiv^{\wh{F}})\ltimes (X, W)$. 
		
	Notice that in the case $V(F)\cap e\ne \vn$ the edge $e'$ we have just exhibited 
	lies in the same wagon of $(\wh{F}, \equiv^{\wh{F}})\ltimes (X, W)$ as the given 
	edge $e$. This proves that the moreover-part of Definition~\ref{dfn:n854}
	is satisfied as well.	
\end{proof}

\subsection{Linear pretrains}
\label{sssec:lp}

There are two possible linearity conditions one can impose on pretrains
depending on whether one wants the edges or the wagons to avoid intersections
in two or more vertices. The linear pretrains considered in this article avoid
both possibilities at the same time. 

\begin{dfn}\label{dfn:838}
	A pretrain $(H, \equiv^H)$ is said to be {\it linear} if 
	\index{linear pretrain}
	\begin{enumerate}[label=\rmlabel]
		\item\label{it:838a} the hypergraph $H$ is linear
		\item\label{it:838b} and $|V(W')\cap V(W'')|\le 1$ holds for 
			any two distinct wagons $W'$, $W''$.	
	\end{enumerate}
	A system of pretrains $(H, \equiv, \ccH)$ is {\it linear} if its underlying 
	pretrain $(H, \equiv)$ is linear. 
	\index{linear system of pretrains}
\end{dfn}

For later use we record an easy consequence of linearity. 

\begin{fact}\label{f:wagstr}
	Every wagon $W$ of a linear pretrain $(H, \equiv^H)$ is strongly induced in $H$.
\end{fact}

\begin{proof}
	Due to the linearity of $H$ it suffices to check the three 
	statements~\ref{it:1836a}\,--\,\ref{it:1836c} in Fact~\ref{fact:1836}.
	The last two of them are clear, since $W$ has at least one edge and no 
	isolated vertices. Now suppose that some edge $e$ of $H$ intersects $V(W)$
	in at least two vertices. If~$W'$ denotes the wagon to which~$e$ belongs, 
	then $|V(W)\cap V(W')|\ge |V(W)\cap e|\ge 2$ and 
	Definition~\ref{dfn:838}\ref{it:838b}
	imply $W=W'$, whence $e\in E(W)$.
\end{proof}

The {\it extension process}, which we shall now describe, is an operation 
transforming two given ordered linear Ramsey construction $\Phi$ and $\Psi$ applicable 
to hypergraphs into a Ramsey construction $\Ext(\Phi, \Psi)$ applicable to ordered 
linear pretrains.
\index{extension process}
 
Given $\Phi$ and $\Psi$, an ordered linear pretrain $(F, \equiv^F)$, and a number of colours $r$,
we explain how to construct the system of 
pretrains $\Ext(\Phi, \Psi)_r(F, \equiv^F)=(H, \equiv^H, \ccH)$ in eight steps. 
 
\begin{enumerate}[label=\nlabel]
	\item\label{it:ext1} Let $(\wh{F}, \equiv^{\wh{F}})$ be obtained from $(F, \equiv^F)$
		by wagon assimilation (see Fact~\ref{exmp:2345}). So all wagons 
		of $(\wh{F}, \equiv^{\wh{F}})$ are isomorphic to the same ordered hypergraph $W$.
		Moreover, $(\wh{F}, \equiv^{\wh{F}})$ is again a linear pretrain.
	\item\label{it:ext2} Construct $\Phi_r(W)=(X, \ccX)$ and assume, without loss 
		of generality, that $X$ has no isolated vertices. 
		Notice that, by hypothesis on $\Phi$, the hypergraph $X$ is ordered and linear.
	\item\label{it:ext3} Define the ordered pretrain $(G, \equiv^G)$ to be the disjoint union of
				\[
			\bigl\{(\wh{F}, \equiv^{\wh{F}})\ltimes (X, W_\star)\colon W_\star\in\ccX\bigr\}\,.
		\]
				This is a linear pretrain containing $|\ccX|$ standard copies of $(F, \equiv^F)$.
		All of its wagons are order-isomorphic to $X$ (see Figure~\ref{fig:32}).
		\begin{figure}[ht]
	\centering
	\begin{subfigure}[b]{.15\textwidth}
			\centering
	\begin{tikzpicture}[scale=.7]
		
		\def\h{1.5}
	\def\w{.7}
	
\draw [thick, purple] (-\w, 1) [out=-90, in= 180] to (0,1-\h)  [out=0, in=-90] to (\w,1) [out=90, in=0] to (0,1+\h)  [out=-180, in=90] to (-\w,1);

\foreach \y in {.1, 1, 1.9}{
	\draw [thick] (0, \y) circle (.53 cm);
	\node at (0, \y) {$W$};
}
			\end{tikzpicture}
		
	\caption{{$(X, \ccX)$}}
	\label{fig:32a} 
	\end{subfigure}
	\hfill    
	\begin{subfigure}[b]{0.2\textwidth}
		\centering
		
			\begin{tikzpicture}[scale=.7]
				
		\def\r{.6}
	\foreach \x in {-2*\r, 0, 2*\r}{
		\draw [thick] (\x,1) circle (.98*\r cm);
		\node at (\x,1) {$W$};
	}
		\fill (\r,1) circle (2.5pt);
		\fill (-\r,1) circle (2.5pt);

			\end{tikzpicture}
		
					\caption{$(\widehat{F}, \equiv^{\widehat{F}})$}
				\label{fig:32b}

		\end{subfigure}    
	\hfill    
	\begin{subfigure}[b]{0.6\textwidth}
		\centering
		
		\begin{tikzpicture}[scale=1.1]
			
		\newcommand{\edge}[5]{\draw [rotate around={#5:(#1,#2)}, purple] (#1-#3, #2) [out=-90, in= 180] to (#1,#2-#4)  [out=0, in=-90] to (#1+#3,#2) [out=90, in=0] to (#1,#2+#4)  [out=-180, in=90] to (#1-#3,#2);}
		
		\def\w{.32}
		\def\h{.85}
		
		\edge{-3.56}{-.05}{\w}{\h}{20}
		\edge{-2.78}{0}{\w}{\h}{0}
		\edge{-2}{-.05}{\w}{\h}{-20}
		
		\edge{-.65}{0}{\w}{\h}{0}
		\edge{0}{0}{\w}{\h}{0}
		\edge{.65}{0}{\w}{\h}{0}
		
		\edge{2}{.05}{\w}{\h}{-20}
		\edge{2.78}{0}{\w}{\h}{0}
		\edge{3.56}{.05}{\w}{\h}{20}
		
		\fill (.33,0) circle (1.2pt);
		\fill (-.33,0) circle (1.2pt);
		\fill (-2.46,-.49) circle (1.2pt);
		\fill (-3.09,-.49) circle (1.2pt);
		\fill (2.46,.49) circle (1.2pt);
		\fill (3.09,.49) circle (1.2pt);
		
		\def\r{.27}
		
		\draw (0,0) circle (.315);
		\draw (-.635,0) circle (.3);
		\draw (.635,0) circle (.3);
		\draw (-.65,.5) circle (\r);
		\draw (-.65,-.5) circle (\r);
		\draw (.65,.5) circle (\r);
		\draw (.65,-.5) circle (\r);
		\draw (0,.5) circle (\r);
		\draw (0,-.5) circle (\r);
		
		\draw (-2.78,0) circle (\r);
		\draw (-3.56,-.04) circle (\r);
		\draw (-2,-.04) circle (\r);
		\draw (-3.73,.42) circle (\r);
		\draw (-3.39,-.5) circle (.29);
		\draw (-1.83,.42) circle (\r);
		\draw (-2.17,-.5) circle (.29);
		\draw (-2.78,.49) circle (\r);
		\draw (-2.78,-.5) circle (.3);
		
		\draw (2.78,0) circle (\r);
		\draw (3.56,.04) circle (\r);
		\draw (2,.04) circle (\r);
		\draw (3.73,-.42) circle (\r);
		\draw (3.39,.5) circle (.29);
		\draw (1.83,-.42) circle (\r);
		\draw (2.17,.5) circle (.29);
		\draw (2.78,-.49) circle (\r);
		\draw (2.78,.5) circle (.3);

		\end{tikzpicture}
		
		\caption{$(G, \equiv^G)$ and $M$}
		\label{fig:32c} 		
	\end{subfigure}  
		\caption{Steps~\ref{it:ext1}\,--\,\ref{it:ext4} of the extension process}	
		\label{fig:32}
	\end{figure}
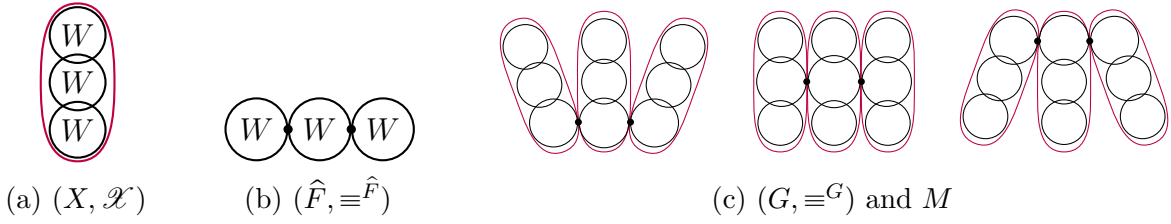 	\item\label{it:ext4} Define $M$ to be the ordered $|V(X)|$-uniform linear 
		hypergraph with $V(M)=V(G)$ whose edges correspond to the wagons of $(G, \equiv^G)$. 
	\item\label{it:ext5} Construct the linear system $(N, \ccN)=\Psi_{r^{e(X)}}(M)$. Notice 
		that, in particular, $N$ is a $|V(X)|$-uniform hypergraph arrowing $M$ with $r^{e(X)}$
		colours. 
	\item\label{it:ext6} Let $(H, \equiv^H)$ be the ordered linear pretrain obtained 
			by inserting ordered copies of~$X$ into the edges of $N$ and declaring 
			them to be wagons.
	\item\label{it:ext7} Now every copy $M_\star\in \ccN$ gives rise 
		to a copy of $(G, \equiv^G)$ in $(H, \equiv^H)$, which has the same vertex set. 
		Let $\ccH_\bullet\subseteq\binom{(H, \equiv^H)}{(G, \equiv^G)}$ be the system 
		of all these copies. Every copy in $\ccH_\bullet$ 
		contains a system $|\ccX|$ standard copies of $(F, \equiv^F)$. 
		We write $\ccH$ for the system of all (at most) $|\ccN|\cdot |\ccX|$ 
		copies of $(F, \equiv^F)$ in $(H, \equiv^H)$ that arise in this manner. 
	\item\label{it:ext8} Finally, we set 
		$\Ext(\Phi, \Psi)_r(F, \equiv^F)=(H, \equiv^H, \ccH)$.
\end{enumerate}

The next result explains our interest in this extension process. 
			
\begin{lemma}\label{lem:0059}
	If $\Phi$ and $\Psi$ denote linear ordered Ramsey constructions, 
	$(F, \equiv^F)$ is an ordered linear pretrain, $r\in\NN$, 
	and $\Ext(\Phi, \Psi)_r(F, \equiv^F)=(H, \equiv^H, \ccH)$, then 
		\[
		\ccH\lra (F, \equiv^F)_r\,.
	\]
	\end{lemma}
  
\begin{proof}
	Let $\gamma\colon E(H)\lra [r]$ be an arbitrary colouring. Each wagon 
	of $(H, \equiv^H)$ is isomorphic to $X$ and thus it receives one 
	of $r^{e(X)}$ possible colour patterns. These colour patterns induce an 
	auxiliary colouring of $E(N)$ and by our construction of $(N, \ccN)$ 
	in Step~\ref{it:ext5} there exists a copy
	$M_\star\in\ccN$ which is monochromatic with respect to this auxiliary 
	colouring. The common colour pattern of its wagons can be regarded as 
	a colouring $\delta\colon E(X)\lra [r]$. By Step~\ref{it:ext2} there 
	exists a copy $W_\star\in \ccX$ that is monochromatic with respect 
	to $\delta$. The copy of $(G, \equiv^G)$ corresponding to $M_\star$ 
	contains a copy of $(\wh{F}, \equiv^{\wh{F}})\ltimes (X, W_\star)$.
	The standard copy of $(\wh{F}, \equiv^{\wh{F}})$ therein is monochromatic with
	respect to $\gamma$. In particular, there exists a monochromatic copy 
	of $(F, \equiv^F)$ belonging to $\ccH$.      
\end{proof}

It should be clear that 
\begin{enumerate}
	\item[$\bullet$] if $\Phi$ and $\Psi$ deliver systems of induced subhypergraphs, then 
		$\Ext(\Phi, \Psi)$ delivers systems of induced pretrains
	\item[$\bullet$] and that $\Phi$ and $\Psi$ are $f$-partite, then so is $\Ext(\Phi, \Psi)$.
\end{enumerate}
 \section{Pretrains in partite constructions}
\label{subsec:PCEP}

We organise the material in this section in such a manner that it constitutes 
a proof of the induced Ramsey theorem for pretrains. This result, however, only 
serves as a point of reference and the arguments occurring in the proof will be more 
relevant in the sequel than the theorem itself. 

\begin{prop} \label{prop:ofpt}
	Given an ordered $f$-partite pretrain $(F, \equiv^F)$ as well as a number of 
	colours $r$, there exists an ordered $f$-partite system of pretrains 
	$(H, \equiv^H, \ccH)$ such that 
		\[
		\ccH\lra(F, \equiv^F)_r\,.
	\]
\end{prop} 

Observe that this result is incomparable in strength to Lemma~\ref{lem:0059}.
For instance, we defined constructions of the form $\Ext(\Omega, \Phi)$ as applying 
to linear pretrains only.\footnote[1]{One can relax this requirement somewhat, but not 
indefinitely.} 
On the other hand, the proof of Proposition~\ref{prop:ofpt} presented in 
\S\ref{sssec:314} relies on a construction that often yields nonlinear 
pretrains $(H, \equiv^H)$ even when the given pretrain $(F, \equiv^F)$ is linear.

\subsection{A partite lemma for pretrains}
\label{sssec:311}

Let $(F, \equiv^F)$ be a $k$-partite $k$-uniform pretrain and let~$r\in \NN$
be a number of colours. The Hales-Jewett construction introduced in \S\ref{sssec:HJ}
leads to a system $\HJ_r(F)=(H, \ccH)$. 
\index{Hales-Jewett construction}

On $E(H)$ we can define an equivalence relation $\equiv^H$ by taking, essentially, 
the product of the equivalence relations $\equiv^F$ that we have on the ``factors'' of~$H$. 
More precisely, if $H$ is the $n^{\mathrm{th}}$ Hales-Jewett power of $F$ and 
$\lambda\colon E(F)^n\longrightarrow E(H)$ denotes the canonical 
bijection, then we define
\[
	\lambda(e_1, \ldots, e_n)\equiv^H \lambda(e'_1, \ldots, e'_n)
	\,\,\, \Longleftrightarrow \,\,\,
	\forall \nu\in [n]\,\, e_\nu\equiv^F e'_\nu
\]
for all $e_1, \ldots, e_n, e'_1, \ldots, e'_n\in E(F)$. 
 
\begin{clm} \label{clm:1128}
	If $F_\star\in \ccH$ and $\equiv^{F_\star}$ denotes the equivalence relation on 
	$E(F_\star)$ rendering \mbox{$(F, \equiv^F)$} and $(F_\star, \equiv^{F_\star})$ naturally 
	isomorphic, then $(F_\star, \equiv^{F_\star})$ is an induced subpretrain 
	of~$(H, \equiv^H)$.
\end{clm}

\begin{proof}
	We already saw in Lemma~\ref{lem:hj-str} that $F_\star$ is an induced subhypergraph 
	of $H$. It remains to show that any two edges of $F_\star$ are equivalent with respect 
	to $\equiv^{F_\star}$ if and only if they are equivalent with respect to $\equiv^H$.
	
	Suppose that $F_\star$ is given by the 
	combinatorial embedding $\eta\colon E(F)\lra E(F)^n$, which in 
	turn depends, as in the proof of Lemma~\ref{lem:hj-str}, on the partition $[n]=C\dcup M$ 
	of $[n]$ into constant and moving coordinates and on the 
	map $\wt{\eta}\colon C\lra E(F)$. 
	\index{constant coordinate}
	\index{moving coordinate}
	Let $e, e'\in E(F)$ be arbitrary and write $\eta(e)=(e_1, \ldots, e_n)$ as well
	as $\eta(e')=(e'_1, \ldots, e'_n)$. Now it remains to observe  
		\begin{align*}
		(\lambda\circ \eta)(e)\equiv^H (\lambda\circ\eta)(e') 
		&\,\,\, \Longleftrightarrow \,\,\, 
		\forall \nu\in [n]\,\, e_\nu\equiv e'_\nu \\
		&\,\,\, \Longleftrightarrow \,\,\, 
		\forall \nu\in M\,\, e_\nu\equiv e'_\nu \\
		&\,\,\, \Longleftrightarrow \,\,\,
		e \equiv^F e' \,,
	\end{align*}
		where the last equivalence exploits $M\ne\vn$. 
\end{proof}

Concerning notation, it seems best to denote the system of pretrains
$\{(F_\star, \equiv^{F_\star})\colon F_\star\in\ccH\}$ by $\ccH$ again, 
so that we may refer to the system of pretrains $(H, \equiv^H, \ccH)$. 
Besides, we write 
\[
	\HJ_r(F, \equiv^F)=(H, \equiv^H, \ccH)
\]
for the above construction; so from now on $\HJ_r(\cdot)$ applies 
to pretrains as well.   

\subsection{Pretrain pictures}
\label{sssec:2340}

Consider a pretrain $(F, \equiv^F)$ as well as a system of hypergraphs $(G, \ccG)$
with $\ccG\subseteq\binom{G}{F}_\nni$.
A {\it pretrain picture over} $(G, \ccG)$ is defined 
to be a quadruple $(\Pi, \equiv, \ccP, \psi)$ such 
that 
\begin{enumerate}
	\item[$\bullet$] $(\Pi, \ccP, \psi)$ is a picture over $(G, \ccG)$ (in the 
		sense of~\S\ref{sssec:pict}),
	\item[$\bullet$] $(\Pi, \equiv)$ is a pretrain,
	\item[$\bullet$] and every copy $(F_\star, \equiv^{F_\star})\in\ccP$ is a 
		subpretrain of $(\Pi, \equiv)$.
\end{enumerate}
\index{pretrain picture}

In this context, the {\it pretrain picture zero} $(\Pi_0, \equiv_0, \ccP_0, \psi_0)$
is defined in the expected way: One starts with picture zero $(\Pi_0, \ccP_0, \psi_0)$
as defined in~\S\ref{sssec:pict} and determines the equivalence relation $\equiv_0$
on $E(\Pi)$ in such a manner that
\begin{enumerate}
	\item[$\bullet$] all copies in $\ccP_0$ become isomorphic to $(F, \equiv^F)$ as 
		pretrains,
	\item[$\bullet$] and edges belonging to different copies are nonequivalent 
		with respect to $\equiv_0$. 
\end{enumerate}
\index{pretrain picture zero}

The second bullet may seem arbitrary at this moment, but proves to be useful later. 
Essentially, there will be notions of cycles and $\GTH$ for systems of pretrains,
and in those cycles wagons can serve as connectors. Now if the wagons were allowed 
to spread over several copies in $\ccP_0$, then we could jump from one copy to the 
next using wagon connectors and already the $\GTH$ of picture zero could be out of 
control.    

\subsection{Partite amalgamations} 
\label{sssec:0021}

Now suppose that $(\Pi, \equiv^\Pi, \ccP, \psi_\Pi)$ is a pretrain picture 
over a system of hypergraphs $(G, \ccG)$,
that $e\in E(G)$, and that $(H, \equiv^H, \ccH)$ is a system of pretrains all of whose copies 
are isomorphic to $\bigl(\Pi^e, \equiv^{\Pi^e}\bigr)$. As demanded by the partite construction 
method, we aim at defining a new picture
\[
	(\Sigma, \equiv^\Sigma\ccQ, \psi_\Sigma) 
	= 
	(\Pi, \equiv^\Pi, \ccP, \psi_\Pi) \conc (H, \equiv^H, \ccH)
\]
over $(G, \ccG)$. 
As in \S\ref{sssec:pict} we construct
\[
	(\Sigma, \ccQ, \psi_\Sigma) = (\Pi, \ccP, \psi_\Pi) \conc (H, \ccH)
\]
and it remains to define an equivalence relation $\equiv^\Sigma$ on $E(\Sigma)$. 
For every standard copy $\Pi_\star$ in~$\Sigma$ 
we copy $\equiv^\Pi$ onto $\Pi_\star$, thus getting a 
pretrain $(\Pi_\star, \equiv^{\Pi_\star})$. Essentially $\equiv^{\Sigma}$ is going to 
be the transitivisation of the free amalgamation 
of $\{\equiv^{\Pi_\star}\colon \Pi_\star \text{ is a standard copy}\}$ over~$\equiv^H$. 

Our official definition of this equivalence relation is somewhat lengthy, but 
its main properties can be summarised as follows. 

\begin{lemma}\label{lem:1825}
	If $(\Pi, \equiv^\Pi, \ccP, \psi_\Pi)$ is a pretrain picture over $(G, \ccG)$,
	$(H, \equiv^H, \ccH)$ is a system of pretrains,
	and 
		\[
		(\Sigma, \ccQ, \psi_\Sigma) = (\Pi, \ccP, \psi_\Pi) \conc (H, \ccH)\,,
	\]
		then there is an equivalence relation $\equiv^\Sigma$ on $E(\Sigma)$ 
	with the following properties. 
	\begin{enumerate}[label=\alabel]
		\item\label{it:1832a} If $\Pi_\star$ is a standard copy in $\Sigma$, then 
			$(\Pi_\star, \equiv^{\Pi_\star})$ is a subpretrain of $(\Sigma, \equiv^\Sigma)$.
		\item\label{it:1832b} $(H, \equiv^H)$ is a subpretrain of $(\Sigma, \equiv^\Sigma)$.
		\item\label{it:1832c} 
				\begin{enumerate}[label=\rmlabel]
					\item\label{it:1832c1} If $\Pi_\star$ is a standard copy 
							and $e_\star\in E(\Pi_\star)$, $e_0\in E(H)$ satisfy 
							$e_\star\equiv^\Sigma e_0$, then there exists an 
							edge $e'\in E(H)\cap E(\Pi_\star)$ with 
							$e_\star \equiv^\Sigma e'\equiv^\Sigma e_0$.
					\item\label{it:1832c2} Similarly, if $\Pi_\star$, $\Pi_{\star\star}$ are distinct 
							standard copies and two edges $e_\star\in E(\Pi_\star)$, 
							$e_{\star\star}\in E(\Pi_{\star\star})$ satisfy
							$e_\star \equiv^\Sigma e_{\star\star}$, then there exist edges 
							$e'\in E(H)\cap E(\Pi_\star)$ and $e''\in E(H)\cap E(\Pi_{\star\star})$ 
							with 
							$e_\star \equiv^\Sigma e'\equiv^\Sigma e'' \equiv^\Sigma e_{\star\star}$.
					\end{enumerate}
		\item\label{it:1832d} $(\Sigma, \equiv^\Sigma, \ccQ, \psi_\Sigma)$ is a pretrain 
			picture over $(G, \ccG)$. 
	\end{enumerate}
\end{lemma}

\begin{proof}
	Since the members of $\ccH$ are subpretrains of $(H, \equiv^H)$, every standard 
	copy $\Pi_\star$ satisfies
		\begin{equation} \label{eq:2052}
	\forall e', e''\in E(\Pi_\star)\cap E(H) \,\,\, 
		\bigl[e'\equiv^{\Pi_\star} e''
		\,\,\, \Longleftrightarrow \,\,\, 
		e'\equiv^H e''\bigr]\,.
	\end{equation}
	
	Let us call the edges in $E(\Sigma)\sm E(H)$ {\it new} and the edges of $H$ {\it old}.
	Observe that every new edge belongs to a unique standard copy. Our first step is to 
	define a relation $\equiv^\Sigma$ on~$E(\Sigma)$. To this end we consider any 
	two edges $e_\star$ and $e_{\star\star}$ of $\Sigma$. If $e_\star$ happens to be new,
	we denote the standard copy it belongs to by $\Pi_\star$. Similarly, if $e_{\star\star}$ 
	is new, its standard copy is denoted by $\Pi_{\star\star}$. 
	We define $e_\star\equiv^\Sigma e_{\star\star}$ to hold if one of the following 
	five cases occurs.
	\begin{enumerate}[label=\glabel]
		\item\label{it:2001} Both $e_\star$ and $e_{\star\star}$ are new, 
			$\Pi_\star\ne \Pi_{\star\star}$, and there exist edges $e'\in E(H)\cap E(\Pi_\star)$
			and $e''\in E(H)\cap E(\Pi_{\star\star})$ with 
			$e_\star \equiv^{\Pi_\star} e'\equiv^H e'' \equiv^{\Pi_{\star\star}}e_{\star\star}$. 
		\item\label{it:2002} Both $e_\star$ and $e_{\star\star}$ are new, 
			$\Pi_\star= \Pi_{\star\star}$, and $e_\star\equiv^{\Pi_\star} e_{\star\star}$.
		\item\label{it:2003} The edge $e_\star$ is new, $e_{\star\star}$ is old, and 
			there is an edge $e'\in E(H)\cap E(\Pi_\star)$ with 
			$e_\star \equiv^{\Pi_\star} e'\equiv^H e_{\star\star}$.
		\item\label{it:2004} The edge $e_\star$ is old, $e_{\star\star}$ is new, and 
			there is an edge $e''\in E(H)\cap E(\Pi_{\star\star})$ with 
			$e_\star \equiv^H e'' \equiv^{\Pi_{\star\star}}e_{\star\star}$.
		\item\label{it:2005} Both $e_\star$ and $e_{\star\star}$ are old and 
			$e_\star\equiv^H e_{\star\star}$.
	\end{enumerate}
	
	Observe that the hypotheses of these five cases are mutually exclusive and cover 
	all possibilities.
	
	\begin{clm}
		The relation $\equiv^\Sigma$ is indeed an equivalence relation.
	\end{clm}
	
	\begin{proof}
		Reflexivity and symmetry are clear. The proof of transitivity is not 
		difficult but requires to look at a large number of cases depending on whether
		the three edges under consideration are old or new and on which of the standard
		copies the new ones belong to coincide. Leaving the other cases as exercises we 
		will only display the case of three new edges $e_1\equiv^\Sigma e_2\equiv^\Sigma e_3$
		living in three standard copies $\Pi_1$, $\Pi_2$, and $\Pi_3$ 
		with $\Pi_1\ne \Pi_2\ne\Pi_3$. Now both equivalences are in case~\ref{it:2001} 
		and we obtain four 
		auxiliary edges $e'_1\in E(H)\cap E(\Pi_1)$, $e'_2, e''_2\in E(H)\cap E(\Pi_2)$,
		and $e''_3\in E(H)\cap E(\Pi_3)$ 
		with 
				\[
			e_1\equiv^{\Pi_1} e'_1\equiv^H e'_2\equiv^{\Pi_2} e_2
			\quad \text{ and } \quad
			e_2\equiv^{\Pi_2} e''_2\equiv^H e''_3\equiv^{\Pi_3} e_3\,.
		\]
				
		As $\equiv^{\Pi_2}$ is an equivalence relation, it follows 
		that $e'_2\equiv^{\Pi_2} e''_2$, which in view of~\eqref{eq:2052} 
		yields $e'_2\equiv^H e''_2$.
		The transitivity of $\equiv^H$ leads 
		to $e_1\equiv^{\Pi_1} e'_1\equiv^H e''_3\equiv^{\Pi_3} e_3$ and in 
		case $\Pi_1\ne \Pi_3$ this, together with~\ref{it:2001}, proves the 
		desired relation $e_1\equiv^\Sigma e_3$.
		In the special case $\Pi_1= \Pi_3$ we appeal to~\eqref{eq:2052} again and 
		obtain $e_1\equiv^{\Pi_1} e_3$, which due to~\ref{it:2002} 
		implies $e_1\equiv^\Sigma e_3$.	
	\end{proof}
	
	Let us prove part~\ref{it:1832a} next. Given a standard copy $\Pi_\star$ 
	and two edges $e_\star, e_{\star\star}\in E(\Pi_\star)$, we are to prove that 
		\begin{equation}\label{eq:2210}
		e_\star\equiv^{\Sigma} e_{\star\star}
		\,\,\,\Longleftrightarrow\,\,\, 
		e_\star\equiv^{\Pi_\star} e_{\star\star}\,.
	\end{equation}
		
	Depending on whether $e_\star$ and $e_{\star\star}$ are old or new the  
	clauses~\ref{it:2002}\,--\,\ref{it:2005} provide a statement equivalent 
	to $e_\star\equiv^{\Sigma} e_{\star\star}$ and in all four cases~\eqref{eq:2052}
	shows that the forward implication holds. 
	
	The backward implication in~\eqref{eq:2210} is clear if $e_\star$ and $e_{\star\star}$ 
	are either both old or both new. If only $e_\star$ is new,~\ref{it:2003} asks for a 
	witness $e'$ and we can simply take $e'=e_{\star\star}$.
	Similarly, if only $e_{\star\star}$ is new, then $e''=e_\star$ exemplifies~\ref{it:2004}.
	This concludes the proof of~\eqref{eq:2210} and, hence, of part~\ref{it:1832a} of the lemma. 
	
	Part~\ref{it:1832b} is much easier and follows from the fact that 
	$\equiv^\Sigma$-equivalence of old edges is decided by~\ref{it:2005} alone. 
	
	Condition~\ref{it:1832c}\ref{it:1832c1} follows from~\ref{it:2003} if $e_\star$ is new 
	and if $e_\star$ is old we just need to set $e'=e_\star$.
	Similarly, for the verification of~\ref{it:1832c}\ref{it:1832c2} one needs to consider 
	four possibilities 
	depending on whether $e_\star$ and $e_{\star\star}$ are old or new. The main case is 
	that both are new and then~\ref{it:2001} yields the desired edges. If only $e_\star$
	is new but $e_{\star\star}$ is old we use~\ref{it:2003} and set $e''=e_{\star\star}$.
	The cases where $e_\star$ is old are similar using~\ref{it:2004},~\ref{it:2005}, and
	$e'=e_\star$. This completes the proof of~\ref{it:1832c}.
	
	The only thing we need to prove for~\ref{it:1832d} is that the copies 
	in $\ccQ$ are subpretrains of $(\Sigma, \equiv^\Sigma)$. Owing to the transitivity 
	of the subpretrain relation this is a direct consequence of~\ref{it:1832a}.  
\end{proof}

\subsection{Proof of Proposition~\ref{prop:ofpt}}
\label{sssec:314}

Given an ordered $f$-partite pretrain $(F, \equiv^F)$ together with a number of colours $r$
we can construct the system $\Rms_r(F)=(G, \ccG)$ as in Section~\ref{subsec:PC}
and enumerate $E(G)=\{e(1), \ldots, e(N)\}$ as usual. Starting with the pretrain picture 
zero $(\Pi_0, \equiv_0, \ccP_0, \psi_0)$ introduced in \S\ref{sssec:2340} we  
construct recursively a 
sequence 
\[
	(\Pi_\alpha, \equiv_\alpha, \ccP_\alpha, \psi_\alpha)_{\alpha\le N}
\]
of pretrain pictures in the 
expected way. That is, if the pretrain picture 
\[
	(\Pi_{\alpha-1}, \equiv_{\alpha-1}, \ccP_{\alpha-1}, \psi_{\alpha-1})
\]
has just been constructed for some positive $\alpha\le N$, we apply the pretrain
construction~$\HJ_r(\cdot)$ to its constituent over $e(\alpha)$, thus getting 
a system of pretrains $(H_\alpha, \equiv^{H_\alpha}, \ccH_\alpha)$, and as explained 
in~\S\ref{sssec:0021} we construct the next pretrain picture 
\[
	(\Pi_\alpha, \equiv_\alpha, \ccP_\alpha, \psi_\alpha)
	=
	(\Pi_{\alpha-1}, \equiv_{\alpha-1}, \ccP_{\alpha-1}, \psi_{\alpha-1})
	\conc
	(H_\alpha, \equiv^{H_\alpha}, \ccH_\alpha)\,.
\]
For the usual reason, the final picture satisfies the partition relation 
\[
	 (\Pi_N, \equiv_N, \ccP_N)\lra (F, \equiv^F)_r\,,
\]
and thereby Proposition~\ref{prop:ofpt} is proved. \qed

Similarly, whenever we have a Ramsey construction $\Phi$ for hypergraphs and a 
partite lemma $\Xi$ for pretrains, we obtain the Ramsey construction $\PC(\Phi, \Xi)$
applicable to pretrains. For instance, we can now regard $\CPL=\PC(\HJ, \HJ)$ also
as a partite lemma for pretrains and then we can move on to the pretrain 
construction $\Omega^{(2)}=\PC(\Rms, \CPL)$. 
\index{clean partite lemma}
\index{$\Omega^{(2)}$}
This construction provides an alternative
proof of Proposition~\ref{prop:ofpt}, which has the obvious advantage to yield systems
of pretrains whose underlying hypergraphs have clean intersections.

\subsection{Orientation}
\label{subsec:7o}

To motivate the material in the next two sections we would brief\-ly like to discuss
the following problem, which seems rather important to us: 
Given a linear pretrain $(F, \equiv^F)$ and a number of colours $r$,
can we find a linear system of pretrains $(H, \equiv^H, \ccH)$ 
such that 
\begin{enumerate}
	\item[$\bullet$] $\ccH\lra (F, \equiv^F)_r$ 
	\item[$\bullet$] and any two copies in $\ccH$ have a clean intersection? 
		(Recall that due to the linearity of $F$ this means, roughly speaking, 
		that any two copies intersect ``at most in an edge''---cf. the discussion 
		before Lemma~\ref{lem:1527}).
\end{enumerate}

As the aforementioned construction $\Omega^{(2)}$ yields such clean intersections, 
it may be tempting to just set $(H, \equiv^H, \ccH)=\Omega^{(2)}_r(F, \equiv^F)$.
However, it can happen for linear pretrains $(F, \equiv^F)$ that the 
pretrain $(H, \equiv^H)$ obtained by means of~$\Omega^{(2)}$ fails to be linear. 
Indeed, Corollary~\ref{cor:0059} only tells us that the hypergraph $H$ needs to 
be linear, but there is no reason why any two wagons of $\equiv^H$ should intersect 
in at most one vertex. As a matter of fact, it could be shown that already our initial 
construction $\HJ$ does not preserve this kind of linearity
and the same holds for the constructions $\CPL$ and $\Omega^{(2)}$ derived from it. 
 
Fortunately we already know a pretrain construction
yielding linear pretrains, namely $\Ext(\Omega^{(2)}, \Omega^{(2)})$ 
(see Lemma~\ref{lem:0059} and the eight-step definition preceding it). 
But the copies provided by this construction can intersect in entire wagons and, 
hence, their intersections are in general quite far from being clean. 
The method of the present subsection suggests a way to remedy this situation by 
looking instead at the 
construction~$\PC\bigl(\Omega^{(2)}, \Ext(\Omega^{(2)}, \Omega^{(2)})\bigr)$. 
At this moment, however, it is not yet clear whether this is a sensible construction, 
for it is difficult to foresee whether at some moment we encounter a picture whose 
constituents are nonlinear pretrains,
in which case it would be impossible to apply the partite lemma 
$\Ext(\Omega^{(2)}, \Omega^{(2)})$.    

An important insight we shall only gain in Section~\ref{sec:1912} is that actually this 
obstruction does not arise. 
So, in other words, $\PC\bigl(\Omega^{(2)}, \Ext(\Omega^{(2)},\Omega^{(2)})\bigr)$
applied to a linear pretrain yields a linear Ramsey pretrain and it is immediately 
clear that this construction solves our problem.  \section{Basic properties of \texorpdfstring{$\GTH$}{Girth}}
\label{subsec:EPAG}

In systems of pretrains there are certain undesirable kinds of 
``short cyclic configurations of copies'', such as two copies intersecting 
the same two wagons in edges, but these cannot be detected by Girth
as defined in \S\ref{sssec:coc}. The goal of the present section is to define 
and study a concept of $\GTH$ applicable to systems of pretrains which, roughly 
speaking, takes care of cyclic configurations of copies that become visible at 
the level of wagons. In the next section, we shall then see that $\GTH$ is highly 
compatible both with the extension process and with the partite construction method. 

\subsection{Basic concepts}
\label{sssec:1657}

Recall that a linear system of hypergraphs has small $\Gth$ 
if all its cycles of copies of low order have one of two properties rendering 
them negligible: either they fail to be semitidy or they have a master copy. 
In the same way we shall define a system of pretrains to have small $\GTH$ 
if all its so-called big cycles of low order behave similarly. 
The notion of a big cycle is defined as follows. 
  
\begin{dfn}\label{dfn:1605}
	A {\it big cycle} in a system of pretrains $(H, \equiv, \ccH)$ is a cyclic sequence
		\[
			\ccC=F_1q_1F_2q_2\ldots F_nq_n
	\]
		with $n\ge 2$ satisfying the following conditions.
	\index{big cycle (in a system of pretrains)}
	\begin{enumerate}[label=($B\arabic*$)]
		\item\label{it:B1} The {\it copies} $F_1, \ldots, F_n$ are in $\ccH$ and we 
			have $F_i\ne F_{i+1}$ for every $i\in \ZZ/n\ZZ$.		
		\item\label{it:B2} The {\it connectors} $q_1, \ldots, q_n$ are distinct and each of them 
			is either a vertex or a wagon of $(H, \equiv)$.  
			\index{connector}
		\item\label{it:B3} If $i\in \ZZ/n\ZZ$ and $q_i$ is a vertex, 
			then $q_i\in V(F_i)\cap V(F_{i+1})$.
		\item\label{it:B4} If $i\in \ZZ/n\ZZ$ and $q_i$ is a wagon, 
			then $E(q_i)\cap E(F_i)\ne\varnothing$ and $E(q_i)\cap E(F_{i+1})\ne\varnothing$.
	\end{enumerate}
\end{dfn}

The main difference between big cycles and the cycles of copies introduced 
in Definition~\ref{dfn:coc} is that now wagons rather than edges act as 
connectors, which causes a corresponding change in the fourth condition 
from~\ref{it:L4} to~\ref{it:B4}. A similar modification allows us to define the order 
of a big cycle. 
Notably, if $\ccC=F_1q_1\ldots F_nq_n$ is a big cycle, we call an 
index~$i\in\ZZ/n\ZZ$
\begin{enumerate}
\item[$\bullet$] {\it pure} if either both of $q_{i-1}$ and $q_i$ are wagons or both are 
	vertices and
	\index{pure index}
\item[$\bullet$] {\it mixed} if one of $q_{i-1}$ and $q_i$ is a wagon while the other 
	one is a vertex,
	\index{mixed index}
\end{enumerate}
and then we call the integer
\[
	\ord{\ccC}=\big|\{i\in\ZZ/n\ZZ\colon i \text{ is pure}\}\big|
		+\tfrac12 \big|\{i\in\ZZ/n\ZZ\colon i \text{ is mixed}\}\big|
\]
the {\it order} of $\ccC$. 
\index{order}

Extended systems can be introduced in the same way as in Section~\ref{subsec:AG}.
Explicitly, if~$(H, \equiv)$ is a pretrain, then every edge $e\in E(H)$ gives rise 
to a subpretrain with vertex set~$e$, having~$e$ as its only edge, and endowed with the 
only possible equivalence relation. There cannot arise confusion if we denote this 
subpretrain again by $e^+$. Moreover, we keep using the notation $E^+(H)$ for 
$\{e^+\colon e\in E(H)\}$, but this time meaning the collection of all subpretrains 
of $(H, \equiv)$ of the form $e^+$. Besides, if $(H, \equiv, \ccH)$ is a system of
pretrains, we shall again write $\ccH^+=\ccH\cup E^+(H)$ and call $(H, \equiv, \ccH^+)$
an {\it extended system of pretrains}. 
\index{extended system of pretrains}
Finally, we keep referring to copies of the 
form $e^+$ as {\it edge copies}, 
\index{edge copy}
while other copies of extended systems of pretrains will be 
called {\it real copies}. 
\index{real copy}
The $M(f)$-notation we have already seen in 
Section~\ref{subsec:AG} can now be applied to wagons. 

\begin{dfn}
	Given a big cycle $\ccC=F_1q_1\ldots F_nq_n$
	in an extended system of pretrains $(H, \equiv, \ccH^+)$ we set 
		\[
		M^\ccC(W)=\bigl\{i\in \ZZ/n\ZZ\colon \text{$q_i$ is a vertex and $q_i\in V(W)$}\bigr\}
	\]
		for every wagon $W$ of the pretrain $(H, \equiv)$. If $\ccC$ is clear from the context, 
	we may omit it and just write $M(W)$ instead.
\end{dfn}

The semitidy cycles of copies considered earlier lead us to acceptable 
big cycles. 

\begin{dfn}\label{dfn:1633}
	A big cycle 
		\[
		\ccC=F_1q_1\ldots F_nq_n
	\]
		in an extended system of pretrains $(H, \equiv, \ccH^+)$ is said to be {\it acceptable} 
	if it has the following three properties. 
	\index{acceptable big cycle}
	\begin{enumerate}[label=($A\arabic*$)]
		\item\label{it:A1} If $\ord{\ccC}=1$, then at least one copy of $\ccC$ is real.
		\item\label{it:A2} If $i\in \ZZ/n\ZZ$ and $q_i$ is a wagon, 
				then $M(q_i) \subseteq\{i-1, i+1\}$. Moreover, if $|M(q_i)|=2$, then there exists 
				no edge $f\in E(H)$ with $q_{i-1}, q_{i+1}\in f$ (see Figure~\ref{fig:A2}).
		\item\label{it:A3} If a wagon $W$ of $(H, \equiv)$ does not appear 
				among $q_1, \ldots, q_n$, then there exists some $i(\star)\in \ZZ/n\ZZ$
				with $M(W)\subseteq \{i(\star), i(\star)+1\}$.
	\end{enumerate}
\end{dfn}

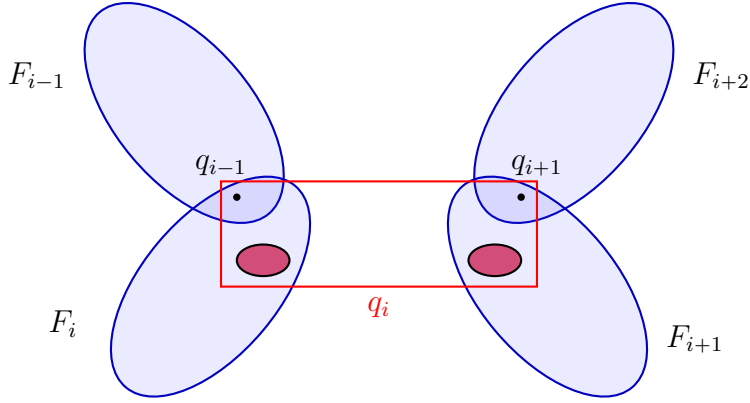
\begin{figure}[ht]
	\centering	
			\begin{tikzpicture}[scale=.7]
	
	\def\w{2.5};
	\def\h{1.3};
	
\fill[blue!80!white, opacity = .1, rotate around={-40:(3.7,2.3)}] (3.7,2.3) ellipse (\h cm and \w cm);
\draw[blue!75!black,thick, rotate around={-40:(3.7,2.3)}] (3.7,2.3) ellipse (\h cm and \w cm);
\fill[blue!80!white, opacity = .1, rotate around={40:(3.2,-1)}] (3.2,-1) ellipse (\h cm and \w cm);
\draw[blue!75!black,thick, rotate around={40:(3.2,-1)}] (3.2,-1) ellipse (\h cm and \w cm);
\fill[blue!80!white, opacity = .1, rotate around={40:(-3.7,2.3)}] (-3.7,2.3) ellipse (\h cm and \w cm);	
\draw[blue!75!black,thick, rotate around={40:(-3.7,2.3)}] (-3.7,2.3) ellipse (\h cm and \w cm);
\fill[blue!80!white, opacity = .1, rotate around={-40:(-3.2,-1)}] (-3.2,-1) ellipse (\h cm and \w cm);
\draw[blue!75!black,thick, rotate around={-40:(-3.2,-1)}] (-3.2,-1) ellipse (\h cm and \w cm);

\fill[purple!70] (2.2,-.5) ellipse (.5cm and .3 cm);
	\draw[black,thick] (2.2,-.5) ellipse (.5cm and .3 cm);
\fill[purple!70] (-2.2,-.5) ellipse (.5cm and .3 cm);
	\draw[black,thick] (-2.2,-.5) ellipse (.5cm and .3 cm);		
	
\draw [red, thick](-3,-1) rectangle (3,1);
\fill (2.7,0.7) circle (2pt);
\fill (-2.7,0.7) circle (2pt);
	
\node [red]at (0,-1.4) {$q_i$};
\node at (-6.5,3) {$F_{i-1}$};
\node at (6.5,3) {$F_{i+2}$};
\node at (-6,-1.7) {$F_{i}$};
\node at (6,-2) {$F_{i+1}$};
\node at (-3,1.3) {$q_{i-1}$};
\node at (3,1.3) {$q_{i+1}$};
\end{tikzpicture}
			
\caption{$|M(q_i)|=2$ and there is no $f\in E(H)$ such 
that $\{q_{i-1}, q_{i+1}\}\subseteq f$.
The purple ellipses represent edges that $q_i$ has in common with $F_i$ and $F_{i+1}$.}\label{fig:A2} 	
\end{figure} 
Let us offer two reasons why we believe that acceptability is a natural concept. 
First, it turns out in Claim~\ref{clm:2036} below that if all wagons consist 
of single edges, then in an obvious sense acceptable cycles are the same as semitidy
cycles. Second, if a big cycle $\ccC=F_1q_1\ldots F_nq_n$ fails to have the 
properties~\ref{it:A2} and~\ref{it:A3}, then it ``decomposes'' into two or more 
shorter cycles of at most the same order. 

For instance, if contrary to~\ref{it:A2} the connector $W=q_n$ 
is a wagon with $M(W)=\{i\}$ for some $i\ne 1, n-1$, then there exists an 
edge $e\in E(W)$ with $q_i\in e$ and~$\ccC$ decomposes into the cycles 
$F_1q_1\ldots F_iq_ie^+W$ and $e^+q_iF_{i+1}q_{i+1}\ldots F_nW$. Similar 
decompositions of $\ccC$ can be found whenever $M(W) \not\subseteq\{1, n-1\}$.
Moreover, if ${M(W)=\{1, n-1\}}$ holds and an edge $f\in E(H)$ with $q_1, q_{n-1}\in f$
exists, then one can decompose~$\ccC$ into the cycles
$\ccA=f^+ q_1F_2q_2\ldots F_{n-1}q_{n-1}$ and $\ccB=F_1 q_1 f^+q_{n-1} F_n W$. 
Without giving details we remark that for linear systems of pretrains $(H, \equiv, \ccH)$ 
condition~\ref{it:A3} can be motivated in a 
similar way. Finally,~\ref{it:A1} just deals with the simplest possible case and the 
reason for declaring big cycles of order $1$ consisting of two edge copies to be 
inacceptable will become clearer in Example~\ref{exmp:2011}.
The next concept has no analogue in Section~\ref{subsec:AG}.

\begin{dfn}\label{dfn:1643}
	Given an extended system of pretrains $(H, \equiv, \ccH^+)$
	we say that a sequence~$P$ is a {\it piece} if 
	\index{piece}
	\begin{enumerate}[label=\rmlabel]
		\item\label{it:1741a} either $P=f^+$ consists of a single edge copy 
		\item\label{it:1741b} or $P=f_1^+Wf_2^+$, where $W$ is a wagon and $f_1^+$, $f_2^+$
			are distinct edge copies with $f_1, f_2\in E(W)$.
	\end{enumerate}
	
	Pieces of the form~\ref{it:1741a} are called {\it short}, 
	\index{short piece}
	while pieces of type~\ref{it:1741b} are said to be {\it long}. 
	\index{long piece}
	If $F_\star\in \ccH^+$ is
	a copy and $P$ is a piece, we shall refer to $P$ as an {\it $F_\star$-piece} if 
	either $P=f^+$ is short and $f\in E(F_\star)$, or if $P=f_1^+Wf_2^+$ is long 
	and $f_1, f_2\in E(F_\star)$.
\end{dfn}

The natural generalisation of master copies to the present context allows to collapse 
copies to such pieces, but the rules concerning long pieces are rather restrictive.

\begin{dfn}\label{dfn:1746}
	Let $\ccC=F_1q_1\ldots F_nq_n$ be a big cycle in an extended system 
	of pretrains $(H, \equiv, \ccH^+)$. A copy $F_\star$ occurring in $\ccC$ is 
	said to be a {\it supreme copy} of $\ccC$ if there exists a family of 
	$F_\star$-pieces $\{P_i\colon i\in\ZZ/n\ZZ \text{ and } F_i\ne F_\star\}$
	with the following properties.
	\index{supreme copy}
	\begin{enumerate}[label=\rmlabel]
		\item\label{it:1751a} The cyclic sequence $\ccD$ obtained from $\ccC$ upon 
			replacing every copy $F_i$ distinct from~$F_\star$ by the 
			corresponding $F_\star$-piece $P_i$ is again a big cycle. 
		\item\label{it:1751b} If $i\in\ZZ/n\ZZ$, $F_i\ne F_\star$, and 
			$P_i=(f'_i)^+W_i(f''_i)^+$ is long, then 
				\begin{enumerate}[label=\glabel]
					\item\label{it:alpha} the connectors $q_{i-1}$, $q_i$ are vertices
					\item\label{it:beta} and there is no edge $f$ with $q_{i-1}, q_i\in f\in E(H)$.
				\end{enumerate}
	\end{enumerate}
\end{dfn}

As this definition is central to everything that follows, we would like to illustrate it 
with two figures. 

\begin{example}
	Suppose first that $\ccC=F_1q_1F_2q_2F_3q_3$ is a big cycle with vertex 
	connectors $q_1$, $q_2$ and a wagon connector $W=q_3$ (see Figure~\ref{fig:74}).
	Under what circumstances is $F_2$ a supreme copy of $\ccC$? 
	Certainly this requires $F_1$ and $F_3$ to be collapsible to 
	some $F_2$-pieces $P_1$,~$P_3$. 
	Due to condition~\ref{it:1751b}\ref{it:alpha} these pieces need to be short. 
	Therefore $P_1=f_1^+$ and $P_3=f_3^+$ have to be edge copies 
	with $f_1, f_3\in E(F_2)$. Moreover, $f_1^+q_1F_2q_2f_3^+q_3$ has to be a big cycle, 
	which requires $q_1\in f_1$, $q_2\in f_3$, and $f_1, f_3\in E(q_3)$.

	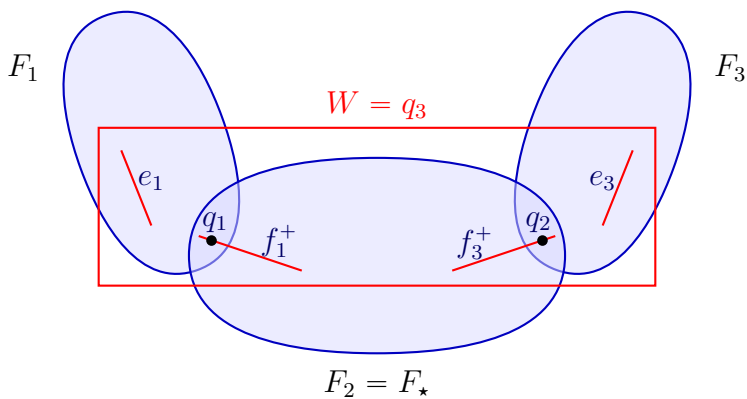
\begin{figure}[ht]
\centering	
\begin{tikzpicture}[scale=1]
	
\newcommand{\edge}[6]{\draw [color={#6},thick, rotate around={#5:(#1,#2)}] (#1-#3, #2) [out=-90, in= 180] to (#1,#2-#4)  [out=0, in=-90] to (#1+#3,#2) [out=90, in=0] to (#1,#2+#4)  [out=-180, in=90] to (#1-#3,#2);}
\newcommand{\fulledge}[6]{\fill [color={#6}, opacity=.5, rotate around={#5:(#1,#2)}] (#1-#3, #2) [out=-90, in= 180] to (#1,#2-#4)  [out=0, in=-90] to (#1+#3,#2) [out=90, in=0] to (#1,#2+#4)  [out=-180, in=90] to (#1-#3,#2);}
	
\def\w{1.8};
\def\h{1};
\coordinate (a) at (-1,1.5);
\coordinate (b) at (2,1);
\coordinate (c) at (2,-3);

\fulledge{-3}{0}{\w}{\h}{-70}{blue!15!white};
\edge{-3}{0}{\w}{\h}{-70}{blue!75!black};
\fulledge{3}{0}{\w}{\h}{70}{blue!15!white};
\edge{3}{0}{\w}{\h}{70}{blue!75!black};
\fulledge{0}{-1.5}{2.5}{1.3}{0}{blue!15!white};
\edge{0}{-1.5}{2.5}{1.3}{0}{blue!75!black};
	
\draw [thick, red] (-3.7,.2) rectangle (3.7, -1.9);
\draw [thick, red](-2.37, -1.24)--(-1,-1.7);
\draw [thick, red](2.37, -1.24)--(1,-1.7);
	
\fill [black] (-2.2,-1.3) circle (2pt);
\fill [black](2.2,-1.3) circle (2pt);
	
\draw [thick, red](-3.4, -.1)--(-3,-1.1);
\draw [thick, red](3.4, -.1)--(3,-1.1);
	
	\node [blue!40!black]at (2.15,-1.05) {$q_2$};
	\node [blue!40!black]at (-2.15,-1.05) {$q_1$};
	
	\node [blue!40!black]at (-3,-.5) {$e_1$};
	\node [blue!40!black]at (3,-.5) {$e_3$};
	
	\node [blue!40!black] at (1.3, -1.33) {$f^+_3$};
	\node [blue!40!black] at (-1.3, -1.3) {$f^+_1$};

	\node [black] at (-4.7,1) {$F_{1}$};
	\node [black] at (4.7,1){$F_{3}$};
	\node [black] at (0,-3.2){$F_2=F_\star$};
	\node [red] at (0,.5) {$W=q_3$};			
\end{tikzpicture}
			
\caption{A cycle with supreme copy $F_2$. Since the connector $q_3$ is a wagon, 
the copies $F_1$, $F_3$ need to get collapsed to short $F_2$-pieces.}
\label{fig:74} 
\end{figure} \end{example}

\begin{example}
	Suppose next that $\ccC=F_1q_1F_2q_2F_3q_3F_4q_4$ is a big cycle 
	all of whose connectors are vertices (see Figure~\ref{fig:75}). 
	In order to determine whether $F_2$
	is a supreme copy of $\ccC$ we first ask ourselves which of the pairs $\{q_4, q_1\}$,
	$\{q_2, q_3\}$, and $\{q_3, q_4\}$ are covered by edges. Assume that there are edges
	$f_1$, $f_3$ such that $q_4, q_1\in f_1$ and $q_2, q_3\in f_3$, but that no edge 
	contains~$q_3$ and $q_4$. Due to condition~\ref{it:1751b}\ref{it:beta} the 
	copies $F_1$, $F_3$ need to get collapsed to short pieces and by~\ref{it:B3} 
	these pieces need to be $f_1^+$, $f_3^+$. Moreover, there exists no short 
	piece~$F_4$ could be collapsed to and thus there needs to be some long $F_2$-piece 
	$(f'_4)^+W(f''_4)^+$ we can use for this purpose.  
			
	\begin{figure}[ht]
\centering	
\begin{tikzpicture}[scale=1]

\newcommand{\edge}[6]{\draw [color={#6},thick, rotate around={#5:(#1,#2)}] (#1-#3, #2) [out=-90, in= 180] to (#1,#2-#4)  [out=0, in=-90] to (#1+#3,#2) [out=90, in=0] to (#1,#2+#4)  [out=-180, in=90] to (#1-#3,#2);}
\newcommand{\fulledge}[6]{\fill [color={#6}, opacity=.5, rotate around={#5:(#1,#2)}] (#1-#3, #2) [out=-90, in= 180] to (#1,#2-#4)  [out=0, in=-90] to (#1+#3,#2) [out=90, in=0] to (#1,#2+#4)  [out=-180, in=90] to (#1-#3,#2);}
	
\def\w{1.8};
\def\h{1.2};
\coordinate (a) at (-1,1.5);
\coordinate (b) at (2,1);
\coordinate (c) at (2,-3);
	
\fulledge{-1.9}{.2}{\w}{\h}{-60}{blue!15!white};
\edge{-1.9}{.2}{\w}{\h}{-60}{blue!75!black};
\fulledge{1.9}{.2}{\w}{\h}{60}{blue!15!white};
\edge{1.9}{.2}{\w}{\h}{60}{blue!75!black};
\fulledge{0}{.7}{\h}{\w}{0}{blue!15!white};
\edge{0}{.7}{\h}{\w}{0}{blue!75!black};
\fulledge{0}{-1.2}{2.5}{1.3}{0}{blue!15!white};
\edge{0}{-1.2}{2.5}{1.3}{0}{blue!75!black};
	
\draw [thick, red](-.78, -.3)--(-.2,-1.8);
\draw [thick, red](.78, -.3)--(.2,-1.8);

	\fill [black] (-1.9,-1.24) circle (2pt);
	\fill [black](1.9,-1.24) circle (2pt);
	\fill (.7,-.5) circle (2pt);
	\fill (-.7,-.5) circle (2pt);
	
	\draw [thick, blue!40!black](-2.3, -1.48)--(-.4,-.32);
	\draw [thick, blue!40!black](2.3, -1.48)--(.4,-.32);

	\node [black]at (2.05,-1.08) {\tiny $q_2$};
	\node [black]at (-2.05,-1.08) {\tiny $q_1$};
	
	\node [black]at (.8,-.15) {\tiny $q_3$};
	\node [black]at (-.8,-.15) {\tiny $q_4$};
	
	\node [blue!40!black] at (1.65, -.7) {\tiny $f^+_3$};
	\node [blue!40!black] at (-1.65, -.7) {\tiny $f^+_1$};

\node [black] at (-3.7,1.3) {$F_{1}$};
\node [black] at (3.7,1.3){$F_{3}$};
\node [black] at (1.1,2.2){$F_4$};
\node [black] at (0,-2.9){$F_2=F_\star$};
\node [red] at (0,-1.5) {$W$};
\node [red] at (-.7,-1.6){\tiny $(f_4'')^+$};
\node [red] at (.7,-1.6){\tiny $(f_4')^+$};
\end{tikzpicture}
\caption{A big cycle with supreme copy $F_2$. The copies $F_1$, $F_3$ are collapsed 
to the short pieces $f_1^+$, $f_3^+$, respectively. Moreover, $F_4$ gets collapsed 
to the long piece $(f_4')^+ W (f_4'')^+$. This requires that no edge 
through $q_3, q_4$ exists.}
\label{fig:75} 
\end{figure}
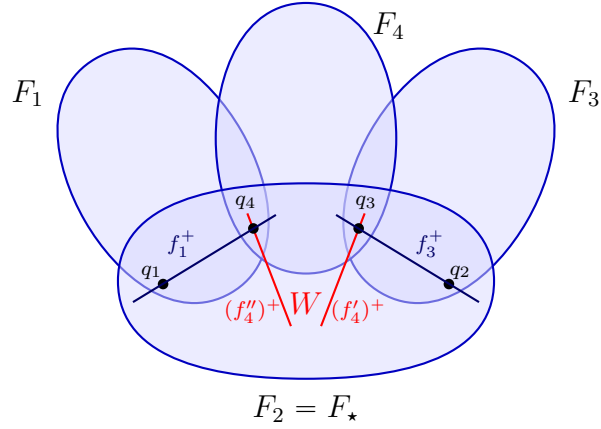 \end{example}
 
\begin{remark}\label{rem:1852}
	\begin{enumerate}[labelsep=0pt, itemindent=20pt, leftmargin=0pt, label=\nlabel]
		\item\label{it:1852a} $\,\,$ In the situation of Definition~\ref{dfn:1746} one can 
			prove with the help of~\ref{it:1751b}\ref{it:alpha} that the new big cycle~$\ccD$ 
			obtained in~\ref{it:1751a} has the same 
			order as $\ccC$. The point of condition~\ref{it:1751b}\ref{it:beta} is that in case
			such an edge $f$ exists it should be better to choose $P_i=f^+$. In other words, 
			long pieces are only allowed if they preserve the order and are
			unavoidable. 
		\item\label{it:1852b} $\,\,$ The same argument as in Fact~\ref{rem:1200} shows that 
			supreme copies are always real copies.
		\end{enumerate} 
\end{remark}

We are now sufficiently prepared for the central concept of this section. 
Recall that in Definition~\ref{dfn:838} we defined a system of pretrains to be 
linear if its underlying pretrain is linear, which in turn means that neither 
edges nor wagons can intersect in more than one vertex.
	 
\begin{dfn}\label{dfn:1820}
	Given an extended system of pretrains $(H, \equiv, \ccH^+)$ as well as a 
	positive integer $g$ we write $\GTH(H, \equiv, \ccH^+)>g$ if 
	\index{$\GTH$}
	\begin{enumerate}
		\item[$\bullet$] the pretrain $(H, \equiv)$ is linear
		\item[$\bullet$] and every acceptable big cycle $\ccC$ in $(H, \equiv, \ccH^+)$ 
			whose order is at most $g$ has a supreme copy.
	\end{enumerate}
\end{dfn}

Let us clarify that the notions ``acceptability'' and ``supreme copy'' are defined 
in all systems of pretrains no matter whether they are linear or not. Accordingly,
one can also ask for nonlinear systems whether they satisfy the statement ``every big cycle 
of order at most~$g$ has a supreme copy''. The answer can very well be affirmative for
a system $(H, \equiv, \ccH^+)$ that fails to be linear, but in such cases we do 
not regard $\GTH(H, \equiv, \ccH^+)>g$ as being true. 

\begin{example}\label{exmp:2011}
	Let $(H, \equiv, \ccH^+)$ be an extended linear system of pretrains.
	Up to symmetry big cycles of order $1$ are of the form $\ccC=F_1xF_2W$, 
	where $x$ is a vertex and $W$ is a wagon. Such cycles always satisfy the acceptability 
	conditions~\ref{it:A2} and~\ref{it:A3}, so the only possibility for them to be 
	inacceptable is that~\ref{it:A1} fails and both $F_1$, $F_2$ are edge copies. 
	This configuration is tantamount to a vertex $x\in V(W)$ having at least the degree $2$
	in $W$. 
	As a matter of fact, we only impose condition~\ref{it:A1} in the definition of 
	acceptability for the reason that we want this to be permissible without impairing 
	the $\GTH$ of $(H, \equiv, \ccH^+)$. 
	
	Next, we figure out what it means to say 
	that $F_1$ is a supreme copy of $\ccC$. By Remark~\ref{rem:1852}\ref{it:1852b} this 
	requires $F_1$ to be real. Moreover, due to 
	Definition~\ref{dfn:1746}\ref{it:1751b}\ref{it:alpha} the copy~$F_2$ can only be collapsed 
	to a short piece $f_1^+$ and an edge $f_1\in E(F_1)$ is appropriate for this purpose 
	if $x\in f_1\in E(W)$. Suppose now that this happens and that $F_2$ is a real copy as well.
	Then $f_1^+xF_2W$ is an acceptable big cycle too and $F_2$ is a supreme 
	copy of this cycle if there exists a further edge $f_2\in E(F_2)$ with $x\in f_2\in E(W)$.
\end{example}

For later reference we summarise this discussion as follows. 

\begin{lemma}\label{lem:GTH1}
	Given an extended linear system of pretrains $(H, \equiv^H, \ccH^+)$ the following 
	two statements are equivalent.
		\begin{enumerate}[label=\alabel]
		\item\label{it:8234a} $\GTH(H, \equiv^H, \ccH^+)>1$;
		\item\label{it:8234b} For every big cycle $\ccC=F_1xF_2W$ there are 
			edges 
						\[
				f_1\in E(F_1)\cap E(W)
				\quad \text{ and } \quad
				f_2\in E(F_2)\cap E(W)
			\]
						such that $x\in f_1\cap f_2$.
	\end{enumerate}
\end{lemma}

\begin{proof}
	Suppose first that~\ref{it:8234a} holds and let $\ccC=F_1xF_2W$
	be a big cycle.
	If $F_1$ is a real copy the existence of $f_1$ was proved in 
	Example~\ref{exmp:2011} and if $F_1$ is an edge copy it suffices to take~$f_1$ to be 
	the only edge of $F_1$. Similarly, the desired edge $f_2$ exists as well and 
	thereby~\ref{it:8234b} is proved. 
	
	For the converse direction we let $\ccC=F_1xF_2W$ be an acceptable big cycle 
	of order~$1$. Due to~\ref{it:A1} and symmetry we can suppose that the copy $F_1$
	is real. If the edge $f_1$ is obtained from assumption~\ref{it:8234b}, then $F_2$
	is collapsible to the short $F_1$-piece $f_1^+$ and thus~$F_1$ is a supreme 
	copy of $\ccC$. 	 
\end{proof} 

We proceed with some further observations on cycles of length $2$.

\begin{lemma}\label{lem:n1624}
	Let $(H, \equiv^H, \ccH^+)$ be an extended linear system of pretrains. 
	Provided that $\GTH(H, \equiv^H, \ccH^+)>2$, there are no big cycles of length $2$
	both of whose connectors are wagons. 
\end{lemma}

\begin{proof}
	Assume that $\ccC=F_1W_1F_2W_2$ is a big cycle in $(H, \equiv^H, \ccH^+)$ with
	wagons $W_1$, $W_2$. Since $\ccC$ has no vertex connectors, it satisfies all 
	three acceptability conditions, and thus it needs to have a supreme copy, say $F_1$. 
	Now we can collapse $F_2$ to a short piece $f^+$ 
	(see Definition~\ref{dfn:1746}\ref{it:1751b}\ref{it:alpha}) 
	and the fact that $F_1W_1f^+W_2$ is again a big cycle yields the 
	contradiction $f\in E(W_1)\cap E(W_2)$.
\end{proof}

In analogy with Lemma~\ref{lem:1522}\ref{it:1522a} we have the following result.

\begin{lemma}\label{lem:2037}
	If $\ccC$ is a big cycle of length $2$ in an extended linear system of pretrains
	$(H, \equiv, \ccH^+)$ with $\GTH(H, \equiv, \ccH^+)>2$, then every real copy 
	of $\ccC$ is a supreme copy of $\ccC$.
\end{lemma}

\begin{proof}
	The case $\ord{\ccC}=1$ has already been discussed in Example~\ref{exmp:2011} 
	and the case that both connectors of $\ccC$ are wagons was excluded in the 
	previous lemma. 
	 	
	Thus it remains to treat the case that $\ccC$ is of the form $F_1xF_2y$ for some 
	vertices $x$ and~$y$, and that, without loss of generality, $F_1$ is a real copy. 
	Evidently, $\ccC$ is acceptable and, therefore, it needs to possess a supreme
	copy. The interesting case occurs if $F_2$ is a supreme copy of~$\ccC$. Let 
	the $F_2$-piece $P$ exemplify this fact. This means that $PxF_2y$ is a big cycle, 
	and it follows that $\ccD=F_1xPy$ is a big cycle, too. 
	Clearly, $\ccD$ has order $2$. 
			
	\smallskip
	
	{\it \hskip2em First Case. The piece $P=f^+$ is short.}
	
	\smallskip
	
	Now $\ccD$ is acceptable and due to Remark~\ref{rem:1852}\ref{it:1852b} 
	its supreme copy is $F_1$. If the $F_1$-piece~$Q$ witnesses 
	this state of affairs, then~$Q$ exemplifies the supremacy of $F_1$ in $\ccC$ 
	as well.    
	
	\smallskip
	
	{\it \hskip2em Second Case. The piece $P=(f')^+W(f'')^+$ is long.}
	
	\smallskip
	
	Again $\ccD$ is acceptable, the moreover-part of~\ref{it:A2} being ensured by 
	Definition~\ref{dfn:1746}\ref{it:1751b}\ref{it:beta}. As before, $F_1$ is 
	a supreme copy of $\ccD$ and by Definition~\ref{dfn:1746}\ref{it:1751b}\ref{it:alpha}
	there exist short $F_1$-pieces $f_\star^+$, $f_{\star\star}^+$ such that 
	$F_1xf_\star^+Wf_{\star\star}^+y$ is a big cycle. Now $f_\star^+Wf_{\star\star}^+$
	is a long $F_1$-piece verifying that~$F_1$ is a supreme copy of $\ccC$.
\end{proof}

$\GTH$ relates to strong inducedness in the following way. 

\begin{lemma}\label{lem:2107}
	If an extended system of pretrains $(H, \equiv, \ccH^+)$ satisfies
		\[
		\GTH(H, \equiv, \ccH^+)>2\,,
	\]
	then every edge $f$ intersecting a copy $F_\star\in \ccH^+$ 
	in at least two vertices belongs to that copy.  
\end{lemma}  

Notice that due to the linearity of $H$ the three items~\ref{it:1836a}\,--\,\ref{it:1836c} 
listed in Fact~\ref{fact:1836} tell us what it means to say that a copy 
in $\ccH^+$ is strongly induced in $H$. The first of these conditions coincides with 
the conclusion of the lemma above. The remaining conditions,~\ref{it:1836b} 
and~\ref{it:1836c}, are usually satisfied in practice but we cannot deduce them 
from the assumption $\GTH(H, \equiv, \ccH^+)>2$. 

\begin{proof}[Proof of Lemma~\ref{lem:2107}]
	Suppose that $x, y\in V(F_\star)\cap f$ 
	are distinct. 
	If $F_\star=f^+$, then $f\in E(F_\star)$ is clear, so suppose $F_\star\ne f^+$ from 
	now on. The big cycle $\ccC=F_\star x f^+ y$ is acceptable, its order is $2$ and, 
	therefore, it possesses a supreme copy. By Remark~\ref{rem:1852}\ref{it:1852b}
	this supreme copy can only be $F_\star$ and the existence of an edge containing $x$
	and $y$ ensures that $f^+$ collapses to a short piece $(f')^+$ with $f'\in E(F_\star)$.
	Since $x, y\in f\cap f'$, the linearity of $H$ implies $f=f'$ and thus we have indeed 
	$f\in E(F_\star)$.
\end{proof}

We conclude our list of preliminary observations on $\GTH$ as follows.

\begin{fact}\label{f:1644}
	If a big cycle $\ccC$ in an extended linear system of pretrains has a supreme 
	copy~$F_\star$, then all vertex connectors of $\ccC$ belong to $F_\star$.  
\end{fact} 

\begin{proof}
	Let $q_i$ be a vertex connector of $\ccC=F_1q_1\dots F_nq_n$. Due to~\ref{it:B1}
	at least one of the copies $F_i$, $F_{i+1}$ is distinct from $F_\star$, so by symmetry 
	we may assume $F_i\ne F_\star$. Consider the $F_\star$-piece $P_i$ to which $F_i$ 
	collapses. If $P_i=f^+$ is short, then $q_i\in f\subseteq V(F_\star)$ follows. 
	Similarly, if $P_i=(f')^+W(f'')^+$ is long, then we have $q_i\in f''\subseteq V(F_\star)$.
\end{proof}

\subsection{Special cases}
\label{sssec:1758}

The next item on our agenda are simplified characterisations of 
$\GTH(H, \equiv, \ccH^+)$ applicable to systems of pretrains with special properties. 
First, the simplest case for $\ccH^+$ is that $\ccH^+=E^+(H)$ and we address this 
situation in Lemma~\ref{lem:2310} below. Second, the case where $\ccH^+$ is arbitrary 
and every wagon consists of a single edge is dealt with in Lemma~\ref{lem:0217}; 
it turns out that $\GTH$ is then essentially the same as $\Gth$.   

\begin{dfn}\label{dfn:2256}
	For a pretrain $(H, \equiv)$ and an integer $g\ge 2$ we write 
	$\ggth(H, \equiv)> g$ if~$H$ is linear and  
	the girth of the set system one obtains from~$(H, \equiv)$ 
	upon replacing the wagons by new edges (and deleting the original edges) 
	exceeds~$g$. 
	\index{$\ggth$}
\end{dfn}

Due to Fact~\ref{fact:231a} this means, explicitly, that if for some $n\in [2, g]$ we have 
a cyclic sequence $W_1q_1\ldots W_nq_n$ such that 
\begin{enumerate}
	\item[$\bullet$] $W_1, \ldots, W_n$ are wagons of $(H, \equiv)$, 
	\item[$\bullet$] the vertices $q_1, \ldots, q_n$ of $H$ are distinct,
	\item[$\bullet$] and $q_i\in V(W_i)\cap V(W_{i+1})$ holds for every $i\in\ZZ/n\ZZ$,
\end{enumerate}
then $W_1=\dots=W_n$. 
Notice that a pretrain $(H, \equiv)$ is linear in the sense 
of Definition~\ref{dfn:838} if and only if $\ggth(H, \equiv) > 2$. 

Let us emphasise that the lowercase ``$\mathfrak{g}$'' in Definition~\ref{dfn:2256} as 
opposed to the capital~``$\mathfrak{G}$'' in Definition~\ref{dfn:1820} indicates that
here we deal with individual objects as opposed to the systems treated there. So the rule
on capitalisation is the same as in Section~\ref{subsec:AG}. It might further be helpful 
to bear in mind that German letters highlight the importance of equivalence relations 
on the edge sets, while Roman letters are used in cases where such pretrain structures
are absent or ignored.
 
\begin{lemma}\label{lem:2310}
	If $(H, \equiv)$ is a linear pretrain and $g\ge 2$, then 
		\[
		\ggth(H, \equiv)>g
		\, \Longleftrightarrow \,
		\GTH\bigl(H, \equiv, E^+(H)\bigr)>g
		\, \Longleftrightarrow \,
		\GTH\bigl(H, \equiv, E^+(H)\cup\{(H, \equiv)\}\bigr) > g\,.
	\]
	\end{lemma}

\begin{proof}
	Evidently, the last condition implies the middle one. Conversely, the middle 
	condition also implies the last one, for $(H, \equiv)$ is a supreme copy of every big 
	cycle in the system of pretrains $\bigl(H, \equiv, E^+(H)\cup\{(H, \equiv)\}\bigr)$ 
	that contains $(H, \equiv)$. It therefore remains to show that the first and 
	middle statement are equivalent.
	
	Suppose first that $\GTH\bigl(H, \equiv, E^+(H)\bigr)>g$, but that for some $n\in[2, g]$
	there exists a cyclic sequence $W_1q_1\ldots W_nq_n$ violating $\ggth(H, \equiv)>g$.
	If we choose this counterexample with $n$ as small as possible, 
	then $W_1, \ldots, W_n$ are distinct. Moreover, the linearity of $(H, \equiv)$ 
	yields $n\ge 3$.
		
	Let us select for every index $i\in \ZZ/n\ZZ$ a piece $P_i$ as follows. If there 
	exists an edge $f_i\in W_i$ with $q_{i-1}, q_i\in f_i$, then let $P_i=f_i^+$.
	If there is no such edge, then we pick edges $f'_i, f''_i\in E(W_i)$ with 
	$q_{i-1}\in f'_i$ and $q_i\in f''_i$, which is possible due to the fact that wagons 
	have no isolated vertices. Moreover, the absence of $f_i$ entails $f'_i\ne f''_i$
	and thus $P_i=(f'_i)^+W_i(f''_i)^+$ is a long piece. Now $\ccC=P_1q_1\ldots P_nq_n$
	is a big cycle in $\bigl(H, \equiv, E^+(H)\bigr)$ with $\ord{\ccC}=n\in [3, g]$.
	By $n\ge 3$ it has the property~\ref{it:A1} of acceptability 
	and one checks easily that 
	any violation of~\ref{it:A2} or~\ref{it:A3} would yield a contradiction to the minimality 
	of~$n$. Thus $\ccC$ is acceptable and Definition~\ref{dfn:1746} entails that $\ccC$
	has a supreme copy. However,~$\ccC$ does not even contain a real copy, so we 
	arrive at a contradiction to Remark~\ref{rem:1852}\ref{it:1852b}.
	
	It remains to prove that, conversely, assuming $\ggth(H, \equiv)>g$ 
	we can derive 
		\[	
		\GTH\bigl(H, \equiv, E^+(H)\bigr)>g\,. 
	\]
		Consider any acceptable big cycle 
	$\ccC=f_1^+q_1\ldots f_n^+q_n$ in $\bigl(H, \equiv, E^+(H)\bigr)$ with $\ord{\ccC}\le g$.
	Due to~\ref{it:A1} the order of $\ccC$ is at least $2$. 
	Notice that owing to~\ref{it:B2} and~\ref{it:B4} is cannot be the case that two 
	consecutive connectors of $\ccC$ are wagons. Therefore, the order of $\ccC$
	is the number of its vertex connectors. By symmetry we may suppose that 
		\[
		1\le i(1)<i(2)<\ldots <i(m)= n
	\]
		are the indices of the vertex connectors of $\ccC$,
	where $m=\ord{\ccC}\in [2, g]$. As $q_1$ and~$q_2$ cannot be consecutive wagons, $i(1)$
	is either $1$ or $2$ and in both cases there exists a wagon~$W_1$ containing all edges 
	$f_i$ with $i\in [1, i(1)]$. Similarly, for every $\mu\in [2, m]$   
	there exists a wagon~$W_\mu$ containing all edges $f_i$ with $i\in [i(\mu-1)+1, i(\mu)]$.
	Now $W_1q_{i(1)}\ldots W_mq_{i(m)}$ is a ``wagon cycle'' and $\ggth(H, \equiv)>g\ge m$
	tells us that there exists a wagon~$W$ with 
		\[
		W=W_1=\dots = W_m\,.
	\]
		
	In particular,
		\[
		f_1, \ldots, f_n\in E(W)\,,
		\qquad 
		q_{i(1)}, \ldots, q_{i(m)} \in V(W)\,,
	\]
		and the only wagon possibly appearing on $\ccC$ is $W$. 
		
	If $W$ does not appear on $\ccC$, then~\ref{it:A3} yields $n=m=|M^\ccC(W)|\le 2$, 
	contrary to the linearity of $H$.
	This shows that~$W$ appears on $\ccC$ and~\ref{it:A2} discloses $m=1$, contrary 
	to~\ref{it:A1}.
\end{proof}

The next result essentially says that if every edge is its own wagon, then $\GTH$ is the same 
as $\Gth$.

\begin{lemma} \label{lem:0217}
	Let $(H, \ccH)$ be a linear system and let $g\ge 2$ be a natural number.
	If $\equiv$ denotes the equivalence relation on $E(H)$ whose equivalence classes are 
	single edges, then 
		\[
		\Gth(H, \ccH^+) >g
		\,\,\, \Longleftrightarrow \,\,\,
		\GTH(H, \equiv, \ccH^+) >g\,.
	\]
	\end{lemma}

\begin{proof}
	For simplicity we treat cycles of copies in $(H, \ccH^+)$ as if they were big cycles 
	in $(H, \equiv, \ccH^+)$ and vice versa. 
	Pedantically speaking, this is not quite precise, for in the former case the 
	non-vertex connectors are just edges of $H$, while in the latter 
	case they are wagons and thus of the form $(e, \{e\})$ with $e\in E(H)$. Moreover, 
	in the former case the copies are just subhypergraphs of $H$ belonging to $\ccH^+$,
	whereas in the latter case they are, officially, subpretrains of $(H, \equiv)$ 
	and thus accompanied by equivalence relations on their sets of 
	edges. 
	Ignoring these extremely minor differences it will be convenient 
	to identify the two concepts for the current purposes; thus when speaking 
	of cycles in the remainder of this proof we shall mean either cycles of copies 
	in $(H, \ccH^+)$ or the corresponding big cycles in $(H, \equiv, \ccH^+)$.   
	It turns out that there is a quite literal translation from 
	Definition~\ref{dfn:1343} to Definition~\ref{dfn:1820}.
	
	\begin{clm} \label{clm:2036}
		A cycle is semitidy if and only if it is acceptable.
	\end{clm}    
	
	\begin{proof}
		Consider a semitidy cycle $\ccC$. Clearly, $\ccC$ satisfies~\ref{it:A1},
		and~\ref{it:S1},~\ref{it:S2} imply~\ref{it:A2} and~\ref{it:A3}, respectively.
		Thus semitidiness does indeed imply acceptability.
		
		Conversely, let $\ccC$ be an acceptable cycle. 
		If~\ref{it:S1} fails for some edge connector $e$, then~\ref{it:A2} shows that
		the connectors next to $e$ are two distinct vertices in~$e$. However, this 
		causes~$e$ itself to violate the moreover-part of~\ref{it:A2}. This proves~\ref{it:S1}.
		Condition~\ref{it:S2} follows from~\ref{it:A3}.  
	\end{proof}
	
	\begin{clm} \label{clm:2055}
		A cycle has a master copy if and only if it has a supreme copy. 
	\end{clm} 
	
	\begin{proof}
		If $F_\star$ is a master copy of a cycle $\ccC$, then the family 
		of edges of $F_\star$ exemplifying this fact leads to a corresponding 
		family of short $F_\star$-pieces exemplifying that $F_\star$ is a supreme copy
		of~$\ccC$.
		Taking into account that our special choice of $\equiv$ precludes the existence 
		of long pieces the converse direction is proved similarly.
	\end{proof}
	
	Owing to Lemma~\ref{lem:1342} our claims show that we are done. 
\end{proof}

\subsection{Two further facts}
\label{sssec:1815}

We conclude this section by returning to general systems of pretrains and proving 
two statements concerning big cycles and $\GTH$. The first of them asserts that not 
too many edge copies of an acceptable big cycle can belong to the same wagon.
 
\begin{lemma}\label{lem:2145}
	Given an acceptable big cycle $\ccC=F_1q_1\ldots F_nq_n$ of length $n\ge 3$
	in an extended system of pretrains $(H, \equiv, \ccH^+)$ and a wagon $W$
	of $(H, \equiv )$ there exists an index~$i(\star)$ such that the set 
		\[
		Q(W)=\bigl\{i\in\ZZ/n\ZZ\colon F_i=e_i^+
		\text{ is an edge copy with } e_i\in E(W)\bigr\}
	\]
		satisfies $Q(W)\subseteq \{i(\star), i(\star)+1\}$. Moreover, if $|Q(W)|=2$,
	then $q_{i(\star)}=W$ and the connectors $q_{i(\star)-1}$, $q_{i(\star)+1}$
	are vertices (see Figure~\ref{fig:71}). 
	In particular, if a vertex connector is between two edge copies, 
	then their underlying edges belong to distinct wagons. 
\end{lemma}

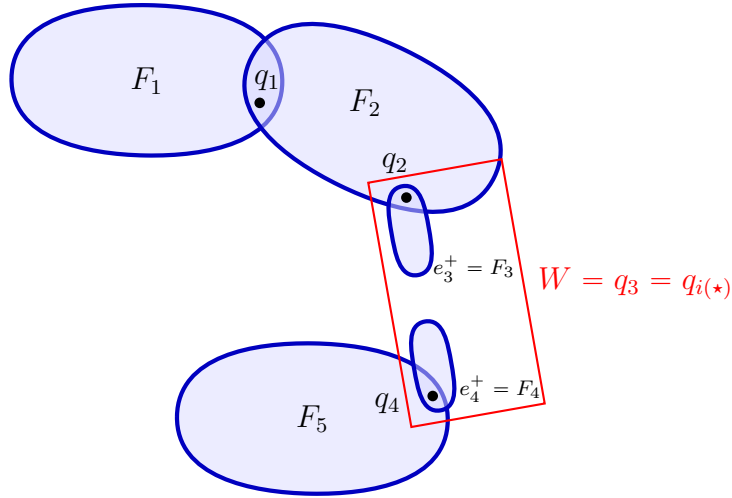
\begin{figure}[ht]
	\centering	
			\begin{tikzpicture}[scale=1]
	
		\newcommand{\edge}[6]{\draw [color={#6},ultra thick, rotate around={#5:(#1,#2)}] (#1-#3, #2) [out=-90, in= 180] to (#1,#2-#4)  [out=0, in=-90] to (#1+#3,#2) [out=90, in=0] to (#1,#2+#4)  [out=-180, in=90] to (#1-#3,#2);}
		\newcommand{\fulledge}[6]{\fill [color={#6}, opacity=.5, rotate around={#5:(#1,#2)}] (#1-#3, #2) [out=-90, in= 180] to (#1,#2-#4)  [out=0, in=-90] to (#1+#3,#2) [out=90, in=0] to (#1,#2+#4)  [out=-180, in=90] to (#1-#3,#2);}
	
	\def\w{1.8};
	\def\h{1};
	\coordinate (a) at (-1,1.5);
	\coordinate (b) at (2,1);
	\coordinate (c) at (2,-3);
	
	\fulledge{-1}{1.5}{\w}{\h}{-1}{blue!15!white};
	\edge{-1}{1.5}{\w}{\h}{-1}{blue!75!black};
	\fulledge{2}{1}{\w}{\h}{-26}{blue!15!white};
	\edge{2}{1}{\w}{\h}{-26}{blue!75!black};
	\fulledge{1.2}{-3}{\w}{\h}{-1}{blue!15!white};
	\edge{1.2}{-3}{\w}{\h}{-1}{blue!75!black};
	
		\draw [red, thick, rotate around ={10:(3,-1.25)}](2.2,.3) rectangle (4,-3);
	
	\fulledge {2.5}{-.5}{.25}{.6}{10}{blue!15!white};		
	\edge {2.5}{-.5}{.25}{.6}{10}{blue!75!black};
	\fulledge {2.8}{-2.3}{.25}{.6}{10}{blue!15!white};
	\edge {2.8}{-2.3}{.25}{.6}{10}{blue!75!black};
	
	\fill (2.45,-.06) circle (2pt);
	\fill (2.8,-2.7) circle (2pt);
	\fill (.5,1.2) circle (2pt);
	
	\node at (.6,1.5) {$q_1$};
	\node at (2.3,.4) {$q_2$};
	\node at (2.2,-2.8) {$q_4$};
	
	\node [red] at (5.5,-1.2) {$W=q_3=q_{i(\star)}$};
	\node at (-1,1.5) {$F_1$};
	\node at (1.9,1.2){$F_2$};
	\node at (1.2,-3){$F_5$};
	
	\node at (3.35,-1) {\tiny $e^+_3=F_3$};
	\node at (3.7,-2.6) {\tiny $e^+_4=F_4$};
	
\end{tikzpicture}
			
\caption{The case $Q(W)=\{i(\star), i(\star)+1\}$ (and $i(\star)=3$) of Lemma~\ref{lem:2145}}
\label{fig:71} 
\end{figure} 
\begin{proof}
	Notice that if $i\in Q(W)$, then by~\ref{it:B3} and~\ref{it:B4}
	each of the connectors $q_{i-1}$, $q_i$ is either a vertex belonging to $V(W)$
	or $W$ itself. Thus if $W$ does not appear on $\ccC$, then the acceptability 
	condition~\ref{it:A3} implies $|Q(W)|\le 1$. Furthermore, if $W=q_{i(\star)}$
	appears on~$\ccC$, then~\ref{it:A2} discloses $Q(W)\subseteq \{i(\star), i(\star)+1\}$.
	Finally, if this holds with equality, then~$q_{i(\star)-1}$ and $q_{i(\star)+1}$
	are indeed vertices.  
\end{proof}

We proceed with a natural analogue of Lemma~\ref{lem:949} that 
reads as follows. 

\begin{lemma}\label{lem:1448} 
	Let $\ccC=F_1q_1\ldots F_nq_n$ be an acceptable big cycle of length $n\ge 3$ 
	in a linear extended system of pretrains $(H, \equiv, \ccH^+)$. Suppose that for 
	some set of indices $K\subseteq \ZZ/n\ZZ$ we have a family of pieces $\{P_k\colon k\in K\}$
	such that if for some $k\in K$ the piece $P_k=(f'_k)^+W_k(f''_k)^+$ is long, 
	then 
	\begin{enumerate}[label=\glabel]		
		\item\label{it:1801a} the connectors $q_{k-1}$, $q_k$ are vertices,
		\item\label{it:1801b} there is no edge $f\in E(H)$ with $q_{k-1}, q_k\in f$,
		\item\label{it:1801c} and $W_k\not\in\{q_1, \ldots, q_n\}$.
	\end{enumerate}
	Let $\ccD$ be the cyclic sequence obtained from $\ccC$ upon replacing 
	every copy $F_k$ with $k\in K$ by the corresponding piece $P_k$. 
	\begin{enumerate}[label=\alabel]
		\item\label{it:1448a} If $\ccD$ satisfies~\ref{it:B3} and~\ref{it:B4}, then it 
			is an acceptable big cycle.  
		\item\label{it:1448b} Moreover, if $\ccD$ has a supreme copy, then so does $\ccC$.
	\end{enumerate}
\end{lemma}

Notice that the conditions~\ref{it:1801a} and~\ref{it:1801b} are the same as in 
Definition~\ref{dfn:1746}\ref{it:1751b}.

\begin{proof}[Proof of Lemma~\ref{lem:1448}]
	A wagon connector of $\ccD$ will be called {\it new} if 
	it occurs in the middle of a long piece $P_k$ and otherwise, i.e., if has been inherited 
	from  $\ccC$, it will be called {\it old}. This distinction is not necessary for vertex
	connectors, since all of them are ``old''. Notice that condition~\ref{it:1801c} imposed 
	on admissible long pieces $P_k$ says that no wagon is at the same time old and new. 
	
	\smallskip
	\noindent
	{\bf Proof of part~\ref{it:1448a}.} The two claims that follow will establish that 
	$\ccD$ is indeed a big cycle. 
	
	\begin{clm}
		The wagons of $\ccD$ are distinct and, hence, $\ccD$ satisfies~\ref{it:B2}.
	\end{clm}
	
	\begin{proof}
		This could only fail, if some new wagon occurs twice, i.e., if there are 
		distinct indices $k, \ell\in K$ such that the pieces $P_k$ and~$P_{\ell}$ 
		are long and contain the same wagon $W_k=W_{\ell}$ in the middle.
		However, this would necessitate $\{k-1, k, \ell-1, \ell\}\subseteq  M^\ccC(W_k)$ 
		and, hence, $|M^\ccC(W_k)|\ge 3$, contrary to~\ref{it:A3}. 
	\end{proof}
	
	\begin{clm}
		The cyclic sequence $\ccD$ satisfies~\ref{it:B1} and, hence, it is a big cycle.  
	\end{clm}
	
	\begin{proof}	
		Otherwise, some edge copy $e^+$ occurs twice in consecutive positions, i.e., $\ccD$ 
		has a subsequence of the form $r_{i-1}e^+r_ie^+r_{i+1}$. 
		Let $W_\star$ denote the unique wagon with ${e\in E(W_\star)}$. Notice that all wagons
		among $r_{i-1}$, $r_{i}$, and $r_{i+1}$ are equal to $W_\star$, while all vertices 
		among these connectors are in $e$. Thus, if all three of these connectors are vertices,
		then $|M^\ccC(W_\star)|\ge 3$, contrary to~\ref{it:A2} or~\ref{it:A3}. 
		Owing to~\ref{it:B2} it remains to consider the case that two of these connectors 
		are vertices, while the third one is equal to $W_\star$.
	
		Suppose first that $W_\star=r_i$. Now $W_\star$ cannot be new, 
		for $e^+W_\star e^+$ does not qualify as a piece. However, if $W_\star$ is
		old, then the edge $e$ contradicts the moreover-part of $W_\star$ satisfying~\ref{it:A2}. 
		
		By symmetry, it only remains to discuss the case that $W_\star=r_{i-1}$. 
		If $W_\star$ is old, then~$r_i$,~$r_{i+1}$ 
		correspond to indices in $M^\ccC(W_\star)$ and we obtain a contradiction to~\ref{it:A2}. 
		Finally, let $W_\star$ be new. The existence of a new wagon 
		implies $\len{\ccD} > \len{\ccC}\ge 3$, i.e., $\len{\ccD}\ge 4$. 
		So $r_{i-2}$, $r_i$, and $r_{i+1}$ are three distinct vertices of $W_\star$  
		that witness $\big|M^\ccC(W_\star)\big|\ge 3$. This contradiction to~\ref{it:A3} concludes 
		the proof that $\ccD$ is indeed a big cycle. 
	\end{proof}
	
	It remains to show that~$\ccD$ is acceptable. Notice that condition~\ref{it:A1} is clear.
	
	\begin{clm}
		The big cycle $\ccD$ satisfies~\ref{it:A2}.
	\end{clm}
	
	\begin{proof}
		The first part of~\ref{it:A2} holds for old wagons because it 
		holds in $\ccC$ and it holds for new wagons because of~\ref{it:A3}.
		If for some wagon $W$ of $\ccD$ 
		there is an edge $f$ as in the moreover-part of~\ref{it:A2}, then~\ref{it:beta}
		shows that $W$ is old and we reach a contradiction to the fact that $\ccC$
		satisfies~\ref{it:A2}. 	
	\end{proof}
	
	\begin{clm}
		Moreover, $\ccD$ satisfies~\ref{it:A3} and is, hence, acceptable.
	\end{clm}
	
	\begin{proof}
		Let $W$ be a wagon not appearing in $\ccD$. Since $W$ cannot belong to $\ccC$
		either and~$\ccC$ satisfies~\ref{it:A3}, the only problem that could arise is
		that the set $M^\ccC(W)$ consists of two consecutive indices and that we want 
		to insert a long piece $P=(f')^+W_\star(f'')^+$ between the corresponding vertex
		connectors. If this happens, then these connectors belong to both wagons $W$ 
		and $W_\star$. Due to the linearity of $(H, \equiv)$ it follows that $W=W_\star$.
		As $\ccD$ contains the piece $P$, this contradicts the assumption that $W$ lies 
		outside $\ccD$. 
	\end{proof}
	
	\noindent
	{\bf Proof of part~\ref{it:1448b}.} 
	Rewrite $\ccD=G_1r_1\ldots G_mr_m$ and suppose that 
		\[
		\bigl\{Q_i\colon i\in \ZZ/m\ZZ \text{ and } G_i\ne F_\star\bigr\}
	\]
		is a family of $F_\star$-pieces exemplifying that $F_\star$ is a supreme copy of $\ccD$.
	Recall that $F_\star$ is a real copy and, consequently, it appears in $\ccC$. 
	In the light of part~\ref{it:1448a} it suffices to prove that for every 
	index $k\in\ZZ/n\ZZ$ with $F_k\ne F_\star$ there exists an $F_\star$-piece 
	$R_k$ such that 
	\begin{enumerate}[label=\nlabel]
		\item\label{it:1254a} the cyclic sequence arising from $\ccC$ if one exchanges $F_k$
			by $R_k$ satisfies~\ref{it:B3} and~\ref{it:B4}
		\item\label{it:1254b} and if $R_k$ is long, then it has the 
			properties~\ref{it:1801a},~\ref{it:1801b}, and~\ref{it:1801c}.
	\end{enumerate} 
	
	If $k\not\in K$, then in the passage from $\ccC$ to $\ccD$ the copy $F_k$ is preserved 
	and receives a new index $i(k)\in \ZZ/m\ZZ$. Moreover, we can take $R_k=Q_{i(k)}$.
	Suppose next that $k\in K$ and that the piece $P_k$ is short. Now the copy $F_k$ of $\ccC$ 
	got replaced by an edge copy $P_k=G_{i(k)}$ in $\ccD$. Since $F_\star$ is a real copy, 
	we have 
	$F_\star\ne G_{i(k)}$ and, therefore, we can again take $R_k=Q_{i(k)}$. 
	
	Suppose finally that $k\in K$ and that the piece $P_k=(f'_k)^+W_k(f''_k)^+$ is long. 
	In~$\ccD$ we have new indices $(f'_k)^+=G_{i(k)}$, $W_k=r_{i(k)}$, 
	and $(f''_k)^+=G_{i(k)+1}$. Moreover, the pieces~$Q_{i(k)}$ and~$Q_{i(k)+1}$ are short
	and $R_k=Q_{i(k)}W_kQ_{i(k)+1}$ is a long $F_\star$-piece with the desired properties. 
\end{proof} \section{\texorpdfstring{$\GTH$}{Girth} in constructions}
\label{sec:1912}

As we shall see in this section, the extension process and partite constructions 
are in perfect harmony with $\GTH$. 

\subsection{The extension lemma}
\label{subsec:ExtL}

Roughly speaking, we shall prove now that the extension process described in
Section~\ref{subsec:EP} converts $\Gth$ into $\GTH$---in the sense that if a 
hypergraph construction $\Psi$ yields systems with large $\Gth$, then the 
pretrain construction $\Ext(\Phi, \Psi)$ yields systems with 
large $\GTH$, provided that $\Phi$ generates strongly induced copies.  
In other words, German $\GTH$ seems to be the correct notion for analysing 
cycles in systems of pretrains obtained by means of the extension process. 
 
\begin{lemma}\label{lem:1955}
	Suppose that $g\ge 2$, 
	\begin{enumerate}
		\item[$\bullet$] that $\Phi$ is a linear ordered Ramsey construction for 
			hypergraphs delivering systems with strongly induced copies
		\item[$\bullet$] and that $\Psi$ denotes a Ramsey construction applicable to 
			ordered hypergraphs $M$ with $\gth(M) > g$ and producing 
			ordered systems of hypergraphs $\Psi_r(M)=(N, \ccN)$ with $\Gth(N, \ccN^+)>g$.
	\end{enumerate}
	If $(F, \equiv^F)$ denotes a pretrain with $\ggth(F, \equiv^F)>g$, then
	for every number of colours $r$ the system
		\[
		\Ext(\Phi, \Psi)_r(F, \equiv^F)=(H, \equiv^H, \ccH)
	\]
		is defined and satisfies $\GTH(H, \equiv^H, \ccH^+)>g$. 
\end{lemma}
	
The proof of any statement addressing pretrain systems of 
the form $\Ext(\Phi, \Psi)_r(F, \equiv^F)$ ultimately needs to refer back 
to the eight-step description of the extension process given immediately before Lemma~\ref{lem:0059}.
It turns out, however, that our argument becomes both more transparent 
and more reusable on later occasions if we look at the construction from the 
perspective of the system called $(N, \equiv^N, \ccN)$ there. 
For this reason some additional concepts seem to be useful. The first of 
them describes the relationship between the hypergraphs $H$ and $N$. 

\begin{dfn}\label{dfn:n150}    
	Let $H$ and $N$ be two hypergraphs on the same vertex set. 
	We say that~$H$ is {\it living} in $N$ (see Figure~\ref{fig:HNA})
	if \index{living}
	\begin{enumerate}[label=\rmlabel]
		\item\label{it:n150-1} for every edge $e$ of $H$ there exists an 
			edge $f$ of $N$ covering it,
		\item\label{it:n150-2} and every $f\in E(N)$ induces a subhypergraph
			of $H$ without isolated vertices. 
	\end{enumerate}
\end{dfn}

Clearly if $N$ is linear, then the edge $f$ guaranteed by~\ref{it:n150-1}
is uniquely determined by $e$. Next we look at the transition from $\ccN$
to $\ccH_\bullet$ in Step~\ref{it:ext7} of the extension process.

\begin{dfn}\label{dfn:n151}    
	 Let $(N, \ccN)$ be a linear system of hypergraphs. Suppose further that the 
	 hypergraph $H$ is living in $N$.
	 For every copy $M\in\ccN$ the subhypergraph $M_H$ of $H$ defined 
	 by $V(M_H)=V(M)$ and $E(M_H)=\bigcup_{f\in E(M)}E(H[f])$ 
	 is said to be {\it derived} from $M$ (see Figure~\ref{fig:HNB}). 
	 Setting $\ccH=\{M_H\colon M\in\ccN\}$ we call $(H, \ccH)$ 
	 a system of copies {\it derived} from $(N, \ccN)$.
\end{dfn}

\begin{figure}[ht]
\centering	

\begin{subfigure}[b]{0.45\textwidth}
		\centering

			\begin{tikzpicture}[scale=.7]
				
	\foreach \i in {0,...,5} {
		\coordinate (v\i) at (60*\i:1.57);
}

	\coordinate (b0) at (-1.55,0);
	\coordinate (b1) at (-2.4,1.1);
	\coordinate (b2) at (-3.9,1.1);
	\coordinate (b4) at (-3.9, -1.1);
	\coordinate (b5) at (-2.4,-1.1);
	
	\coordinate (a0) at (2.2,3.9);
	\coordinate (a1) at (.5,3.6);
	\coordinate (a2) at (.8,2.9);
	\coordinate (a3) at (2.1, 2.2);
	\coordinate (a4) at (2.8,2.2);
	
		\coordinate (c0) at (1.5,-2.7);
	\coordinate (c1) at (2.9,-3);
	\coordinate (c2) at (2.1,-1.6);
	\coordinate (c3) at (.3,-2.7);
	\coordinate (c4) at (1,-4);
	
	\draw [thick, green!70!black] (v5)--(c0)--(c1)--(c2)--(v5)--(c3)--(c4)--(c0);
	
	\draw [thick, green!70!black] (v1) --(a2)--(a0)--(a3)--(v1);
	\draw [thick, green!70!black] (a1)--(a2)--(a3)--(a4);	
	\draw [thick, green!70!black] (v0)--(v1)--(v2)--(v3)--(v4)--(v5)--(v0);	
	\draw [thick, green!70!black] (b0)--(b1)--(-3.05,0)--(b1)--(b2)--(-3.05,0)--(b4)--(b5)--(b0);		
	\draw [thick, blue!70!black](.0,0) circle (1.57cm);
	\draw [thick, blue!70!black](-3.14,0) circle (1.57cm);
	\draw [thick, blue!70!black] (1.57,2.72) circle (1.57cm);
	\draw [thick, blue!70!black] (1.57,-2.72) circle (1.57cm);
				
	\fill [blue!70!black] (-1.525,0) circle (1.5pt);
	\fill (v1) circle (1.5pt);
		\fill (v5) circle (1.5pt);

\node [blue!70!black] at (3.5,0) {\Large $N$};
\node [green!70!black] at (1,3.5) {$H$};

\end{tikzpicture}
			
\caption{A graph $H$ living in a $6$-uniform hypergraph $N$}
\label{fig:HNA} 		
\end{subfigure}
\hfill   
\begin{subfigure}[b]{0.45\textwidth}
\centering
			\begin{tikzpicture}[scale=.7]
	\foreach \i in {0,...,5} {
		\coordinate (v\i) at (60*\i:1.57);
}

	\coordinate (b0) at (-1.55,0);
	\coordinate (b1) at (-2.4,1.1);
	\coordinate (b2) at (-3.9,1.1);
	\coordinate (b4) at (-3.9, -1.1);
	\coordinate (b5) at (-2.4,-1.1);
	
	\coordinate (a0) at (2.2,3.9);
	\coordinate (a1) at (.5,3.6);
	\coordinate (a2) at (.8,2.9);
	\coordinate (a3) at (2.1, 2.2);
	\coordinate (a4) at (2.8,2.2);
	
		\coordinate (c0) at (1.5,-2.7);
	\coordinate (c1) at (2.9,-3);
	\coordinate (c2) at (2.1,-1.6);
	\coordinate (c3) at (.3,-2.7);
	\coordinate (c4) at (1,-4);
	
	\draw [dashed, gray, thick] (v5)--(c0)--(c1)--(c2)--(v5)--(c3)--(c4)--(c0);
	\draw [dashed, gray, thick]  (v1) --(a2)--(a0)--(a3)--(v1);
	\draw [dashed, gray, thick] (a1)--(a2)--(a3)--(a4);	
	\draw [thick, green!70!black] (v0)--(v1)--(v2)--(v3)--(v4)--(v5)--(v0);	
	\draw [thick, green!70!black] (b0)--(b1)--(-3.05,0)--(b1)--(b2)--(-3.05,0)--(b4)--(b5)--(b0);		
	\draw [thick, blue!70!black](.0,0) circle (1.57cm);
	\draw [thick, blue!70!black](-3.14,0) circle (1.57cm);
	\draw [dashed, gray, thick](1.57,2.72) circle (1.57cm);
	\draw [dashed, gray, thick](1.57,-2.72) circle (1.57cm);
				
	\fill [blue!70!black] (-1.525,0) circle (1.5pt);
	\fill (v1) circle (1.5pt);
		\fill (v5) circle (1.5pt);
			
\node [blue!70!black] at (-1.5,1.5) {\Large $M$};
\node [green!70!black] at (-1.5,-1.5) {$M_H$};
\end{tikzpicture}		
			
\caption{The copy $M_H$ is derived from $M$ (which has two edges)}
\label{fig:HNB} 		
\end{subfigure}  		
\caption{}		
\end{figure} 
Let us briefly pause and check that being strongly induced is a property of copies preserved under such derivations.  

\begin{fact}\label{f:strder}
	Let $(N, \ccN)$ be a linear system of hypergraphs with strongly induced copies.
	If a hypergraph $H$ is living in $N$, then the copies of the derived system 
	$(H, \ccH)$ are strongly induced as well. 
\end{fact}

\begin{proof}
	Consider an arbitrary copy $M\in\ccN$, its derived copy $F\in\ccH$, as well as 
	an edge $e\in E(H)$. We are to prove that the set $x=V(F)\cap e$ can be covered 
	by an appropriate edge $e'$ of $F$. To this end we let $f\in E(N)$ be the 
	edge covering $e$. Due to $M\Str N$ there is an edge $f'\in E(M)$ such 
	that $x\subseteq V(M)\cap f\subseteq f'$.
	
	If $|x|\ge 2$, then the linearity of $N$ and $x\subseteq f\cap f'$ imply $f=f'$,
	whence $e\in E(F)$, which allows us to take $e'=e$. 
	If, on the other hand, $|x|\le 1$, then the fact that $H[f']$ has no isolated 
	vertices yields an edge $e'\in E(H[f'])\subseteq E(F)$ such that $x\subseteq e'$.
\end{proof}

\begin{dfn}\label{dfn:n152}    
	 Let $(N, \equiv^N)$ be a pretrain whose underlying 
	 hypergraph~$N$ is linear. Suppose further that the hypergraph $H$ 
	 is living in $N$. The equivalence relation $\equiv^H$ on~$E(H)$ 
	 {\it derived} from $\equiv^N$ is defined by declaring $e\equiv^H e'$  
	 for two edges $e, e'\in E(H)$ if there are edges $f, f'\in E(N)$
	 such that $e\subseteq f$, $e'\subseteq f'$, and $f\equiv^N f'$.
	 In this situation the pretrain $(H, \equiv^H)$ is said to be {\it derived}
	 from $(N, \equiv^N)$.
\end{dfn}

Another way to think about this construction is that to every wagon $W_N$ 
of $(N, \equiv^N)$ there corresponds a unique wagon $W_H$ of $(H, \equiv^H)$ 
such that $E(W_H)=\bigcup_{f\in E(W_N)}E(H[f])$, called the wagon {\it derived} 
from $W_N$. Owing to Definition~\ref{dfn:n150}\ref{it:n150-2} we have $V(W_N)=V(W_H)$.
Combined with the fact that, conversely, every wagon of $(H, \equiv^H)$ is 
derived from some wagon of $(N, \equiv^N)$ this shows that if $(N, \equiv^N)$
and $H$ are linear, then so is $(H, \equiv^H)$. 
Roughly speaking, our next result asserts that derivations cannot decrease $\GTH$.
	
\begin{lemma}\label{lem:n023}
	Let the linear hypergraph $H$ be living in the linear hypergraph $N$. 
	If a system of pretrains $(N, \equiv^N, \ccN)$ 
	satisfies $\GTH(N, \equiv^N, \ccN^+)>g$ for some integer $g\ge 1$, then the 
	system $(H, \equiv^H, \ccH)$ derived from it satisfies $\GTH(H, \equiv^H, \ccH^+)>g$ 
	as well. 
\end{lemma}

\begin{proof}
	Let $\delta$ be the bijective map assigning to each wagon of $(H, \equiv^H)$ 
	the wagon of $(N, \equiv^N)$ it is derived from. 
	We already know that the pretrain $(H, \equiv^H)$ is linear. Thus it 
	remains to show that every acceptable big cycle  
		\[
		\ccC=F_1q_1\ldots F_nq_n
	\]
		in $(H, \equiv^H, \ccH^+)$ with $\ord{\ccC}\le g$ has a supreme copy. 
	Assume for the sake of contradiction that $\ccC$ is a counterexample to
	this statement and, moreover, that among all possibilities~$\ccC$ has been 
	chosen in such a way that 
		\[
		\nu=\big|\bigl\{i\in\ZZ/n\ZZ\colon F_i \text{ is a real copy}\bigr\}\big|
	\]
		is minimal. 
	
	\begin{clm}\label{clm:n181}
		If $n\ge 3$, the index $i\in \ZZ/n\ZZ$ is mixed, and the vertex among
		$q_{i-1}$, $q_i$ belongs to the wagon, then $F_i$ is an edge 
		copy. 
	\end{clm}  
	
	\begin{proof}
		By symmetry it suffices to deal with the case that $q_{i-1}$ is a vertex, 
		$q_i$ is a wagon, and $q_{i-1}\in V(q_i)$. Assume towards a contradiction that 
		$F_i$ is a real copy. Pick an edge $e_i\in E(H)$ with $q_{i-1}\in e_i\in E(q_i)$
		and denote the cyclic sequence obtained from $\ccC$ upon collapsing $F_i$ to $e_i^+$
		by $\ccD$. Owing to Lemma~\ref{lem:1448}\ref{it:1448a} applied to $\{i\}$ and $e_i^+$
		here in place of~$K$ and~$P_i$ there $\ccD$ is an acceptable big cycle. 
		(Notice that the clauses~\ref{it:1801a}\,--\,\ref{it:1801c} of Lemma~\ref{lem:1448}
		are irrelevant here, because the piece $e_i^+$ is short). 
		 
		Now $\ord{\ccD}=\ord{\ccC}\le g$ and $\ccD$ contains fewer real copies than $\ccC$.
		So by the minimality of $\nu$ we know that $\ccD$ possesses a supreme copy. 
		Due to Lemma~\ref{lem:1448}\ref{it:1448b} the cycle $\ccC$ 
		has a supreme copy as well, contrary to the choice of $\ccC$.  
	\end{proof}
	
	After these preliminaries we briefly describe our strategy for finding 
	a supreme copy of~$\ccC$, which consists of three major steps. 
	\begin{enumerate}[label=\nlabel]
		\item\label{it:pss1} We translate $\ccC$ into a cyclic sequence $\ccD$
			with respect to the system $(N, \equiv^N, \ccN^+)$.
		\item\label{it:pss2} Second we check that $\ccD$ is an acceptable big cycle 
			whose order is at most $g$. 
			Now the hypothesis $\GTH(N, \equiv^N, \ccN^+)>g$
			yields a supreme copy $M_\star$ of $\ccD$ and a family~$\ccQ$ of pieces 
			witnessing the supremacy of $M_\star$.
		\item\label{it:pss3} Finally, we need to translate $M_\star$ and $\ccQ$ 
			back to a supreme copy $F_\star$ of $\ccC$ and a family $\ccP$
			of $F_\star$-pieces witnessing the supremacy of $F_\star$.
	\end{enumerate}
	
	The first step of this plan depends on the notion of 
	a twin in $\ccC$, that we shall now explain. 
	Suppose that $Z=(e')^+W(e'')^+$ is a subsequence
	of~$\ccC$, where $e', e''\in E(H)$ 
	are edges and $W$ is a wagon connector. Owing to~\ref{it:B2} 
	and~\ref{it:B4} the other connectors of~$\ccC$ next to~$(e')^+$ and~$(e'')^+$ 
	are vertices.

\usetikzlibrary{decorations.pathreplacing}

\begin{figure}[ht]
	\centering

			\begin{tikzpicture}[scale=1]

		\draw [thick, red!80!black, rounded corners](-3.5, -.5) rectangle (3.5,2);
		
			\draw [green!75!black, thick] (-3.3,0)--(3.3,0);
			
			\draw (-2.9,-.3) -- (-2.2,1.8);
			\draw (2.9,-.3)--(2.2,1.8);
	
		\fill (-2.8,0) circle (1.5pt);
		\fill (2.8,0) circle (1.5pt);
		
		\node [red!80!black] at (4, 1.7) {\Large $W$};
		\node [green!75!black] at (0,.25) {$f$};
		\node at (-2.5,1.5) {$e'$};
		\node at (2.6, 1.5){$e''$};
		\node at (-2.9, .25) {$q'$};
		\node at  (3, .25) {$q''$};
		
		\end{tikzpicture}

				\caption{A twin $(e')^+W(e'')^+$ with conductor $f$.}
				\label{fig:Zwilling} 			
	\end{figure}
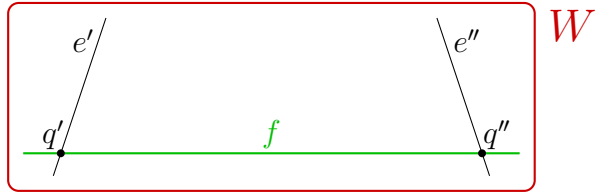 	
	In other words, $\ccC$ actually has a subsequence of the form $q'Zq''$, 
	where~$q'$ and~$q''$ are vertices. If there exists an edge $f\in E(N)$ 
	containing $q'$ and $q''$ we call $Z$ a {\it twin} and an edge~$f$ 
	verifying this fact is said to be a {\it conductor} of $Z$ 
	(see Figure~\ref{fig:Zwilling}).
	\index{twin}
	\index{conductor}
	
	Due to~\ref{it:A1} there cannot exist any twins if $n=2$. 
	Moreover, if for $n\ge 3$ we have a twin $Z=(e')^+X(e'')^+$, then the 
	vertices $q'$, $q''$ mentioned in the previous paragraph are distinct and 
	by the linearity of $N$ the conductor of $Z$ is uniquely determined.   
	
	Clearly the twins are mutually non-overlapping in the sense that each edge 
	copy $F_i$ belongs to at most one twin. So by symmetry we may 
	suppose that~$F_1$ and $F_n$ do not appear together in a twin. 
	
	We say that 
	\begin{equation}\label{eq:n403}
		\ccC=Z_1r_1\ldots Z_mr_m
	\end{equation}
	is the {\it twin decomposition} of $\ccC$ 
	if $\{r_1, \ldots, r_m\}\subseteq \{q_1, \ldots, q_n\}$, every $Z_j$ is 
	either a single edge copy, a single real copy, or a twin, and conversely every 
	twin is in the set $\{Z_1, \ldots, Z_m\}$. We have $m\ge 2$, as there are 
	no twins in the case $n=2$. Set 
		\[
		\overline{r}_j=
			\begin{cases}
					\delta(r_j) & \text{ if $r_j$ is a wagon} \cr
					r_j & \text{ if $r_j$ is a vertex}
			\end{cases}
	\]
	for every $j\in \ZZ/m\ZZ$. According to the rules 
	that follow we shall determine certain copies $M_1, \dots, M_m\in \ccN^+$, 
	the intention being that 
		\[
		\ccD=M_1\overline{r}_1\ldots M_m\overline{r}_m
	\]
	will turn out to be the desired big cycle in $(N, \equiv^N, \ccN^+)$.
	Let $j\in\ZZ/m\ZZ$ be given.
	\begin{enumerate}[label=\alabel]
		\item\label{it:4156a} If $Z_j$ is a real copy, let $M_j\in\ccN$ be the 
			copy it is derived from.
		\item\label{it:4156b} If $Z_j=e^+_j$ is an edge copy, 
			let $f_j\in E(N)$ be the edge covering $e_j$ and set $M_j=(f_j)^+$.
		\item\label{it:4156c} Finally, if $Z_j=(e'_j)^+W_j(e''_j)^+$ is a twin 
			with conductor $f_j\in E(N)$, then we set $M_j=f_j^+$.
	\end{enumerate}
	
	Having thus defined $\ccD$ we proceed with the second part of our plan. 
	
	\begin{clm}
		The cyclic sequence $\ccD$ is a big cycle 
		in $(N, \equiv^N, \ccN^+)$.
	\end{clm}
	
	\begin{proof}
		The demands~\ref{it:B2},~\ref{it:B3} are clear and~\ref{it:B4} follows 
		easily from the fact that twins are surrounded by vertex connectors. 
		So it remains to prove~\ref{it:B1}.
		
		Assume contrariwise that some copy $M_\star$ occurs twice 
		in consecutive positions in $\ccD$, say $M_\star=M_j=M_{j+1}$ for 
		some $j\in\ZZ/m\ZZ$. Due to~\ref{it:4156a} and the fact that $\ccC$
		satisfies~\ref{it:B1} we know that $M_\star=f_\star^+$ is an edge copy. 
		Let $X_\star$ be the wagon of $(N, \equiv^N)$ containing $f_\star$ and write 
		$W_\star=\delta^{-1}(X_\star)$ for its derived wagon. Each of the three 
		rules~\ref{it:4156a}\,--\,\ref{it:4156c} could in principle give rise to 
		the edge copy~$f_\star^+$. Depending on which of them was used in the
		definitions of $M_j$, $M_{j+1}$ there arise three possibilities for 
		each of $Z_j$, $Z_{j+1}$. Both of them are either the copy $H[f_\star]$
		derived from~$f_\star^+$ (as in~\ref{it:4156a}), 
		a single edge copy covered by $f_\star$ (as in~\ref{it:4156b}), 
		or a twin conducted by $f_\star$ (as in~\ref{it:4156c}). 
		Furthermore, each of the connectors $r_{j-1}$, $r_j$, and $r_{j+1}$ is 
		either $W_\star$ or a vertex belonging to $f_\star$.
		
		\smallskip
		
		{\hskip2em \it First case: We have $m=2$}
		
		\smallskip
		
		Suppose first that none of $Z_j$, $Z_{j+1}$ is a twin, so that $n=2$. 
		Due to the linearity of $H$ and~\ref{it:A1} it cannot be the case 
		that $\ccC$ consists of two edge copies. Thus~$H[f]$ appears in~$\ccC$ 
		and by collapsing the other copy, which has to be an edge copy, 
		to itself we see that~$H[f]$ is a supreme 
		copy of~$\ccC$, contrary to the choice of~$\ccC$. 
		
		Thus we can assume by symmetry that $Z_j$ is a twin, which causes the 
		connectors of~$\ccD$ to be vertices. As the wagon in the middle of $Z_j$ 
		contains those vertices, the linearity of~$(N, \equiv^N)$ shows that this 
		wagon has to be $W_\star$. Now $Z_{j+1}$ cannot be a twin as well (because 
		it also would need to contain the wagon $W_\star$), and by~\ref{it:A2}
		$Z_{j+1}$ cannot be a single edge copy; so only the case $Z_{j+1}=H[f]$ remains.
		Utilising the fact that this copy has no isolated vertices we find two edges
		$e', e''\in E(H[f])$ containing the connectors $r_{j-1}$, $r_j$. By collapsing 
		the two edge copies in $Z_j$ to $(e')^+$ and $(e'')^+$ we see that $H[f]$ is a 
		supreme copy, which again contradicts the choice of $\ccC$. 		
		 		 		
		\smallskip
		
		{\hskip2em \it Second case: We have $n\ge m\ge 3$}
		
		\smallskip
		
		Now $r_{j-1}$, $r_j$, and $r_{j+1}$ are three distinct connectors 
		and due to $|M^\ccC(W_\star)|\le 2$ one of them is equal to $W_\star$, while 
		the two other ones are vertices in $f_\star$. 
		If $r_j=W_\star$, then Claim~\ref{clm:n181} tells us that $Z_j$, $Z_{j+1}$
		are single edge copies. So $Z_jr_jZ_{j+1}$ is a twin conducted
		by $f_\star$, which contradicts the fact that~\eqref{eq:n403} is the twin
		decompositions of $\ccC$. 
		
		By symmetry it remains to consider the case that $r_{j-1}=W_\star$, 
		whilst $r_j$ and $r_{j+1}$ are vertices. 
		Since the wagon $W_\star$ satisfies~\ref{it:A2}
		with respect to $\ccC$, this is only possible if $r_{j-1}$ and $r_{j+1}$
		are consecutive in $\ccC$. So $m=3$ and Claim~\ref{clm:n181} implies 
		that $Z_{j-1}$, $Z_j$ are single edge copies, which in turn 
		causes $Z_{j-1}r_{j-1}Z_j$ to be a twin with conductor~$f_\star$. 
		Again this contradicts~\eqref{eq:n403} being the twin decomposition 
		of $\ccC$.
	\end{proof}
	
	\begin{clm}\label{clm:n420}
		The big cycle $\ccD$ is acceptable.
	\end{clm}
	
	\begin{proof} 	
		Assume first that contrary to~\ref{it:A1} we have $\ccD=f_1^+Xf_2^+x$,
		where $f_1, f_2\in E(N)$, the connector $X$ is a wagon of $(N, \equiv^N)$, 
		and $x$ is a vertex. Since twins are surrounded by vertex connectors, there 
		cannot be any twins in $\ccC$. Thus $\ccC$ is of the 
		form $\ccC=F_1WF_2x$, where~$W$ is derived from~$X$. 
		As $\ccC$ is required to satisfy~\ref{it:A1}, we may 
		assume that $F_1$ is real, and by~\ref{it:4156a} $F_1=H[f_1]$ is the 
		copy derived from $f_1^+\in \ccN^+$. So $F_1$ has no isolated vertices and 
		due to $x\in f_1$  
		there exists an edge $e_2\in E(F_1)$ passing through $x$. 
		By collapsing~$F_2$ to $e_2^+$ we see that $F_1$ is a supreme copy of $\ccC$,
		which contradicts the assumption that $\ccC$ be a counterexample.  		 
		
		Working towards~\ref{it:A2} we consider any wagon connector
		$\overline{r}_j$ of $\ccD$. If $r_j=q_i$, then the acceptability of $\ccC$ 
		yields $M^\ccC(q_i)\subseteq \{i-1, i+1\}$, whence 
		$M^\ccD(\overline{r}_j)\subseteq\{j-1, j+1\}$. Now suppose that 
		$|M^\ccD(\overline{r}_j)|=2$ and that some edge $f\in E(N)$ satisfies 
		$r_{i-1}, r_{i+1}\in f$. Evidently this is only possible if $n\ge m\ge 3$ and
		$M^\ccD(\overline{r}_j)=\{j-1, j+1\}$, which in turn implies  
		$M^\ccC(q_i)=\{i-1, i+1\}$. Claim~\ref{clm:n181} tells us that $F_i$, $F_{i+1}$
		are edge copies and, therefore, $F_iq_iF_{i+1}$ is a twin conducted by $f$.
		This contradiction to the fact that~\eqref{eq:n403} is the twin decomposition 
		of $\ccC$ establishes that $\ccD$ satisfies the 
		moreover-part of~\ref{it:A2} as well. 
		
		It remains to verify~\ref{it:A3}. To this end we consider an arbitrary 
		wagon $W_\star$ of $(H, \equiv^H)$ such that 
		$X_\star=\delta(W_\star)$ fails to be a wagon connector of~$\ccD$. 
		If~$W_\star$ is a connector of $\ccC$, then it appears in the 
		middle of a twin and by~\ref{it:A2} the set $M^\ccD(X_\star)$ can 
		be covered by a pair of consecutive indices.
		If, on the other hand, $W_\star$ is absent from $\ccC$, 
		we can appeal to~\ref{it:A3}. Thereby Claim~\ref{clm:n420} is proved. 	
	\end{proof}
	
	Owing to $\ord{\ccD}=\ord{\ccC}\le g$ 
	and $\GTH(N, \equiv^N, \ccN^+)>g$ Claim~\ref{clm:n420} shows that $\ccD$ 
	has a supreme copy~$M_\star$. Let the family of $M_\star$-pieces 
		\[
		\ccQ=\{Q_j\colon j\in\ZZ/m\ZZ \text{ and } M_j\ne M_\star\}
	\]
	exemplify the supremacy of $M_\star$. Recall that $M_\star$ is a real copy by 
	Remark~\ref{rem:1852}\ref{it:1852b}, whence there is a real 
	copy $F_\star\in\ccH$  
	derived from it and appearing in~$\ccC$.
	Coming to Step~\ref{it:pss3} of the plan outlined above 
	we shall now show that~$F_\star$ is a supreme copy of $\ccC$.  
	 
	\begin{clm}\label{clm:n465}
		For the index set $K=\{i\in\ZZ/n\ZZ\colon F_i\ne F_\star\}$ there is a family 
				\[
			\ccP = \bigl\{P_i\colon i\in K\bigr\}
		\]
				of $F_\star$-pieces such that for every $i\in K$ the following holds. 
		\begin{enumerate}[label=\nlabel]
			\item\label{it:4652a} The cyclic sequence obtained from $\ccC$ 
					upon replacing~$F_i$ by~$P_i$ has the 
					properties~\ref{it:B3} and~\ref{it:B4}.
			\item\label{it:4652b} If the piece $P_i$ is long, then it satisfies the 
					clauses~\mbox{\ref{it:1801a}\,--\,\ref{it:1801c}} 
					from Lemma~\ref{lem:1448}.
 		\end{enumerate}	
	\end{clm}
		
	\begin{proof}
		Given any $i\in K$ we need to explain how to find the required piece $P_i$.
		Let us start with the case that for some twin $Z_j=(e'_j)^+W_j(e''_j)^+$,
		say with conductor~$f_j$, we have $F_i\in\{(e'_j)^+, (e''_j)^+\}$. 
		By Definition~\ref{dfn:1746}\ref{it:1751b}\ref{it:beta}
		applied to $\ccD$ and $(N, \equiv^N, \ccN^+)$ here rather than $\ccC$
		and $(H, \equiv, \ccH^+)$ there
		the piece $Q_j$ to which $M_j=f_j^+$ collapses is short; due to 
		the linearity of $N$ the only possibility is $Q_j=f_j^+$,
		which implies $f_j\in E(M_\star)$. Since $H[f_j]$ has no isolated vertices, 
		we can collapse $(e'_j)^+$, $(e''_j)^+$ to two edge 
		copies $(e^\star_j)^+$, $(e^{\star\star}_j)^+$ 
		with $r_{j-1}\in e^\star_j\in E(H[f_j])$ 
		and $r_j\in e^{\star\star}_j\in E(H[f_j])$.
		
		So from now on we can assume that $q_{i-1}F_iq_i=r_{j-1}Z_jr_j$ 
		holds for some $j\in \ZZ/m\ZZ$.
		Because of $i\in K$ we have $M_j\ne M_\star$ and, therefore, $\ccQ$ provides
		an $M_\star$-piece $Q_j$.
		
		\smallskip
		
		{\it \hskip2em First Case. The $M_\star$-piece $Q_j=f_j^+$ is short.}
		
		\smallskip
		
		Suppose first that one of the connectors $q_{i-1}$ and~$q_i$, say $q_i$, 
		is a wagon. Now $q_{i-1}$ needs to be a vertex, there exists an 
		edge $e_i\in H[f_j]$ with $q_{i-1}\in e_i$, 
		and the short $F_\star$-piece $P_i=e_i^+$ is as desired. 
				
		It remains to consider the case that $q_{i-1}$ and $q_i$ are vertices. 
		If there exists an edge $e_i\in E(H)$ containing both of them, 
		then the linearity of $N$ yields $e_i\subseteq f_j$ 
		and it is permissible to set $P_i=e_i^+$.
		From now on we assume that such an edge $e_i$ does not exist. 
		If $e'_i, e''_i$ are two edges of $H[f_j]$ with $q_{i-1}\in e'_i$ as well 
		as $q_i\in e''_i$, and $W_i$ denotes the unique wagon of $(H, \equiv ^H)$
		with $f_j\in E(\delta(W_i))$, then $P_i=(e'_i)^+ W_i (e''_i)^+$ is 
		a long $F_\star$-piece satisfying~\ref{it:4652a}
		and the demands~\ref{it:1801a},~\ref{it:1801b} mentioned in~\ref{it:4652b}.
		If~\ref{it:1801c} fails, then the first part of~\ref{it:A2} discloses $n=3$ 
		and $q_{i+1}=W_i$. 
		But then $q_{i-1}, q_i\in V(q_{i+1})$, so by Claim~\ref{clm:n181}
		both~$F_{i+1}$ and~$F_{i+2}$ are edge copies. 
		Thus the supreme copy $M_\star$ fails to appear on $\ccD$, which is absurd.    
		 
		\smallskip
		
		{\it \hskip2em Second Case. The $M_\star$-piece $Q_j=(f'_j)^+ X_j(f''_j)^+$ 
		is long.}
		
		\smallskip
		
		Here $X_j$ is a wagon of $(N, \equiv^N)$ that fails to appear on $\ccD$, 
		whence $W_j=\delta^{-1}(X_j)$ is a wagon of $(H, \equiv^H)$ that is 
		distinct from $r_1, \ldots, r_m$. As $H[f'_j]$ and $H[f''_j]$ have no 
		isolated vertices, there are edges~$e'_j, e''_j\in E(F_\star)$ 
		such that $r_{j-1}\in e'_j$ and $r_j\in e''_j$.
		Now $P_i=(e'_j)^+ W_j (e''_j)^+$ is an $F_\star$-piece 
		satisfying~\ref{it:1801a} and~\ref{it:1801b}, because $Q_j$ has these 
		properties as well. 
		Proceeding with~\ref{it:1801c} we assume 
		contrariwise that $W_j\in\{q_1, \ldots, q_n\}$. As we already 
		know $W_j\not\in\{r_1, \ldots, r_m\}$ this is only possible if $W_j$ 
		is in the middle of some twin $Z_k$, where $k\in (\ZZ/m\ZZ)\sm\{j\}$. 
		Now $\{j-1, j, k-1, k\}\subseteq M^\ccD(X_j)$ and the acceptability 
		of~$\ccD$ yield~$m=2$.
		Combined with $M_\star\ne M_k, M_j$ this contradicts the fact that $M_\star$
		appears in~$\ccD$. 
	\end{proof}
	
	If $n=2$ this shows immediately that $F_\star$ is a supreme copy of $\ccC$
	and for $n\ge 3$ we need to point out additionally that owing to 
	Lemma~\ref{lem:1448}\ref{it:1448a} all collapses suggested 
	by the claim can be carried out simultaneously. 
\end{proof}

The next concept is motivated by the relationship between the systems $\ccH_\bullet$
and $\ccH$ occurring in Step~\ref{it:ext7} of the extension process. 

\begin{dfn}\label{dfn:n155}
\begin{enumerate}[labelsep=0pt, itemindent=20pt, leftmargin=0pt, label=\alabel]
	\item\label{it:n155-1} $\,\,$
	If a pretrain $(G, \equiv^G)$ is the disjoint union of the family of pretrains
	$\{(G_i, \equiv^i)\colon i\in I\}$ and for every $i\in I$ the pretrain 
	$(G_i, \equiv^i)$ is a tame extension of its subpretrain $F_i$, then the system
	of pretrains $\ccS_G=\{F_i\colon i\in I\}$ is said to be {\it scattered} in $G$. 
	\item\label{it:n155-2} $\,\,$
	Let $(H, \equiv^H, \ccH_\bullet)$ be a system of pretrains. If for every 
	copy $G\in \ccH_\bullet$ we have a system~$\ccS_G$ of subpretrains scattered 
	in $G$, then the system $\ccH=\bigcup_{G\in\ccH_\bullet} \ccS_G$ is said to 
	be {\it scattered} in $\ccH_\bullet$. 
\end{enumerate} 
	\index{scattered}
\end{dfn}

\begin{lemma}\label{lem:n540}
	Let $(H, \equiv^H, \ccH_\bullet)$ be a system of pretrains satisfying 
	$\GTH(H, \equiv^H, \ccH_\bullet^+)>g$ for some integer $g\ge 1$. 
	If $\ccH$ is scattered in $\ccH_\bullet$, 
	then $\GTH(H, \equiv^H, \ccH^+)>g$ follows. 
\end{lemma}

\begin{proof}
	As in Definition~\ref{dfn:n155}\ref{it:n155-2} 
	we write $\ccH=\bigcup_{G\in\ccH_\bullet} \ccS_G$,
	where for every copy $G\in\ccH_\bullet$ the system $\ccS_G$ is scattered in $G$.
	
	\begin{clm}\label{clm:1452}
		Suppose $(G_\star, \equiv^{G_\star})\in\ccH_\bullet$ 
		and that $\ccC=F_1q_1\ldots F_nq_n$ is an acceptable 
		big cycle in $(H, \equiv^H, \ccH^+)$ with $\ord{\ccC}\le g$. If all
		copies of $\ccC$ belong to $\ccS_{G_\star}\cup E^+(G_\star)$, 
		then $\ccC$ has a supreme copy.  
	\end{clm}
	
	\begin{proof}
		Since $\ccS_{G_\star}$ is scattered in $G_\star$, we can express $G_\star$
		as a disjoint union of a family $\{G_i\colon i\in I\}$ of subpretrains 
		such that for every member of $\ccS_{G_{\star}}$ there is a unique $G_i$
		tamely extending it. The entire cycle $\ccC$ cannot jump from one $G_i$ 
		to another one and thus there is a pretrain $G_\circ$ in this family such 
		that all edge copies of $\ccC$ correspond to edges of $G_\circ$ and all real 
		copies of $\ccC$ coincide with the unique copy $F_\circ\in\ccS_{G_\star}$ 
		contained in $G_\circ$. Due to $\GTH(H, \equiv^H, \ccH_\bullet^+)>g$ only the 
		case that $F_\circ$ is a real copy appearing on $\ccC$ is interesting.
		We shall show that then $F_\circ$ itself is a supreme copy of $\ccC$. 
		
		As a first step towards this goal we prove that all vertex connectors 
		of $\ccC$ belong to~$F_\circ$.
		Assume (reductio ad absurdum) that this fails for some vertex connector $q_i$. 
		Now $F_i$, $F_{i+1}$ are edge copies and $n\ge 3$, so Lemma~\ref{lem:2145} shows 
		that the underlying edges of $F_i$, $F_{i+1}$ belong to distinct wagons, 
		say~$W_i$,~$W_{i+1}$. Since $G_\circ$ is an extension of $F_\circ$,
		we have indeed $q_i\in V(W_i)\cap V(W_{i+1})\subseteq V(F_\circ)$.
		
		Now let $K=\{i\in\ZZ/n\ZZ\colon F_i\ne F_\circ\}$. We shall show that for 
		every $i\in K$ there is a short $F_\circ$-piece $P_i$ such that in $\ccC$
		we can collapse $F_i$ to $P_i$. If $n=2$ this will show immediately 
		that~$F_\circ$ is a supreme copy of $\ccC$ (because then we have $|K|=1$), 
		and if $n\ge 3$ Lemma~\ref{lem:1448}\ref{it:1448a} tells us that all 
		these collapses can be carried out simultaneously. 
		
		Consider an arbitrary index $i\in K$. We already know that $F_i=f_i^+$ is an 
		edge copy. If the connectors $q_{i-1}$ and $q_i$ are vertices, 
		then $F_\circ\Str G_\circ$ and $q_{i-1}, q_i\in V(F_\circ)\cap f_i$ 
		imply $f_i\in E(F_\circ)$, for which reason it is permissible to set 
		$P_i=F_i$. 
		
		If, on the other hand, one of the connectors $q_{i-1}$, $q_i$, say $q_i$,
		is a wagon, then $q_{i-1}$ is a vertex and the moreover-part of 
		Definition~\ref{dfn:n854} shows that there is an edge $e_i$ such that 
		$q_{i-1}\in e_i\in E(q_i)\cap E(F_\circ)$. Clearly $P_i=e_i^+$ is as desired. 
	\end{proof}
		
	Now let $\ccC=F_1q_1\dots F_nq_n$ be an arbitrary acceptable big cycle 
	in $(H, \equiv^H, \ccH^+)$ whose order is at most $g$. 
	We would like to construct an auxiliary acceptable big cycle 
		\[
		\ccD=G_1q_1\dots G_nq_n
	\]
		in $(H, \equiv^H, \ccH_\bullet^+)$, which inherits its connectors from $\ccC$.
	To this end we associate with every copy $F_i\in \ccH^+$ 
	a copy $G_i\in \ccH_\bullet^+$ such that 
		\begin{enumerate}
		\item[$\bullet$] either $F_i=G_i$ is an edge copy
		\item[$\bullet$] or $F_i$ is a real copy and $F_i\in\ccS_{G_i}$.
	\end{enumerate}
		 
	Evidently $\ccD$ has the properties~\ref{it:B2}\,--\,\ref{it:B4}	of an 
	acceptable big cycle in $(H, \equiv^H, \ccH_\bullet^+)$. Moreover, if 
	some copy $G_\star\in \ccH_\bullet^+$ appears twice in consecutive positions 
	of $\ccD$, then $G_\star$ is a real copy and we get a contradiction to the 
	fact that $\ccH_{G_\star}$ is scattered in $G_\star$. 
	So altogether,~$\ccD$ is a big cycle. 
	
	The acceptability of $\ccD$ and $\ord{\ccD}=\ord{\ccC}\le g$ are clear.
	Thus our assumption $\GTH(H, \equiv^H, \ccH_\bullet^+)>g$ leads to 
	a supreme copy $G_\star$ of $\ccD$. Recall that $G_\star$ is a real copy by 
	Remark~\ref{rem:1852}\ref{it:1852b}. 
	Set $K=\{i\in \ZZ/n\ZZ \text{ and } G_i\ne G_\star\}$
	and let the family of $G_\star$-pieces 
	$\{P_i\colon i\in K\}$ exemplify the supremacy 
	of $G_\star$ in $\ccD$. Starting with $\ccC$ we can still replace every 
	copy $F_i$ with $i\in K$ by the piece $P_i$. If $n=2$ this shows immediately 
	that $\ccC$ has a supreme copy and for $n\ge 3$ 
	Lemma~\ref{lem:1448}\ref{it:1448a} tells us that we obtain an 
	acceptable big cycle~$\ccE$ in this manner. 
	Due to Claim~\ref{clm:1452} $\ccE$ has a supreme 
	copy and by Lemma~\ref{lem:1448}\ref{it:1448b} so does~$\ccC$. 
\end{proof}

\begin{proof}[Proof of Lemma~\ref{lem:1955}]
	We assume that the reader has the eight-step 
	description~\ref{it:ext1}\,--\,\ref{it:ext8}
	of $\Ext(\Phi, \Psi)$ given immediately before Lemma~\ref{lem:0059} on the desk 
	and without further explanation we use the same notation as there. Let us remark that 
	the pretrain $(G, \equiv^{G})$ generated in Step~\ref{it:ext3} 
	satisfies $\ggth\bigl(G, \equiv^{G}\bigr)>g$, whence 
	$\gth(M) >g$. So the application of~$\Psi_{r^{e(X)}}$ in Step~\ref{it:ext5} is 
	justified,
	meaning that the system $(H, \equiv^H, \ccH^+)$ does indeed exist. 
	It remains to establish that its $\GTH$ exceeds $g$.  
	
	Let $\equiv^N$ be the equivalence relation on $E(N)$ whose equivalence classes
	are single edges. By Step~\ref{it:ext6} the hypergraph $H$ is living in $N$ 
	and the pretrain $(H, \equiv^H)$ is derived from $(N, \equiv^N)$. Similarly, 
	the first part of Step~\ref{it:ext7} says that the system $\ccH_\bullet$ is 
	derived from~$\ccN$.  
	Due to $\Gth(N, \ccN^+)>g$ and Lemma~\ref{lem:0217} we 
	know $\GTH(N, \equiv^N, \ccN^+)>g$ and, therefore, Lemma~\ref{lem:n023} 
	discloses $\GTH(H, \equiv^H, \ccH_\bullet^+)>g$.
	
	For every copy $(G_\star, \equiv^{G_\star})\in \ccH_\bullet$ the $|\ccX|$ 
	standard copies of $(F, \equiv^F)$ form a scattered system 
	(cf.\ Lemma~\ref{lem:5757}) and thus $\ccH$ is scattered in $\ccH_\bullet$.
	Now $\GTH(H, \equiv^H, \ccH^+)>g$ follows from Lemma~\ref{lem:n540}.
\end{proof} 

\subsection{\texorpdfstring{$\GTH$}{Girth} preservation}
\label{subsec:GTHpres}

The next result and its proof share strong similarities with the material in 
Section~\ref{subsec:PCAG}. However, the focus there was on making incremental 
progress, whereas here we only care about maintaining what we already have. 

\begin{prop} \label{prop:2235}
	Suppose that $g\ge 2$ and that $\Xi$ is a partite lemma 
	\begin{enumerate}
		\item[$\bullet$] applicable to $k$-partite $k$-uniform 
			pretrains $(B, \equiv^B)$ with $\ggth(B, \equiv^B)>g$
	 	\item[$\bullet$] and producing $k$-partite $k$-uniform systems 
			of pretrains $(H, \equiv^H, \ccH)$ with 
						\[
				\GTH(H, \equiv^H, \ccH^+)>g\,.
			\]
				\end{enumerate}	
	If $\Phi$ denotes an arbitrary linear Ramsey construction for hypergraphs 
	generating systems of strongly induced copies, 
	then the pretrain construction $\PC(\Phi, \Xi)$  
	\begin{enumerate}
		\item[$\bullet$] applies to all  
			pretrains $(F, \equiv^F)$ with $\ggth(F, \equiv^F)>g$
	 	\item[$\bullet$] and delivers systems 
			of pretrains $(H, \equiv^H, \ccH)$ with $\GTH(H, \equiv^H, \ccH^+)>g$. 
	\end{enumerate}	
\end{prop}

One may brief\-ly wonder why this is a useful result. After all, the 
main property $\PC(\Phi, \Xi)$ is shown to have is already assumed for $\Xi$.   
Indicating only one example of a successful application of 
Proposition~\ref{prop:2235} we would like to point out that by Lemma~\ref{lem:cleancap} the copies in the systems 
produced by $\PC(\Phi, \Xi)$ will always have clean intersections, while this need not 
be the case for the systems generated by $\Xi$ (whose copies can intersect in entire 
wagons). For this reason, $\PC(\Phi, \Xi)$ can have more 
desirable properties than $\Xi$ itself.   
As usual, the main work going into the proof of Proposition~\ref{prop:2235}
concerns a picturesque lemma.     

\begin{lemma}\label{lem:0052}
	Suppose that $g\ge 2$ and that 
		\begin{equation}\label{eq:0912}
		(\Sigma, \equiv^\Sigma, \ccQ, \psi_\Sigma)
		=
		(\Pi, \equiv^\Pi, \ccP, \psi_\Pi)
		\conc
		(H, \equiv^H, \ccH)
	\end{equation}
	holds for two pretrain pictures $(\Pi, \equiv^\Pi, \ccP, \psi_\Pi)$ 
	and $(\Sigma, \equiv^\Sigma, \ccQ, \psi_\Sigma)$ over a linear system~$(G, \ccG)$ 
	with strongly induced copies and for a $k$-partite $k$-uniform system of 
	pretrains~$(H, \equiv^H, \ccH)$.
	If 
		\[
		\GTH(\Pi, \equiv^\Pi, \ccP^+)>g
		\quad \text{ and } \quad 
		\GTH(H, \equiv^H, \ccH^+)>g\,, 
	\]
		then $\GTH(\Sigma, \equiv^\Sigma, \ccQ^+)>g$.
\end{lemma}  

\begin{proof}
	The picturesque statement in the proof of Lemma~\ref{lem:2207} reveals 
	that $\Sigma$ is a linear hypergraph. So it remains to establish
		\begin{enumerate}[label=\alabel]
		\item\label{it:2127a} that any two wagons of $(\Sigma, \equiv^\Sigma)$ 
			intersect in at most one vertex 
		\item\label{it:2127b} and that every acceptable big cycle $\ccC$ 
			in $(\Sigma, \equiv^\Sigma, \ccQ^+)$ with $\ord{\ccC}\le g$ has a supreme copy.
	\end{enumerate}
		
	\noindent
	{\bf Stage A: First observations.}
	The assumption that the copies in $\ccG$ be strongly induced will be used in the following 
	way.
	
	\begin{clm}\label{clm:2131}
		If $F_\star\in \ccQ^+$ and $x, y\in V(F_\star)\cap V(H)$
		are two distinct vertices, then there exists an edge $f$ with 
		$x, y\in f\in E(F_\star)\cap E(H)$. 
	\end{clm}
	
	\begin{proof}
		Let the amalgamation~\eqref{eq:0912} be constructed over the edge $e\in E(G)$. 
		Since $F_\star$ intersects every music line of the pretrain 
		picture $(\Sigma, \equiv^\Sigma, \ccQ, \psi_\Sigma)$ at most once, 
		the vertices $\psi_\Sigma(x), \psi_\Sigma(y)\in e\cap \psi_\Sigma[V(F_\star)]$ 
		are distinct.  
		So if $F_\star=f_\star^+$ is an edge copy, then the linearity of $G$ 
		yields $\psi_\Sigma[f_\star]=e$, for which reason $f=f_\star$ is as desired. 
		If the copy $F_\star$ is real, then~$\psi_\Sigma$ projects it into $\ccG$
		and thus onto a strongly induced subhypergraph of $G$. In particular, this projection 
		needs to contain the edge $e$. The edge $f\in E(F_\star)$ projected to $e$ has the 
		desired property. 
	\end{proof}
	
	We proceed by transferring the hypothesis $\GTH(\Pi, \equiv^\Pi, \ccP^+)>g$ to standard 
	copies. More precisely, we deal with the special case of cycles 
	in $(\Sigma, \equiv^\Sigma, \ccQ^+)$ all of whose copies belong to the same standard copy 
	of the previous picture. 
	  	 
	\begin{clm}\label{clm:1327}
		Let $(\Pi_\star, \equiv^{\Pi_\star}, \ccP_\star)$ be a standard copy. If 
		all copies of an acceptable 
		big cycle~$\ccC$ in $(\Sigma, \equiv^\Sigma, \ccQ^+)$ with $\ord{\ccC}\le g$
		belong to $\ccP_\star^+$, then $\ccC$ 
		possesses a supreme copy. 
	\end{clm}

	\begin{proof}
		Let $\ccC=F_1q_1\ldots F_nq_n$. 
		Strictly speaking, $\ccC$ does not need to be a big cycle in
		the pretrain system $(\Pi_\star, \equiv^{\Pi_\star}, \ccP_\star^+)$,
		for the wagon connectors of $\ccC$ are wagons of $\equiv^\Sigma$ 
		rather than wagons of $\equiv^{\Pi_\star}$. However,~\ref{it:B4} shows that every 
		wagon among $q_1, \ldots, q_n$ contracts to~$(\Pi_\star, \equiv^{\Pi_\star})$.
		Thus we can define 
				\[
			\overline{q}_i=
				\begin{cases}
					\text{the contraction of $q_i$ to $(\Pi_\star, \equiv^{\Pi_\star})$,} & 
							\text{ if $q_i$ is a wagon} \cr
					q_i, & \text{ if $q_i$ is a vertex}
				\end{cases}
		\]
		for every $i\in\ZZ/n\ZZ$ and then 
		\[
			\overline{\ccC}=F_1\overline{q}_1\ldots F_n\overline{q}_n
		\]
		will be a big cycle in the system of pretrains 
		$(\Pi_\star, \equiv^{\Pi_\star}, \ccP_\star^+)$ with 
		$\ord{\overline{\ccC}}=\ord{\ccC}\le g$. It is easily confirmed 
		that $\overline{\ccC}$ is acceptable and, therefore, $\overline{\ccC}$
		possesses a supreme copy $F_\star\in\{F_1, \dots, F_n\}$. 
		Let 
				\[
			\overline{\ccM}
			=
			\{\overline{P}_i\colon i\in\ZZ/n\ZZ \text{ and } F_i\ne F_\star\}
		\]
				be a family of $F_\star$-pieces with respect to the ambient 
		system $(\Pi_\star, \equiv^{\Pi_\star}, \ccP_\star^+)$ that exemplifies 
		the supremacy of $F_\star$ in $\overline{\ccC}$. 
		
		We intend to show that $F_\star$ is a supreme copy of $\ccC$ as well and to
		this end we need to convert $\overline{\ccM}$ into an appropriate family 
		of $F_\star$-pieces with respect to $(\Sigma, \equiv^\Sigma, \ccQ^+)$.
		For every index~$i$ such that $\overline{P}_i$ is short we just 
		set $P_i=\overline{P}_i$. If $\overline{P}_i=(f'_i)^+\overline{W}_i(f''_i)^+$
		is long, however, then we put $P_i=(f'_i)^+ W_i (f''_i)^+$, 
		where $W_i$ denotes the wagon of $(\Sigma, \equiv^\Sigma)$ contracting 
		to~$\overline{W}_i$. 
		
		It will turn out that  
				\[
			\ccM
			=
			\{P_i\colon i\in\ZZ/n\ZZ \text{ and } F_i\ne F_\star\}
		\]
				is the desired family of $F_\star$-pieces. The main point to be checked here 
		is that the long pieces in this family still satisfy 
		Definition~\ref{dfn:1746}\ref{it:1751b}\ref{it:beta} with respect 
		to the larger pretrain~$(\Sigma, \equiv^\Sigma)$. 
		Assume for the sake of contradiction that $P_i$
		is a long piece and that some edge~$f$ in~$E(\Sigma)\sm E(\Pi_\star)$
		contains the connectors $q_{i-1}$, $q_i$. Due to the linearity of $G$ this 
		requires $f\in E(H)$. But now Lemma~\ref{lem:2107} applied to the system 
		$(H, \equiv^H, \ccH^+)$ shows that~$f$ belongs to the 
		copy $\Pi^e_\star\in\ccH$ extended by~$\Pi_\star$, contrary 
		to $f\not\in E(\Pi_\star)$. This concludes the proof that~$F_\star$ is a
		supreme copy of $\ccC$. 
	\end{proof}
	
	Next we study a peculiar kind of big cycles in our amalgamation that will 
	occasionally require a special treatment later. 
	
	\begin{clm}\label{clm:2203}
		Let $(\Pi_\star, \ccP_\star)$ and $(\Pi_{\star\star}, \ccP_{\star\star})$ be two 
		distinct standard copies of $(\Pi, \ccP)$. 
		If 
				\[
			\ccC=F_1W_\star F_2xF_3W_{\star\star}F_4y
		\]
				is a big cycle in $(\Sigma, \equiv^\Sigma, \ccQ^+)$ 
		such that the connectors $x$, $y$ are vertices, $W_\star$, $W_{\star\star}$ 
		are wagons, and $F_1, F_2\in \ccP_\star^+$ as well 
		as $F_3, F_4\in \ccP_{\star\star}^+$,
		then there exists an edge $f\in E(H)$ with $x, y\in f$. 
 	\end{clm}

\usetikzlibrary{decorations.pathreplacing}

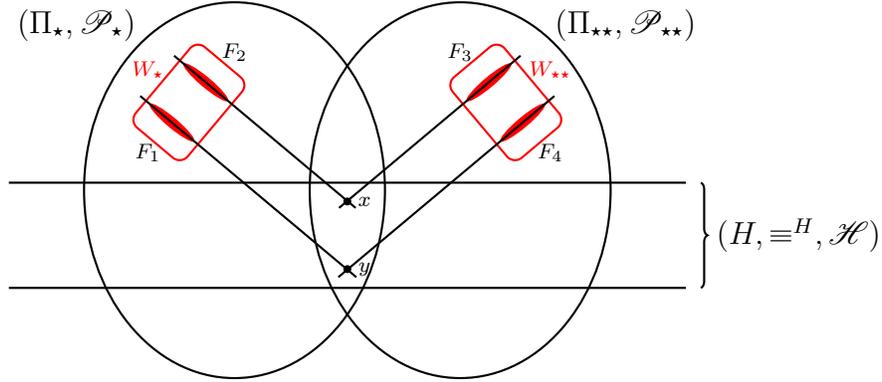
\begin{figure}[ht]
	\centering

			\begin{tikzpicture}[scale=1]

		\def\ob{50:(0,.5)};
		\def\oo{-50:(0,.5)};
		
		\draw [thick](-4.5,-.2) -- (4.5,-.2);
			\draw [thick](-4.5,1.2) -- (4.5,1.2);
		
	\fill [red, rotate around={\ob}](-.35, 3.25) ellipse (2pt and 11pt);
	\fill [red, rotate around={\ob}](.35, 3.25) ellipse (2pt and 11pt);
			
	\fill [red, rotate around={\oo}](-.35, 3.25) ellipse (2pt and 11pt);
	\fill [red, rotate around={\oo}](.35, 3.25) ellipse (2pt and 11pt);

		\draw [thick](-1.5,1.1) ellipse (2cm and 2.5cm);
		\draw [thick](1.5,1.1) ellipse (2cm and 2.5cm);
		
		\draw [thick, red, rounded corners, rotate around={\ob}](-.7, 2.8) rectangle (.7, 3.7);
		
		\draw [thick, rotate around={\ob}] (-.35,3.8)--(-.35,.05);
		\draw [thick, rotate around={\ob}] (.35,3.8)--(.35,.65);

			\draw [thick, red, rounded corners, rotate around={\oo}](-.7, 2.8) rectangle (.7, 3.7);
		
		\draw [thick, rotate around={\oo}] (-.35,3.8)--(-.35,.65);
		\draw [thick, rotate around={\oo}] (.35,3.8)--(.35,.05);
		
		\fill (0,.05) circle (1.5pt);
		\fill (0,.95) circle (1.5pt);
		
		\node at (.23, .05) {\tiny $y$};
		\node at (.23, .95) {\tiny $x$};
		\node at (-2.65,1.6) {\tiny $F_1$};
		\node at (-1.5, 2.93){\tiny $F_2$};
		\node [red]at (-2.65, 2.7) {\tiny $W_{\star}$};
		\node at  (-3.6, 3.3) {$(\Pi_\star, \ccP_\star)$};
		
		\node at (2.7,1.6) {\tiny $F_4$};
		\node at (1.5, 2.93){\tiny $F_3$};
		\node [red]at (2.7, 2.7) {\tiny $W_{\star\star}$};
		\node at  (3.7, 3.3) {$(\Pi_{\star\star}, \ccP_{\star\star})$};

			\draw [thick, decorate,
			decoration = {brace}] (4.7,1.2) --  (4.7,-.2);
			
			\node at (6,.5) {$(H, \equiv^H, \ccH)$};

		\end{tikzpicture}

				\caption{The configuration analysed in Claim~\ref{clm:2203}}
				\label{fig:912}

	\end{figure} 	
	\begin{proof}
		Notice that $x, y\in V(H)$.
		Due to the linearity of the pretrain $(H, \equiv^H)$ it cannot be the case 
		that $W_{\star}$, $W_{\star\star}$ contract to wagons $W^H_{\star}$, $W^H_{\star\star}$ 
		of $(H, \equiv^H)$ with $x, y\in V(W^H_{\star})\cap V(W^H_{\star\star})$. 
		By symmetry we may therefore suppose that 
				\begin{equation}\label{eq:1428}
			\text{if the contraction $W^H_{\star}$ exists, 
				then $\{x, y\}\not\subseteq V(W^H_{\star})$.}
		\end{equation}
				
		Let $\Pi^e_\star, \Pi^e_{\star\star}\in \ccH$ be the copies extended by $\Pi_\star$, 
		$\Pi_{\star\star}$. Now $\ccD=\Pi^e_\star x \Pi^e_{\star\star} y$ is a big cycle 
		in~$(H, \ccH^+)$. If either of its copies is an edge copy, then the existence of 
		the desired edge $f$ is clear, so we can assume that $\Pi^e_\star$,  
		$\Pi^e_{\star\star}$ are both real.
		
		Owing to $\GTH(H, \equiv^H, \ccH^+)>g\ge 2$ Lemma~\ref{lem:2037} tells us 
		that $\Pi^e_\star$ is a supreme copy of~$\ccD$. Let the $\Pi^e_\star$-piece $P$
		exemplify this fact. If $P=f^+$ is short we have found the desired edge $f$. It 
		remains to consider the case that $P=(f')^+W^H(f'')^+$ is long, where, let us
		emphasise, $W^H$ is a wagon of $(H, \equiv^H)$. 
		
		By $x, y\in V(W^H)$ and~\eqref{eq:1428} it cannot be the case 
		that $W^H= W^H_\star$. Thus the wagon~$W$ of $(\Sigma, \equiv^\Sigma)$
		contracting to $W^H$ is distinct from $W_\star$. A fortiori the contractions 
		$\overline{W}$, $\overline{W}_\star$ of~$W$,~$W_\star$ 
		to $(\Pi_\star, \equiv^{\Pi_\star})$ are distinct as well and, therefore,   
		\[
			\ccE=F_1\overline{W}_\star F_2x(f')^+\overline{W}(f'')^+y
		\]
		(see Figure~\ref{fig:912e}) is a big cycle 
		in $(\Pi_\star, \equiv^{\Pi_\star}, \ccP_\star^+)$.
		
		\begin{figure}[ht]
\centering	
\begin{tikzpicture}[scale=1]
\def\ob{50:(0,.5)};
\def\oo{-50:(0,.5)};
		
\draw [thick](-3.3,-.2) -- (.5,-.2);
\draw [thick](-3.6,1.2) -- (.8,1.2);
			
\fill [red, rotate around={\ob}](-.35, 3.25) ellipse (2pt and 11pt);
\fill [red, rotate around={\ob}](.35, 3.25) ellipse (2pt and 11pt);
			
\draw [thick](-1.4,1.1) ellipse (2.1cm and 2.5cm);
\draw [thick, red, rounded corners, rotate around={\ob}](-.7, 2.8) rectangle (.7, 3.7);
\draw [thick, rotate around={\ob}] (-.35,3.8)--(-.35,.05);
\draw [thick, rotate around={\ob}] (.35,3.8)--(.35,.65);
		
\draw [green!70!black, thick] (-.8,-.2)--(.17,1.2);
\draw [green!70!black, thick] (-.8, 1.2)--(.17, -.2);
		
\fill (0,.05) circle (1.5pt);
\fill (0,.95) circle (1.5pt);
		
\node at (.23, .05) {\tiny $y$};
\node at (.23, .95) {\tiny $x$};
\node at (-2.65,1.6) {\tiny $F_1$};
\node at (-1.5, 2.93){\tiny $F_2$};
\node [red]at (-2.65, 2.7) {\tiny $\overline{W}_{\star}$};
\node at (-3.6, 3.3) {$(\Pi_\star, \ccP_\star)$};
		
\draw [green!70!black, thick, rounded corners] (-1,1.3) rectangle(.33, -.3);
	
\node [green!70!black] at (-.8,.85) {\tiny $f''$};
\node [green!70!black] at (-.75,.05) {\tiny $f'$};
\node [green!70!black] at (0,1.6) {\tiny $\overline{W}$};
			
\end{tikzpicture}

\caption{The big cycle $\ccE$.}
\label{fig:912e} 
\end{figure} 		
		Clearly, it satisfies the acceptability condition~\ref{it:A1}, and~\ref{it:A2}
		could only fail if the desired edge $f$ exists. Moreover,~\ref{it:A3} holds 
		due to the fact that the only wagon of the linear 
		pretrain $(\Pi_\star, \equiv^{\Pi_\star})$ that contains $x$ and $y$ is $\overline{W}$ 
		and, hence, a connector of $\ccE$. 
		
		Altogether we may assume that $\ccE$ is acceptable.
		Due to 
				\[
			\GTH(\Pi_\star, \equiv^{\Pi_\star}, \ccP^+)
			>
			g
			\ge 
			2
			=
			\ord{\ccE}
		\]
				this implies that $\ccE$ has a supreme copy. By Fact~\ref{f:1644} such a 
		supreme copy needs to contain both~$x$ and~$y$ and thus Claim~\ref{clm:2131} 
		leads again to the desired edge $f$.
	\end{proof}
	
	\noindent
	{\bf Stage B: Segmentations.}
	When considering any (not necessarily acceptable) big cycle $\ccC=F_1q_1\ldots F_nq_n$ 
	in $(\Sigma, \equiv^\Sigma)$
	we can define {\it segments} of $\ccC$, their {\it leaders}, and {\it segmentations} 
	of~$\ccC$ as in the proof of Lemma~\ref{lem:251}. Unless the copies of $\ccC$ are 
	contained in a single standard copy of $(\Pi, \ccP^+)$ we obtain a segmentation
		\[
		\ccC=I_1r_1\ldots I_tr_t
	\]
		with $t\ge 2$ such that any two consecutive ones among the 
	leaders $\Pi^e_1, \ldots, \Pi^e_t$ are distinct. As usual we want 
	to pass to a big cycle in $(H, \equiv^H, \ccH^+)$ by replacing every segment 
	$I_\tau$ by its leader $\Pi^e_\tau$. 
	
	The resulting cyclic sequence clearly 
	has the properties~{\ref{it:B1}\,--\,\ref{it:B3}}. However, the wagons among 
	$r_1, \ldots, r_t$ usually fail to be wagons of $(H, \equiv^H)$ and thus they
	are, pedantically speaking, not allowed to serve as wagons of such cycles. 
	Due to the clauses~\ref{it:1832b} and~\ref{it:1832c} of Lemma~\ref{lem:1825}, 
	however, we know that such wagon connectors $r_\tau$ contract to $(H, \equiv^H)$.
	More precisely, if the leaders $\Pi^e_\tau$ and $\Pi^e_{\tau+1}$ are in $\ccH$, 
	then~\ref{it:1832c}\ref{it:1832c2} shows that they contain edges of $r_\tau$.
	Similarly, if only one of them is in $\ccH$ while the other one is an edge copy,
	we appeal to~\ref{it:1832c}\ref{it:1832c1}
	and the case that both are edge copies is unproblematic due to~\ref{it:1832b}.
	   
	Therefore we can define  
		\[
		\overline{r}_\tau=
			\begin{cases}
				\text{the contraction of $r_\tau$ to $(H, \equiv^H)$} &
					 	\text{ if $r_\tau$ is a wagon} \cr
				r_\tau & \text{ if $r_\tau$ is a vertex}
			\end{cases}
	\]
		for every $\tau\in\ZZ/t\ZZ$ and the cyclic sequence 
		\[
		\ccD=\Pi^e_1\overline{r}_1\ldots \Pi^e_t\overline{r}_t
	\]
	will be a big cycle in $(H, \equiv^H, \ccH^+)$. 
			Let us observe that Fact~\ref{fact:h}
	also holds for big cycles instead of cycles of copies, because its proof is based on
	a computation that can be carried out in both scenarios. Utilising the thus modified
	form of Fact~\ref{fact:h} repeatedly we obtain $\ord{\ccD}\le \ord{\ccC}$. 
	It will be convenient to call $\ccD$ a {\it reduct} of $\ccC$. 
	\index{reduct}
	
	\smallskip
	\noindent
	{\bf Stage C: The proof of~\ref{it:2127a}.} Recall that by Fact~\ref{f:wagstr} 
	the wagons of linear pretrains are strongly induced. Thus the following claim 
	can be regarded as a first step towards proving that $(\Sigma, \equiv^\Sigma)$ 
	is linear.
	
	\begin{clm}\label{clm:1516}
		Every wagon of $(\Sigma, \equiv^\Sigma)$ is strongly induced. 
	\end{clm}

	\begin{proof}
		Let $W$ denote an arbitrary wagon of $(\Sigma, \equiv^\Sigma)$. 
		Due to the definition of 
		wagons we have $E(W)\ne \vn$ and no vertex of $W$ is isolated. Hence it remains 
		to check that if an edge $f\in E(\Sigma)$ intersects $V(W)$ in two distinct vertices 
		$x$ and $y$, then $f\in V(W)$. To this end we consider pairs of 
		edges $e_x, e_y\in E(W)$
		with $x\in e_x$ and $y\in e_y$. If $e_x$ and $e_y$ can be chosen in such a way that
		any two of the three edges $e_x$, $e_y$, and $f$
		coincide, we conclude $f\in E(W)$ either directly or by appealing to the linearity 
		of the hypergraph $\Sigma$.
		Assuming that this is not the case we obtain, for each choice of $(e_x, e_y)$, 
		a big cycle 
				\[
			\ccC=f^+xe_x^+We_y^+y
		\]
				of order $2$. 
		It should perhaps be pointed out that $\ccC$ is always inacceptable, for 
		the edge~$f$ exemplifies that the wagon $W$ violates the moreover-part of~\ref{it:A2}.
		Nevertheless, $\ccC$ can have acceptable reducts and we can learn something by 
		looking at supreme copies of such reducts.  
		
		If the edges $e_x$, $e_y$ can be selected in such a way that together with $f$
		they belong to a common standard copy $(\Pi_\star,  \equiv^{\Pi_\star})$, 
		then $f\in E(W)$ follows from the fact that the contraction of $W$ 
		to $(\Pi_\star, \equiv^{\Pi_\star})$ is strongly induced in $\Pi_\star$. 
		
		So we can suppose that such a choice of $e_x$, $e_y$ is impossible. Now let 
		these edges and a segmentation 
		\begin{equation}\label{eq:1556}
			\ccC=I_1r_1\ldots I_tr_t
		\end{equation}
		be determined in such a manner that $t\in \{2, 3\}$ is minimal.
		Write $\ccD=\Pi^e_1\overline{r}_1\ldots \Pi^e_t\overline{r}_t$ for 
		the cor\-re\-sponding reduct. Up to symmetry one of the following three cases occurs. 
		
		\smallskip
		
		{\it \hskip2em First Case. $\ccD=f^+ x \Pi_2^e y$}
	
		\smallskip
		
		Let $\Pi_2$ be the standard copy of $\Pi$ extending $\Pi^e_2$ and containing 
		the edges $e_x$, $e_y$. Lemma~\ref{lem:2107} applied to the 
		system $(H, \equiv^H, \ccH^+)$, the copy $\Pi^e_2$, and the edge $f$ yields 
		$f\in E(\Pi^e_2)$. Thus all copies of $\ccC$ can be found in the same standard 
		copy $\Pi_2$ and we are in the case discussed immediately before~\eqref{eq:1556}.  
		
		\smallskip
		
		{\it \hskip2em Second Case. $\ccD=\Pi^e_1 \overline{W} \Pi^e_2 y$}
	
		\smallskip
		
		Here $\Pi^e_1$ extends to a standard copy $\Pi_1$ of $\Pi$ that contains the 
		edges $f$, $e_x$. Moreover,~$\overline{W}$ denotes the contraction of $W$ 
		to $(H, \equiv^H)$.
		Due to Lemma~\ref{lem:GTH1} there exists an edge~$e'_y$ 
		in $E(\Pi^e_1)\cap E(\overline{W})$ passing through~$y$. 
		By choosing~$e'_y$ instead of~$e_y$ we again arrive at a situation where 
		everything relevant happens within the same standard copy, namely $\Pi_1$.    	
		
		\smallskip
		
		{\it \hskip2em Third Case. $\ccD=f^+ x \Pi^e_2 \overline{W} \Pi^e_3 y$ and, 
		hence, $t=3$.}
	
		\smallskip
		
		If $x, y\in V(\overline{W})$, then $\overline{W}\Str H$ 
		yields $f\in E(\overline{W})\subseteq E(W)$ and we are done immediately. 
		Now assume for the sake of contradiction that $\{x, y\}\not\subseteq V(\overline{W})$, 
		which causes $\ccD$ to be acceptable. Without loss of generality we may 
		suppose that $\Pi^e_2$ is a supreme copy of~$\ccD$. The $\Pi^e_2$-piece 
		$\Pi^e_3$ collapses to needs to be short. If~$(e'_y)^+$ denotes that piece, 
		then by choosing~$e'_y$ instead of~$e_y$ we are led to a case with $t=2$, 
		contrary to the minimality of~$t$. 
	\end{proof}

	Now we are ready to proceed to the linearity of the pretrain $(\Sigma, \equiv^\Sigma)$. 
	Given any two vertices~$x$ and $y$, we are to prove that the set $\ccW$ of all wagons $W$
	with $x, y\in V(W)$ has at most one element. If there exists an edge $f\in E(\Sigma)$ with 
	$x, y\in f$, then Claim~\ref{clm:1516} discloses $f\in E(W)$ for every $W\in \ccW$ 
	and $|\ccW|\le 1$ follows. Thus we may assume that 
		\begin{equation}\label{eq:1518}
		\text{ no $f\in E(\Sigma)$ contains both $x$ and $y$}.
	\end{equation}
		If $x, y\in V(H)$, then the linearity of $(H, \equiv^H)$ implies that at most one 
	wagon $W_\star^H$ of this pretrain satisfies $x, y\in V(W^H_\star)$. If such a wagon
	exists, then we denote the unique wagon of $(\Sigma, \equiv^\Sigma)$ contracting to it 
	by $W_\star$. 
	
	Now assume for the sake of contradiction that there exist distinct $W', W''\in \ccW$. 
	In case $W_\star$ exists we suppose, additionally, that $W_\star\in \{W', W''\}$. 
	For every quadruple $\seq{e}=(e'_x, e'_y, e''_x, e''_y)$ of edges with 
		\[
		x\in e'_x\in E(W')\,, \quad
		y\in e'_y\in E(W')\,, \quad
		x\in e''_x\in E(W'')\,, 
		\quad\text{ and }\quad
		y\in e''_y\in E(W'')\,, 
	\]
	the statement~\eqref{eq:1518} shows that 
	\[
		\ccC=(e'_x)^+ W' (e'_y)^+ y (e''_y)^+ W'' (e''_x)^+ x
	\]
	is a big cycle. As the standard copies of $(\Pi, \equiv^\Pi)$ are linear 
	pretrains, there is no such quadruple $\seq{e}$ whose four edges are in the 
	same standard copy. 
	Thus every choice of~$\seq{e}$ gives rise to at least one segmentation 
	\begin{equation}\label{eq:1533}
		\ccC=I_1r_1\ldots I_tr_t
	\end{equation}
	of the associated cycle $\ccC$. For the rest of the argument we fix $\seq{e}$
	and the segmentation~\eqref{eq:1533} in such a manner that $t\in \{2, 3, 4\}$ is minimal. 
	Let $\ccD=\Pi^e_1\overline{r}_1\ldots \Pi^e_t\overline{r}_t$ denote the corresponding 
	reduct of $\ccC$.

	\smallskip
		
		{\it \hskip2em First Case. $t=2$}
	
	\smallskip
	
	Let us consider the connectors of $\ccD$. Due to~Claim~\ref{clm:2203} 
	and~\eqref{eq:1518} it cannot be the case that both of them are vertices. 
	Moreover, Lemma~\ref{lem:n1624} excludes the case that $r_1$, $r_2$ are the 
	two wagons of $\ccC$. Up to symmetry this means that only the possibility 
	$\ccD=\Pi^e_1 \overline{W'} \Pi^e_2 x$ remains, where $\Pi^e_1$ denotes the 
	leader of $(e'_x)^+$. 
	
	Now Lemma~\ref{lem:GTH1} provides an 
	edge $f$ such that $x\in f\in E(\overline{W'})\cap E(\Pi^e_2)$ and 
	the quadruple $(f, e'_y, e''_x, e''_y)$ can play the r\^{o}le of~$\seq{e}$.
	However, all four edges of this quadruple pertain to the same standard 
	copy $\Pi_2$, which is absurd.  
	 
	\smallskip
		
		{\it \hskip2em Second Case. $t\in \{3, 4\}$}
	
	\smallskip
		
	We contend that $\ccD$ is acceptable. 		
	The property~\ref{it:A1} is clear,~\ref{it:A2} follows from~\eqref{eq:1518}, 
	and~\ref{it:A3} only requires attention if~$t=4$ and the wagon $W$ we want to 
	test contains both~$x$ and~$y$. We already know, however, 
	that there exists at most one such wagon, namely~$W^H_\star$. By our choice of $W'$,~$W''$ 
	and by $t=4$ this wagon is, if it exists, a connector of $\ccD$ and thus irrelevant 
	for~\ref{it:A3}. 
	
	SO $\ccD$ is indeed acceptable and due to $\ord{\ccD} = 2$ it follows 
	that $\ccD$ has a supreme copy, say $\Pi_2^e$. If both segments $I_1$, $I_3$ 
	contained more than one copy, then the contradiction 
	$2\le |\ccC|-|\ccD|=4-t\le 1$ would arise. Consequently we may assume 
	that $I_1$ consists of a single copy, which causes the corresponding 
	index to be mixed in $\ccD$. If $P_1=f^+$ denotes the short $\Pi_2^e$-piece 
	$\Pi^e_1$ collapses to, then by including $f$ into the quadruple $\seq{e}$ 
	(instead of the edge corresponding to $I_1$) we reach a contradiction to 
	the minimality of $t$. 
	   
	This concludes the proof that $(\Sigma, \equiv^\Sigma)$ is linear.
	
	\smallskip
	
	\noindent
	{\bf Stage D: The proof of~\ref{it:2127b}.}
	Consider any acceptable big cycle $\ccC=F_1q_1\ldots F_nq_n$ 
	in $(\Sigma, \equiv^\Sigma, \ccQ^+)$ whose order is at most~$g$. 
	We are to prove that $\ccC$ has a supreme copy. 
	
	Assuming that Claim~\ref{clm:1327} does not imply this immediately, 
	we pick a segmentation 
		\[
		\ccC=I_1r_1\ldots I_tr_t
	\]
		of $\ccC$ whose length $t\ge 2$ is minimal and pass to the associated reduct
		\[
		\ccD=\Pi^e_1\overline{r}_1\ldots \Pi^e_t\overline{r}_t\,.
	\]
	We commence with the simplest case. 
	
	\begin{clm}\label{clm:GTH1}
		If $\ord{\ccD}=1$, then $\ccC$ has a supreme copy. 
	\end{clm}
	
	\begin{proof}
		Suppose $\ccD=\Pi^e_1 x \Pi^e_2 \overline{W}$, where $x$ is a vertex 
		and $\overline{W}$ is a wagon of $(H, \equiv^H)$.
		
		Let us first deal with the special case $n=|\ccC|=2$, where $\ccC=F_1 x F_2 W$.
		Due to~\ref{it:A1} and symmetry we may assume that the copy $F_1$ is real. 
		Now Lemma~\ref{lem:GTH1} applied to $\ccD$ and~$\Pi^e_1$ yields an edge $f$ such that 
		$x\in f\in E(\overline{W})\cap E(\Pi^e_1)$. Since all copies of the 
		acceptable big cycle $\ccE=F_1 x f^+ W$ are in the same standard copy of $\Pi$,
		Claim~\ref{clm:1327} shows that $\ccE$ has a supreme copy. 
		Remark~\ref{rem:1852}\ref{it:1852b} tells us that this supreme copy must be $F_1$. 
		The short $F_1$-piece $Q$ to which we can collapse $f^+$ witnesses that $F_1$
		is a supreme copy of $\ccC$ as well. 
		 
		It remains to treat the case $n\ge 3$. If both segments $I_1$, $I_2$ contained 
		at least two copies of $\ccC$, then the wagon $W$ would exemplify that $\ccC$ 
		violates~\ref{it:A2}. So we may assume that $I_1$ consists of a single copy, 
		say~$F_1$. This time we utilise Lemma~\ref{lem:GTH1} for obtaining an edge $f$ 
		with $x\in f\in E(\overline{W})\cap E(\Pi^e_2)$. Lemma~\ref{lem:1448}\ref{it:1448a} 
		tells us that in $\ccC$ we can replace $F_1$ by $f^+$, thus arriving at another 
		acceptable big cycle $\ccE=f^+ x I_2 W$. The advantage of $\ccE$ is that all its 
		copies belong to the same standard copy of $\Pi$. Thus Claim~\ref{clm:1327} 
		tells us that $\ccE$ has some supreme copy and owing to 
		Lemma~\ref{lem:1448}\ref{it:1448b} so does $\ccC$.
	\end{proof}
	
	So in the sequel we can assume $\ord{\ccD}\in [2, \ord{\ccC}]$. In particular, 
	$\ccD$ satisfies~\ref{it:A1} and one confirms easily that it inherits the 
	properties~\ref{it:A2},~\ref{it:A3} from~$\ccC$. Altogether $\ccD$ is an acceptable 
	big cycle whose order is at most $g$; thus 
		\begin{equation}\label{eq:1821}
		\ccD \text{ has a supreme copy.}
	\end{equation}
			
	Next, we isolate the following special case. 

	\begin{clm}\label{clm:2337}
		If $t=2$, then $\ccC$ has a supreme copy. 
	\end{clm}
	
	\begin{proof}
		As a consequence of Lemma~\ref{lem:n1624} it cannot be the case that 
		both connectors of~$\ccD$ are wagons. So Claim~\ref{clm:GTH1} allows 
		us to assume that $r_1$, $r_2$ are vertices. We distinguish several 
		cases depending on whether the segments $I_1$, $I_2$ consist of single 
		copies. Up to symmetry there are three possibilities. 
		
		\smallskip
		
		{\it \hskip2em First Case. $n=2$}
	
		\smallskip
		 
		Claim~\ref{clm:2131} leads to an edge $f\in E(F_1)\cap E(F_2)\cap E(H)$.
		By~\ref{it:B1} it cannot be the case that both copies of $\ccC$ are equal to $f^+$ 
		and, hence, we may suppose that $F_1\ne f^+$. So we can collapse $F_2$ to
		the $F_1$-piece $f^+$, thus inferring that $F_1$ is a supreme copy of $\ccC$.
		
		\smallskip
		
		{\it \hskip2em Second Case. $n\ge 3$ and $I_2$ consists of a singly copy.}
	
		\smallskip
		
		To simplify the notation we assume $I_2=F_n$. Due to Claim~\ref{clm:2131} 
		there exists an edge $f\in E(F_n)\cap E(H)$. 
		Now Lemma~\ref{lem:1448}\ref{it:1448a} informs us that $\ccE=I_1 r_1 f^+ r_2$
		is an acceptable big cycle in $(\Sigma, \equiv^\Sigma, \ccQ^+)$. Moreover,   
		$r_1, r_2\in f\cap V(\Pi^e_1)$ and Lemma~\ref{lem:2107} imply $f\in E(\Pi^e_1)$,
		for which reason all copies of $\ccE$ belong to a common standard copy of $\Pi$. 
		By Claim~\ref{clm:1327} it follows that $\ccE$ has a supreme copy  
		and, finally, by Lemma~\ref{lem:1448}\ref{it:1448b} so does $\ccC$. 
		
		\smallskip
		
		{\it \hskip2em Third Case. Both $I_1$, $I_2$ contain at least two copies.}
	
		\smallskip 
		
		Suppose for the sake of contradiction that some edge $f$ of $\Sigma$ 
		covers $\{r_1, r_2\}$. Let $W$ be the wagon of $(\Sigma, \equiv^\Sigma)$
		containing $f$. Due to~\ref{it:A2} we know that $W$ cannot serve as a 
		connector of $\ccC$ and thus the description of the present case entails 
		that $W$ violates~\ref{it:A3}. This proves that 
				\begin{equation}\label{eq:norrf}
			\text{ there is no $f\in E(\Sigma)$ such that $r_1, r_2\in f$}.
		\end{equation}
				
		In particular, the copies $\Pi^e_1$ and $\Pi^e_2$ need to be real 
		and due to Lemma~\ref{lem:2037}
		both of them are supreme copies of~$\ccD$. By Claim~\ref{clm:2203} and symmetry 
		we may suppose that $I_1$ does not consist of two copies and an interposed wagon. 
		Let $P$ be a $\Pi^e_1$-piece with the property that collapsing~$\Pi^e_2$ to $P$ 
		establishes the supremacy of $\Pi_e^1$ in $\ccD$. Because of~\eqref{eq:norrf}
		we know that $P$ is long. Write $P=(f')^+\overline{W}(f'')^+$
		and denote the wagon of $(\Sigma, \equiv^\Sigma)$ contracting 
		to~$\overline{W}$ by~$W$.
		Owing to~\ref{it:A3} this wagon needs to appear in~$\ccC$. 
		Now~\ref{it:A2} tells us that the connectors of~$\ccC$ next to $W$ 
		are $r_1$ and $r_2$. 
		
		So our special choice of~$I_1$ guarantees that $I_2$ consists of two 
		copies and the wagon $W$ between them. 
		Without loss of generality we may suppose that $I_2=F_{n-1}WF_n$. 	
		Now Lemma~\ref{lem:1448} applies to $K=\{n-1, n\}$, $P_{n-1}=(f')^+$, 
		and $P_n=(f'')^+$.
		So by part~\ref{it:1448a} of the lemma $I_1r_1(f')^+W(f'')^+r_2$ 
		is an acceptable big cycle, by 
		Claim~\ref{clm:1327} this cycle has a supreme copy, and by 
		Lemma~\ref{lem:1448}\ref{it:1448b} $\ccC$ has a supreme copy as well. 
	\end{proof}
		
	Throughout the remainder of this proof we suppose that $n\ge t\ge 3$.
	Let $\Pi^e_\star$ be a supreme copy  of $\ccD$ (see~\eqref{eq:1821})  
	and let 
		\[
		\bigl\{Q_\tau\colon \tau\in \ZZ/t\ZZ \text{ and } \Pi_\tau^e \ne \Pi_\star^e\bigr\}
	\]
		be a family of $\Pi^e_\star$-pieces acting as supremacy witnesses.	
	Our goal is to associate with every copy $F_k$ of $\ccC$ that belongs 
	to a segment $I_\tau$ whose leader $\Pi_\tau^e$
	is distinct from $\Pi_\star^e$ a $\Pi^e_\star$-piece~$P_k$ such that 
	\begin{enumerate}[label=\nlabel]
		\item\label{it:2006a} if one eliminates $F_k$ from $\ccC$ and inserts $P_k$
			instead of it, then the resulting cyclic sequence satisfies~\ref{it:B3} 
			and~\ref{it:B4},
		\item\label{it:2006b} and if $P_k$ is long, then it has the 
			properties~\ref{it:1801a},~\ref{it:1801b}, and~\ref{it:1801c} 
			from Lemma~\ref{lem:1448}.
	\end{enumerate} 
	
	Once we have such pieces $P_k$ at our disposal, we can perform all replacements 
	indicated in~\ref{it:2006a} simultaneously and Lemma~\ref{lem:1448}\ref{it:1448a} 
	informs us that the cyclic sequence that results is an acceptable big cycle.
	All copies of this cycle will be contained in the same standard 
	copy~$(\Pi_\star, \equiv^{\Pi_\star}, \ccP_\star^+)$, 
	where $\Pi_\star$ extends $\Pi^e_\star$. Owing 
	to Claim~\ref{clm:1327}	the modified cycle needs to have a supreme copy, 
	so by Lemma~\ref{lem:1448}\ref{it:1448b} the 
	original big cycle $\ccC$ has a supreme copy as well. 
	
	Thus it suffices to exhibit the $\Pi^e_\star$-pieces described above.
	Notice that the connectors~$\overline{r}_{\tau-1}$ and~$\overline{r}_\tau$ cannot be wagons 
	at the same time, for then the piece $Q_\tau$ could not exist. 
	Suppose next that one of those connectors, say $\overline{r}_{\tau-1}$, is a vertex 
	while the other one, i.e. $\overline{r}_{\tau}$, is a wagon. In this case the 
	piece $Q_\tau$ needs to be short. Besides, since the wagon $r_\tau$ satisfies~\ref{it:A2}
	in $\ccC$, the segment~$I_\tau$ has length $1$ and we can set $P_k=Q_\tau$.
	
	The last case we need to consider is that both $r_{\tau-1}$ and $r_\tau$ are vertices. 
	Suppose first that the piece $Q_\tau=f^+$ is short and let $W$ be the wagon 
	of $(\Sigma, \equiv^\Sigma)$ with $f\in E(W)$. An easy argument using~\ref{it:A2}
	shows that $W$ cannot be a connector of $\ccC$. So~\ref{it:A3} applies to~$W$.
	We deduce that the segment~$I_\tau$ has length $1$ and again we can set $P_k=Q_\tau$.
	
	So henceforth we may assume that the piece $Q_\tau=(f')^+\overline{W}(f'')^+$ is long. 
	The easiest possibility would be that the interval $I_\tau$ is of the 
	form $F_jW F_{j+1}$, where $W$ denotes the wagon of $(\Sigma, \equiv^\Sigma)$ 
	contracting to $\overline{W}$, for then we just need to set $P_j=(f')^+$ 
	and $P_{j+1}=(f'')^+$.  
	
	Let us now assume for the sake of contradiction that $I_\tau$ is not of this 
	special form. 
	Observe first that the admissibility of $Q_\tau$ as a supremacy witness
	entails that 
		\begin{equation}\label{eq:0014}
		\text{no edge of $H$ contains both $r_{\tau-1}$ and $r_{\tau}$.}
	\end{equation}
		If the connectors $r_{\tau-1}$ and $r_\tau$ are consecutive 
	in $\ccC$, then Claim~\ref{clm:2131} applied to the copy between them
	yields an edge that contradicts~\eqref{eq:0014}. This shows that   
		\begin{equation}\label{eq:1850}
		\text{the connectors $r_{\tau-1}$ and $r_\tau$ fail to be consecutive in $\ccC$}.
	\end{equation}
		
	Now condition~\ref{it:A3} tells us that the wagon~$W$ 
	has to appear somewhere in $\ccC$. On the other hand, $\overline{W}$ cannot appear 
	in $\ccD$, for then the piece $Q_\tau$ would again not be admissible. 
	In other words, $W$ is hidden in one of the segments. But by~\ref{it:A2}
	and $t\ge 3$ this is only possible if $I_\tau$ is indeed of the desired form.  
\end{proof}

The following should now be straightforward, but let us elaborate.

\begin{proof}[Proof of Proposition~\ref{prop:2235}]
	Consider a pretrain $(F, \equiv^F)$ with $\ggth(F, \equiv^F)>g$ and a 
	number of colours $r$. Construct the linear system $\Phi_r(F)=(G, \ccG)$ with 
	strongly induced copies and enumerate the edges of $G$ arbitrarily 
	as $E(G)=\{e(1), \ldots, e(N)\}$. By Lemma~\ref{lem:2310} the picture zero 
	$(\Pi_0, \equiv_0, \ccP_0, \psi_0)$ corresponding to this situation satisfies 
	${\GTH(\Pi_0, \equiv_0, \ccP_0^+)>g}$. 
	
	Starting with picture zero we intend to execute a partite construction, i.e., to
	construct recursively a sequence of pictures 
	$(\Pi_\alpha, \equiv_\alpha, \ccP_\alpha, \psi_\alpha)_{\alpha\le N}$ 
	over $(G, \ccG)$.
	Suppose that for some $\alpha\in [N]$ we have just constructed a picture 
	$(\Pi_{\alpha-1}, \equiv_{\alpha-1}, \ccP_{\alpha-1}, \psi_{\alpha-1})$ with 
		\[
		\GTH(\Pi_{\alpha-1}, \equiv_{\alpha-1}, \ccP_{\alpha-1}^+)>g\,.
	\]
		In particular, we have 
	$\GTH\bigl(\Pi_{\alpha-1}, \equiv_{\alpha-1}, E^+(\Pi_{\alpha-1})\bigr)>g$
	and Lemma~\ref{lem:2310} reveals 
		\[
		\ggth(\Pi_{\alpha-1}, \equiv_{\alpha-1}) > g\,.
	\]
		
	This in turn implies that $\Xi_r(\cdot)$ applies to the 
	constituent 
	$\bigl(\Pi_{\alpha-1}^{e(\alpha)}, \equiv_{\alpha-1}^{e(\alpha)}\bigr)$, 
	where $\equiv_{\alpha-1}^{e(\alpha)}$ denotes the restriction 
	of $\equiv_{\alpha-1}$ 
	to $E\bigl(\Pi_{\alpha-1}^{e(\alpha)}\bigr)$. In other words, 
		\[
		\Xi_r\bigl(\Pi_{\alpha-1}^{e(\alpha)}, \equiv_{\alpha-1}^{e(\alpha)}\bigr)
		= 
		(H_\alpha, \equiv^{H_\alpha}, \ccH_\alpha)
	\]
	is defined and satisfies $\GTH(H_\alpha, \equiv^{H_\alpha}, \ccH_\alpha^+) >g$.
	By Lemma~\ref{lem:0052} the new picture
		\[
		   (\Pi_\alpha, \equiv_\alpha, \ccP_\alpha, \psi_\alpha)
		   =
		   (\Pi_{\alpha-1}, \equiv_{\alpha-1}, \ccP_{\alpha-1}\psi_{\alpha-1})
		   \conc
		   (H_\alpha, \equiv^{H_\alpha}, \ccH_\alpha)
	\]
		has the property $\GTH(\Pi_\alpha, \equiv_\alpha, \ccP_\alpha^+)>g$
	and thus the partite construction goes on. 
	
	Eventually we obtain the final picture 
	$\PC(\Phi, \Xi)_r(F, \equiv^F)=(\Pi_N, \equiv_N, \ccP_N)$ with
		\[
		\GTH(\Pi_N, \equiv_N, \ccP_N^+)>g\,. \qedhere
	\]
	\end{proof}

\subsection{Orientation}\label{subsec:9o} 
The main results of this section lead to a quick solution of the problem 
discussed in~\S\ref{subsec:7o}. The construction 
$\ups=\Ext(\Omega^{(2)}, \Omega^{(2)})$ applies to all linear, ordered, 
\mbox{$f$-partite} 
pretrains. Due to the case $g=2$ of Lemma~\ref{lem:1955} the systems of pretrains
$(H, \equiv^H, \ccH)$ generated by $\ups$ satisfy $\GTH(H, \equiv^H, \ccH^+)>2$
(the hypotheses of this lemma were verified in Proposition~\ref{prop:1738} 
and Corollary~\ref{cor:2201}).

Now we can use $\ups$ as a partite lemma and look at the 
construction $\Lambda=\PC(\Omega^{(2)}, \ups)$. The case $g=2$ 
of Proposition~\ref{prop:2235} informs us that for every linear, 
ordered, $f$-partite pretrain $(F, \equiv^F)$ and every number of 
colours $r$ the system of pretrains $\Lambda_r(F, \equiv^F)=(H, \equiv^H, \ccH)$
is defined and satisfies $\GTH(H, \equiv^H, \ccH^+)>2$. So, in particular, 
$(H, \equiv^H)$ is a linear pretrain arrowing $(F, \equiv^F)$ with $r$ colours. 
Moreover, and this is the main advantage of $\Lambda$ over $\ups$, the copies 
in $\ccH$ have clean intersections due to Lemma~\ref{lem:cleancap}.     \section{Trains}
\label{sec:trains}

Let us start this section with an observation of central importance to our
girth Ramsey theoretic endeavour: a ``typical'' $k$-partite $k$-uniform hypergraph 
can never occur as a constituent of a picture that arises when the partite 
construction method is carried out over a linear system~$(G, \ccG)$. 

For instance, all constituents of picture zero consist of a matching together with 
some isolated vertices (cf. Figure~\ref{fig:21}). The first picture is obtained by 
applying the intended partite lemma to the constituent over some edge $e(1)\in E(G)$ 
and performing the usual amalgamation. 
Evidently, every constituent of the new picture over an edge $e\in E(G)$ 
with $e(1)\cap e=\vn$ is still a matching augmented by some isolated vertices. 
More interestingly, if an edge $e\ne e(1)$ 
intersects $e(1)$, then by the linearity of $G$ this intersection occurs in a single 
vertex~$x$ and, accordingly,
the new constituent over $e$ is obtained by taking several disjoint copies of the old 
constituent and identifying a couple of vertices on the music line $V_x$. However, there
occur no such identifications on the other music lines $V_z$ with $z\in e\sm\{x\}$.
Notice that we can endow the new constituent with a pretrain structure 
by declaring two edges to be equivalent if they belong to the same copy of the old 
constituent. 
It should also be clear that this pretrain is going to be linear if our partite 
lemma delivers copies with clean intersections.

Similarly, the most complicated possibility for a constituent of the second picture
is that one takes plenty of disjoint copies of such a pretrain and identifies some 
of their vertices, all identified vertices belonging to a common music line. 
As the partite construction progresses, there iteratively arise more and more complex 
objects all of which fit into the framework we shall now develop (see also 
Figure~\ref{fig:N3}). 

\subsection{Quasitrains and parameters}
\label{sssec:3133}

We begin by encoding the recursive formation of our 
objects in terms of a sequence of nested equivalence relations.    
  
\begin{dfn} \label{dfn:2206}
	Given a positive integer $m$ we say that 
	$\seq{F}=(F, \equiv_0, \ldots, \equiv_m)$ is a {\it quasitrain of height $m$} 
	if $F$ is a hypergraph and $\equiv_0, \ldots, \equiv_m$ are equivalence relations
	on~$E(F)$ such that the following holds.  
	\index{quasitrain}
	\index{height}
	\begin{enumerate}[label=\rmlabel]
			\item\label{it:2206a} Every wagon of the pretrain $(F, \equiv_0)$ consists of 
					a single edge. 
			\item\label{it:2206b} If $e', e''\in E(F)$, $\mu\in [0, m)$, 
					and $e'\equiv_\mu e''$, then $e'\equiv_{\mu+1} e''$.
			\item\label{it:2206c} Any two edges of $F$ are equivalent with respect 
					to $\equiv_m$.
	\end{enumerate}
	For every $\mu\in [0, m]$ the wagons of the pretrain $(F, \equiv_\mu)$ 
	are called the {\it $\mu$-wagons} of $\seq{F}$.
	\index{$\mu$-wagon}
\end{dfn}

\begin{example}\label{exmp:0814}
	Quasitrains of height $1$ are triples $(F, \equiv_0, \equiv_1)$, where $F$
	is a hypergraph, the \mbox{$0$-wagons} are single edges 
	(by~\ref{it:2206a}),
	and all edges belong to the same $1$-wagon
	(by~\ref{it:2206c}). Consequently, there is a natural bijective correspondence  
	between quasitrains of height $1$ and hypergraphs. At a later moment it will be 
	convenient to call $(F, \equiv_0, \equiv_1)$ the quasitrain of height $1$ 
	{\it associated with} the hypergraph $F$. 
\end{example}

\begin{example}\label{exmp:3140}
	Similarly, quasitrains of height $2$ are quadruples 
	$(F, \equiv_0, \equiv_1, \equiv_2)$
	with the properties that $(F, \equiv_0, \equiv_2)$ is a quasitrain of height $1$
	and $(F, \equiv_1)$ is a pretrain. In particular, quasitrains of height $2$ 
	are in natural bijective correspondence with pretrains. 	
\end{example}

The train concept to which we proceed next is taking into account what 
we said earlier about amalgamations occurring only in prespecified music lines. 
Recall that in Definition~\ref{dfn:n38}\ref{it:n38a} we introduced the following
convenient piece of notation for such situations: given an $f$-partite 
hypergraph $F$ with index set $I$ and a subset $A\subseteq I$ the 
union $\bigcup_{i\in A}V_i(F)$ of the vertex classes 
whose indices belong to $A$ is denoted by $V_A(F)$.

\begin{dfn}\label{dfn:3140}    
	Suppose that $m\in\NN$, that $f\colon I\longrightarrow \NN$ is a function 
	from a finite index set $I$ to $\NN$ with $\sum_{i\in I} f(i)\ge 2$,
	and that $\seq{A}=(A_1, \ldots, A_m)\in \powerset(I)^{m}$ 
	is a sequence of subsets of $I$.
	A structure $\seq{F}=(F, \equiv_0, \ldots, \equiv_m)$
	is said to be an {\it $f$-partite train of height $m$ with parameter $\seq{A}$}
	(see Figure~\ref{fig:monday}) 
	if it has the following two properties.
	\index{train}
	\index{parameter}
	\begin{enumerate}[label=\rmlabel]
		\item\label{it:3140a} $\seq{F}$ is an $f$-partite quasitrain of height $m$.
		\item\label{it:3140b} If $\mu\in[m]$ and $W'_{\mu-1}$, $W''_{\mu-1}$ are
			two distinct $(\mu-1)$-wagons of $\seq{F}$ included in a common $\mu$-wagon,
			then 
			$V(W'_{\mu-1})\cap V(W''_{\mu-1})\subseteq V_{A_\mu}(F)$.
	\end{enumerate} 
\end{dfn}
	 
\begin{example}\label{exmp:2230}
	Trains of height $1$ have one-term sequences serving as their parameters.
	If $\seq{F}=(F, \equiv_0, \equiv_1)$ is an $f$-partite quasitrain of height $1$ 
	(cf. Example~\ref{exmp:0814}), $I$ denotes the domain of $f$, 
	and $A_1\subseteq I$, then $\seq{F}$ is a train with parameter $(A_1)$ 
	if and only if 
	$e'\cap e''\subseteq V_{A_1}(F)$ holds for any two distinct 
	edges $e', e''\in E(F)$, i.e., if $F$ is $A_1$-intersecting in the sense
	of Definition~\ref{dfn:n38}\ref{it:n38b}.
\end{example}

\begin{figure}
\centering

\begin{tikzpicture}[scale=1]

	\def\a{1.3}
	\def\b{0}
	\def\c{-1.3}
	
\coordinate (a1) at (-2.7,\a);
\coordinate (a2) at (-1.4,\a);
\coordinate (a3) at (1,\a);
\coordinate (a4) at (2.05,\a);
\coordinate (a5) at (2.75,\a);

\coordinate (b1) at (-3.1,\b);
\coordinate (b2) at (-2.3,\b);
\coordinate (b3) at (-1.4,\b);
\coordinate (b4) at (-.5,\b);
\coordinate (b5) at (.4,\b);
\coordinate (b6) at (1,\b);
\coordinate (b7) at (1.7,\b);
\coordinate (b8) at (2.4,\b);
\coordinate (b9) at (3.1,\b);

\coordinate (c1) at (-2.7,\c);
\coordinate (c2) at (-1.4,\c);
\coordinate (c3) at (1,\c);
\coordinate (c4) at (2.05,\c);
\coordinate (c5) at (2.75,\c);
	
\draw [thick] (-4,\a) -- (4,\a);
\draw [thick] (-4,\b)--(4,\b);
\draw [thick] (-4,\c)--(4,\c);

\redge{(b2)}{(c1)}{(b1)}{(a1)}{4pt}{.6pt}{green!70!black}{opacity=0};
\redge{(b3)}{(c2)}{(b2)}{(a2)}{4pt}{.6pt}{green!70!black}{opacity=0};
\redge{(b5)}{(c3)}{(b4)}{(a2)}{3pt}{.6pt}{blue!80!black}{opacity=0};
\redge{(b6)}{(a3)}{(b7)}{(c3)}{4pt}{.6pt}{blue!80!black}{opacity=0};
\redge{(b7)}{(a4)}{(b8)}{(c4)}{4pt}{.6pt}{red!70!black}{opacity=0};
\redge{(b8)}{(a5)}{(b9)}{(c5)}{4pt}{.6pt}{red!70!black}{opacity=0};

\foreach \i in {1,...,5} {
				\fill (a\i) circle (1.5pt);
				\fill (c\i) circle (1.5pt);
			}
\foreach \i in {1,...,9} 
				\fill (b\i) circle (1.5pt);

\node [green!70!black] at (-2.1,\a+.5) {$W'$};
\node [blue!80!black] at (0,\a+.5) {$W''$};
\node [red!70!black] at (2.5,\a+.5) {$W'''$};

\node at (-4.5,\a){$1$};
\node at (-4.5,\b){$2$};
\node at (-4.5,\c){$3$};

			\end{tikzpicture}
\caption{An $f$-partite train of height~$1$ with parameter $(A_1)$, 
where $f\colon [3]\longrightarrow\NN$ is defined by $f(1)=f(3)=1$, 
$f(2)=2$, and $A_1=\{1, 2\}$.}
\label{fig:monday}
\end{figure}
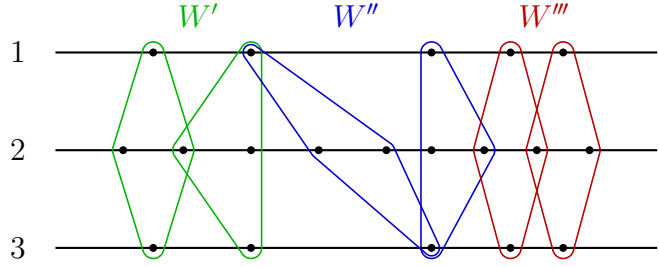
 
When compared to the discussion at the beginning of this section, 
there are two directions in which our definition of trains might 
appear to be too general. 
First, we only motivated the need for $k$-partite $k$-uniform 
trains, whereas Definition~\ref{dfn:3140} allows the underlying hypergraphs of trains
to be $f$-partite for general functions $f$. The reason for this is that we intend 
to subject trains to the extension process and, as we saw in Section~\ref{subsec:EP},
when we apply a construction of the form $\Ext(\Phi, \Psi)$ to a $k$-partite 
$k$-uniform pretrain, the construction $\Psi$ needs to be applicable to 
general $f$-partite hypergraphs. 

Another surprise in Definition~\ref{dfn:3140} might be that we allow the sets 
$A_1, \dots, A_m$ constituting the parameter of the train to be arbitrary subsets 
of $I$. 
After all, we have just tried to convey the idea that for trains created by partite 
constructions we can demand $|A_\mu|\le 1$ for every $\mu\in [m]$, and that this 
very fact will be responsible for the girth increment the achievement of which
is the reason for starting to consider trains at all. 
On the other hand, when we have some construction applicable to $f$-partite trains
and attempt to clean it by means of the partite construction method, then from the 
point of view of the constituents of the involved pictures the sizes of the 
sets in the parameter can appear to have been enlarged. 
This tension between conflicting demands on 
the parameters will completely resolve itself in Section~\ref{subsec:0136}.    
 
It is probably clear what we mean by subtrains, but let us elaborate for the sake of 
completeness.

\begin{dfn}\label{dfn:0033}
	Given two quasitrains 
		\[	
		\seq{F}=(F, \equiv_0^F, \ldots, \equiv_m^F)
		\quad \text{ and } \quad 
		\seq{G}=(G, \equiv_0^G, \ldots, \equiv_m^G)
	\]
		of the same height $m$ we say that $\seq{F}$ 
	is a {\it subquasitrain} of $\seq{G}$ if
	\index{subquasitrain} 
	\begin{enumerate}[label=\rmlabel]
		\item\label{it:0033a} $F$ is a subhypergraph of $G$
		\item\label{it:0033b} and $\forall e', e''\in E(F)\,\, \forall \mu\in [0, m] \,\,
			\bigl[e'\equiv^F_\mu e'' \,\,\, \Longleftrightarrow \,\,\, e'\equiv^G_\mu e''\bigr]$.
	\end{enumerate}
	If, moreover, both $\seq{F}$ and $\seq{G}$ are $f$-partite trains having the same 
	parameter $\seq{A}$ and~\ref{it:0033a} holds in the stronger form that 
	$F$ is an $f$-partite subhypergraph of $G$, then $\seq{F}$ is called 
	a {\it subtrain} of $\seq{G}$.
	\index{subtrain}
\end{dfn}   

The subtrain relation is, of course, reflexive and transitive. 
A {\it quasitrain system of height~$m$} is a structure of the form 
\[
	(H, \equiv_0^H, \ldots, \equiv_m^H, \ccH)\,,
\]
where $\seq{H}=(H, \equiv_0^H, \ldots, \equiv_m^H)$ is a quasitrain of height $m$  
and $\ccH$ is a collection of subquasitrains of $\seq{H}$. 
\index{quasitrain system}
It will be convenient to denote this system by $(\seq{H}, \ccH)$ as well. 
We regard the notion of an {\it extended quasitrain system} $(\seq{H}, \ccH^+)$ 
to be self-explanatory.

\subsection{More German girth}
\label{sssec:3135}

Next we need to generalise $\ggth$ to quasitrains, and $\GTH$ to 
systems of quasitrains. In both contexts we work with the free monoid~$\gM$ 
generated by $\NN_{\ge 2}$. 
Thus the elements of~$\gM$ can be thought of as sequences of elements 
from $\NN_{\ge 2}$ and the composition of $\gM$, denoted by~$\circ$, 
is the concatenation of sequences.
The {\it empty sequence} $\vn$ is the neutral element of $\gM$. 
For a sequence $\seq{g}=(g_1, \dots, g_m)\in\gM$ we call $m=|\seq{g}|$ 
the {\it length} of $\seq{g}$ and we use $\inf(\seq{g})$ as an abbreviation 
for $\inf\{g_1, \dots, g_m\}$. 
So we have $\inf(\varnothing)=\infty$ and in all other cases the infimum is just 
a minimum. 
For integers $g\ge 2$ and $m\ge 1$ the sequence $(g, \dots, g)$ consisting 
of $m$ terms equal to $g$ is the $m^\mathrm{th}$ power of the one-term sequence $(g)$ 
and thus we shall denote it by $(g)^m$. 
Finally, we write $\gM_\le$ for the subset of $\gM$ consisting of all nondecreasing 
sequences (including $\vn$), and we set $\gM^\times_\le=\gM_\le\sm\{\vn\}$.
   
A straightforward adaptation of Definition~\ref{dfn:2256} leads to the following 
$\ggth$ notion for quasitrains. 

\begin{dfn} \label{dfn:0027}
	Given a quasitrain $\seq{F}=(F, \equiv_0, \ldots, \equiv_m)$ 
	of height~$m$ and a sequence $\seq{g}=(g_1, \dots, g_m)\in \gM$ of length $m$ 
	we write $\ggth(\seq{F}) > \seq{g}$ if for every 
	$\mu\in [m]$ and every $\mu$-wagon $W_\mu$ we have 
		\[
		\ggth(W_\mu, \equiv_{\mu-1}^{W_\mu}) >g_\mu\,,
	\]
		where $\equiv^{W_\mu}_{\mu-1}$ denotes the restriction of $\equiv_{\mu-1}$
	to $E(W_\mu)$. A quasitrain $\seq{F}$ of height $m$ with $\ggth(\seq{F}) > (2)^m$ 
	is said to be {\it linear}.
	\index{$\ggth$}
	\index{linear quasitrain}
\end{dfn} 

The transitivity of ordinary girth leads to the following statement.

\begin{lemma}\label{lem:1557}
	Let $\seq{g}=(g_1, \dots, g_m)\in \gM^\times_\le$ be given. 
	If $\seq{F}=(F, \equiv_0^F, \ldots, \equiv_m^F)$
	denotes a quasitrain of height $m$ with $\ggth(\seq{F}) > \seq{g}$, then 
	$\ggth(F, \equiv_{\mu-1}) >g_\mu$
	holds for every $\mu\in [m]$. In particular, we have $\gth(F) >g_1$.  
\end{lemma}

\begin{proof}
	The case $\mu=m$ is clear, for by Definition~\ref{dfn:2206}\ref{it:2206c}
	there is only one $m$-wagon of~$\seq{F}$ and by Definition~\ref{dfn:0027}
	applied to that wagon we obtain indeed $\ggth(F, \equiv_{m-1}) >g_m$. 
	
	Arguing by decreasing induction 
	on $\mu$ we now assume that some $\mu\in[m-1]$ has the property 
	$\ggth(F, \equiv_{\mu}) > g_{\mu+1}\ge g_\mu$, meaning that the $\mu$-wagons 
	of $\seq{F}$ form a set system whose girth exceeds~$g_\mu$. 
	According to Definition~\ref{dfn:0027} each of these $\mu$-wagons is composed of 
	$(\mu-1)$-wagons the vertex sets of which form a set system whose girth 
	is larger than~$g_\mu$ as well. Due to Fact~\ref{fact:girth-trans} this 
	implies $\ggth(F, \equiv_{\mu-1}) > g_\mu$
	and the induction is complete. 
	
	As the $0$-wagons consist of single edges, the special case~$\mu=1$ 
	yields~$\gth(F) >g_1$. 
\end{proof}

The reason why we are keeping track of suitable parameters when discussing trains 
is that they are needed for the following important yet innocent looking variant of 
this argument.   
  
\begin{lemma} \label{lem:0036}
	Let $\seq{F}=(F, \equiv_0, \ldots, \equiv_m)$ be a $k$-partite 
	$k$-uniform train whose parameter $\seq{B}=(B_1, \dots, B_m)$ 
	satisfies $|B_\mu|\le 1$ for every $\mu\in [m]$. 
	If $\ggth(\seq{F}) > (g)^m$ holds for some integer $g\ge 2$, 
	then $\gth(F) > g+1$.
\end{lemma} 

\begin{proof}
	We are to prove that the girth of the unique $m$-wagon of $F$ exceeds $g+1$.
	Arguing indirectly, we let $\mu\in [0, m]$ be minimal with the property that 
	for some $\mu$-wagon~$W_\mu$ of $\seq{F}$ the statement $\gth(W_\mu)>g+1$ fails.
	
	As the $0$-wagons of $\seq{F}$ are single edges, we have $\mu>0$. Therefore $W_\mu$ 
	is comprised of $(\mu-1)$-wagons and by the minimality of $\mu$ the girth of each of 
	them is larger than $g+1$. According to Definition~\ref{dfn:0027} the vertex sets 
	of these $(\mu-1)$-wagons form a set system whose girth is larger than $g$. 
	Due to Fact~\ref{fact:girth-trans} it follows that $W_\mu$ contains a $(g+1)$-cycle
	$e_1v_1\ldots e_{g+1}v_{g+1}$ and 
	distinct $(\mu-1)$-wagons $W^1_{\mu-1}, \ldots, W^{g+1}_{\mu-1}$
	with $e_i\subseteq V(W^{i}_{\mu-1})$ for every $i\in \ZZ/(g+1)\ZZ$.
	Now clause~\ref{it:3140b} of Definition~\ref{dfn:3140} 
	implies $v_1, \ldots, v_{g+1}\in V_{B_\mu}$. Together with $|B_\mu|\le 1$ 
	this shows, in particular, that $v_1$, $v_{g+1}$ are in the same vertex 
	class of $F$. On the other hand, these two vertices are distinct and belong 
	to the edge~$e_1$. We have thereby reached a contradiction to the assumption 
	that $F$ be $k$-partite and $k$-uniform for some $k\ge 2$, and the proof is complete. 
\end{proof}
   
\begin{dfn}\label{dfn:0115}
	Given an extended quasitrain 
	system $(H, \equiv_0, \ldots, \equiv_m, \ccH^+)=(\seq{H}, \ccH^+)$ of height $m$,
	and a sequence $\seq{g}=(g_\ell, \dots, g_m)\in\gM$ whose length is at most $m$ 
	we write $\GTH(\seq{H}, \ccH^+)>\seq{g}$
	if \index{$\GTH$}
		\begin{enumerate}
		\item[$\bullet$]
			for every $\mu\in [\ell, m]$ we have $\GTH(H, \equiv_{\mu-1}, \ccH^+)>g_\mu$,
		\item[$\bullet$] and $\GTH(H, \equiv_m, \ccH^+)>1$.
	\end{enumerate} 
	\end{dfn} 

In practice we will always have $\ell\in\{1, 2\}$ when working with this definition. 
The case $\ell=1$ is certainly more natural, as it corresponds to having a demand on
$\GTH(H, \equiv_{\mu-1}, \ccH^+)$ for every possible value of $\mu$. 
If a train system $(\seq{H}, \ccH)$ is created by means of the extension process, 
however, there will be copies intersecting in entire $1$-wagons 
(see~\S\ref{subsec:0135}) and thus there is nothing interesting that could be said 
about $\GTH(H, \equiv_0, \ccH^+)$.
This lack of information is the reason why we sometimes have to resort to the 
case $\ell=2$. 
Fortunately, the partite construction method allows us to regain control 
over $\GTH(H, \equiv_0, \ccH^+)$ and thus we can prevent this deficit from 
spreading (see~\S\ref{subsec:ngr}). So we never need to deal with the possibility 
$\ell>2$. 
Let us end this discussion by recording a direct consequence of Lemma~\ref{lem:2310}.

\begin{cor}\label{cor:1921}
	If a sequence $\seq{g}\in\gM$ and a quasitrain 
	$\seq{H} = (H, \equiv_0, \ldots, \equiv_m)$ 
	of height $m=|\seq{g}|$
	satisfy $\GTH\bigl(\seq{H}, E^+(H)\bigr) > \seq{g}$,
	then $\ggth(\seq{H})>\seq{g}$ holds as well. \qed
\end{cor}

\subsection{Diamonds}
\label{sssec:Karo}

We shall now give a precise description of the induction scheme the proof of 
the girth Ramsey theorem is based on. The central statement concerns 
trains of large $\ggth$ and reads as follows. 
     
\begin{dfn}\label{dfn:0121}
	For every sequence $\seq{g}\in \gM_\le^\times$ of length $m$ the 
	Ramsey theoretic principle~$\karo_{\seq{g}}$
	asserts that there exists a construction $\Psi^{\seq{g}}$ associating with 
	every ordered $f$-partite train 
	\[
		\seq{F}=(F, \equiv^F_0, \ldots, \equiv^F_m)
	\]
	of height $m$ satisfying $\ggth(\seq{F})>\seq{g}$ and with every number of 
	colours $r$ an ordered $f$-partite train system 
	\[
		\Psi^{\seq{g}}_r(\seq{F})
		=
		(H, \equiv^H_0, \ldots, \equiv^H_m, \ccH)
		=
		(\seq{H}, \ccH)
	\]
	with the same parameter as $\seq{F}$ such that the copies in $\ccH$ are strongly 
	induced, \index{$\karo$}
	\[
		\ccH\longrightarrow (\seq{F})_r\,,
		\quad \text{ and } \quad 
		\GTH(\seq{H}, \ccH^+)>\seq{g}\,. 
	\]
\end{dfn}

Essentially we shall prove all these principles ``by induction on $\seq{g}$''.
Up to some changes in the language, the base case has already been analysed. 

\begin{lemma}\label{lem:0120}
	The principle $\karo_{(2)}$ holds. 
\end{lemma}

\begin{proof}
	Let an ordered $f$-partite train $(F, \equiv^F_0, \equiv^F_1)$ of height~$1$
	with parameter~$(A)$ and a number of colours~$r$ be given. For reasons of 
	simplicity we assume first that $F$ has no isolated vertices.
	
	Now we form the hypergraph system $\Omega^{(2)}_r(F)=(H, \ccH)$, 
	and consider the quasitrain $\seq{H}=(H, \equiv^H_0, \equiv^H_1)$
	associated with~$H$ (as in Example~\ref{exmp:0814}). 
	Corollary~\ref{cor:2201} informs us that~$H$ is an ordered $f$-partite 
	hypergraph, $\ccH\lra(F)_r$, and $\Gth(H, \ccH^+)>2$.
	Since $F$ is $A$-intersecting (cf.\ Example~\ref{exmp:2230}),
	Lemma~\ref{lem:n383} shows that $H$ is $A$-intersecting as well 
	or, in other words, that $\seq{H}$ is a train with parameter $(A)$.
	\index{$A$-intersecting hypergraph}
	\index{$\Omega^{(2)}$}
 
	By Proposition~\ref{prop:1738} the copies in $\ccH$ are strongly induced in $H$.
	Moreover, Lemma~\ref{lem:0217} yields $\GTH(H, \equiv^H_0, \ccH^+)>2$.
	Since $F$ has no isolated vertices, Lemma~\ref{lem:GTH1} implies
	$\GTH(H, \equiv^H_1, \ccH^+)>1$.  
	According to Definition~\ref{dfn:0115} this shows 
   ${\GTH(\seq{H}, \ccH^+)>(2)}$ and altogether the system 
	$(\seq{H}, \ccH)$ is as demanded by~$\karo_{(2)}$.
	This concludes our discussion of the case that $F$ has no isolated vertices. 
		
	In the general case we remove the isolated vertices from $F$, 
	perform the construction we have just seen, and then we put the isolated 
	vertices back. 
	This needs to be done in such a way that each copy in $\ccH$ ends up getting 
	its own isolated vertices, so that $\GTH(H, \equiv^H_1, \ccH^+)>1$ remains valid.  
\end{proof}

There will be two quite distinct kinds of induction steps for 
traversing~$\gM_\le^\times$.
First, the extension process together with some subsequent cleaning steps
will allow us to prove the following implication.
\begin{prop}\label{prop:0142}
	For every nonempty nondecreasing sequence $\seq{g}\in\gM_\le^\times$ the 
	principle~$\karo_{\seq{g}}$ implies $\karo_{(2)\circ\seq{g}}$.
\end{prop}

Second, we shall later describe a diagonal variant of the 
partite construction method and establish the following result.

\begin{prop}\label{prop:0139}
	Suppose that $\seq{g}\in\gM_\le$ is a (possibly empty) nondecreasing sequence
	and that $g$ is an integer satisfying $\inf(\seq{g})>g\ge 2$. If for every 
	positive integer $m$ the principle $\karo_{(g)^m\circ\seq{g}}$ holds, then 
	$\karo_{(g+1)\circ\seq{g}}$ is likewise valid.    
\end{prop}

Throughout the remainder of this subsection, we assume that these two 
propositions are true and explore some of their consequences.
In particular, we show that they really yield all karo principles. 

\begin{lemma}
	For every sequence $\seq{g}\in\gM_\le^\times$ and every 
	integer $g$ such that $\inf(\seq{g})\ge g\ge 2$ the principle 
	$\karo_{\seq{g}}$ implies $\karo_{(g)\circ \seq{g}}$.
\end{lemma}

\begin{proof}
	We argue by induction on $g$. Proposition~\ref{prop:0142} provides
	the base case $g=2$. Now suppose that the lemma 
	holds for some integer $g\ge 2$ and that some nonempty nondecreasing 
	sequence $\seq{g}\in\gM_\le^\times$ satisfying $\karo_{\seq{g}}$ 
	and $\inf(\seq{g})\ge g+1$ is given. 
	Iterative applications of the induction hypothesis 
	establish $\karo_{(g)^m\circ\seq{g}}$ for every positive integer $m$
	and thus Proposition~\ref{prop:0139} yields the desired principle 
	$\karo_{(g+1)\circ\seq{g}}$.
\end{proof}

\begin{lemma}\label{lem:karog}
	For every integer $g\ge 2$ the principle $\karo_{(g)}$ holds.
\end{lemma}

\begin{proof}
	Referring to Lemma~\ref{lem:0120} as a base case we argue by induction on $g$. 
	In the induction step we suppose that $g\ge 2$ denotes an integer such 
	that $\karo_{(g)}$ is true. Repeated applications of the previous lemma disclose 
	$\karo_{(g)^m}$ for every positive integer $m$ and appealing once more to 
	Proposition~\ref{prop:0139}, this time with the empty sequence, 
	we infer $\karo_{(g+1)}$.	  
\end{proof}

It should be clear that the two foregoing lemmata suffice for proving the following 
statement by induction on $|\seq{g}|$.

\begin{cor}\label{cor:1412}
	For every sequence $\seq{g}\in\gM_\le^\times$ the principle 
	$\karo_{\seq{g}}$ is valid. \qed
\end{cor}

Moreover, when unravelling the meaning of Lemma~\ref{lem:karog} one arrives 
at a fairly strong form of the girth Ramsey theorem first announced
in \S\ref{sssec:9902}. 
 
\begin{thm}\label{thm:6653}
	For every integer $g\ge 2$ there exists a Ramsey construction 
	$\Omega^{(g)}$ that given an ordered $f$-partite hypergraph $F$ 
	with $\gth(F)>g$ and a number of colours~$r$ produces an ordered 
	$f$-partite system $\Omega^{(g)}_r(F)=(H, \ccH)$ satisfying
	$\ccH\lra (F)_r$ and $\Gth(H, \ccH^+)>g$. 
\end{thm}   

\begin{proof}
	Let $\seq{F}=(F, \equiv^F_0, \equiv^F_1)$ be the ordered $f$-partite 
	quasitrain associated to $F$ (as in Example~\ref{exmp:0814}).
	We can view $\seq{F}$ as a train with parameter $(I)$, where $I$ denotes 
	the domain of $f$. Since $\ggth(\seq{F})>(g)$, the principle $\karo_{(g)}$ 
	yields a train system $(\seq{H}, \ccH)$ such that $\ccH\longrightarrow(\seq{F})_r$
	and $\GTH(\seq{H}, \ccH^+)>(g)$. Writing $\seq{H}=(H, \equiv^H_0, \equiv^H_1)$
	it remains to observe that $\GTH(H, \equiv^H_0, \ccH^+)>g$ 
	implies $\Gth(H, \ccH^+)>g$ (see Lemma~\ref{lem:0217}).
\end{proof}

In analogy to Lemma~\ref{lem:n383} we could add that the construction 
$\Omega^{(g)}$ also preserves being $A$-intersecting, but due to a lack 
of known applications we do not state this more carefully here. 
Let us finally recall that every $k$-uniform hypergraph $F$ can be viewed 
as an $f$-partite hypergraph, where $f$ denotes any function from a one-element 
set to $\{k\}$. Therefore, Theorem~\ref{thm:6653} immediately yields a 
Ramsey theorem for arbitrary hypergraphs without short cycles. As the 
conclusion $\Gth(H, \ccH^+)>g$ implies $\gth(H)>g$ (see Lemma~\ref{lem:Gth-gth})
we thereby see that Theorem~\ref{thm:6653} yields Theorem~\ref{thm:grth1}.
Let us collect all these implications into a single statement. 

\begin{summary}\label{sum:1020}
	The conjunction of Proposition~\ref{prop:0142} and Proposition~\ref{prop:0139}
	leads to Corollary~\ref{cor:1412}, Theorem~\ref{thm:6653}, and to the girth 
	Ramsey theorem as formulated in Theorem~\ref{thm:grth1}.
\end{summary}
 
\begin{remark}
	With the help of ordinal numbers our induction scheme can be reformulated 
	in a, perhaps, more transparent way. 
	The map $\phi\colon \gM_{\le}\lra \omega^\omega$ defined by 
		\[
		\phi(g_1, \dots, g_m)=\omega^{g_m-2}+\dots+\omega^{g_1-2}
	\]
		is bijective. Thus Corollary~\ref{cor:1412} asserts 
	that $\karo_{\phi^{-1}(\alpha)}$ holds for every positive 
	ordinal~$\alpha$ beneath~$\omega^\omega$. The proof is by induction 
	on $\alpha$, its base case $\alpha=1$ agrees with Lemma~\ref{lem:0120},
	and Proposition~\ref{prop:0142} takes care of the successor step 
	$\karo_{\phi^{-1}(\alpha)}\Longrightarrow\karo_{\phi^{-1}(\alpha+1)}$.
	Finally, Proposition~\ref{prop:0139} provides the limit step, for in the 
	notation employed there
	the sequence $\langle \phi((g)^m\circ\seq{g})\colon m<\omega\rangle$
	converges to the limit number $\phi((g+1)\circ\seq{g})$.   
\end{remark} \section{Trains in the extension process}
\label{sec:TUD}

In this section we take the first major step towards proving 
Proposition~\ref{prop:0142}. Notice that we are given there 
a construction applicable to certain trains of height $m=|\seq{g}|$,
and we are to exhibit another construction capable of handling certain
trains of height $m+1$. The plan for accomplishing this height increment 
is to use the extension process; the main result of this section describes
how far we can go with this idea (see Lemma~\ref{lem:1904}). The main deficit 
of the construction we obtain is that the copies it produces can still intersect 
in entire $1$-wagons. 
Following the arguments we saw in~\S\ref{subsec:9o} this can be remedied 
by means of the partite construction method, but we defer the details of this 
cleaning step to~\S\ref{subsec:ngr}.

\subsection{Extensions of trains}
\label{sssec:3136}

The extension process discussed in Section~\ref{subsec:EP} generalises
from pretrains to quasitrains of arbitrary height and our next immediate goal 
is to develop an appropriate language for the description of such constructions.    
In other words, we repeat~\S\ref{sssec:1448} in a more intricate setting. 

\begin{dfn}\label{dfn:4127}
	Given a subquasitrain $\seq{F}=(F, \equiv^F_0, \ldots, \equiv^F_m)$ of 
	a quasitrain 
		\[
		\seq{H}=(H, \equiv^H_0, \ldots, \equiv^H_m)
	\]
		we call $\seq{H}$ a {\it $1$-extension} of $\seq{F}$ provided
	that $(H, \equiv^H_1)$ is an extension of $(F, \equiv^F_1)$.
\end{dfn}

Let us recall that the latter notion was introduced in Definition~\ref{dfn:1811}
and in the present situation it means that the following three conditions hold.  
\begin{enumerate}[label=\rmlabel]
	\item\label{it:4127a} The hypergraphs $F$ and $H$ have the same isolated vertices.
	\item\label{it:4127b} Every $1$-wagon $W$ of $\seq{H}$ contracts to a $1$-wagon 
			of $\seq{F}$, i.e., satisfies $E(W)\cap E(F)\ne\varnothing$.
	\item\label{it:4127c} If two distinct $1$-wagons $W_1^\star$ and $W_1^{\star\star}$ 
			of $\seq{H}$ contract to the $1$-wagons $\overline{W}_1^\star$ 
			and $\overline{W}_1^{\star\star}$ of~$\seq{F}$, then 
			$V(W_1^\star)\cap V(W_1^{\star\star})
			=V(\overline{W}_1^\star)\cap V(\overline{W}_1^{\star\star})$.
\end{enumerate} 

The next lemma expresses the intuitively obvious fact that there is a natural 
bijective correspondence between the $1$-extensions 
of a quasitrain $\seq{F}=(F, \equiv^F_0, \ldots, \equiv^F_m)$ and the extensions 
of the pretrain $(F, \equiv^F_1)$.
 
\begin{lemma} \label{lem:1455}
	Let $\seq{F}=(F, \equiv^F_0, \ldots, \equiv^F_m)$ be a quasitrain. 
	If the pretrain~${(H, \equiv^H_1)}$ is an extension of $(F, \equiv^F_1)$, 
	then there exists a unique quasitrain structure   
		\[
		\seq{H}=(H, \equiv^H_0, \ldots, \equiv^H_m)
	\]
		such that $\seq{F}$ is a subquasitrain of $\seq{H}$. 
\end{lemma}

Notice that in this situation $\seq{H}$ is a $1$-extension of $\seq{F}$.  

\begin{proof}[Proof of Lemma~\ref{lem:1455}]
	Dealing with the uniqueness first, we consider any such quasitrain
	${\seq{H}=(H, \equiv^H_0, \ldots, \equiv^H_m)}$. 
	Definition~\ref{dfn:2206}\ref{it:2206a} determines the relation 
	$\equiv^H_0$. For every $\mu\in [2, m]$ we contend that 
		\begin{equation}\label{eq:1755}
	\forall e', e''\in E(H)\,\,\,\bigl[e'\equiv^H_\mu e'' \,\,\, \Longleftrightarrow \,\,\,
		\exists e_\star, e_{\star\star} \in E(F)\,\,\,\, e'\equiv^H_1 e_\star\equiv^F_\mu 
		e_{\star\star} \equiv^H_1 e''\bigr]\,.
	\end{equation}
		
	Indeed, let $\mu\in [2, m]$ and $ e', e''\in E(H)$. For the backwards implication 
	we just need to observe that 
	because of Definition~\ref{dfn:2206}\ref{it:2206b}
	and Definition~\ref{dfn:0033}\ref{it:0033b} the formula 
	$e'\equiv^H_1 e_\star\equiv^F_\mu e_{\star\star} \equiv^H_1 e''$
	yields $e'\equiv^H_\mu e_\star\equiv^H_\mu e_{\star\star} \equiv^H_\mu e''$,
	whence $e'\equiv^H_\mu e''$.
	
	For the forwards direction, we suppose $e'\equiv^H_\mu e''$ and take two edges 
	$e_\star, e_{\star\star}\in E(F)$ which are in the same $1$-wagons of $H$ 
	as $e'$, $e''$, respectively. Since $(H, \equiv^H_1)$ is an extension 
	of $(F, \equiv^F_1)$,
	Definition~\ref{dfn:1811}\ref{it:1811b} informs us that such edges do indeed exist. 
	Now $e_\star \equiv^H_1 e' \equiv^H_\mu e'' \equiv^H_1 e_{\star\star}$ 
	implies $e_\star\equiv^H_\mu e_{\star\star}$. Due to the fact that $\seq{F}$ is a 
	subquasitrain of $\seq{H}$ we can conclude $e_\star\equiv^F_\mu e_{\star\star}$, 
	and altogether the edges $e_\star$, $e_{\star\star}$ are as desired. 
	This proves~\eqref{eq:1755} and the uniqueness of $\seq{H}$ follows. 
	
	Addressing the claim on the existence of $\seq{H}$ we define $\equiv^H_0$ 
	as demanded by Definition~\ref{dfn:2206}\ref{it:2206a}   
	and for $\mu\in [2, m]$ we define $\equiv^H_\mu$ by~\eqref{eq:1755}. 
	We omit the easy proof that these relations are indeed equivalence relations
	satisfying the clauses~\ref{it:2206b} and~\ref{it:2206c} of Definition~\ref{dfn:2206}.
\end{proof}

The following special cases are going to be important soon. 

\begin{example}\label{exmp:2339}
	Suppose that $\seq{F}=(F, \equiv^F_0, \ldots, \equiv^F_m)$ is an arbitrary 
	ordered quasitrain of height $m$ and that $(H, \equiv^H_1)$ denotes the 
	ordered pretrain obtained from $(F, \equiv^F_1)$ by wagon assimilation 
	as discussed in Example~\ref{exmp:2345}. 
	\index{wagon assimilation}
	Lemma~\ref{lem:1455} leads us to a (unique) ordered quasitrain 
	$\seq{H}=(H, \equiv^H_0, \ldots, \equiv^H_m)$
	of height $m$ which is a $1$-extension of $\seq{F}$. It will be convenient 
	to say that $\seq{H}$ arises from $\seq{F}$ by {\it assimilation of its $1$-wagons}.  
\end{example}

\begin{example}\label{exmp:2340} 
	Consider an ordered hypergraph pair $(X, W)$ such that neither $X$ nor~$W$ 
	has isolated vertices. Given an ordered quasitrain 
	$\seq{F}=(F, \equiv^F_0, \ldots, \equiv^F_m)$ all of whose $1$-wagons are 
	order-isomorphic to $W$ we can form 
	the ordered pretrain 
		\[
		(G, \equiv^G_1)=(F, \equiv^F_1)\ltimes (X, W)
	\]
		as in 
	Definition~\ref{dfn:1529}. 
	Since $(G, \equiv^G_1)$ is an extension 
	of $(F, \equiv^F_1)$, Lemma~\ref{lem:1455} gives rise to a (unique)
	$1$-extension $\seq{G}=(G, \equiv^G_0, \ldots, \equiv^G_m)$ of $\seq{F}$.
	We shall write $\seq{G}=\seq{F}\ltimes (X, W)$ for this construction (see 
	Figure~\ref{fig:FWX}).  
	\index{$(F, \equiv^F)\ltimes (X, W)$}
\end{example}

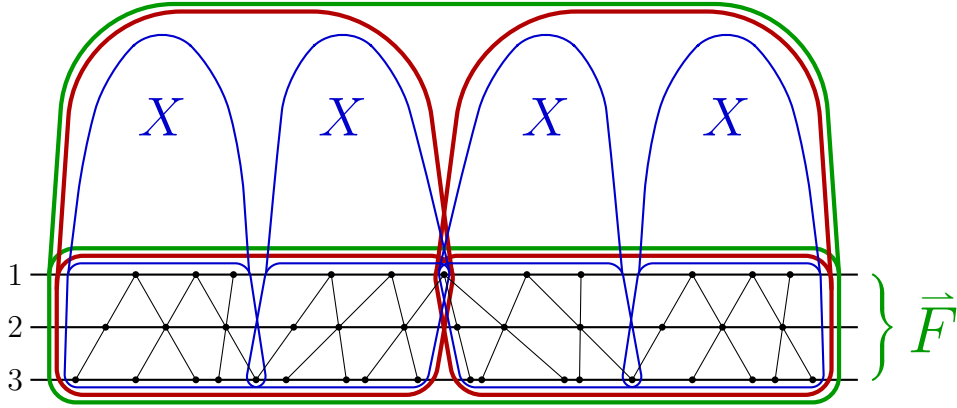
\begin{figure}
\centering
\begin{tikzpicture}[scale=1]
	
	\def\a{1.4}
	\def\b{.7}
	\def\c{0}

\draw [thick] (-5.5,\a) -- (5.5,\a);
\draw [thick] (-5.5,\b)--(5.5,\b);
\draw [thick] (-5.5,\c)--(5.5,\c);

\draw [rounded corners=10, green!60!black,ultra thick] (-5.25, \a+.35)--(-5.25,\c-.3)--(5.25,\c-.3)--(5.25,\a+.35)--cycle;
\draw [rounded corners=40, green!60!black, ultra thick] (-5.25,\a) -- (-5, 5)--(5,5)--(5.25,\a);

\draw [rounded corners=10, red!70!black, ultra thick] (-5.15, \a+.25)--(-5.15,\c-.2)--(-.15,\c-.2)--(.17,\a+.25)--cycle;
\draw [rounded corners=40, red!70!black, ultra thick] (-5.15,\a-.2) -- (-4.9, 4.9)--(-.35,4.9)--(.1,\a+.05);

\draw [rounded corners=10, red!70!black, ultra thick] (5.15, \a+.25)--(5.15,\c-.2)--(.15,\c-.2)--(-.17,\a+.25)--cycle;
\draw [rounded corners=40, red!70!black, ultra thick] (5.15,\a-.2) -- (4.9, 4.9)--(.35,4.9)--(-.1,\a+.05);

\draw [rounded corners=5, blue!80!black, thick] (-5.05,\c-.1)--(-5,\a+.15)--(-2.6,\a+.15)--(-2.35,\c-.1)--cycle;
\draw [rounded corners=15, blue!80!black, thick] (-5, \a) -- (-4.8, 3)-- (-4.5,4.1)--(-3.75,4.8)--(-3,4.1)--(-2.7,3) -- (-2.59, \a);

\draw [rounded corners=5, blue!80!black, thick] (-2.65,\c-.1)--(-2.35,\a+.15)--(.1,\a+.15)--(-.24,\c-.1)--cycle;
\draw [rounded corners=15, blue!80!black, thick] (-2.37, \a) -- (-2.22, 3)-- (-2.1,4.1)--(-1.35,4.8)--(-.6,4.1)--(-.3,3) -- (0.07, \a+.05);

\draw [rounded corners=5, blue!80!black, thick] (2.65,\c-.1)--(2.35,\a+.15)--(-.1,\a+.15)--(.24,\c-.1)--cycle;
\draw [rounded corners=15, blue!80!black, thick] (2.37, \a) -- (2.22, 3)-- (2.1,4.1)--(1.35,4.8)--(.6,4.1)--(.3,3) -- (-0.07, \a+.05);

\draw [rounded corners=5, blue!80!black, thick] (5.05,\c-.1)--(5,\a+.15)--(2.6,\a+.15)--(2.35,\c-.1)--cycle;
\draw [rounded corners=15, blue!80!black, thick] (5, \a) -- (4.8, 3)-- (4.5,4.1)--(3.75,4.8)--(3,4.1)--(2.7,3) -- (2.59, \a);

\draw (-4.9,\c)--(-4.1,\a)--(-3.3,\c);
\draw (-4.1, \c)--(-3.3,\a)--(-2.5,\c)--(-1.5,\a)--(-1.3,\c);
\draw (-3,\c)--(-2.8,\a);
\draw (-2.1,\c)--(-.7,\a)--(-.35,\c);
\draw (-1.05,\c)--(0,\a)--(.35,\c);

\draw (0,\a)--(1.6,\c);
\draw (1.8,\c)--(1.82,\a);

\draw (4.9,\c)--(4.1,\a)--(3.3,\c);
\draw (4.1, \c)--(3.3,\a)--(2.5,\c)--(1.1,\a)--(.5,\c);
\draw (4.4,\c)--(4.6,\a);

\foreach \x in {-4.1, -3.3,-2.8, -1.5,-.7, 0, 1.1, 1.82, 3.3, 4.1, 4.6} 
				\fill (\x,\a) circle (1.3pt);
\foreach \x in {-4.5, -3.7,-2.9, -2,-1.4, -.53, .175, .8, 1.81, 2.9, 3.7, 4.5} 
				\fill (\x,\b) circle (1.3pt);
\foreach \x in {-4.9, -4.1, -3.3,-3, -2.5,-2.1, -1.3,-1.05,-.35, .35,.5,1.6, 1.8, 2.5, 3.3, 4.1,4.4, 4.9} 
				\fill (\x,\c) circle (1.3pt);

\foreach \x in {-3.75, -1.35, 3.75, 1.35}
		\node [blue!80!black] at (\x,3.5) {\LARGE $X$};

\node at (-5.7,\a) {$1$};
\node at (-5.7,\b) {$2$};
\node at (-5.7,\c) {$3$};

\node [green!60!black] at (6.5,\b) {\Huge $\seq{F}$};
\draw [pen colour={green!60!black},decorate, decoration = {calligraphic brace, amplitude = 7pt}, ultra thick] (5.7,\a) --  (5.7,\c);
\end{tikzpicture}

\caption{The train $\harpp{F}\ltimes (X, W)$, where $\harpp{F}$ is $3$-partite and has
the parameter $(\{1, 2\}, \{3\}, \{1\})$. For space reasons, the $3$-partite structure 
of $X$ is not shown in this figure.}
\label{fig:FWX}
\end{figure} 
\begin{lemma}\label{lem:4015}
	If the quasitrain $\seq{H}=(H, \equiv^H_0, \ldots, \equiv^H_m)$ of height $m$
	is a $1$-extension of the quasitrain $\seq{F}=(F, \equiv^F_0, \ldots, \equiv^F_m)$,
	then for every $\mu\in [m]$ the pretrain $(H, \equiv^H_\mu)$ is an extension 
	of $(F, \equiv^F_\mu)$. 
\end{lemma}

\begin{proof}
	Given $\mu\in [m]$ we check the three conditions in Definition~\ref{dfn:1811}. 
	There is no problem with~\ref{it:1811a}. For~\ref{it:1811b} we consider 
	an arbitrary $\mu$-wagon $W_\mu$ of $\seq{H}$ and take an edge $e\in E(W_\mu)$.
	Since the $1$-wagon of $e$ contracts to $\seq{F}$, there is an edge $e'\in E(F)$
	with $e'\equiv^H_1 e$. In particular, we have $e'\in E(W_\mu)\cap E(F)$, meaning 
	that $W_\mu$ indeed contracts. 
	
	Proceeding with~\ref{it:1811c} we let $\overline{W}'_\mu$, $\overline{W}''_\mu$
	be the contractions of two distinct $\mu$-wagons~$W'_\mu$,~$W''_\mu$ of $\seq{H}$.
	Now 
	$V(\overline{W}'_\mu)\cap V(\overline{W}''_\mu)\subseteq V(W'_\mu)\cap V(W''_\mu)$ 
	is clear and for the reverse inclusion we look at an arbitrary vertex 
	$x\in V(W'_\mu)\cap V(W''_\mu)$. Pick two edges $e'\in E(W'_\mu)$, $e''\in E(W''_\mu)$ 
	with $x\in e'\cap e''$
	and denote the $1$-wagons of $\seq{H}$ these edges belong to by $W'_1$, $W''_1$, 
	respectively. Since $\seq{H}$ is a $1$-extension of $\seq{F}$, the contractions 
	$\overline{W}'_1$, $\overline{W}''_1$ of these wagons satisfy 
		\[
		x\in V(W'_1)\cap V(W''_1)=V(\overline{W}'_1)\cap V(\overline{W}''_1)
	\]
		and, consequently, there exist edges $e_\star\in E(\overline{W}'_1)$,
	$e_{\star\star}\in E(\overline{W}''_1)$ with $x\in e_\star\cap e_{\star\star}$.
	As~$e'$ and~$e_\star$ are in the same $1$-wagon of $\seq{H}$, they are also in the 
	same $\mu$-wagon, for which reason 
	$e_\star\in E(W'_\mu)\cap E(F)=E(\overline{W}'_\mu)$.
	Similarly we have $e_{\star\star}\in E(\overline{W}''_\mu)$
	and altogether $x\in V(\overline{W}'_\mu)\cap V(\overline{W}''_\mu)$
	follows.   
\end{proof}

In general, $1$-extensions of trains are only known to be quasitrains, but they 
may fail to be trains. This is for the reason that two edges of an enlarged $1$-wagon 
might intersect each other in a 'wrong' vertex class. The lemma that follows shows that, 
actually, this is the only obstacle.    

\begin{lemma}\label{lem:4507}
	Suppose that $\seq{F}=(F, \equiv^F_0, \ldots, \equiv^F_m)$ is an $f$-partite 
	train of height $m$ with parameter $\seq{A}=(A_1, \ldots, A_m)$ 
	and that the $f$-partite quasitrain $\seq{H}=(H, \equiv^H_0, \ldots, \equiv^H_m)$
	is a \mbox{$1$-extension} of $\seq{F}$. 
	If every $1$-wagon $X$ of $\seq{H}$ has the property 
	that $e'\cap e''\subseteq V_{A_1}(H)$ holds for any two distinct 
	edges $e', e''\in E(X)$, then $\seq{H}$ is a train with parameter $\seq{A}$.
\end{lemma}

\begin{proof}
	We need to check that Definition~\ref{dfn:3140}\ref{it:3140b} holds for $\seq{H}$. 
	For $\mu=1$ this was stated as a hypothesis, so suppose $\mu\in [2, m]$ from now on. 
	Let $W'_{\mu-1}$ and $W''_{\mu-1}$ be two distinct $(\mu-1)$-wagons of $\seq{H}$
	contained in the same $\mu$-wagon and denote their contractions to $\seq{F}$
	by $\overline{W}'_{\mu-1}$, $\overline{W}''_{\mu-1}$. Lemma~\ref{lem:4015}
	informs us that $(H, \equiv^H_{\mu-1})$ is an extension of $(F, \equiv^F_{\mu-1})$
	and thus we have indeed 
		\[
		V(W'_{\mu-1})\cap V(W''_{\mu-1})
		=
		V(\overline{W}'_{\mu-1})\cap V(\overline{W}''_{\mu-1})
		\subseteq
		V_{A_\mu}(F)
		\subseteq
		V_{A_\mu}(H)\,. \qedhere
	\]
	\end{proof}

Taking a $1$-extension can affect the $\ggth$ of a quasitrain, 
for the new $1$-wagons might contain shorter cycles than the original $1$-wagons. 
The next result shows, that if there are no problems with $1$-wagons, 
then $1$-extensions inherit the $\ggth$ of the original quasitrain.  
 	
\begin{lemma}\label{lem:4508}
	Suppose that $\seq{g}=(g_1, \dots, g_m)\in\gM$ is a sequence of some 
	positive length~$m$, that $\seq{F}=(F, \equiv^F_0, \ldots, \equiv^F_m)$
	is a quasitrain with $\ggth(\seq{F})>\seq{g}$ and that  
	$\seq{H}=(H, \equiv^H_0, \ldots, \equiv^H_m)$ is a $1$-extension of $\seq{F}$.
	If every $1$-wagon~$X$ of $\seq{H}$ satisfies $\gth(X)>g_1$, 
	then $\ggth(\seq{H}) > \seq{g}$. 
\end{lemma}

\begin{proof}
		Given $\mu\in [m]$ and a $\mu$-wagon $W_\mu$ of $\seq{H}$ we are to prove 
	that $\ggth(W_\mu, \equiv^{W_\mu}_{\mu-1}) >g_\mu$. In the special case $\mu=1$
	this was stated as an assumption, so it remains to consider the case $\mu\in [2, m]$.  
	Assume for the sake of contradiction that
	for some $n\in [2, g_\mu]$ there exists an $n$-cycle of distinct $(\mu-1)$-wagons 
		\[
		\ccC=W^1_{\mu-1}q_1 \ldots W^n_{\mu-1}q_n
	\]
	all of which are contained in $W_\mu$. Owing to Lemma~\ref{lem:4015}  
	these $(\mu-1)$-wagons contract to~$\seq{F}$ and the connecting 
	vertices $q_1, \ldots, q_n$ are in $V(F)$. 
	Therefore, the contraction $\overline{W}_\mu$ of~$W_\mu$ fails to have the 
	property $\ggth(\overline{W}_\mu, \equiv^{\overline{W}_\mu}_{\mu-1})>g_\mu$,
	contrary to $\ggth(\seq{F})>\seq{g}$.
\end{proof}
	
Let us state an immediate consequence of Lemma~\ref{lem:4507} and Lemma~\ref{lem:4508}.
 
\begin{cor}\label{cor:5515}
	If $\seq{g}\in\gM$ is a nonempty sequence, $\seq{F}$ denotes a train of 
	height $|\seq{g}|$ with parameter~$\seq{A}$ and $\ggth(\seq{F})>\seq{g}$,
	and the quasitrain $\seq{H}$ arises from $\seq{F}$ by assimilation of 
	its $1$-wagons, then~$\seq{H}$
	is again a train of height $m$ with parameter $\seq{A}$ 
	and $\ggth(\seq{H})>\seq{g}$. \qed
\end{cor} 	

We conclude this subsection with a brief discussion of disjoint unions of quasitrains. 
Suppose first that $\{\seq{G}_j\colon j\in J\}$ is a family of mutually 
vertex-disjoint $f$-partite quasitrains of the same height $m\in \NN$, 
say $\seq{G}_j=(G_j, \equiv^j_0, \ldots, \equiv^j_m)$ for every $j\in J$. 
By the {\it union} of this family we mean the $f$-partite quasitrain 
$\seq{G}=(G, \equiv^G_0, \ldots, \equiv^G_m)$ with 
\[
	V(G)=\bigdcup_{j\in J}V(G_j)
	\quad \text{ and } \quad
	E(G)=\bigdcup_{j\in J}E(G_j)
\]
whose equivalence relations are defined as follows. 
\begin{enumerate}
	\item[$\bullet$] If $e', e''\in E(G)$ and $\mu\in [0, m)$, then $e'\equiv^G_\mu e''$ 
				means that there is an index $j\in J$ with $e', e''\in E(G_j)$
				and $e'\equiv^j_\mu e''$. 
	\item[$\bullet$] Moreover, $e'\equiv^G_m e''$ holds for all edges $e', e'' \in E(G)$.
\end{enumerate}
One checks immediately that $\seq{G}$ is indeed a quasitrain of height $m$ and 
that $\seq{G}_j$ is a subquasitrain of $\seq{G}$ for every $j\in J$. If the 
quasitrains~$\seq{G}_j$ are ordered, we order the vertex classes of $\seq{G}$ in such a 
way that $\seq{G}_j$ is an ordered subquasitrain of $\seq{G}$ for every $j\in J$. 

If a family ${\mathfrak G}=\{\seq{G}_j\colon j\in J\}$ of not necessarily 
vertex-disjoint ordered quasitrains of the same height $m$ is given, we can take a
family of mutually vertex-disjoint ordered 
quasitrains $\{\seq{G}^\star_j\colon j\in J\}$ 
such that $\seq{G}^\star_j$ is order-isomorphic to $\seq{G}_j$ for every~$j\in J$,
and then we can form its union $\seq{G}$ as explained above. 
In this situation $\seq{G}$ is called the {\it disjoint union} of the 
family~$\mathfrak G$. 
\index{disjoint union (of quasitrains)}
One readily confirms that if all 
members of~$\mathfrak G$ are trains with the same parameter~$\seq{A}$, then~$\seq{G}$
is a train with parameter~$\seq{A}$ as well. Moreover, if $\seq{g}\in\gM$ is 
a sequence of length $m$ and $\ggth(\seq{G}_j) >\seq{g}$ holds for every~$j\in J$, 
then $\ggth(\seq{G}) >\seq{g}$ follows. 

We conclude this discussion with an easy fact on disjoint unions that will 
help us in Section~\ref{subsec:0136} to analyse the $\GTH$ of train picture zero.  

\begin{fact}\label{f:picnull}
	Suppose that a sequence $\seq{g}\in \gM_\le^\times$
	and a quasitrain $\seq{F}$ of height $|\seq{g}|$ satisfy $\ggth(\seq{F})>\seq{g}$. 
	If~$\ccP$ denotes a set of mutually vertex-disjoint copies of $\seq{F}$ and the 
	quasitrain~$\seq{P}$ is their union, then $\GTH(\seq{P}, \ccP^+)>\seq{g}$.    
\end{fact}

\begin{proof}
	Writing $\seq{g}=(g_1, \dots, g_m)$	 we shall prove first  
		\begin{equation}\label{eq:n113}
		\GTH(P, \equiv^P_{\mu-1}, \ccP^+)>g_\mu
		\quad \text{ for every $\mu\in [m]$.}
	\end{equation}
		To this end we recall that Lemma~\ref{lem:1557} 
	tells us $\ggth(F, \equiv^F_{\mu-1}) >g_\mu$,
	which due to Lemma~\ref{lem:2310} implies
	$\GTH\bigl(F, \equiv^F_{\mu-1}, E^+(F)\cup\{(F, \equiv^F_{\mu-1})\}\bigr) > g_\mu$.
	Since big cycles of the pretrain system $(P, \equiv^P_{\mu-1}, \ccP^+)$ cannot 
	jump from one copy in $\ccP$ to another copy, this confirms~\eqref{eq:n113}.
	
	Moreover, Lemma~\ref{lem:GTH1} immediately implies $\GTH(P, \equiv^P_m, \ccP^+)>1$.
\end{proof}

\subsection{A generalised extension lemma}
\label{subsec:0135}

The proof of Proposition~\ref{prop:0142} starts with the assumption that for 
some sequence $\seq{g}=(g_1, \dots, g_m)\in\gM_\le^\times$ we have a construction 
$\Psi^{\seq{g}}$ exemplifying $\karo_{\seq{g}}$. The extension process then gives 
rise to a construction $\ups=\Ext(\Omega^{(2)}, \Psi^{\seq{g}})$, which is applicable
to trains of height $m+1$ whose $\ggth$ exceeds $(2)\circ\seq{g}$. 
Roughly speaking, this construction starts by applying $\Omega^{(2)}$ to the 
(assimilated) $1$-wagons and forming a disjoint union over all possible extensions.
\index{$\Omega^{(2)}$} 
We then apply $\Psi^{\seq{g}}$ to an auxiliary train of height~$m$ with ``larger edges'', and insert the earlier $1$-wagons back into the current edges.

We proceed with a full description as to how this construction 
is carried out. 
Suppose that $\seq{F}=(F, \equiv^F_0, \ldots, \equiv^F_{m+1})$ is an 
ordered $f$-partite train of height~$m+1$ with parameter 
$\seq{A}=(A_1, \ldots, A_{m+1})$ 
satisfying $\ggth(\seq{F}) > (2)\circ\seq{g}$, and that $r$ signifies 
a number of colours.  
Now the train system $\ups_r(\seq{F})=(\seq{H}, \ccH)$ is constructed
by means of the following eight steps paralleling the discussion in~\S\ref{sssec:lp}. 

\begin{enumerate}[label=\nlabel]
	\item\label{it:am1} Let $(\wh{F}, \equiv^{\wh{F}}_0, \ldots, \equiv^{\wh{F}}_{m+1})$ 
		be obtained from $\seq{F}$ by assimilating its $1$-wagons.  
		Due to Corollary~\ref{cor:5515} this is a train of height $m+1$ with 
		parameter $\seq{A}$ containing a standard copy of $\seq{F}$ and satisfying 
		$\ggth(\wh{F}, \equiv^{\wh{F}}_0, \ldots, \equiv^{\wh{F}}_{m+1})
		>(2)\circ\seq{g}$. 
		Let $W$ denote an ordered $f$-partite hypergraph all $1$-wagons of this train 
		are isomorphic to. Notice that $W$ is a linear $A_1$-intersecting hypergraph
		without isolated vertices.
	\item\label{it:am2} Construct $\Omega^{(2)}_r(W)=(X, \ccX)$ and assume, without loss 
		of generality, that~$X$ has no isolated vertices. Now $X$ is an ordered $f$-partite 
		hypergraph which is linear and $A_1$-intersecting (by 
		Proposition~\ref{prop:1738}, 
		Corollary~\ref{cor:0059}, and
		Lemma~\ref{lem:n383}).
		Furthermore, Corollary~\ref{cor:2201} entails $\Gth(X, \ccX^+)>2$.  
	\item\label{it:am3} Construct the family 
				\[
				{\mathfrak G}
				=
				\bigl\{
				(\wh{F}, \equiv^{\wh{F}}_0, \ldots, \equiv^{\wh{F}}_{m+1}) 
				\ltimes (X, W_\star)
				\colon 
				W_\star\in\ccX\bigr\}
		\]
				of ordered $f$-partite quasitrains as in Example~\ref{exmp:2340}.  
		By Lemma~\ref{lem:4507} and Lemma~\ref{lem:4508} every 
		$\seq{G}_\star\in{\mathfrak G}$ is a train with parameter $\seq{A}$ 
		that satisfies $\ggth(\seq{G}_\star) >(2)\circ\seq{g}$. 
		
		Let $\seq{G}=(G, \equiv^G_0, \ldots, \equiv^G_{m+1})$ be the disjoint union 
		of the trains in $\mathfrak G$ as defined at the end of \S\ref{sssec:3136}.
		Thus $\seq{G}$ is a train with parameter $\seq{A}$ 
		satisfying $\ggth(\seq{G})>(2)\circ\seq{g}$
		and all $1$-wagons of $\seq{G}$ are order-isomorphic to $X$.
		Moreover, $\seq{G}$ contains $|\ccX|$ standard copies of $(F, \equiv^F)$.		
	\item\label{it:am4} Let $I$ be the index set of the given train $\seq{F}$ 
		and define the function $x\colon I\longrightarrow \NN$ by 
		$x(i)=|V_i(X)|$ for every $i\in I$. Let $M$ be the ordered $x$-partite 
		hypergraph with $V(M)=V(G)$ whose edges correspond to the wagons 
		of $(G, \equiv_1^G)$ (so that $G$ is living in $M$). 
		Moreover, let $\seq{M}=(M, \equiv^M_0, \ldots, \equiv^M_m)$ be the train 
		of height $m$ whose $\mu$-wagons correspond to the $(\mu+1)$-wagons of~$\seq{G}$. 
		In more precise terms, this means that 
		for every $\mu\in [0, m]$ the pretrain $(G, \equiv^G_{\mu+1})$ is derived 
		from $(M, \equiv^M_\mu)$.
		We remark that $\seq{M}$ has the parameter 
		$\seq{A}_\bullet = (A_2, \ldots, A_{m+1})$ 
		and satisfies $\ggth(\seq{M})>\seq{g}$.   
	\item\label{it:am5} Construct the ordered $x$-partite train system
				\[
			\Psi^{\seq{g}}_{r^{e(X)}}(\seq{M})
			=
			(\seq{N}, \ccN)=(N, \equiv^N_0, \ldots, \equiv^N_m, \ccN)
		\]
		of height $m$. Due to $\karo_{\seq{g}}$ the copies of this system are strongly 
		induced and we have $\GTH(\seq{N}, \ccN^+)>\seq{g}$. 
		Moreover, the parameter of $\seq{N}$ is again $\seq{A}_\bullet$. 
	\item\label{it:am6} Let $H$ be the ordered $f$-partite hypergraph obtained from 
		$N$ by inserting ordered copies of $X$ into its edges. We endow $H$ with the 
		following train structure 
				\[
			\seq{H}=(H, \equiv^H_0, \ldots, \equiv^H_{m+1})
		\]
		of height $m+1$. 
		For $\mu\in [m+1]$ we declare two edges $e', e''\in E(H)$ to be in the 
		same $\mu$-wagon of $\seq{H}$ if the edges $f', f''\in E(N)$ 
		with $e'\subseteq f'$ and 
		$e''\subseteq f''$ satisfy $f'\equiv^N_{\mu-1} f''$. In other words, 
		we demand that the pretrain $(H, \equiv^H_\mu)$ be derived 
		from $(N, \equiv^N_{\mu-1})$. Recall that the $0$-wagons 
		of $\seq{H}$ need to be determined according to 
		Definition~\ref{dfn:2206}\ref{it:2206a}. 
	\item\label{it:am7} Every copy $\seq{M}_\star\in \ccN$ gives rise to a derived 
		copy $\seq{G}_\star\in\binom{\seq{H}}{\seq{G}}$ and we write $\ccH_\bullet$
		for the system of all $|\ccN|$ copies arising in this manner. Each member 
		of $\ccH_\bullet$ con\-tains~$|\ccX|$ standard copies of $\seq{F}$. Let 
		$\ccH\subseteq \binom{\seq{H}}{\seq{F}}$ denote the system of all these copies. 
	\item\label{it:am8} Finally, we set $\ups_r(\seq{F})=(\seq{H}, \ccH)$.
\end{enumerate}

The next result summarises all properties of this construction we shall need 
in the sequel. 
 
\begin{lemma}\label{lem:1904}
	Assume $\karo_{\seq{g}}$ for some sequence $\seq{g}\in\gM^\times_\le$ of length $m$.
	There exists a construction~$\ups$ which generates for every 
	ordered $f$-partite train $\seq{F}$ of height~$m+1$ 
	with $\ggth(\seq{F})>(2)\circ\seq{g}$ and every number of colours $r$ 
	a train system $\ups_r(\seq{F})=(\seq{H}, \ccH)$ with the same parameter 
	as~$\seq{F}$ and strongly induced copies such that $\ccH\lra(\seq{F})_r$
	and $\GTH(\seq{H}, \ccH^+)>\seq{g}$. 
\end{lemma}

It might be helpful to point out that the conclusion of this lemma implies
\begin{equation*}		\ggth(\seq{H})>(2)\circ\seq{g}\,.
\end{equation*}
This is because $\GTH(\seq{H}, \ccH^+)>\seq{g}$ presupposes that $H$ is linear,
whence all $1$-wagons of~$H$ are linear. Moreover, for $\mu\in [2, m+1]$ the required 
statement on cycles of $(\mu-1)$-wagons within the same $\mu$-wagon follows 
from $\GTH(\seq{H}, E^+(H))>\seq{g}$ via Lemma~\ref{lem:2310}.
	
\begin{proof}[Proof of Lemma~\ref{lem:1904}]
	Given a train $\seq{F}=(F, \equiv^F_0, \ldots, \equiv^F_{m+1})$ of height $m+1$ 
	with parameter~$\seq{A}$ such that $\ggth(\seq{F}) > (2)\circ\seq{g}$ and a 
	number of colours $r$ we follow the above Steps~\ref{it:am1}\,--\,\ref{it:am8} 
	and construct the train system $\ups_r(\seq{F})=(\seq{H}, \ccH)$.  
	
	Lemma~\ref{lem:0059} (applied to $\Omega^{(2)}$ and $\Psi^{\seq{g}}$ here in place 
	of $\Phi$ and $\Psi$ there) yields the partition relation $\ccH\lra(\seq{F})_r$.  
	Thus it remains to prove 
		\begin{enumerate}[label=\alabel]
		\item\label{it:tea} that the copies in $\ccH$ are strongly induced,
		\item\label{it:teb} and $\GTH(\seq{H}, \ccH^+)>\seq{g}$.
	\end{enumerate} 
		
	We begin by taking a closer look at the train $\seq{G}$ constructed 
	in Step~\ref{it:am3}. Let $\ccS_G$ be the system consisting of the $|\ccX|$ 
	standard copies of $\seq{F}$ in $\seq{G}$. Each of them is contained 
	in a unique train of the 
	form $\seq{G}_\star=(\wh{F}, \equiv^{\wh{F}}_0, \ldots, \equiv^{\wh{F}}_{m+1})\ltimes (X, W_\star)$, where $W_\star\in\ccX$.\footnote[1]{For the sake of transparency we are 
	ignoring here the isomorphism implied in the formation of the disjoint union.} 	
	Since the copies in $\ccX$ are strongly induced, Lemma~\ref{lem:5757} 
	informs us that $(G_\star, \equiv^{G_\star}_1)$ is a tame extension 
	of $(F, \equiv^F_1)$. In particular, 
		\begin{equation}\label{eq:n1053}
		\text{ the system $(G, \ccS_G)$ has strongly induced copies.}
	\end{equation}
		
	Moreover, for every $\mu\in[m+1]$ the equivalence 
	relation $\equiv^{G_\star}_1$ refines $\equiv^{G_\star}_\mu$ and thus  
	the pretrain $(G_\star, \equiv^{G_\star}_\mu)$ is a tame extension 
	of $(F, \equiv^F_\mu)$ as well. Consequently, 
		\begin{equation}\label{eq:n1054}
		\text{ for every $\mu\in[m]$ the system $\ccS_G$ is scattered 
			in $(G, \equiv^G_\mu)$.}
	\end{equation}
	   For clarity we point out that we cannot claim this for $\mu=m+1$;
   the difference is that the 
   unique $(m+1)$-wagon of $\seq{G}$ is the union of the $(m+1)$-wagons of the
   various trains~$\seq{G}_\star$, while for $\mu\in [m]$ the $\mu$-wagons of 
   these trains ``remain separate from each other''. 
   
   Later we will need to know that 
   	\begin{equation}\label{eq:n1056}
		\text{ if $F_\star\in \ccS_G$ and $x$ is an isolated vertex of $F_\star$, 
		then $x$ is isolated in $G$.}
	\end{equation}
		Indeed, by the definition of wagon assimilation, $x$ is also isolated 
	in $\widehat{F}$ and thus in every member of $\mathfrak{G}$. 
   
	We proceed by discussing the hypergraphs $H$ and $N$.  
	Clearly, $H$ is living in $N$ (cf.\ Step~\ref{it:am6}) and the 
	system derived from $\ccN$ is $\ccH_\bullet$ (cf.\ Step~\ref{it:am7}).
	By Step~\ref{it:am5} the copies in $\ccN$ are strongly induced and, therefore, 	
	Fact~\ref{f:strder} tells us that the system $(H, \ccH_\bullet)$ has strongly 
	induced copies, too. 
	Together with~\eqref{eq:n1053} and the fact that strong inducedness
	is a transitive relation this proves~\ref{it:tea}.
		
	Let us now write $\seq{g}=(g_1, \dots, g_m)$ and fix some $\mu\in [m]$.
	Due to
		\begin{equation}\label{eq:n1050}
		\GTH(\seq{N}, \ccN^+)>\seq{g}
	\end{equation}
	   we know $\GTH(N, \equiv^N_{\mu-1}, \ccN^+)>g_\mu$.
	As the pretrain $(H, \equiv^H_\mu)$ is derived from $(N, \equiv^N_{\mu-1})$ 
	(cf.\ Step~\ref{it:am6}), Lemma~\ref{lem:n023} translates this to 
	$\GTH(H, \equiv^H_\mu, \ccH^+_\bullet)>g_\mu$, which together 
	with~\eqref{eq:n1054} and Lemma~\ref{lem:n540} reveals
	$\GTH(H, \equiv^H_\mu, \ccH^+)>g_\mu$.
	
	So in order to establish~\ref{it:teb} it only remains to 
	prove $\GTH(H, \equiv^H_{m+1}, \ccH^+)>1$. 
	The related statement 
		\begin{equation}\label{eq:n1058}
		\GTH(H, \equiv^H_{m+1}, \ccH^+_\bullet)>1
	\end{equation}
		can be shown as in the previous paragraph by observing that~\eqref{eq:n1050}
	contains the information $\GTH(N, \equiv^N_m, \ccN^+)>1$ and invoking 
	Lemma~\ref{lem:n023}.
	
	Now let $F_1, F_2\in \ccH^+$ be two distinct copies having a 
	vertex $x$ in common. According to Lemma~\ref{lem:GTH1} and symmetry it 
	suffices to exhibit an edge $e\in E(F_1)$ such that $x\in e$. Assume contrariwise 
	that $x$ is an isolated vertex of $F_1$. 
	
	In particular, $F_1$ is a real copy. Let $G_1\in\ccH_\bullet$ be 
	the copy one of whose standard copies is $\seq{F}_1$. 
	If $F_2$ is a real copy as well 
	we determine $G_2\in\ccH_\bullet$ similarly
	and if $F_2$ is an edge copy we set $G_2=F_2$.
	
	Due to~\eqref{eq:n1056} we know that $x$ is isolated in $G_1$. 
	Moreover, the copies $G_1, G_2\in \ccH^+_\bullet$ have the vertex $x$
	in common. If they were distinct, then~\eqref{eq:n1058} and Lemma~\ref{lem:GTH1}    
	would lead to the contradiction that $x$ is non-isolated in $G_1$.	
	
	So altogether $G_1=G_2\in \ccH_\bullet$ is a real copy and, 
	therefore, $F_1$ and $F_2$ are two distinct standard copies 
	in $G_1$. But now $x\in V(F_1)\cap V(F_2)$ contradicts $F_1\ne F_2$
	because of the construction of~$\seq{G}$.  
\end{proof} \section{Trains in partite constructions}
\label{subsec:0136}

In this section we complete the proof of Proposition~\ref{prop:0142}
and establish Proposition~\ref{prop:0139}. Both tasks are accomplished 
by means of the partite construction method and thus we begin with 
some general remarks on train pictures. 
 
\subsection{Quasitrain constructions}
\label{sssec:4137}

Since it allows us to ignore parameters, it will be easier to study 
quasitrains in partite constructions first. 
Suppose that $\Phi$ denotes a Ramsey construction for hypergraphs and that $\Xi$ 
is a partite lemma applicable to $k$-partite $k$-uniform quasitrains of a 
fixed height $m\in \NN$. As we shall see below, we can then define a 
construction $\PC(\Phi, \Xi)$ applicable to quasitrains of height~$m$. 

Let us first introduce some terminology for such situations.
Suppose that $(G, \ccG)$ is a system of hypergraphs, 
where $\ccG\subseteq \binom{G}{F}$
holds for some hypergraph $F$ endowed with a fixed quasitrain structure 
$\seq{F}=(F, \equiv^F_0, \ldots, \equiv^F_m)$ of height $m$. 
A {\it quasitrain picture over $(G, \ccG)$} is a structure of the form 
$(\seq{\Pi}, \ccP, \psi)=(\Pi, \equiv^\Pi_0, \ldots, \equiv^\Pi_{m}, \ccP, \psi)$,
where
\begin{enumerate}
	\item[$\bullet$] $\seq{\Pi}$ is a quasitrain of height $m$,
	\item[$\bullet$] $(\Pi, \ccP, \psi)$ is a picture over $(G, \ccG)$,
	\item[$\bullet$] and $\ccP\subseteq \binom{\seq{\Pi}}{\seq{F}}$, i.e., every 
		copy $(F_\star, \equiv_0^{F_\star}, \ldots, \equiv_m^{F_\star})\in\ccP$
		is a subquasitrain of $\seq{\Pi}$ isomorphic to $\seq{F}$.
\end{enumerate} 
\index{quasitrain picture}

For instance, the picture zero $(\Pi_0, \ccP_0, \psi_0)$ introduced in~\S\ref{sssec:pict}
expands uniquely to the corresponding {\it quasitrain picture zero}
$(\seq{\Pi}_0, \ccP_0, \psi_0)$ with the property that for $\mu\in[0, m)$
every $\mu$-wagon of $\seq{\Pi}_0$ is contained in exactly one copy from $\ccP_0$. 
Recall that owing to Definition~\ref{dfn:2206}\ref{it:2206c} 
all edges of $\Pi_0$ need to be in the same $m$-wagon 
of $\seq{\Pi}_0$. In other words, the quasitrain $\seq{\Pi}_0$ is constructed 
to be the disjoint union of the quasitrains in $\ccP_0$. 
 
Now suppose that $(\Pi, \equiv^\Pi_0, \ldots, \equiv^\Pi_{m}, \ccP, \psi_\Pi)$ 
is such a quasitrain picture and that $e\in E(G)$.
The constituent $\Pi^e$ induces a subquasitrain $\seq{\Pi}^e$ of $\seq{\Pi}$. 
For reasons that will become apparent later we only define amalgamations over $e$
when 
\begin{equation}\label{eq:n190}
	E(\Pi^e)\ne \vn\,,
\end{equation}
which will never cause problems in practice.   
Given a $k$-partite $k$-uniform quasitrain system 
\[
	(\seq{H}, \ccH)=(H, \equiv^H_0, \ldots, \equiv^H_{m}, \ccH)
\]
(where $k=|e|$) with  
$\ccH\subseteq \binom{\seq{H}}{\seq{\Pi}^e}$ we can construct a structure 
\begin{equation}\label{eq:2213}
	(\Sigma, \equiv^\Sigma_0, \ldots, \equiv^\Sigma_{m}, \ccQ, \psi_\Sigma)
	=
	(\Pi, \equiv^\Pi_0, \ldots, \equiv^\Pi_{m}, \ccP, \psi_\Pi)
	\conc
	(H, \equiv^H_0, \ldots, \equiv^H_{m}, \ccH)
\end{equation}
by forming, in a first step, the ordinary picture 
$(\Sigma, \ccQ, \psi_\Sigma)=(\Pi, \ccP, \psi_\Pi)\conc (H, \ccH)$
and then defining the equivalence relations 
$\equiv^\Sigma_0, \ldots, \equiv^\Sigma_{m}$
on $E(\Sigma)$ in such a way that 
\[
	(\Sigma, \equiv^\Sigma_\mu, \ccQ, \psi_\Sigma)
	=
	(\Pi, \equiv^\Pi_\mu, \ccP, \psi_\Pi)
	\conc 
	(H, \equiv^H_\mu, \ccH)
\]
holds for every $\mu\in[0, m]$ (see~\S\ref{sssec:0021} and Lemma~\ref{lem:1825}).
One verifies easily that $(\seq{\Sigma}, \ccQ, \psi_\Sigma)$ 
is again a quasitrain 
picture over $(G, \ccG)$, 
where $\seq{\Sigma}=(\Sigma, \equiv^\Sigma_0, \ldots, \equiv^\Sigma_{m})$. 
In particular, one has to check here that the construction in~\S\ref{sssec:0021}
causes the equivalence classes of~$\equiv^\Sigma_0$ to consist of single edges.
Moreover, one needs to convince oneself that all edges of $\Sigma$ are in the 
same wagon with respect to $\equiv^\Sigma_m$, which 
requires~\eqref{eq:n190}. 

Let us now return to the discussion of $\PC(\Phi, \Xi)$, where, let us recall, $\Phi$
is a Ramsey construction for hypergraphs and $\Xi$ denotes a partite lemma for quasitrains 
of height~$m$. Given a quasitrain $\seq{F}=(F, \equiv_0^F, \ldots, \equiv^F_m)$ of 
height $m$ and a number of colours $r$
the quasitrain system $\PC(\Phi, \Xi)_r(\seq{F})$ is constructed as follows. Set
$\Phi_r(F)=(G, \ccG)$ and assume, without loss of generality, that every edge of $G$
appears in at least one copy of $\ccG$ (other edges of $G$ are not needed for 
ensuring the partition relation $\ccG\lra (F)_r$). Now picture zero 
$(\seq{\Pi}_0, \ccP_0, \psi_0)$ over $(G, \ccG)$ satisfies $E(\Pi^e_0)\ne\vn$
for every $e\in E(G)$ and thus we never need to worry about~\eqref{eq:n190}
throughout the ensuing partite construction. As usual we 
fix an enumeration $E(G)=\{e(1), \ldots, e(N)\}$ and recursively we construct 
a sequence $(\seq{\Pi}_\alpha, \ccP_\alpha, \psi_\alpha)_{\alpha\le N}$ of 
quasitrain pictures over~$(G, \ccG)$.
Whenever we have just obtained 
$(\seq{\Pi}_{\alpha-1}, \ccP_{\alpha-1}, \psi_{\alpha-1})$
for some $\alpha\in [N]$ we generate the $k$-partite $k$-uniform quasitrain system 
$\Xi_r(\seq{\Pi}_{\alpha-1}^{e(\alpha)})=(\seq{H}_\alpha, \ccH_\alpha)$
and amalgamate 
\[
	 (\seq{\Pi}_\alpha, \ccP_\alpha, \psi_\alpha)
	 =
	 (\seq{\Pi}_{\alpha-1}, \ccP_{\alpha-1}, \psi_{\alpha-1})
	 \conc
	 (\seq{H}_\alpha, \ccH_\alpha)\,.
\]
Finally, when the final quasitrain picture $(\seq{\Pi}_N, \ccP_N, \psi_N)$ has been 
reached, we stipulate 
\[
	\PC(\Phi, \Xi)_r(\seq{F})=(\seq{\Pi}_N, \ccP_N)\,.
\]

\subsection{Train amalgamations}\label{subsec:tipc}
The next problem is whether the partite construction method can handle trains 
instead of quasitrains as well. That is, we shall need to know 
suitable conditions on constructions $\Phi$, $\Xi$, and trains $\seq{F}$ 
guaranteeing that for every number of colours $r$ the quasitrain
$\PC(\Phi, \Xi)_r(\seq{F})$ turns out to be a train. 
As usual, our strategy is to enforce that all quasitrain pictures 
$(\seq{\Pi}, \ccP, \psi)$ generated along the way have the property 
that their underlying quasitrains 
$\seq{\Pi}$ are trains. This is already somewhat problematic for 
picture zero and we resolve to deal with this situation by demanding that~$\Phi$
be a train construction, so that vertically we have a train system $(\seq{G}, \ccG)$
and not just a hypergraph system. We are thus led to the following notion of train 
pictures. 

\begin{dfn}\label{dfn:7058}
	Suppose that $\seq{F}$ is a train 
	and that $(\seq{G}, \ccG)$ is a train system all of whose copies are isomorphic 
	to $\seq{F}$. We say that $(\seq{\Pi}, \ccP, \psi_\Pi)$ is 
	a {\it train picture over $(\seq{G}, \ccG)$} if 
	\begin{enumerate}[label=\rmlabel]
		\item\label{it:tp1} the trains $\seq{F}$, 
			$\seq{G}=(G, \equiv^G_0, \dots, \equiv^G_m)$, 
			and $\seq{\Pi}=(\Pi, \equiv^\Pi_0, \dots, \equiv^\Pi_m)$
			are $f$-partite for the same function $f$, have the same height~$m$, 
			and the same parameter $\seq{A}$;
		\item\label{it:tp2} $(\seq{\Pi}, \ccP, \psi_\Pi)$ is a quasitrain picture 
			over $(G, \ccG)$;
		\item\label{it:tp3} and $\forall \mu\in [0, m]\,\,\forall e, e'\in E(\Pi) \,\,\,
			[e\equiv^\Pi_\mu e' \,\,\, \Longrightarrow \,\,\, \psi_\Pi(e)\equiv^G_\mu \psi_\Pi(e')]$.
	\end{enumerate} 
	In this situation, we call $m$ and $\seq{A}$ the {\it height} 
	and the {\it parameter} of the picture $(\seq{\Pi}, \ccP, \psi_\Pi)$.  
	\index{train picture}
\end{dfn}    

A more intuitive way of thinking about condition~\ref{it:tp3} is that the 
projection $\psi_\Pi$ is required to be a ``train homomorphism'' 
from $\seq{\Pi}$ to $\seq{G}$ (rather than just a mere hypergraph homomorphism).   
Notice that given $\seq{F}$ and $(\seq{G}, \ccG)$ 
with $\ccG\subseteq \binom{\seq{G}}{\seq{F}}$ as in~\ref{it:tp1} we can always 
form the {\it train picture zero} $(\seq{\Pi}_0, \ccP_0, \psi_0)$ in the usual way. 
The next statement takes a brief look at the parameters of constituents of 
train pictures. 

\begin{fact}\label{f:N107}
	Let $(\seq{\Pi}, \ccP, \psi)$ be an $f$-partite 
	train picture over the train system $(\seq{G}, \ccG)$. 
	If $m$ and $\seq{A}=(A_1, \dots, A_m)$
	denote the height and the parameter of this picture, then for every edge $e\in E(G)$
	the constituent $\seq{\Pi}^e$ is a train of height $m$ whose parameter 
	$\seq{D}=(D_1, \dots, D_m)$ is given by $D_\mu=e\cap V_{A_\mu}(G)$ 
	for every $\mu\in [m]$.
\end{fact}

As usual, the constituent $\seq{\Pi}^e$ is regarded here as a $k$-partite $k$-uniform 
train with index set $e$, where $k=|e|$.

\begin{proof}
	Recall that the $f$-partite structure of $\Pi$ is defined 
	by $V_i(\Pi)=\psi^{-1}\bigl(V_i(G)\bigr)$ for every index~$i$ in the domain of $f$. 
	Thus the sets $D_\mu=e\cap V_{A_\mu}(G)$ satisfy  
		\[
		V_{D_\mu}(\Pi^e)
		=
		\psi^{-1}(D_\mu)
		=
		\psi^{-1}(e)\cap \psi^{-1}\bigl(V_{A_\mu}(G)\bigr)
		=
		V(\Pi^e)\cap V_{A_\mu}(\Pi)
	\]
		for every $\mu\in [m]$. 
	
	Now if two edges $e_\star$, $e_{\star\star}\in E(\Pi^e)$
	belong to the same $\mu$-wagon but not to the same $(\mu-1)$-wagon 
	of $\seq{\Pi}^e$, 
	then $e_\star\cap e_{\star\star}\subseteq V(\Pi^e)\cap V_{A_\mu}(\Pi)
	=V_{D_\mu}(\Pi^e)$. 
\end{proof}

When executing partite constructions with train pictures, we can maintain parameters 
by appealing to the following result. 

\begin{lemma}\label{lem:N128}
	Let~$(\seq{G}, \ccG)$ be a train system of height $m$ with 
	parameter~$\seq{A}=(A_1, \dots, A_m)$. 
	Suppose further that $(\seq{\Pi}, \ccP, \psi_\Pi)$
	is a train picture over~$(\seq{G}, \ccG)$ and that 
		\[
		(\seq{\Sigma}, \ccQ, \psi_\Sigma)
		=
		(\seq{\Pi}, \ccP, \psi_\Pi)
		\conc
		(\seq{H}, \ccH)
	\]
		holds for a train system 
	$(\seq{H}, \ccH)=(H, \equiv^H_0, \ldots, \equiv^H_{m}, \ccH)$ and a quasitrain 
	picture $(\seq{\Sigma}, \ccQ, \psi_\Sigma)$. 
	If the amalgamation occurs over the edge $e\in E(G)$,
		\begin{enumerate}
		\item[$\bullet$] the parameter $\seq{D}=(D_1, \dots, D_m)$ of $\seq{H}$
			is given by $D_\mu=e\cap V_{A_\mu}(G)$ for every $\mu\in [m]$,
		\item[$\bullet$] and $\GTH(H, \equiv_\mu, \ccH^+)>1$ for every $\mu\in [m]$,
	\end{enumerate}
		then $(\seq{\Sigma}, \ccQ, \psi_\Sigma)$ is a train picture over $(\seq{G}, \ccG)$. 
\end{lemma}

Let us emphasise that our demand on the parameter $\seq{D}$ agrees with 
Fact~\ref{f:N107}. As we saw in the above proof, it leads 
to $V_{D_\mu}(H)=V_{A_\mu}(\Sigma)\cap V(H)$ for every $\mu\in [m]$.

\begin{proof}[Proof of Lemma~\ref{lem:N128}]
	We have to show that $\seq{A}$ is a legitimate parameter 
	for $\seq{\Sigma}$ 	and that clause~\ref{it:tp3} of Definition~\ref{dfn:7058} holds for $\Sigma$
	instead of $\Pi$. The latter condition is clear for $\mu=0$, 
	because $\equiv^\Sigma_0$ is the same as equality. 
	So it remains to consider an index $\mu\in [m]$, a $\mu$-wagon $W_\mu$ 
	of~$\seq{\Sigma}$, and two edges $e', e''\in E(W_\mu)$. We are to prove that
		\begin{enumerate}[label=\alabel]
		\item\label{it:N128-1} if $e'\not\equiv^\Sigma_{\mu-1} e''$ and 
			$x\in e'\cap e''$, then $x\in V_{A_\mu}(\Sigma)$;
		\item\label{it:N128-2} and $\psi_\Sigma(e')\equiv^G_\mu \psi_\Sigma(e'')$.
	\end{enumerate}

	For both edges $e'$ and $e''$ we distinguish the cases whether they belong 
	to $H$ or not. By symmetry there are three possibilities. 
	
	\smallskip
	
	{\it \hskip 2em First Case: Neither $e'$ nor $e''$ is in $E(H)$.}
	
	\smallskip
	
	Let $\Pi_\star$ and $\Pi_{\star\star}$ denote the standard copies 
	of~$\Pi$ with $e'\in E(\Pi_\star)$ and $e''\in E(\Pi_{\star\star})$,
	respectively. In the special case $\Pi_\star=\Pi_{\star\star}$
	both claims follow from $(\seq{\Pi}, \ccP, \psi_\Pi)$ being a train 
	picture and $\seq{\Pi}_\star$ being a subtrain of $\seq{\Sigma}$, 
	so we may assume $\Pi_\star\ne \Pi_{\star\star}$ from now on. 
	
	Starting with~\ref{it:N128-1} we observe that the vertex $x$ needs to belong
	to $V(H)$.  
	Lemma~\ref{lem:1825}\ref{it:1832c}\ref{it:1832c2} tells us that the copies 
	$\Pi^e_\star, \Pi^e_{\star\star}\in \ccH$ extended by the standard 
	copies $\Pi_\star$, $\Pi_{\star\star}$ intersect the wagon $W_\mu$ and, 
	for this reason, $W_\mu$ contracts to a wagon $\overline{W}_\mu$ 
	of $(H, \equiv^H_\mu)$. Clearly $\Pi^e_\star x \Pi^e_{\star\star} \overline{W}_\mu$ 
	is a big cycle of order $1$
	in $(H, \equiv^H_\mu, \ccH^+)$. Due to $\GTH(H, \equiv^H_\mu, \ccH^+) > 1$
	and Lemma~\ref{lem:GTH1} there are two edges 
	$e_\star\in E(\Pi^e_\star)\cap E(\overline{W}_\mu)$ 
	and $e_{\star\star} \in E(\Pi^e_{\star\star})\cap E(\overline{W}_\mu)$ 
	with $x\in e_\star\cap e_{\star\star}$. By now we know four edges containing $x$
	and belonging to $W_\mu$, namely $e'$, $e_\star$, $e_{\star\star}$, and $e''$.
	Owing to $e'\not\equiv^\Sigma_{\mu-1} e''$ at least one of the three cases
		\[
		e'\not\equiv^\Sigma_{\mu-1} e_\star\,, 
		\quad 
		e_\star \not\equiv^\Sigma_{\mu-1} e_{\star\star}\,,
		\quad \text{ or } \quad 
		e_{\star\star}\not\equiv^\Sigma_{\mu-1} e''
	\]
		occurs. As $\seq{A}$ parametrises $\Pi_\star$, the first alternative
	implies indeed $x\in V_{A_\mu}(\Pi_\star)\subseteq V_{A_\mu}(\Sigma)$, and the 
	third case is similar. Moreover, the second 
	case rewrites as $e_\star \not\equiv^H_{\mu-1} e_{\star\star}$
	and $x\in V_{D_\mu}(H)\subseteq V_{A_\mu}(\Sigma)$ follows 
	(see Figure~\ref{fig:gmp}).
	This completes the proof of~\ref{it:N128-1}. 
	
\begin{figure}[h]
\centering
\begin{tikzpicture}[scale=.7]
	
\draw (-6,4)--(-6,-4);

\newcommand{\elip}[7]{\draw [color={#5}, fill={#6}, opacity={#7}] (#1-#3, #2) [out=-90, in= 180] to (#1,#2-#4)  [out=0, in=-90] to (#1+#3,#2) [out=90, in=0] to (#1,#2+#4)  [out=-180, in=90] to (#1-#3,#2);}

\elip{-6}{0}{.3}{.85}{black}{white}{.7};
\elip{-6}{.4}{.27}{.43}{black}{brown!80}{.7};
\elip{-2.5}{0}{1.5}{3.8}{black}{yellow!60}{1};
\elip{2.5}{0}{1.5}{3.8}{black}{cyan!30}{1};

\draw [red,fill=red!20] (-2.3,2.4) rectangle (2.3,-2.4);
\draw [red,fill=red!40] (-2.27,.77) rectangle (2.27,-.77);

\draw [thick] (-1.125,2.41) [out=-80, in=180] to (0,.83) [out=0, in=-100] to (1.125,2.41);
\fill [white](-1.125,2.42) [out=-80, in=180] to (0,.83) [out=0, in=-100] to (1.125,2.42)--cycle;

\draw [thick] (-1.125,-2.41) [out=80, in=180] to (0,-.83) [out=0, in=100] to (1.125,-2.41);
\fill [white](-1.125,-2.42) [out=80, in=180] to (0,-.83) [out=0, in=100] to (1.125,-2.42)--cycle;

\draw [rounded corners=10] (-.2,.84)--(.6,.7)--(.6,-.7)--(-.2,-.84);
\draw [rounded corners=10] (.2,.84)--(-.6,.7)--(-.6,-.7)--(.2,-.84);

\draw [red] (0,.43)--(-1.3,1.22);
\draw (0,.4)--(-1.3,1.2);
\node at (-1.05,1.3) {$e'$};

\draw [red] (0,.43)--(1.3,1.22);
\draw (0,.4)--(1.3,1.2);
\node at (1.05,1.25) {$e''$};

\draw (.4,.79)--(-1.22,-.79);
\node at (-1.3,-.5) {$e_\star$};

\draw (-.4,.79)--(1.22,-.79);
\node at (1.35,-.5) {$e_{\star\star}$};

\fill (0,.4) circle (1.5pt);
\node at (.25,.4) {$x$};
\draw [thick, dashed, -Stealth] (0,.4) -- (-5.7,.4); 

\draw (-5,.8)--(5,.8);
\draw (-5,-.8)--(5,-.8);

\node at (-6.75, .4) {$D_\mu$};
\node at (-6.6,-.4) {$e$};

\node [red!70!black] at (0,1.8) {$W_\mu$};
\node [red!70!black] at (3,0) {$\overline{W}_\mu$};

\node [yellow] at (-3.97,3.52) {\large ${\Pi_\star}$};
\node [yellow] at (-3.95,3.53) {\large ${\Pi_\star}$};
\node  at (-4,3.5) {\large ${\Pi_\star}$};
\node [cyan] at (4.13,3.52) {\large ${\Pi_{\star\star}}$};
\node  at (4.1,3.5) {\large ${\Pi_{\star\star}}$};

\end{tikzpicture}
\caption{Part~\ref{it:N128-1} in the first case.}
\label{fig:gmp}
\end{figure}
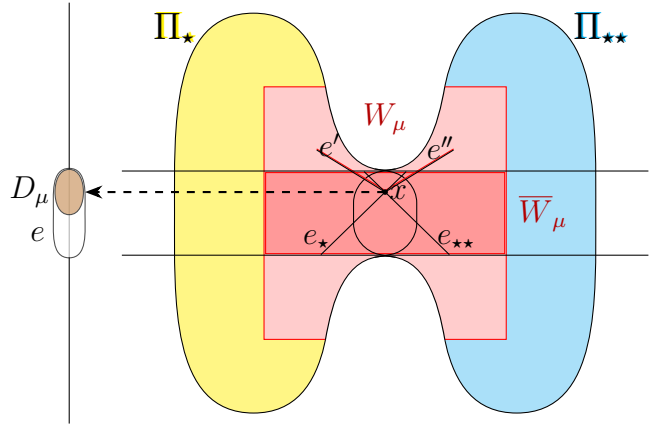 	
	Proceeding with~\ref{it:N128-2} we again invoke
	Lemma~\ref{lem:1825}\ref{it:1832c}\ref{it:1832c2}, thus obtaining  
	two edges $e_\star\in E(H)\cap E(\Pi_\star)$ 
	and $e_{\star\star}\in E(H)\cap E(\Pi_{\star\star})$ with 
	$e' \equiv^\Sigma_\mu e_\star \equiv^\Sigma_\mu e_{\star\star}\equiv^\Sigma_\mu e''$.
 	Since the train picture $(\seq{\Pi}, \ccP, \psi_\Pi)$ satisfies 
	Definition~\ref{dfn:7058}\ref{it:tp3}, we have
	$\psi_\Sigma(e')\equiv^G_\mu \psi_\Sigma(e_\star)$
	and $\psi_\Sigma(e'')\equiv^G_\mu \psi_\Sigma(e_{\star\star})$.
	Together with $\psi_\Sigma(e_\star)=e=\psi_\Sigma(e_{\star\star})$ 
	this yields the desired equivalence $\psi_\Sigma(e')\equiv^G_\mu \psi_\Sigma(e'')$.
	
	\smallskip
	
	{\it \hskip 2em Second Case: We have $e'\not\in E(H)$ and $e''\in E(H)$.}
	
	\smallskip
	
	For the proof of~\ref{it:N128-1} we again observe $x\in V(H)$,  
	denote the standard copy of~$\Pi$ containing~$e'$ by~$\Pi_\star$,
	and let $\Pi^e_\star\in \ccH$ be the copy extended by~$\Pi_\star$. Invoking 
	Lemma~\ref{lem:1825}\ref{it:1832c}\ref{it:1832c1} we infer 
	that $\Pi^e_\star x (e'')^+ \overline{W}_\mu$ is a big cycle of order $1$ 
	in $(H, \equiv^H_\mu, \ccH^+)$ and as before we find an 
	edge $e_\star\in E(\Pi^e_\star)\cap E(\overline{W}_\mu)$ 
	with $x\in e_\star$. Since at least one of the statements  
		\[
		e'\not\equiv^\Sigma_{\mu-1} e_\star
		\quad \text{ or } \quad
		e_\star\not\equiv^\Sigma_{\mu-1} e''
	\]
		holds, we can conclude $x\in V_{A_\mu}(\Sigma)$ as in the first case.
		
	For dealing with~\ref{it:N128-2} we appeal to 
	Lemma~\ref{lem:1825}\ref{it:1832c}\ref{it:1832c1} again, 
	this time getting an edge $e_\star\in E(H)\cap E(\Pi_\star)$ 
	with $e' \equiv^\Sigma_\mu e_\star\equiv^\Sigma_\mu e''$.
	Due to $\psi_\Sigma(e')\equiv^G_\mu \psi_\Sigma(e_\star)$
	and $\psi_\Sigma(e_\star)=e=\psi_\Sigma(e'')$ we have indeed
	$\psi_\Sigma(e')\equiv^G_\mu \psi_\Sigma(e'')$.
	
	\smallskip
	
	{\it \hskip 2em Third Case: Both $e'$ and $e''$ are in $E(H)$.}
	
	\smallskip

	Now~\ref{it:N128-1} follows from $x\in V_{D_\mu}(H)\subseteq V_{A_\mu}(\Sigma)$, 
	and $\psi_\Sigma(e')=e=\psi_\Sigma(e'')$ yields~\ref{it:N128-2}. 
\end{proof}

Next we adapt the $\GTH$ preservation lemma from~\S\ref{subsec:GTHpres} to train 
pictures. 

\begin{lemma}\label{lem:n209}
	Let $(\seq{G}, \ccG)$ be a train system of height $m$ with strongly induced 
	copies. Suppose further that 
		\[
		(\seq{\Sigma}, \ccQ, \psi_\Sigma)
		=
		(\seq{\Pi}, \ccP, \psi_\Pi)\conc(\seq{H}, \ccH)
	\]
		holds for two train pictures $(\seq{\Sigma}, \ccQ, \psi_\Sigma)$ 
	and $(\seq{\Pi}, \ccP, \psi_\Pi)$ over $(\seq{G}, \ccG)$, and a
	$k$-partite $k$-uniform train system $(\seq{H}, \ccH)$.
	If for some sequence $\seq{g}=(g_\ell, \dots, g_m)\in\gM$ whose length 
	is at most $m$ we have 
	$\GTH(\seq{\Pi}, \ccP^+)>\seq{g}$ and $\GTH(\seq{H}, \ccH^+)>\seq{g}$,
	then $\GTH(\seq{\Sigma}, \ccQ^+)>\seq{g}$.
\end{lemma}

\begin{proof}
	According to Definition~\ref{dfn:0115} the $\GTH$ assumptions mean 
		\begin{enumerate}[label=\nlabel]
		\item\label{it:125-1} $\GTH(\Pi, \equiv^\Pi_{\mu-1}, \ccP^+)>g_\mu$ for
			every $\mu\in [\ell, m]$;
		\item\label{it:125-2} $\GTH(\Pi, \equiv^\Pi_m, \ccP^+)>1$;
		\item\label{it:125-3} $\GTH(H, \equiv^H_{\mu-1}, \ccH^+)>g_\mu$ for
			every $\mu\in [\ell, m]$;
		\item\label{it:125-4} and $\GTH(H, \equiv^H_m, \ccH^+)>1$.
	\end{enumerate}
		Due to Lemma~\ref{lem:0052} the statements~\ref{it:125-1} and~\ref{it:125-3} 
	entail $\GTH(\Sigma, \equiv^\Sigma_{\mu-1}, \ccQ^+)>g_\mu$ for 
	every $\mu\in [\ell, m]$ and thus it only remains to 
	prove $\GTH(\Sigma, \equiv^\Sigma_m, \ccQ^+)>1$.
	
	To this end we consider a big cycle 
		\[
		\ccC=F_1xF_2W_\Sigma
	\]
		in $(\Sigma, \equiv^\Sigma_m)$, where $F_1, F_2\in\ccQ^+$ are distinct copies,~$x$ 
	is a vertex, and $W_\Sigma$ denotes the unique $m$-wagon of $\seq{\Sigma}$.  
	Due to Lemma~\ref{lem:GTH1} and symmetry it suffices to show that 
	there exists an edge $f_1\in E(F_1)$ passing through $x$. 
	 
	If $F_1\not\in E(H)^+$ we denote the standard copy of $(\Pi, \ccP^+)$
	to which $F_1$ belongs by $(\Pi_1, \ccP_1^+)$ and we let 
	$\Pi_1^e\in\ccH$ be the copy extended by $\Pi_1$. If, on the other hand,
	$F_1\in E(H)^+$, then we set $\Pi^e_1=F_1$. So in both cases we 
	have $\Pi^e_1\in\ccH^+$. Let~$\Pi_2^e$ be defined similarly with respect to~$F_2$. 
	
	If $\Pi^e_1=\Pi^e_2$, then $F_1$, $F_2$ belong to a common standard copy 
	and the existence of $f_1$ follows from~\ref{it:125-2}. 
	So we can henceforth assume $\Pi^e_1\ne \Pi^e_2$. Since $\ccC$ is a big cycle,
	each of the copies $F_1$, $F_2$ has at least one edge, and 
	Lemma~\ref{lem:1825} implies $E(\Pi^e_1), E(\Pi^e_2)\ne\vn$. 
	Using the unique $m$-wagon $W^H$ of $\seq{H}$ we can thus form a big cycle 
	$\Pi^e_1x\Pi^e_2W^H$ in $(H, \equiv^H_m, \ccH^+)$. Owing to~\ref{it:125-4}
	and Lemma~\ref{lem:GTH1} there exists an 
	edge $f\in E(\Pi^e_1)$ passing through $x$. 
	
	In the special case $F_1=f^+$ we can simply take $f_1=f$ and otherwise 
	we apply $\GTH(\Pi_1, \equiv^{\Pi_1}, \ccP_1^+)>1$ to the big 
	cycle $F_1xf^+W^{\Pi_1}$, where $W^{\Pi_1}$ denotes the unique $m$-wagon 
	of $\seq{\Pi}_1$. Due to Lemma~\ref{lem:GTH1} we thus obtain the desired edge~$f_1$.
\end{proof}

\subsection{Amenable partite lemmata}
\label{subsec:napl}

Both Proposition~\ref{prop:0142} and Proposition~\ref{prop:0139}
assert that under certain inductive assumptions some karo principle holds. 
A commonality of their proofs is that they end with similar partite constructions 
that can be executed for roughly the same reasons. The main result of this subsection 
explains how this works. This involves the following concepts. 

\begin{dfn}\label{dfn:n1034}
	Let $\seq{g}\in \gM_\le^\times$ be a nonempty nondecreasing sequence. 
	Put $g=\inf(\seq{g})$, $m=|\seq{g}|$, and let $\seq{g}_\star\in\gM_\le$ 
	be obtained from $\seq{g}$ by removing its initial term $g$, so that 
	$\seq{g}=(g)\circ \seq{g}_\star$. 
	\index{Ramsey construction for $\seq{g}$-trains}
	\index{amenable partite lemma}
		\begin{enumerate}[label=\alabel]
		\item	We say that $\Phi$ is a {\it Ramsey construction 
		for $\seq{g}$-trains} if for every ordered $f$-partite
		train $\seq{F}$ of height $m$ with $\ggth(\seq{F})>\seq{g}$ 
		and every number of colours $r$ the train system 
		$\Phi_r(\seq{F})=(\seq{G}, \ccG)$ is defined,
		$\seq{G}$ is a linear ordered $f$-partite train of height~$m$
		with the same parameter as $\seq{F}$, the copies in $\ccG$ 
		are strongly induced, and $\ccG\lra (\seq{F})_r$.		
		\item A partite lemma $\Xi$ is said to be {\it $\seq{g}$-amenable}
		if for every $k$-partite $k$-uniform train~$\seq{F}$ of height $m$ 
		with $\ggth(\seq{F})>\seq{g}$ and every number of colours $r$ it 
		generates a train system $\Xi_r(\seq{F})=(\seq{H}, \ccH)$ such that 
		$\seq{H}$ has the same parameter as $\seq{F}$,
				\[
			\ccH\lra(\seq{F})_r\,, 
			\quad
			\GTH(\seq{H}, \ccH^+)>\seq{g}_\star\,, 
			\quad \text{ and } \quad 
			\Gth(H, \ccH^+)>(g, g)\,.
		\]
			\end{enumerate}
\end{dfn}

\begin{lemma}\label{lem:n1035}
	Let $\seq{g}\in \gM_\le^\times$ be a nonempty nondecreasing sequence. 
	If $\Phi$ denotes a Ramsey construction for $\seq{g}$-trains
	and the partite lemma $\Xi$ is $\seq{g}$-amenable, then $\PC(\Phi, \Xi)$
	exemplifies~$\karo_{\seq{g}}$.
\end{lemma}

\begin{proof}
	Consider an arbitrary ordered $f$-partite train $\seq{F}$ 
	with parameter $\seq{A}$ and $\ggth(\seq{F})>\seq{g}$  
	as well as number of colours $r$, and construct the train 
	system $\Phi_r(\seq{F})=(\seq{G}, \ccG)$. 
	Since $\Phi$ is a Ramsey construction for $\seq{g}$-trains,
	we know that $\seq{G}$ is a linear ordered $f$-partite train 
	with parameter $\seq{A}$, that the copies in $\ccG$ are strongly induced, 
	and that $\ccG\lra(\seq{F})_r$. Without loss of generality we can assume 
	that every edge of $G$ belongs to at least one copy in $\ccG$.
	
	As usual we let $\{e(1), \dots, e(N)\}$ enumerate the edges of the 
	underlying hypergraph of~$\seq{G}$. Now we run the partite 
	construction, thereby creating a sequence 
	$(\seq{\Pi}_\alpha, \ccP_\alpha, \psi_\alpha)_{\alpha\le N}$
	of train pictures. Picture zero can clearly be formed and by Fact~\ref{f:picnull}
	it has the property 
	$\GTH(\seq{\Pi}_0, \ccP^+_0)>\seq{g}$. 
	
	Now suppose that for some $\alpha\in [N]$ we have reached 
	the picture $(\seq{\Pi}_{\alpha-1}, \ccP_{\alpha-1}, \psi_{\alpha-1})$
	satisfying 
	$\GTH(\seq{\Pi}_{\alpha-1}, \ccP^+_{\alpha-1})>\seq{g}$.
	Corollary~\ref{cor:1921} entails 
	${\ggth(\seq{\Pi}_{\alpha-1}^{e(\alpha)})>\seq{g}}$ and the 
	parameter $\seq{D}_\alpha$ of $\seq{\Pi}_{\alpha-1}^{e(\alpha)}$
	has been described in Fact~\ref{f:N107}.
	By our assumptions on $\Xi$ there exists a train system 
		\[
		\Xi_r\bigl(\seq{\Pi}_{\alpha-1}^{e(\alpha)}\bigr)
		=
		(\seq{H}_\alpha, \ccH_\alpha)
	\]
		of height $m$ with parameter $\seq{D}_\alpha$ that satisfies 
		\[
		\GTH(\seq{H}_\alpha, \ccH^+_\alpha) > \seq{g}_\star
		\quad \text{ and } \quad 
		\Gth(H_\alpha, \ccH^+_\alpha) > (g, g)\,,
	\]
		where $\seq{g}=(g)\circ \seq{g}_\star$.
	In view of Lemma~\ref{lem:N128} the amalgamation 
		\[
		(\seq{\Pi}_\alpha, \ccP_\alpha, \psi_\alpha)
		=
		(\seq{\Pi}_{\alpha-1}, \ccP_{\alpha-1}, \psi_{\alpha-1})
		\conc
		(\seq{H}_\alpha, \ccH_\alpha)
	\]
		provides the next train picture over $(\seq{G}, \ccG)$ 
	and in order to keep the construction going we need to check 
	$\GTH(\seq{\Pi}_{\alpha}, \ccP_{\alpha}^+)>\seq{g}$. 
	
	To this end we observe that the hypothesis 
	$\GTH(\seq{\Pi}_{\alpha-1}, \ccP^+_{\alpha-1})>\seq{g}$ is equivalent to 
	the conjunction of
		\[
		\GTH(\Pi_{\alpha-1}, \equiv^{\Pi_{\alpha-1}}_0, \ccP^+_{\alpha-1})>g		
		\quad \text{ and } \quad 
		\GTH(\seq{\Pi}_{\alpha-1}, \ccP^+_{\alpha-1})>\seq{g}_\star\,.
	\]
		So Lemma~\ref{lem:n209} immediately yields 
		\begin{equation}\label{eq:n553}
		\GTH(\seq{\Pi}_{\alpha}, \ccP^+_{\alpha})>\seq{g}_\star
	\end{equation}
		and Lemma~\ref{lem:0217} tells us $\Gth(\Pi_{\alpha-1}, \ccP_{\alpha-1}^+)>g$.
	Combined with ${\Gth(H_\alpha, \ccH^+_\alpha) > (g, g)}$
	and Lemma~\ref{lem:251} this leads to $\Gth(\Pi_\alpha, \ccP_\alpha^+)>g$,
	and a further application of Lemma~\ref{lem:0217} translates this back to 
	$\GTH(\seq{\Pi}_{\alpha}, \equiv^{\Pi_\alpha}_0, \ccP_{\alpha}^+)>g$.
	Together with~\eqref{eq:n553} this establishes
	$\GTH(\seq{\Pi}_\alpha, \ccP_\alpha^+)>\seq{g}$ and thus the construction 
	can be continued. 
	
	Eventually, it produces a final picture $(\seq{\Pi}_N, \ccP_N, \psi_N)$
	such that $\GTH(\seq{\Pi}_N, \ccP_N^+)>\seq{g}$. Since $(G, \ccG)$ has 
	strongly induced copies, the same applies to $(\Pi_N, \ccP_N)$ 
	(see Lemma~\ref{lem:cleancap} and its proof). So altogether, when viewed
	as an ordered $f$-partite structure the last picture is as required 
	by~$\karo_{\seq{g}}$.
\end{proof}
		    
\subsection{Girth resurrection}
\label{subsec:ngr}

Now we shall finally complete the proof of Proposition~\ref{prop:0142}.
So we consider an arbitrary nonempty nondecreasing sequence 
$\seq{g}\in\gM^\times_\le$ and assume $\karo_{\seq{g}}$.
Let~$\ups$ be the construction provided by Lemma~\ref{lem:1904}.
Evidently $\ups$ is a Ramsey construction for $((2)\circ\seq{g})$-trains. 
So in view of Lemma~\ref{lem:n1035} it only remains to exhibit an
$((2)\circ\seq{g})$-amenable partite lemma. 

By restricting our attention to unordered $k$-partite $k$-uniform trains
we can regard $\ups$ as a partite lemma for trains and 
in this manner $\Xi=\PC(\ups, \ups)$ becomes a partite lemma as well. 
It will turn out that $\Xi$ has the required amenability property. 

\begin{lemma}\label{lem:2222}
	For every $k$-partite $k$-uniform train $\seq{F}$ of height~$m+1$  
	with $\ggth(\seq{F})>(2)\circ\seq{g}$ and every number of colours $r$ 
	the train system $\Xi_r(\seq{F})=(\seq{H}, \ccH)$ is defined.
	Moreover, the train $\seq{H}$ has the same parameter 
	as $\seq{F}$, the copies in $\ccH$ are strongly induced, their intersections 
	are clean, and $\GTH(\seq{H}, \ccH^+)>\seq{g}$.
\end{lemma}

\begin{proof}
	Owing to Lemma~\ref{lem:1904} the $k$-partite $k$-uniform train 
	system $\ups_r(\seq{F})=(\seq{G}, \ccG)$
	satisfies $\ccG\lra (\seq{F})_r$, its copies are strongly induced, and the 
	parameters of $\seq{F}$ and $\seq{G}$ are the same. Without loss of generality 
	we can assume that every edge of $G$ belongs to some copy in $\ccG$.
	
	Let $E(G)=\{e(1), \dots, e(N)\}$ enumerate the edges of $G$. We will 
	create a sequence $(\seq{\Pi}_\alpha, \ccP_\alpha, \psi_\alpha)_{\alpha\le N}$
	of train pictures over $(\seq{G}, \ccG)$ by means of the partite construction 
	method. Clearly, picture zero $(\seq{\Pi}_0, \ccP_0, \psi_0)$ 
	exists and it has strongly induced copies with clean intersections. Moreover,
	Fact~\ref{f:picnull} yields $\GTH(\seq{\Pi}_0, \ccP_0^+)>(2)\circ \seq{g}$, 
	which implies, in particular, $\GTH(\seq{\Pi}_0, \ccP_0^+)>\seq{g}$.
	
	Now suppose inductively that for some positive $\alpha\le N$ we have just constructed 
	a train picture $(\seq{\Pi}_{\alpha-1}, \ccP_{\alpha-1}, \psi_{\alpha-1})$
	such that the copies in $\ccP_{\alpha-1}$ are strongly induced, their intersections 
	are clean, and $\GTH(\seq{\Pi}_{\alpha-1}, \ccP_{\alpha-1}^+)>\seq{g}$.
	   
	Combined with Lemma~\ref{lem:2310} and the linearity of $\Pi_{\alpha-1}$ this 
	shows, in particular, that $\ggth(\seq{\Pi}_{\alpha-1}^{e(\alpha)})>(2)\circ\seq{g}$
	and thus Lemma~\ref{lem:1904} yields a $k$-partite $k$-uniform train system
		\[
		\ups_r(\seq{\Pi}_{\alpha-1}^{e(\alpha)})=(\seq{H}_\alpha, \ccH_\alpha)
	\]
		with the correct parameter satisfying $\GTH(\seq{H}_\alpha, \ccH_\alpha)>\seq{g}$.
	Due to Lemma~\ref{lem:N128} the next train picture 
		\[
		(\seq{\Pi}_\alpha, \ccP_\alpha, \psi_\alpha)
		=
		(\seq{\Pi}_{\alpha-1}, \ccP_{\alpha-1}, \psi_{\alpha-1})
		\conc
		(\seq{H}_\alpha, \ccH_\alpha)
	\]
		exists. By Lemma~\ref{lem:cleancap} the copies in $\ccP_\alpha$ are again 
	strongly induced and their intersections are clean. Moreover, Lemma~\ref{lem:n209}
	tells us $\GTH(\seq{\Pi}_\alpha, \ccP_\alpha^+)>\seq{g}$. 
	Thus the construction goes on and the final train picture we eventually reach has the 
	desired properties. 	
\end{proof}

By Lemma~\ref{lem:Gth22} the systems $(\seq{H}, \ccH)$ 
produced by $\Xi$ satisfy $\Gth(H, \ccH)>(2, 2)$ and, consequently, $\Xi$ 
is indeed $((2)\circ\seq{g})$-amenable. Thus the discussion at the beginning 
of this subsection shows that the construction $\PC(\ups, \Xi)$ exemplifies 
$\karo_{(2)\circ\seq{g}}$. This concludes the proof of Proposition~\ref{prop:0142}.  
For later reference we formulate a straightforward consequence of 
the results we currently have. 

\begin{cor}\label{cor:k2l}
	For every $\seq{g}\in \gM^\times$ there exists a Ramsey construction
	for $\seq{g}$-trains. 
\end{cor}

\begin{proof}
	Due to Lemma~\ref{lem:0120} and Proposition~\ref{prop:0142} we have 
	$\karo_{(2)^m}$ for every positive integer~$m$. Clearly, every construction 
	exemplifying $\karo_{(2)^m}$ for $m=|\seq{g}|$ is a Ramsey construction 
	for $\seq{g}$-trains.
\end{proof}

\subsection{Revisability}
\label{sssec:1746}

The remainder of this section is devoted to the proof of Proposition~\ref{prop:0139}. 
To this end we fix a (possibly empty) nondecreasing sequence $\seq{g}$ as well 
as an integer $g$ such that $\inf(\seq{g}) > g \ge 2$. Let $\ell$ be the length 
of $\seq{g}$ and write $\seq{g}=(g_1, \dots, g_\ell)$. We assume that
for every positive integer $m$ there exists a construction $\Psi^m$
exemplifying the principle $\karo_{(g)^m\circ\seq{g}}$. 
Our deduction of $\karo_{(g+1)\circ\seq{g}}$ consists of three steps. 
\begin{enumerate}
	\item[$\bullet$] First, we explain how the constructions $\Psi^m$
		are going to be used. 
	\item[$\bullet$] Second, we perform a ``diagonal'' partite construction 
		in order to derive 
		a $((g+1)\circ\seq{g})$-amenable partite lemma. 
	\item[$\bullet$] Third, we conclude $\karo_{(g+1)\circ\seq{g}}$.
\end{enumerate}	

The given constructions $\Psi^m$ will only be used as partite  lemmata, i.e., 
we shall only apply them to unordered $k$-partite $k$-uniform trains. 
More precisely, each $\Psi^m$ will be reinterpreted as a partite lemma applicable 
to certain trains $\seq{F}$ of height $\ell+1$ that will be called $m$-revisable.
Here is the definition of this concept. 

\begin{dfn}\label{dfn:N206}
	A $k$-partite $k$-uniform train $\seq{F}=(F, \equiv_0, \dots, \equiv_{\ell+1})$ 
	of height $\ell+1$ with index set $I$ and 
	parameter $\seq{A}=(A_1, \dots, A_{\ell+1})$ 
		is said to be {\it $m$-revisable} for 
	a positive integer $m$ if there exist equivalence 
	relations $\sim_1, \dots, \sim_{m-1}$ on $E(F)$ and an 
	$m$-tuple $\seq{B}=(B_1, \dots, B_m)\in \powerset(A_1)^m$
	such that 
		\begin{enumerate}
		\item[$\bullet$] $\seq{F}_\bullet
		=(F, \equiv_0, \sim_1, \dots, \sim_{m-1}, \equiv_1, \ldots, \equiv_{\ell+1})$ 
		is a train of height $\ell+m$ with 
		parameter $\seq{B}\circ(A_2, \dots, A_{\ell+1})$; 
	\item[$\bullet$] $\ggth(\seq{F}_\bullet)>(g)^m\circ \seq{g}$;
		\item[$\bullet$] and $|B_\mu|\le 1$ for every $\mu\in [m]$.
	\end{enumerate}
	In this situation we call the pair $(\seq{F}_\bullet, \seq{B})$ an 
	{\it $m$-revision} of $\seq{F}$.
	\index{revisability}
\end{dfn}

We illustrate this notion with an easy example that will later 
be applied to the constituents of a train picture zero. 

\begin{fact}\label{f:pnc}
	Let $\seq{F}$ be a $k$-partite $k$-uniform train of height $\ell+1$. 
	If its underlying hypergraph is a matching, possibly together with some 
	isolated vertices, then $\seq{F}$ is $1$-revisable. 
\end{fact}

\begin{proof}
	For every $\lambda\in [\ell+1]$ any two distinct $(\lambda-1)$-wagons 
	of $\seq{F}$ are vertex-disjoint. Thus we 
	have $\ggth(\seq{F})>(g)\circ \seq{g}$ and $(\seq{F}, (\vn))$
	is a $1$-revision of $\seq{F}$. 
\end{proof}

Here is a necessary condition for revisability that could be shown to be quite
far from being sufficient. 

\begin{fact}\label{f:revis}
	If a $k$-partite $k$-uniform train of height $\ell+1$ is $m$-revisable 
	for some positive integer $m$, then $\ggth(\seq{F})>(g+1)\circ \seq{g}$.
\end{fact}

\begin{proof}
	Let $\seq{A}$, $\seq{B}$, and $\seq{F}_\bullet$ be as in Definition~\ref{dfn:N206}.
	We need to prove that for every $\lambda\in [\ell+1]$ and every $\lambda$-wagon $W$
	of $\seq{F}$ the $(\lambda-1)$-wagons in $W$ form a set system whose girth exceeds
	the $\lambda^{\mathrm{th}}$ entry of $(g+1)\circ\seq{g}$. For $\lambda\ne 1$ 
	this follows immediately from $\ggth(\seq{F}_\bullet)>(g)^m\circ \seq{g}$.
 
	Thus it remains to show $\gth(W)>g+1$ for every $1$-wagon $W$ of $\seq{F}$. 
	To this end we notice that by restricting the first $m+1$ equivalence 
	relations of $\seq{F}_\bullet$ to~$E(W)$ 
	we obtain a train 
	$\seq{W}=(W, \equiv_0^W, \sim_1^W, \dots, \sim_{m-1}^W, \equiv^W_1)$
	of height $m$. We know $\ggth(\seq{W})>(g)^m$ and the parameter of $\seq{W}$
	is $\seq{B}=(B_1, \dots, B_m)$.
	As we required $|B_\mu|\le 1$ for every $\mu\in [m]$, all assumptions of 
	Lemma~\ref{lem:0036} are satisfied and we have indeed $\gth(W)>g+1$.
\end{proof}

Given an $m$-revisable $k$-partite $k$-uniform train $\seq{F}$ 
as well as a number of colours~$r$ we can take an arbitrary $m$-revision
$(\seq{F}_\bullet, \seq{B})$ of $\seq{F}$ and construct the train system 
\begin{equation}\label{eq:N206}
	\Psi^m_r(\seq{F}_\bullet)
	=
	(H, \equiv_0^H, \sim_1^H, \dots, \sim_{m-1}^H, 
		\equiv_1^H, \ldots, \equiv_{\ell+1}^H, \ccH_\bullet)\,.
\end{equation}
The outcome does not depend on $\seq{F}$ alone, for there could be many
distinct $m$-revisions of~$\seq{F}$. 
This ambiguity, however, is irrelevant to the main concern of this subsection 
and it will be convenient to write 
\[
	\Psi^m_r(\seq{F})
	=
	(\seq{H}, \ccH)\,,
\]
where $\seq{H}=(H, \equiv_0^H, \ldots, \equiv_{\ell+1}^H)$ and the copies in
$\ccH\subseteq \binom{\seq{H}}{\seq{F}}$ are obtained from the copies 
of~$\seq{F}_\bullet$ in $\ccH_\bullet$ by forgetting the $\sim$-relations. 

\begin{lemma}\label{clm:0158}
	If $\seq{F}$ denotes an $m$-revisable $k$-partite $k$-uniform train of 
	height $\ell+1$ with parameter $\seq{A}$, and $r$ is a number of colours, 
	then the train system 
	$\Psi^m_r(\seq{F})=(\seq{H}, \ccH)$ has again the parameter $\seq{A}$ 
	and satisfies		
	\[
		\ccH\lra (\seq{F})_r\,, \quad
		\ggth(\seq{H}) > (g+1)\circ \seq{g},
		\quad\text{ and }\quad
		\GTH(\seq{H}, \ccH^+)>(g)\circ\seq{g}\,.
	\]
\end{lemma} 

\begin{proof}
	We keep using the notation from Definition~\ref{dfn:N206} and~\eqref{eq:N206}.
	As a first step we shall show that $\seq{A}$ parametrises $\seq{H}$. 
	Due to~$\karo_{(g)^m\circ\seq{g}}$ the train 
		\[
		\seq{H}_\bullet
		=
		(H, \equiv_0^H, \sim_1^H, \dots, \sim_{m-1}^H, 
		\equiv_1^H, \ldots, \equiv_{\ell+1}^H)
	\]
		has the same parameter as $\seq{F}_\bullet$, i.e., 
	$\seq{B}\circ (A_2, \dots, A_{\ell+1})$.
	Setting for simplicity $\sim^H_0=\equiv^H_0$ and $\sim^H_m=\equiv^H_1$ this means
	that for any two edges $e, e'\in E(H)$ the following two statements hold.  
		\begin{enumerate}[label=\nlabel]
		\item\label{it:19-1} If $\mu\in [m]$, $e\sim^H_\mu e'$, 
			and $e\not\sim^H_{\mu-1} e'$, then $e\cap e'\subseteq V_{B_\mu}(H)$.
		\item\label{it:19-2} If $\lambda\in [2, \ell+1]$, $e\equiv^H_\lambda e'$, 
			and $e\not\equiv^H_{\lambda-1} e'$,
			then $e\cap e'\subseteq V_{A_\lambda}(H)$.
	\end{enumerate}
		
	In view of~\ref{it:19-2} it only remains to be shown that all 
	edges $e, e'\in E(H)$ with $e\equiv^H_1 e'$ and $e\not\equiv^H_0 e'$
	satisfy $e\cap e'\subseteq V_{A_1}(H)$. Due to $e\sim^H_m e'$ there 
	exists a smallest integer $\mu\in [0, m]$ such that $e\sim^H_\mu e'$. 
	Since $e\not\sim^H_0 e'$, we have $\mu>0$, and thus the minimality of $\mu$ 
	yields $e\not\sim^H_{\mu-1} e'$. So~\ref{it:19-1} 
	reveals $e\cap e'\subseteq V_{B_\mu}(H)$ and because of $B_\mu\subseteq A_1$
	the desired inclusion $e\cap e'\subseteq V_{A_1}(H)$ follows.
		
 	This concludes our discussion of parameters and we proceed with the three displayed 
	properties of the train system $(\seq{H}, \ccH)$.	
	As the train construction~$\Psi^m$ 
	exemplifies the principle $\karo_{(g)^m\circ\seq{g}}$,
	the partition relation $\ccH\lra (\seq{F})_r$ is clear. 
	
	Moreover, we have 
		\begin{equation}\label{eq:n713}
		\GTH(\seq{H}_\bullet, \ccH_\bullet^+)>(g)^m\circ\seq{g}\,.
	\end{equation}
		According to Definition~\ref{dfn:0115} this implies		\begin{enumerate}		\item[$\bullet$] $\GTH(H, \equiv^H_0, \ccH^+)>g$,
						\item[$\bullet$] $\GTH(H, \equiv^H_\lambda, \ccH^+)>g_\lambda$
			for every $\lambda\in [\ell]$,
		\item[$\bullet$] and $\GTH(H, \equiv^H_{\ell+1}, \ccH^+)>1$,
	\end{enumerate}
		which in turn yields $\GTH(\seq{H}, \ccH^+)>(g)\circ\seq{g}$.
	By Corollary~\ref{cor:1921} and~\eqref{eq:n713} we 
	have $\ggth(\seq{H}_\bullet)>(g)^m\circ\seq{g}$.
	So $(\seq{H}_\bullet, \seq{B})$ is an $m$-revision of $\seq{H}$ 
	and $\ggth(\seq{H})>(g+1)\circ\seq{g}$ follows from Fact~\ref{f:revis}.	 
\end{proof}

This concludes the first step of our proof of Proposition~\ref{prop:0139}, 
i.e., the conversion of the given $\karo$-principles 
into partite lemmata $\Psi^m$ applicable to certain trains $\seq{F}$ of height $\ell+1$.

\subsection{Train constituents}
\label{sssec:1748}

Recall that our initial motivation for introducing trains was the observation 
that partite constructions produce trains; moreover trains seem to offer a 
chance to handle hypergraphs of the next larger girth in the context of an 
argument by induction. In some sense, this 
subsection is the place of the whole article where all these hopes materialise.  

\begin{lemma}\label{lem:Xig}
	There exists a $((g+1)\circ\seq{g})$-amenable partite lemma $\Xi$.
\end{lemma}

This partite lemma $\Xi$ will be obtained by a diagonal variation 
of the partite construction method. Vertically we use an arbitrary 
Ramsey construction $\Phi$ for $((g+1)\circ\seq{g})$-trains 
(see Corollary~\ref{cor:k2l}). Horizontally the basic observation is that 
it was never written into stone that one has to employ the very same 
partite lemma in each stage of the iterative procedure. 
Rather, it is reasonable to adjust to the increasing complexity 
of the arising pictures by using more and more sophisticated partite lemmata 
as the construction progresses. 
In fact, we plan to use $\Psi^m$ when we need a partite lemma 
for the $m^{\mathrm{th}}$ time, 
so a suggestive notation for the construction we are about to describe 
could be
$\Xi=\PC\bigl(\Phi, (\Psi^{m})_{m\in \NN}\bigr)$.  
Here is the picturesque statement we shall iterate.

\begin{lemma}\label{lem:3111}
	Let $(\seq{G}, \ccG)$ be a $k$-partite $k$-uniform train system 
	of height $\ell+1$ with parameter~$\seq{A}$ 
	whose underlying hypergraph $G$ is linear. 
	Suppose further that 
	\[
		(\seq{\Sigma}, \ccQ, \psi_\Sigma)
		=
		(\seq{\Pi}, \ccP, \psi_\Pi)
		\conc
		(\seq{H}, \ccH)
	\]
		holds for two train pictures $(\seq{\Pi}, \ccP, \psi_\Pi)$, 
	$(\seq{\Sigma}, \ccQ, \psi_\Sigma)$ of height $\ell+1$ with parameter $\seq{A}$
	and a $k$-partite $k$-uniform train system $(\seq{H}, \ccH)$ of height $\ell+1$
	such that 
		\[
		\GTH(\seq{\Sigma}, \ccQ^+)>\seq{g}
		\quad \text{ and } \quad 
		\Gth(H, \ccH^+)>g \,.
	\]
		If this amalgamation occurs over the edge $e\in E(G)$,
		\begin{enumerate}
		\item[$\bullet$] another edge $e_\star \in E(G)$ is distinct to $e$, 
		\item[$\bullet$] and the constituent $\seq{\Pi}^{e_\star}$ is
			$\alpha$-revisable for some $\alpha\in\NN$, 
	\end{enumerate}
		then $\seq{\Sigma}^{e_\star}$ is $(\alpha+1)$-revisable. 
\end{lemma}

\begin{proof}
	Write $\seq{A}=(A_1, \dots, A_{\ell+1})$ and recall that according to 
	Fact~\ref{f:N107} the parameter $\seq{D}=(D_1, \dots, D_{\ell+1})$ 
	of the constituent $\seq{\Pi}^{e_\star}$ is given 
	by $D_\lambda=e_\star\cap V_{A_{\lambda}}(G)$ for every $\lambda\in [\ell+1]$.
	Let $(\seq{\Pi}^{e_\star}_\bullet, \seq{B})$ be an $\alpha$-revision 
	of $\seq{\Pi}^{e_\star}=(\Pi^{e_\star}, \equiv^\Pi_0, \dots, \equiv^\Pi_{\ell+1})$, 
	where 
		\[
		\seq{\Pi}^{e_\star}_\bullet
		=
		(\Pi^{e_\star}, \equiv^\Pi_0, \sim^\Pi_1, \dots, \sim^\Pi_{\alpha-1},
		\equiv^\Pi_1, \dots, \equiv^\Pi_{\ell+1})
		\quad \text{ and } \quad 
		\seq{B}=(B_1, \dots, B_\alpha)\in \powerset(D_1)^\alpha\,.
	\]
		
	Our task is to exhibit an $(\alpha+1)$-revision of the new 
	constituent 
		\[
		\seq{\Sigma}^{e_\star}
		=
		(\Sigma^{e_\star}, \equiv^\Sigma_0, \dots, \equiv^\Sigma_{\ell+1})\,.
	\]
		For every standard copy $\seq{\Pi}_\circ$ 
	of $\seq{\Pi}$ in $\seq{\Sigma}$ we can copy the above revision of 
	$\seq{\Pi}^{e_\star}$ onto the constituent $\seq{\Pi}^{e_\star}_\circ$,
	thus getting a train 
		\begin{equation}\label{eq:N329}
		(\Pi^{e_\star}_\circ, \equiv^{\Pi_\circ}_0, 
		\sim^{\Pi_\circ}_1, \dots, \sim^{\Pi_\circ}_{\alpha-1},
		\equiv^{\Pi_\circ}_1, \dots, \equiv^{\Pi_\circ}_{\ell+1})\,.
	\end{equation}
		It will be convenient to set $\sim^{\Pi_\circ}_0=\equiv^{\Pi_\circ}_0$
	and $\sim^{\Pi_\circ}_\alpha=\equiv^{\Pi_\circ}_1$ for every such standard copy. 
	The edge set of the new constituent~$\Sigma^{e_\star}$ is the disjoint union of 
	the edge sets of all $\Pi_\circ^{e_\star}$ as $\Pi_\circ$ varies over the standard 
	copies of $\seq{\Pi}$ in $\Sigma$. We can thus define a quasitrain 
		\[
		\seq{\Sigma}^{e_\star}_\bullet
		=
		(\Sigma^{e_\star}, \equiv^\Sigma_0, 
		\sim^\Sigma_1, \dots, \sim^\Sigma_\alpha, 
		\equiv^\Sigma_1, \dots, \equiv^\Sigma_{\ell+1})
	\]
		of height $\ell+\alpha+1$ by declaring for every $\mu\in[\alpha]$ 
	and any two edges $e', e''\in E(\Sigma^{e_\star})$ that the statement 
	$e'\sim^\Sigma_\mu e''$ means: there is a common standard copy $\Pi_\circ$
	containing $e'$, $e''$, and $e'\sim_\mu^{\Pi_\circ} e''$ holds.
	Next we define a set $B_{\alpha+1}$ by setting 
		\[
		B_{\alpha+1}
		=
		\begin{cases}
			e\cap e_\star & \text{ if $e\equiv^G_1 e_\star$} \cr
			\vn & \text{ if $e\not\equiv^G_1 e_\star$.}
		\end{cases}
	\]
		
	We shall eventually show 
	that $\bigl(\seq{\Sigma}^{e_\star}_\bullet, \seq{B}\circ (B_{\alpha+1})\bigr)$
	is an $(\alpha+1)$-revision of $\seq{\Sigma}^{e_\star}$.
	
	\begin{clm}
		 We have $B_{\alpha+1}\subseteq D_1$ and $|B_{\alpha+1}|\le 1$.
	\end{clm}
	
	\begin{proof}
		If $B_{\alpha+1}=\vn$ both assertions are clear, so we may assume 
		$B_{\alpha+1}=e\cap e_\star$ and $e\equiv^G_1 e_\star$ from now on. 
		The linearity of $G$ implies $|B_{\alpha+1}|\le 1$. 
		We also know $e\cap e_\star\subseteq V_{A_1}(G)$, because 
		$\seq{G}$ has the parameter $\seq{A}$. Thus we have indeed
				\[
			B_{\alpha+1}
			=
			e\cap e_\star
			\subseteq 
			e_\star\cap V_{A_1}(G)
			=
			D_1\,. \qedhere
		\]
			\end{proof}
	
	\begin{clm}
		The sequence $\seq{B}\circ (B_{\alpha+1}, D_2, \dots, D_{\ell+1})$ 
		parametrises $\seq{\Sigma}^{e_\star}_\bullet$.
	\end{clm}
	  
	\begin{proof}
		For every standard copy $\seq{\Pi}_\circ$ of $\seq{\Pi}$ in $\seq{\Sigma}$
		the train~\eqref{eq:N329} has the same parameter as~$\seq{\Pi}^{e_\star}_\bullet$,
		i.e., $\seq{B}\circ (D_2, \dots, D_{\ell+1})$ 
		and, therefore, the statements involving the sets $B_1, \dots, B_\alpha$ hold.
		Similarly, the claims on $D_2, \dots, D_{\ell+1}$ follow from $\seq{A}$ being 
		the parameter of the entire train $\seq{\Sigma}$ and Fact~\ref{f:N107}.
		
		So it remains to be shown that if two edges $e', e''\in E(\Sigma^{e_\star})$
		satisfy $e'\equiv^\Sigma_1 e''$ but $e'\not\sim_{\alpha}^\Sigma e''$, then 
		$e'\cap e''\in V_{B_{\alpha+1}}(\Sigma^{e_\star})$. Notice that this 
		can only happen if $e'$ and $e''$ are in distinct standard 
		copies of $\seq{\Pi}$. So in view of 
		Lemma~\ref{lem:1825}\ref{it:1832c}\ref{it:1832c2} there needs to exist,
		in particular, an edge $e_0\in E(H)$ satisfying 
		$e'\equiv^\Sigma_1 e_0\equiv^\Sigma_1 e''$. 
		Now Definition~\ref{dfn:7058}\ref{it:tp3} discloses 
		$e=\psi_\Sigma(e_0)\equiv^G_1\psi_\Sigma(e')=e_\star$, which in turn 
		leads to $B_{\alpha+1}=e\cap e_\star$. Since $\psi_\Sigma$ projects 
		all vertices of $e'\cap e''$, if there exist any, into $e\cap e_\star$,
		this proves that $e'\cap e''\in V_{B_{\alpha+1}}(\Sigma^{e_\star})$ 
		is indeed true. 
	\end{proof}

	Now it remains to be shown 
	that $\ggth(\seq{\Sigma}^{e_\star}_\bullet)>(g)^{\alpha+1}\circ\seq{g}$. 
	Most parts of this claim are straightforward consequences 
	of $\ggth(\seq{\Pi}^{e_\star}_\bullet)>(g)^\alpha\circ\seq{g}$ 
	and of $\GTH(\seq{\Sigma}, \ccQ^+)>\seq{g}$. In fact, the only part 
	requiring some attention is that if $W$ denotes an $(\alpha+1)$-wagon 
	of $\seq{\Sigma}^{e_\star}_\bullet$, i.e., a wagon 
	of $(\Sigma^{e_\star}, \equiv^\Sigma_1)$, then $\ggth(W, \sim_\alpha^W)>g$.
 
	Assume contrariwise that for some $n\in [2, g]$ there is an $n$-cycle
		\[
		\ccC = W_1v_1 \ldots W_nv_n\,,
	\]
		where $W_1, \dots, W_n$ denote wagons of $(W, \sim_\alpha^W)$ 
	and $v_1, \dots, v_n$ are vertices. Due to the definition of $\sim^\Sigma_\alpha$
	there exists for every $t\in\ZZ/n\ZZ$ a standard copy $\seq{\Pi}_t$ of $\seq{\Pi}$
	such that the wagon $W_t$ is entirely contained in $\seq{\Pi}^{e_\star}_t$.
	Moreover, $W_t\ne W_{t+1}$ yields $\seq{\Pi}_t\ne \seq{\Pi}_{t+1}$ and 
	due to $v_t\in V(W_t)\cap V(W_{t+1})$ we have $v_t\in V(H)$. We thus arrive 
	at a cycle of copies   
		\[
		\ccD = \Pi^{e}_1v_1 \ldots \Pi^{e}_nv_n
	\]
		in $(H, \ccH^+)$ whose length is~$n$ and all of whose connectors are vertices sitting 
	on the same music line of $\Sigma$ (namely the line projected to the unique vertex 
	in $e\cap e_\star$).   
	In particular, $\ccD$ is tidy and none of its copies is 
	collapsible, meaning that~$\ccD$ has no master copy. 
	But because of $\ord{\ccD}=n\le g$ this contradicts $\Gth(H, \ccH^+)>g$. 
\end{proof}

\begin{proof}[Proof of Lemma~\ref{lem:Xig}]
	Suppose that $\seq{F}$ is a $k$-partite $k$-uniform train of height $\ell+1$
	with 
		\begin{equation}\label{eq:N921}
		\ggth(F) > (g+1)\circ \seq{g}
	\end{equation}
		and parameter $\seq{A}$, and that $r$ is a number of colours. 
	Pick a Ramsey construction $\Phi$ for $((g+1)\circ \seq{g})$-trains 
	(cf.\ Corollary~\ref{cor:k2l}) and set $\Phi_r(\seq{F})=(\seq{G}, \ccG)$.
	Recall that $\seq{G}$ is a linear train of height $\ell+1$ with parameter~$\seq{A}$
	and that the copies in $\ccG$ are strongly induced. Without loss of generality 
	we can assume that all edges of $G$ belong to some copy in $\ccG$. 
	 
	Let $E(G)=\{e(1), \ldots, e(N)\}$ enumerate the edges of~$G$. 
	We intend to construct recursively a sequence of train 
	pictures $(\seq{\Pi}_\alpha, \ccP_\alpha, \psi_\alpha)_{0\le \alpha\le N}$ 
	over $(\seq{G}, \ccG)$ that starts with picture zero. This is to be done 
	in such a way that for every $\alpha\in [0, N]$ the $\alpha^{\mathrm{th}}$ 
	picture has the properties
		\begin{enumerate}[label=\alabel]
		\item\label{it:0811a} $\Gth(\Pi_\alpha, \ccP_\alpha^+) > (g+1, g+1)$;
		\item\label{it:0811b} $\GTH(\seq{\Pi}_\alpha, \ccP_\alpha^+)>\seq{g}$;
		\item\label{it:0811c} and for every $\beta\in (\alpha, N]$ the 
			constituent $\seq{\Pi}_\alpha^{e(\beta)}$
			is $(\alpha+1)$-revisable. 
	\end{enumerate}  

	Before we can move any further we need to check that picture zero has 
	these properties for $\alpha=0$. Fact~\ref{f:picnull} and~\eqref{eq:N921} imply 
	$\GTH(\seq{\Pi}_0, \ccP_0^+)>(g+1)\circ\seq{g}$, which entails~\ref{it:0811b} 
	immediately and in view of Lemma~\ref{lem:0217} clause~\ref{it:0811a} follows 
	as well. Furthermore, Fact~\ref{f:pnc} yields~\ref{it:0811c}.  
	
	Now let any positive $\alpha\le N$ be given for which we have already managed 
	to reach a train picture $(\seq{\Pi}_{\alpha-1}, \ccP_{\alpha-1}, \psi_{\alpha-1})$ 
	satisfying~\ref{it:0811a},~\ref{it:0811b}, and~\ref{it:0811c} for $\alpha-1$ 
	instead of $\alpha$. As a consequence of the last condition we know that its 
	constituent $\Pi_{\alpha-1}^{e(\alpha)}$ is $\alpha$-revisable. 
	So Lemma~\ref{clm:0158} provides a train system 
		\[
		\Psi^\alpha_r\bigl(\seq{\Pi}_{\alpha-1}^{e(\alpha)}\bigr) 		
		=
		(\seq{H}_{\alpha}, \ccH_{\alpha})
	\]
		with the correct parameter satisfying 
		\[
		\ccH_\alpha\lra \bigl(\seq{\Pi}_{\alpha-1}^{e(\alpha)}\bigr)_r\,, \quad
		\ggth(\seq{H}_\alpha) > (g+1)\circ\seq{g},
		\quad\text{ and }\quad
		\GTH(\seq{H}_\alpha, \ccH_\alpha^+)>(g)\circ\seq{g}\,.
	\]
		
	In view of Lemma~\ref{lem:N128} the quasitrain picture 
		\[
		(\seq{\Pi}_{\alpha}, \ccP_{\alpha}, \psi_{\alpha})
		=
		(\seq{\Pi}_{\alpha-1}, \ccP_{\alpha-1}, \psi_{\alpha-1})
		\conc
		(\seq{H}_{\alpha}, \ccH_{\alpha})
	\]
		is actually a train picture over $(\seq{G}, \ccG)$. We need to check that 
	it has the properties~\ref{it:0811a}\,--\,\ref{it:0811c}.
	
	Observe first that $\ggth(\seq{H}_\alpha) > (g+1)\circ\seq{g}$, our assumption 
	$\inf(\seq{g})>g$, and Lemma~\ref{lem:1557} imply 
		\begin{equation}\label{eq:n135}
		\gth(H_\alpha)>g+1\,.
	\end{equation}
		Moreover, $\GTH(\seq{H}_\alpha, \ccH_\alpha^+)>(g)\circ\seq{g}$ contains the 
	information $\GTH(H_\alpha, \equiv^{H_\alpha}_0, \ccH_\alpha^+)>g$, which in view 
	of Lemma~\ref{lem:0217} leads to $\Gth(H_\alpha, \ccH_\alpha^+)>g$. 
	Together with our induction 
	hypothesis $\Gth(\Pi_{\alpha-1}, \ccP_{\alpha-1}^+) > (g+1, g+1)$
	and~\eqref{eq:n135}
	this allows us to apply Lemma~\ref{lem:1753} with $g+1$ here in place of $g$ there.
	This proves~\ref{it:0811a}.
	
	Clause~\ref{it:0811b} follows from Lemma~\ref{lem:n209}.	
	For the verification of~\ref{it:0811c} we consider an arbitrary 
	integer $\beta\in (\alpha, N]$. 
	We already know that $\Pi_{\alpha-1}^{e(\beta)}$
	is $\alpha$-revisable and we want to deduce the   
	$(\alpha+1)$-revisability of $\Pi_{\alpha}^{e(\beta)}$ 
	with the help of Lemma~\ref{lem:3111}. Its assumptions 
	${\GTH(\seq{\Pi}_\alpha, \ccP_\alpha^+)>\seq{g}}$
	and $\Gth(H_\alpha, \ccH_\alpha^+)>g$ have already been 
	established and $e(\beta)\ne e(\alpha)$ is clear. 
	Thereby~\ref{it:0811c} has been confirmed as well and the 
	partite construction continues.  
	
	Eventually we arrive at a final picture 
	$(\seq{\Pi}_N, \ccP_N, \psi_N)$, which due to~\ref{it:0811a}
	and~\ref{it:0811b} satisfies   
		\[
		\Gth(\Pi_N, \ccP_N^+) > (g+1, g+1)
		\quad \text{ as well as } \quad
		\GTH(\seq{\Pi}_N, \ccP_N^+)>\seq{g}\,.
	\]
		Thus the partite lemma $\Xi$ defined 
	by $\Xi_r(\seq{F})=(\seq{\Pi}_N, \ccP_N^+)$
	is $((g+1)\circ\seq{g})$-amenable. 	
\end{proof}
	
It is now immediate from Lemma~\ref{lem:n1035} that the construction 
$\PC(\Phi, \Xi)$, where $\Phi$ again denotes a Ramsey construction for 
$((g+1)\circ\seq{g})$-trains, exemplifies $\karo_{(g+1)\circ\seq{g}}$. 
This concludes the proof of Proposition~\ref{prop:0139}. 
In the light of our earlier work
Corollary~\ref{cor:1412}, Theorem~\ref{thm:6653}, and 
Theorem~\ref{thm:grth1} have thereby been demonstrated as well (see Summary~\ref{sum:1020}). 

\section{Paradise}
\label{sec:paradise}

This section is concerned with the proofs of   
Theorem~\ref{thm:1522} and Theorem~\ref{cor:19}. To this end we need
a connection between $\Gth$ and forests of copies, which is established 
in~\S\ref{subsec:6142}. The main result there (Lemma~\ref{lem:6822})
asserts that large $\Gth$ forces small sets of copies to be forests.
In~\S\ref{subsec:2142} this lemma will allow us to analyse a partite 
construction utilising the construction $\Omega^{(g)}$ provided by Theorem~\ref{thm:6653} 
as a partite lemma. As a result, we shall prove Theorem~\ref{thm:1522} in a very 
strong form (Theorem~\ref{thm:6643}). 
The deduction of Theorem~\ref{cor:19} will then be straightforward.

\subsection{Forests of copies}
\label{subsec:6142}

Let us recall that forests of copies were introduced in Definition~\ref{dfn:1606} 
as being certain systems $\ccN\subseteq \binom{H}{F}$ consisting of mutually 
isomorphic copies. 

Marginally generalising the underlying setup, we shall henceforth say for a linear system 
of hypergraphs $(H, \ccN)$ that~$\ccN$ is a {\it forest of copies}
if one can write $\ccN=\{F_1, \ldots, F_{|\ccN|}\}$ in such a way that 
for every $j\in [2, |\ccN|]$ the set 
$z_j=V(F_j)\cap\bigl(\bigcup_{i<j}V(F_i)\bigr)$ is either an edge 
in $E(F_j)\cap\bigl(\bigcup_{i<j}E(F_i)\bigr)$ or it consists of at most 
one vertex. 
\index{forest of copies}
In this situation, $(F_1, \ldots, F_{|\ccN|})$
is called an {\it admissible enumeration} of $\ccN$. 
\index{admissible enumeration}
It may be helpful to point out that forests of copies are thereby allowed 
to contain both edge copies and real copies. 

Most results in this subsection can be regarded as appropriate adaptations
of well-known facts on ordinary forests---a notion brief\-ly recapitulated immediately 
before Theorem~\ref{thm:7338}. To distinguish such forests from forests of copies we shall
call them {\it edge forests} in the discussion that follows. 
\index{edge forest}
Here is a list of four standard facts on edge forests. 

\begin{enumerate}[label=\rmlabel]
	\item\label{it:6353a} Every subset of an edge forest is again an edge forest.
	\item\label{it:6353b} The result of gluing two edge forests together along a 
		vertex or an edge is again an edge forest. 
	\item\label{it:6353c} If a hypergraph $H=(V, E)$ satisfies $\gth(H) > |E|$, then
		$E$ is an edge forest. 
	\item\label{it:6353d} Every edge forest consisting of at least two edges has 
		at least two leaves (belonging to different edges).
\end{enumerate}

Now we already saw in the introduction that forests of copies are not closed 
under taking subsets, meaning that the na\"{i}ve generalisation of~\ref{it:6353a} fails
(see Figure~\ref{fig:11}).
Nevertheless, we shall show in Lemma~\ref{lem:6140} that deleting edge copies preserves 
being a forest of copies. In Lemma~\ref{lem:7032} and Lemma~\ref{lem:6032} we shall 
then see that~\ref{it:6353b} generalises in the obvious way. 
Statement~\ref{it:6353c} remains valid if we replace lower case $\gth$ by 
capital $\Gth$ (see Lemma~\ref{lem:6822}). Finally, when dealing 
with~\ref{it:6353d} we shall work with the following analogue of edges acting as ``leaves''.
 
\begin{dfn}   
	Given a forest of copies $\ccN$ and a copy $F_\star\in\ccN$ we say that $F_\star$ 
	is {\it terminal in $\ccN$} if there exists an admissible enumeration
	$(F_1, \ldots, F_{|\ccN|})$ of $\ccN$ such that $F_\star=F_{|\ccN|}$.
	\index{terminal copy}
\end{dfn}

For instance, if $\ccN$ denotes a forest consisting of two copies, then both enumerations 
of $\ccN$ are admissible and, consequently, both copies in $\ccN$ are terminal. 
Now the variant of~\ref{it:6353d} we have been alluding to reads as follows.

\begin{lemma}\label{lem:7207}
	Every forest of at least two copies possesses at least two terminal copies. 
\end{lemma}
  
In the argument that follows and at several other occasions occurring in the rest of 
this section it will be convenient to employ the following notation. Given a set of 
hypergraphs $\ccN$, we shall write $V(\ccN)$ for $\bigcup_{F_\star\in\ccN}V(F_\star)$
and, similarly, $E(\ccN)$ will abbreviate $\bigcup_{F_\star\in\ccN}E(F_\star)$.
   
\begin{proof}[Proof of Lemma~\ref{lem:7207}]
	We argue by induction on the size $|\ccN|$ of the forest of copies 
	under consideration. The base case $|\ccN|=2$ has already been studied. 
	For the induction step we suppose that~$\ccN$
	is a forest consisting of $n\ge 3$ copies and that every forest 
	of copies~$\ccN'$ with $|\ccN'|=n-1$ has at least two terminal copies. 
	
	Since $\ccN$ admits an admissible enumeration, it possesses some terminal 
	copy~$F_\star$. Notice that the 
	set 
		\[
		z_\star
		=
		V(F_\star)\cap V(\ccN\sm\{F_\star\})
	\]
		is either an edge in $E(F_\star)\cap E(\ccN\sm\{F_\star\})$
	or it consists of at most one vertex. Therefore there exists a copy 
	$F_\circ\in \ccN\sm\{F_\star\}$ such that 
	\begin{enumerate}		\item[$\bullet$] $z_\star\subseteq V(F_\circ)$
		\item[$\bullet$] and if $z_\star$ is an edge, then $z_\star\in E(F_\circ)$.
	\end{enumerate} 	  
	
	The terminality of $F_\star$ implies that $\ccN\sm\{F_\star\}$ is a forest 
	of copies and the induction hypothesis yields an admissible 
	enumeration $(F_1, \ldots, F_{n-1})$ of $\ccN\sm\{F_\star\}$
	whose terminal copy is distinct from $F_\circ$. We contend that $F_{n-1}$
	is terminal in $\ccN$ as well and that the enumeration 
		\begin{equation}\label{eq:6310}
		(F_1, \ldots, F_{n-2}, F_\star, F_{n-1})
	\end{equation}
		exemplifies this fact. The claims that we need to check concerning the 
	copies $F_2, \ldots, F_{n-2}$ follow immediately 
	from the admissibility of $(F_1, \ldots, F_{n-1})$, and the claim 
	on $F_\star$ is a consequence of $F_\circ\in \{F_1, \ldots, F_{n-2}\}$.
	It remains to verify that the set 
		\[	
		z^+ 
		= 
		V(F_{n-1}) 
		\cap 
		V\bigl(\{F_1, \ldots, F_{n-2}, F_\star\}\bigr)
	\]
		is either an edge 
	in $E(F_{n-1})\cap E\bigl(\{F_1, \ldots, F_{n-2}, F_\star\}\bigr)$ or that it 
	consists of at most one vertex. 
	Due to the admissibility of $(F_1, \ldots, F_{n-1})$ we know that 
	the set 
		\[
		z^- 
		= 
		V(F_{n-1}) 
		\cap 
		V\bigl(\{F_1, \ldots, F_{n-2}\}\bigr)
	\]
		has analogous properties and thus it suffices to prove $z^+=z^-$.
	Because of 
		\[
		V(F_{n-1})\cap V(F_\star)
		\subseteq 
		V(F_{n-1})\cap z_\star 
		\subseteq 
		V(F_{n-1})\cap V(F_\circ)
		\subseteq 
		z^-
	\]
		and 
		\[
		z^+=z^-\cup\bigl(V(F_{n-1})\cap V(F_\star)\bigr)
	\]
		this is indeed the case.  
	Altogether we have thereby proved that the enumeration~\eqref{eq:6310} is 
	admissible. Hence $F_\star$ and $F_{n-1}$ are two terminal copies of $\ccN$, 
	and the induction step is complete.
\end{proof}      

The concept ``dual'' to a terminal copy is that of an {\it initial copy}, 
by which we mean a copy in a forest that is capable of standing in the first position 
of an admissible enumeration.
\index{initial copy}

\begin{lemma}\label{lem:6339}
	In a forest of copies every copy is initial. 
\end{lemma}

\begin{proof}
	Again we argue by induction on the size of the forest under consideration. 
	In the base case, when the forest consists of at most two copies, every 
	enumeration is admissible and, therefore, every copy is initial.
	
	For the induction step we look at a forest of copies $\ccN$ 
	with $n=|\ccN|\ge 3$ and at an arbitrary copy $F_1\in \ccN$. 
	Lemma~\ref{lem:7207}
	discloses that $\ccN$ has a terminal copy $F_n\ne F_1$. 
	Now $\ccN'=\ccN\sm\{F_n\}$ is again a forest of copies and due to the
	induction hypothesis $F_1$ is an initial copy of $\ccN'$, i.e., there is 
	an admissible enumeration $(F_1, \ldots, F_{n-1})$ of $\ccN'$ that starts 
	with $F_1$. Clearly $(F_1, \ldots, F_n)$ is an admissible enumeration 
	of $\ccN$ and, consequently, $F_1$ is indeed an initial copy of $\ccN$. 
\end{proof}

Next we address ``subforests'' of copies obtained by the removal of edge copies.

\begin{lemma}\label{lem:6140}
	The result of deleting an edge copy from a forest of copies is again 
	a forest of copies. 
\end{lemma}

\begin{proof}
	Let $\ccN$ be a forest of copies containing some edge copy $e^+$.
	We are to prove that $\ccN'=\ccN\sm\{e^+\}$ is a forest of copies as well. 
	To avoid trivialities we may suppose that $\ccN'\ne \vn$. If $e\in E(\ccN')$ 
	we choose a copy $F_1\in \ccN'$ satisfying $e\in E(F_1)$ and otherwise we 
	let $F_1\in \ccN'$ be arbitrary. By Lemma~\ref{lem:6339} there is 
	an admissible enumeration 
		\[
	 (F_1, \ldots, F_i, e^+, F_{i+1}, \ldots, F_{|\ccN'|})
	\]
		of $\ccN$ starting with $F_1$ and one checks immediately that 
	the enumeration $(F_1, \ldots, F_{|\ccN'|})$ of~$\ccN'$ obtained by deleting $e^+$
	is admissible. 
\end{proof}

Roughly speaking, the lemma that follows asserts that if we glue two forests 
of copies $\ccA$ and $\ccB$ together at a vertex, then the result is again 
a forest of copies. 
 
\begin{lemma}\label{lem:7032}
	Suppose that $(H, \ccN)$ is a linear system of hypergraphs admitting a 
	partition $\ccN=\ccA\dcup \ccB$ such that 
	$|V(\ccA)\cap V(\ccB)|\le 1$.
	If both $\ccA$ and $\ccB$ are forests of copies, then so is $\ccN$. 
\end{lemma}
 
\begin{proof}
	The case $\ccB=\vn$ being clear we may suppose that $\ccB$ contains at least 
	one copy. Since the set $z=V(\ccA)\cap V(\ccB)$ consists of at most one vertex,
	there exists a copy $F'_1\in \ccB$ such that $z\subseteq V(F'_1)$. Now 
	Lemma~\ref{lem:6339} shows that $\ccB$ has an admissible 
	enumeration $(F'_1, \ldots, F'_{|\ccB|})$ starting with $F'_1$.
	It is not hard to check that if $(F_1, \ldots, F_{|\ccA|})$ is an arbitrary
	admissible enumeration of $\ccA$, then the enumeration 
	$(F_1, \ldots, F_{|\ccA|}, F'_1, \ldots, F'_{|\ccB|})$ of $\ccN$ is admissible 
	as well. 
\end{proof}

There is a similar statement addressing the case that the glueing occurs 
in an entire edge of the underlying hypergraph.  

\begin{lemma}\label{lem:6032}
	Let $(H, \ccN)$ be a linear system of hypergraphs and let 
	$\ccN=\ccA\dcup \ccB$ be a partition such that $V(\ccA)\cap V(\ccB)\subseteq e$ 
	holds for some edge $e$ of $H$.
	If both $\ccA\cup\{e^+\}$ and $\ccB\cup\{e^+\}$ are forests of copies, 
	then so is $\ccN$. 
\end{lemma}

\begin{proof}
	Take an arbitrary admissible enumeration $(F_1, \ldots, F_a)$ of $\ccA\cup\{e^+\}$,
	as well as an admissible enumeration $(e^+, F'_2, \ldots, F'_b)$ of $\ccB\cup\{e^+\}$, 
	where $a=|\ccA\cup\{e^+\}|$ and $b=|\ccB\cup\{e^+\}|$, respectively. 
	Now $(F_1, \ldots, F_a, F'_2, \ldots, F'_b)$ is an admissible enumeration
	of $\ccN\cup\{e^+\}$. If $e^+\in \ccN$ this proves that $\ccN$ is indeed 
	a forest of copies and in case $e^+\not\in \ccN$ Lemma~\ref{lem:6140} leads us to the 
	same conclusion.  
\end{proof}

An iterative application of the next result will allow us to relate
forests of copies to $\Gth$. 

\begin{lemma}\label{lem:6349}
	Given a linear system of hypergraphs $(H, \ccN)$ 
	with ${\Gth(H, \ccN^+)>|\ccN|\ge 2}$ there exists a copy $F_\star\in\ccN$
	such that the set 
		\begin{equation}\label{eq:6622}
		z_\star=V(F_\star)\cap V(\ccN\sm\{F_\star\})
	\end{equation}
		is either an edge in $E(F_\star)\cap E(\ccN\sm\{F_\star\})$  
	or it consists of at most one vertex.
\end{lemma}

\begin{proof}
	Assume for the sake of contradiction that no such copy $F_\star$ exists.
	We aim at building a tidy cycle of copies in $(H, \ccN^+)$ consisting 
	of at most $|\ccN|$ distinct copies, which does not possess a master copy. 
	Looking at the 
	potential connectors of such a cycle, we call a vertex $x\in V(H)$
	{\it useful} if there exist distinct copies $F', F''\in \ccN$ such that 
	$x\in V(F')\cap V(F'')$. 
	Similarly, an edge $e\in E(H)$ is said to be 
	{\it useful} if there are distinct copies $F', F''\in \ccN$ satisfying
	$e\in E(F')\cap E(F'')$. 
	For instance, for every copy $F_\star\in\ccN$ all vertices belonging to the 
	corresponding set $z_\star$ defined in~\eqref{eq:6622} are useful. 
	
	\begin{clm}\label{clm:1455}
		For every useful vertex $x$ there are two distinct copies $F'_x, F''_x\in\ccN$
		such that	
		\begin{enumerate}[label=\rmlabel]
			\item\label{it:1455i} $x\in V(F'_x)\cap V(F''_x)$;
			\item\label{it:1455ii} and each of $F'_x$, $F''_x$ has at most one 
				useful edge containing $x$.
		\end{enumerate}  
	\end{clm}
	
	\begin{proof}
		Given $x$ we consider an auxiliary set system $S_x$ with vertex set 
				\[
			V(S_x) = \bigl\{F_\star\in \ccN\colon x\in V(F_\star)\bigr\}\,.
		\]
				For every useful edge $e$ containing $x$ the set
		$\phi_e =\bigl\{F_\star\in V(S_x)\colon e\in E(F_\star)\bigr\}$
		has at least the size $2$. Thus we can define the edge set of our set system $S_x$
		by 
				\[
			E(S_x) = \bigl\{\phi_e\colon \text{$e$ is useful and $x\in e$}\bigr\}\,.
		\]
				
		We are to prove that $S_x$ has at least two vertices whose degree
		is at most one. Due to $v(S_x)\ge 2$ the failure of this statement would imply 
		that $S_x$ contains some cycle $\phi_{e(1)}F_1\ldots \phi_{e(n)}F_n$ (in the 
		sense of Definition~\ref{dfn:girth}). But then $F_1e(2)\ldots F_ne(1)$
		is a tidy cycle of copies in $(H, \ccN)$ all of whose connectors are edges,
		which in view of 
				\[
			\Gth(H, \ccN^+) > |\ccN| \ge v(S_x)\ge n
		\]
				contradicts Lemma~\ref{lem:1447}. 
	\end{proof}
	
	Next we shall construct a cycle of copies in $(H, \ccN)$ which has only vertex
	connectors and some further special property. 
	We commence by examining desirable subsequences 
	of the form $q'F_\star q''$ of the envisaged cycle. 
	Given a copy $F_\star\in\ccN$ and distinct useful vertices 
	$q', q''\in V(F_\star)$ 
	we say that {\it $F_\star$ is secure between $q'$ and $q''$} provided 
	the following statement holds: If there is an edge $f$ 
	satisfying $q', q''\in f\in E(F_\star)$, then this edge is not useful. 
	The intuition here is that it is impossible to collapse copies
	that occur securely between their neighbouring connectors.  
	
	\begin{clm}\label{clm:6640}
		If $F_\star\in \ccN$ and $q'\in V(F_\star)$ is useful,
		then there exists 	a copy $F_{\star\star}\in\ccN\sm\{F_\star\}$ satisfying 
		$q'\in V(F_{\star\star})$ together with another useful vertex 
		$q''\in V(F_{\star\star})$ such that 
		$F_{\star\star}$ is secure between $q'$ and $q''$.
	\end{clm}
	
	\begin{proof}
		Let $F'_{q'}, F''_{q'}\in \ccN$ be the copies Claim~\ref{clm:1455} provides 
		for $x=q'$, choose $F_{\star\star}\in \{F'_{q'}, F''_{q'}\}$ such that 
		$F_{\star\star}\ne F_\star$, and observe that Claim~\ref{clm:1455}\ref{it:1455i}
		ensures $q'\in V(F_{\star\star})$.  
		Let the set $z^{\star\star}$ be defined 
		with respect to $F_{\star\star}$ as in~\eqref{eq:6622}. The failure 
		of our lemma entails $|z^{\star\star}|\ge 2$ and thus there exists 
		a vertex $q^\circ\in z^{\star\star}$ distinct from~$q'$. 
		If $F_{\star\star}$ is secure between $q'$ and $q^\circ$, we can just 
		take $q''=q^\circ$, so suppose from now on that this is not the case. 
		This means that there exists a useful edge $f$ satisfying  
		$q', q^\circ\in f\in E(F_{\star\star})$. 
		
		\begin{figure}[ht]
	\centering	
			\begin{tikzpicture}[scale=.8]
	
	\def\w{2.5};
	\def\h{1.5};
	\def\x{2}
	
	\fill[blue!15!white, opacity = .5] (-\x,0) ellipse (\w cm and \h cm);
	\draw[blue!75!black,thick] (-\x,0) ellipse (\w cm and \h cm);
	
	\fill[blue!15!white, opacity = .5] (\x,0) ellipse (\w cm and \h cm);
	\draw[blue!75!black,thick] (\x,0) ellipse (\w cm and \h cm);
	
	\coordinate (a) at (2,1);
	\coordinate (b) at (0,-.5);
	\coordinate (c) at (4,-.5);
	
	\draw [green!70!black, thick] (a) -- (b) -- (c);
		
		\foreach \i in {a,b,c}{
			\fill (\i) circle (2pt);}
						
	\node at (1,.6) { $f$};
	\node at (2,-.9) {$f'$};
	\node at (-4.5,1) {$F_{\star}$};
	\node at (4.7,1) {$F_{\star\star}$};
	\node at (2.4,1.1) {$q^\circ$};
			\node at (-.15,-.1) {$q'$};
	\node at (3.9,-.1) {$q''$};
			
			\end{tikzpicture}
			
				\caption{Security of $F_{\star\star}$}
				\label{fig:91}

	\end{figure} 		
		Clearly $f\subseteq z^{\star\star}$ and by appealing to the failure of 
		our lemma again we learn $z^{\star\star}\ne f$. Hence 
		there exists a vertex $q''\in z^{\star\star}\sm f$ and it suffices 
		to prove that $F_{\star\star}$ is secure between~$q'$ and~$q''$ (see Figure~\ref{fig:91}). 
		If this were not the case, there had to exist a useful edge $f'$
		such that $q', q''\in f'\in E(F_{\star\star})$. Now $q''\in f'\sm f$
		yields $f\ne f'$ and, therefore, $F_{\star\star}$ violates clause~\ref{it:1455ii}
		of Claim~\ref{clm:1455}. This contradiction concludes the proof of Claim~\ref{clm:6640}. 
	\end{proof}
	
	Let us call a cycle of copies $F_1q_1\ldots F_nq_n$ in 
	$(H, \ccN)$ or $(H, \ccN^+)$ {\it special} if 
	\begin{enumerate}
		\item[$\bullet$] the copies $F_1, \ldots, F_n$ are distinct,
		\item[$\bullet$] the connectors $q_1, \ldots, q_n$ are vertices,
		\item[$\bullet$] and with at most one exception every copy $F_i$ 
			is secure between $q_{i-1}$ and $q_i$.
	\end{enumerate}
	
	\begin{clm}\label{clm:6650}
		There exist a special cycle of copies in $(H, \ccN)$.	
	\end{clm}
	
	\begin{proof}
		An iterative application of Claim~\ref{clm:6640} allows us to 
		construct an infinite sequence
				\[
			F_1q_1F_2q_2\ldots
		\]
				consisting of copies $F_1, F_2, \ldots \in\ccN$ and useful 
		vertices $q_1, q_2, \ldots,$ such that for every $i\in \NN$ 
		\begin{enumerate}[label=\nlabel]
			\item\label{it:6651} the vertex $q_i$ belongs to $V(F_i)\cap V(F_{i+1})$, 
			\item\label{it:6652} the copies $F_i$, $F_{i+1}$ are distinct,
			\item\label{it:6653} and $F_{i+1}$ is secure between $q_i$ and $q_{i+1}$. 
		\end{enumerate}
		
		Indeed, let $F_1\in \ccN$ be arbitrary and let $q_1\in V(F_1)$ be a
		useful vertex. If for some natural number $m$ we have already constructed 
		the initial segment $F_1q_1\ldots F_mq_m$ of our infinite sequence, 
		we apply Claim~\ref{clm:6640} to $(F_m, q_m)$ here in place of $(F_\star, q')$ there,
		thus obtaining a copy $F_{m+1}$ and a useful vertex $q_{m+1}$ that allow us to continue. 
		
		Since $\ccN$ is finite, there exists some $n\in\NN$ such that 
		$F_{n+1}\in \{F_1, \ldots, F_n\}$. If $n$ denotes the least such natural
		number, then $F_1, \ldots, F_n$ are distinct and, hence, there 
		is a unique index $i\in [n]$ such that $F_i=F_{n+1}$. For notational 
		simplicity we may suppose that $i=1$. Since~\ref{it:6652} implies $n\ge 2$,
		the cyclic sequence 
				\[
			\ccC=F_1q_1\ldots F_nq_n
		\]
				has all properties of a cycle of copies except that we do not 
		know whether its connectors are distinct. Moreover,~\ref{it:6653} 
		tells us that $F_1$ is the only copy in $\ccC$ that might be insecure.
		
		Thus if the connectors of $\ccC$ happen to be distinct, then $\ccC$ is 
		the desired special cycle of copies. If $q_1, \ldots, q_n$ fail to be distinct,
		we take a pair $(r, s)$ of indices with $1\le r<s\le n$,~$q_r=q_s$, 
		and, subject to this, such that $s-r$ is minimal. 
		Now one checks easily that $F_{r+1}q_{r+1}\ldots F_sq_s$ is a special cycle 
		of copies.
	\end{proof}
	
	Throughout the rest of the proof we consider a special cycle of copies 
	$\ccC=F_1q_1\ldots F_nq_n$ in the extended system $(H, \ccN^+)$ whose 
	length~$n$ is minimal. Claim~\ref{clm:6650} discloses $n\le |\ccN|$ and, 
	consequently, we have $\ord{\ccC} < \Gth(H, \ccN^+)$. In other words,
	\begin{enumerate}[label=\alabel]
		\item\label{it:6730a} either $\ccC$ fails to be tidy
		\item\label{it:6730b} or $\ccC$ has a master copy. 
	\end{enumerate}
	
	We shall show that both alternatives lead to a contradiction. 
	Let us deal with case~\ref{it:6730a} first. Since $\ccC$ has no edge 
	connectors, this is only possible if~\ref{it:T2} fails for some 
	edge $f\in E(H)$. By symmetry we can suppose that $1\in M(f)$. 
	Write $M(f)=\{i(1), \ldots, i(m)\}$ with $1=i(1)< i(2) < \ldots i(m)\le n$ 
	and $m\ge 2$. We observe that the cyclic sequences 
		\begin{align*}
		\ccD_\mu &=
			f^+q_{i(\mu)}F_{i(\mu)+1}q_{i(\mu)+1}\ldots F_{i(\mu+1)}q_{i(\mu+1)}
			\quad \text{ for } \mu\in[m-1] \\
		\text{ and } \quad 
			\ccD_m &=
			f^+q_{i(m)}F_{i(m)+1}q_{i(m)+1}\ldots F_{1}q_{1}
	\end{align*}
		are shorter than $\ccC$. Indeed, this is clear for $m\ge 3$ and in case $m=2$ 
	it follows from the fact that the two members of $M(f)$ cannot be consecutive 
	in $\ZZ/n\ZZ$.
	 
	We aim at showing that at least one of $\ccD_1, \ldots, \ccD_m$
	contradicts the minimality of $n$. Notice that the newly inserted 
	edge copy $f^+$ can cause trouble in two different ways. 
	First, it might be insecure in all of $\ccD_1, \ldots, \ccD_m$ and, 
	second, it might happen that one of $\ccD_1, \ldots, \ccD_m$ contains $f^+$
	twice. However, since $\ccC$ itself is special, at most one 
	of $\ccD_1, \ldots, \ccD_m$ contains an insecure appearance 
	of one of $F_1, \ldots, F_n$. Thus at most two among $\ccD_1, \ldots, \ccD_m$
	can fail to contradict the minimality of $n$, one due to containing
	two insecure copies and the (potential) other one due to 
	containing $f^+$ twice. 
	 
	In other words, the only case where we are not done yet occurs if $m=|M(f)|=2$ 
	and $f^+$ is among $F_1, \ldots, F_n$. As $f^+=F_i$ 
	implies $i-1, i\in M(f)$, this case requires $M(f)=\{i-1, i\}$,
	which contradicts the choice of $f$. Altogether, we have thereby shown 
	that~\ref{it:6730a} is indeed impossible, i.e., that our 
	minimal special cycle of copies $\ccC$ in $(H, \ccN^+)$ is tidy.
	
	Therefore~\ref{it:6730b} holds, i.e., $\ccC$ has a master copy. 
	By symmetry we may suppose that $F_1$ is a master copy of $\ccC$
	and that the family of edges $\{f_i\in E(F_1)\colon i\in [2, n]\}$
	exemplifies this fact. For every $i\in [2, n]$ Fact~\ref{rem:5027}
	tells us $f_i\in E(F_i)\cap E(F_1)$, wherefore 
	the edge $f_i$ is useful. Combined with $q_{i-1}, q_i\in f_i$ this 
	implies that $F_i$ fails to be secure between $q_{i-1}$ and~$q_i$.
	Thus despite being special $\ccC$ contains at least $n-1$ insecure copies, 
	which is only possible if $n=2$. But now the useful edge $f_2$ 
	exemplifies that $F_1$ is insecure as well and we have obtained 
	the final contradiction that rules out option~\ref{it:6730b} and 
	thereby concludes the proof of Lemma~\ref{lem:6349}. 
\end{proof}

\begin{lemma}\label{lem:6822}
	Let $(H, \ccN)$ be a linear system of hypergraphs. 
	If $\Gth(H, \ccN^+)>|\ccN|$, then $\ccN$ is a forest of copies.  
\end{lemma}

\begin{proof}
	We argue by induction on $n=|\ccN|$, the base case $n=1$ being clear. 
	In the induction step we have to deal with a linear system $(H, \ccN)$	
	consisting of $n\ge 2$ copies that satisfies $\Gth(H, \ccN^+)>|\ccN|$.
	Let $F_\star\in \ccN$ be a copy obtained by applying Lemma~\ref{lem:6349}
	to~$\ccN$ and set $\ccN_\star=\ccN\sm\{F_\star\}$. 
	Since $\Gth(H, \ccN_\star^+)>|\ccN|>|\ccN_\star|$, the induction hypothesis 
	shows that $\ccN_\star$ is a forest of copies. If $(F_1, \ldots, F_{n-1})$
	is an admissible enumeration of $\ccN_\star$, 
	then $(F_1, \ldots, F_{n-1}, F_\star)$ is the 
	desired admissible enumeration of $\ccN$.
\end{proof}
  
\subsection{The final partite construction}
\label{subsec:2142}

We shall prove the following strong form of 
Theorem~\ref{thm:1522} alluded to in the introduction.

\begin{thm}\label{thm:6643}
	Let a hypergraph $F$ and a natural number $g\ge 2$ satisfy 
	$\gth(F) > g$. If~$r, n\ge 2$ are two further natural numbers, then there 
	exists a linear system of hypergraphs $(H, \ccH)$ such that
	\begin{enumerate}
		\item[$\bullet$] $\ccH\lra(F)_r$
		\item[$\bullet$] and for every $\ccN\subseteq \ccH^+$ with $|\ccN|\in [2, n]$
			there exists some $\ccX\subseteq \ccH$ 
			for which $\ccN\cup \ccX$ is a forest of copies
			and $|\ccX|\le \frac{|\ccN|-2}{g-1}$. 
	\end{enumerate}
\end{thm}

Before we embark on the proof of this result we would like to point out
that for~$n\le g$ it follows quickly from a pair of statements that have 
been obtained earlier. Indeed, the construction $\Omega^{(g)}$ (see 
Theorem~\ref{thm:6653}) delivers a linear system of hypergraphs 
$(H, \ccH)$ such that $\ccH\lra(F)_r$ and $\Gth(H, \ccH^+)>g$.
Now for every $\ccN\subseteq \ccH^+$ with $|\ccN|\in [2, g]$ we have 
$\Gth(H, \ccN^+) >g\ge |\ccN|$ and owing to Lemma~\ref{lem:6822} 
this implies that $\ccN$ is a forest of copies. In other words, the second
bullet holds for $\ccX=\vn$, as desired. 

The general case of Theorem~\ref{thm:6643} will be proved by means of a 
further application of the partite construction method. This will involve 
pictures of the following kind. 

\begin{dfn}
	Given integers $g, n\ge 2$ a picture $(\Pi, \ccP, \psi)$ is said to be 
	{\it $(g, n)$-certified} if each of its constituents $\Pi^e$ satisfies 
	$\gth(\Pi^e)>n$ and for every $\ccN\subseteq \ccP^+$ with $|\ccN|\in [2, n]$
	there exists some $\ccX\subseteq \ccP$ such that $\ccN\cup\ccX$ is a forest of 
	copies and $|\ccX|\le \frac{|\ccN|-2}{g-1}$. \index{$(g, n)$-certified}
\end{dfn}

Let us show first that being certified is a property we can expect any reasonable 
picture zero to have.

\begin{lemma}\label{lem:6156}
	Suppose that a hypergraph $F$ and an integer $g\ge 2$ satisfy $\gth(F)>g$.
	If~$(G, \ccG)$ denotes any linear system of hypergraphs 
	with $\ccG\subseteq\binom GF$, then the picture zero over this system 
	is $(g, n)$-certified for every integer $n\ge 2$. 
\end{lemma}

\begin{proof}
	Let $(\Pi, \ccP, \psi)$ denote the picture zero under discussion. Since all of
	its constituents are matchings, their girth is greater than $n$ for every $n\ge 2$.
	Proceeding with the second property, we suppose that 
	any $\ccN\subseteq \ccP^+$ with $|\ccN|\ge 2$ is given. 
	For every copy $F_\star\in\ccP$ we set 
	$\ccN(F_\star)=\ccN \cap \bigl(E^+(F_\star)\cup\{F_\star\}\bigr)$. Since $\Pi$ 
	is the disjoint union of the copies in $\ccP$, this stipulation yields 
	a partition $\ccN=\bigdcup_{F_\star\in \ccP}\ccN(F_\star)$.
	
	There are two sufficient conditions ensuring that a partition class 
	$\ccN(F_\star)$ is a forest of copies. First, if $F_\star\in N(F_\star)$,
	then every enumeration of $\ccN(F_\star)$ starting with $F_\star$ itself 
	is admissible. Second, if $F_\star\not\in \ccN(F_\star)$ and $|\ccN(F_\star)|\le g$,
	then $\ccN(F_\star)$ is a forest of edge copies due to $\gth(F_\star)>g$.
	For these reasons, the set 
	\[
		\ccX=\bigl\{F_\star\in \ccP\colon |\ccN(F_\star)|\ge g+1\bigr\}
	\]
	has the property that $\ccN\cup\ccX$ is a forest of copies. 
	
	So it remains to prove $|\ccX|\le \frac{|\ccN|-2}{g-1}$. 
	The obvious bound $|\ccX|\le |\ccN|/(g+1)$ rewrites as
	$(g-1)|\ccX|+2(|\ccX|-1)\le |\ccN|-2$
	and provided that $\ccX$ is nonempty the desired upper bound on $|\ccX|$
	follows. Moreover, if $\ccX$ is empty, then we just need to appeal 
	to $|\ccN|\ge 2$. 
\end{proof}

The picturesque lemma appropriate for the present context reads as follows. 

\begin{lemma} \label{lem:6229}
	Suppose that $g, n\ge 2$ are integers and that 
		\begin{equation}\label{eq:6726}
		(\Sigma, \ccQ, \psi_\Sigma)
		=
		(\Pi, \ccP, \psi_\Pi)
		\conc
		(H, \ccH)
	\end{equation}
		holds for two pictures $(\Pi, \ccP, \psi_\Pi)$ and $(\Sigma, \ccQ, \psi_\Sigma)$
	over a linear system of hypergraphs~$(G, \ccG)$ and a linear $k$-partite $k$-uniform
	system of hypergraphs $(H, \ccH)$. If 
	\begin{enumerate}
		\item[$\bullet$] $(\Pi, \ccP, \psi_\Pi)$ is $(g, n)$-certified
		\item[$\bullet$] and $\Gth(H, \ccH^+)>n$, 
	\end{enumerate}
	then $(\Sigma, \ccQ, \psi_\Sigma)$ is $(g, n)$-certified as well. 
\end{lemma}    

\begin{proof}
	The demand on the girth of the constituents of $(\Sigma, \ccQ, \psi_\Sigma)$ 
	follows easily from the corresponding property of the constituents
	of $(\Pi, \ccP, \psi_\Pi)$ combined with $\Gth(H, \ccH^+)>n$ and the linearity
	of the vertical hypergraph $G$. So it remains to deal with the forest extension 
	property. 
	
	We shall say that $(\ccN, \phi)$ is a {\it good pair} if $\ccN\subseteq \ccQ^+$
	and $\phi\colon \ccN\lra\ccH^+$ is a map such that for every copy $F_\star\in \ccN$
	one of the following two cases occurs:
	\begin{enumerate}
		\item[$\bullet$] Either $F_\star\in E^+(H)$ is an edge copy and
			$\phi(F_\star)=F_\star$  
		\item[$\bullet$] or $\phi(F_\star)=\Pi_\star^e\in \ccH$ 
			and the standard copy $(\Pi_\star, \ccP_\star)$ extending $\Pi^e_\star$
			satisfies $F_\star\in \ccP_\star^+$. 
	\end{enumerate}
	Notice that for every $\ccN\subseteq \ccQ^+$ there is a map $\phi\colon \ccN\lra\ccH^+$
	such that $(\ccN, \phi)$ is a good pair. But for clarity we would like to remark 
	that $\phi$ does not need to be uniquely determined by~$\ccN$, since for 
	$f^+\in \ccN\cap E^+(H)$ there may be several legitimate choices for~$\phi(f^+)$. 
	
	A good pair $(\ccX, \xi)$ is said to {\it resolve} another good pair $(\ccN, \phi)$
	if 
	\begin{enumerate}[label=\rmlabel]
		\item\label{it:6149a} $\ccX\subseteq\ccQ$, 
		\item\label{it:6149b} $\xi[\ccX]\subseteq \phi[\ccN]$,
		\item\label{it:6149c} and $\ccN\cup\ccX$ is a forest of copies.
	\end{enumerate}
	
	We shall prove the following strengthening of our claim: 
	Every good pair $(\ccN, \phi)$ such that $|\ccN| \in [2, n]$ is resolved by another 
	good pair $(\ccX, \xi)$ satisfying $|\ccX|\le \frac{|\ccN|-2}{g-1}$.  
	
	Arguing indirectly we fix a good pair $(\ccN, \phi)$ with $|\ccN|\in [2, n]$
	\begin{enumerate}[label=\nlabel]
		\item\label{it:6311a} that is not resolved by any good pair $(\ccX, \xi)$ 
			with $|\ccX|\le \frac{|\ccN|-2}{g-1}$ 		
		\item\label{it:6311b} and subject to this in such a way 
			that $\big|\phi[\ccN]\big|$ minimal. 
	\end{enumerate}
	
	Assume first that $\big|\phi[\ccN]\big|\le 1$. Because of $|\ccN|\ge 2$ this is only 
	possible if there exists some standard copy $(\Pi_\star, \ccP_\star)$ such that 
	$\ccN\subseteq \ccP_\star^+$ and $\phi[\ccN]=\{\Pi^e_\star\}$, where $\Pi_\star$
	extends~$\Pi^e_\star$. 
	As the picture $(\Pi, \ccP, \psi_\Pi)$ is $(g, n)$-certified, there exists  
	a set $\ccX\subseteq \ccP_\star\subseteq \ccQ$ such that $\ccN\cup\ccX$
	is a forest of copies and $|\ccX|\le \frac{|\ccN|-2}{g-1}$. 
	Let $\xi$ be the unique map from $\ccX$ to $\{\Pi^e_\star\}$. 
	Now the good pair $(\ccX, \xi)$ resolves $(\ccN, \phi)$
	and thus it contradicts~\ref{it:6311a}.   
	This argument proves
		\begin{equation}\label{eq:6034}
		\big|\phi[\ccN]\big|\ge 2\,.
	\end{equation}
		
	Since $|\phi[\ccN]|\le |\ccN|\le n<\Gth(H, \ccH^+)$, Lemma~\ref{lem:6822}
	tells us that $\phi[\ccN]$ is a forest of copies. We choose a terminal 
	copy $\Pi^e_\star$ in this forest, partition $\ccN$ into the sets 
		\[
		\ccA=\phi^{-1}(\Pi^e_\star)
		\quad\text{ and } \quad
		\ccB=\ccN\sm\ccA\,,
	\]
		and remark that~\eqref{eq:6034} implies
		\begin{equation}\label{eq:6035}
		1\le \big|\phi[\ccA]\big|, \big|\phi[\ccB]\big| < \big|\phi[\ccN]\big|\,.
	\end{equation}
		
	Due to $\phi[\ccA]=\{\Pi^e_\star\}$ 
	and $\phi[\ccB]=\phi[\ccN]\sm \{\Pi^e_\star\}$ the terminal choice of~$\Pi^e_\star$
	guarantees that the set $x= V(\phi[A])\cap V(\phi[\ccB])$ is either an edge 
	in $E(\phi[A])\cap E(\phi[\ccB])$ or it consists of at most one vertex. 
	We begin with the former case, as it illustrates better how the upper 
	bound $\frac{|\ccN|-2}{g-1}$ arises in the proof. 
	
	\smallskip
		
	{\it \hskip2em First Case. $x\in E(\phi[A])\cap E(\phi[\ccB])$.}
		
	\smallskip
	
	Let us consider the sets $\ccA'=\ccA\cup\{x^+\}$ and $\ccB'=\ccB\cup\{x^+\}$.
	We claim that there are maps $\alpha\colon \ccA'\lra \ccH^+$ and  
	$\beta\colon \ccB'\lra \ccH^+$ such that 
		\begin{equation}\label{eq:6111}
		\text{$(\ccA', \alpha)$ and $(\ccB', \beta)$ are good pairs, 
		$\alpha[\ccA']=\phi[\ccA]$, and $\beta[\ccB']=\phi[\ccB]$}\,.
	\end{equation}
		Concerning the existence of $\alpha$, we observe that in case $x^+\in\ccA$ 
	we may just take the restriction of $\phi$ to $\ccA$. If $x^+\not\in\ccA$ we 
	additionally need to specify an appropriate value of~$\alpha(x^+)$ 
	in $\phi[\ccA]$, which is possible because of $x\in E(\phi[\ccA])$. Thus the
	desired map $\alpha$ does indeed exist and we may argue similarly with respect 
	to $\beta$. Thereby~\eqref{eq:6111} is proved. 
	
	Next we seek to establish
		\begin{equation}\label{eq:6031}
		|\ccA'|, |\ccB'| \in [2, n]\,. 
	\end{equation}
		Since~\eqref{eq:6035} implies $\ccA, \ccB\ne\ccN$, the upper bounds follow
	from $|\ccN|\le n$. Assume towards a contradiction that $|\ccA'|\le 1$. 
	By $\ccA\ne\vn$ this is only possible if $\ccA=\{x^+\}$,
	whence~$\ccB'=\ccN$. Now $(\ccN, \beta)$ is a good 
	pair, which by~\eqref{eq:6035} and~\eqref{eq:6111} satisfies 
	$\big|\beta[\ccN]\big| < \big|\phi[\ccN]\big|$. So by the minimality 
	demand~\ref{it:6311b} there is a good pair $(\ccX, \xi)$
	resolving $(\ccN, \beta)$ and satisfying $|\ccX|\le \frac{|\ccN|-2}{g-1}$.
	In particular, $(\ccX, \xi)$ resolves $(\ccN, \phi)$, contrary
	to~\ref{it:6311a}. This whole argument shows $|\ccA'|\ge 2$ and in a similar 
	fashion one confirms $|\ccB'|\ge 2$ as well. Thereby~\eqref{eq:6031} is proved. 
	
	For all these reasons~\ref{it:6311b} tells us that the good pairs $(\ccA', \alpha)$ 
	and $(\ccB', \beta)$ are resolved by certain good pairs $(\ccY, \upsilon)$ 
	and $(\ccZ,\zeta)$ that satisfy
		\begin{equation}\label{eq:6141}
		|\ccY|\le \frac{|\ccA'|-2}{g-1}
		\quad \text{ and } \quad
		|\ccZ|\le \frac{|\ccB'|-2}{g-1}\,,
	\end{equation}
		respectively. 
	Set $\ccX=\ccY\cup\ccZ$ and let $\xi\subseteq\upsilon\cup\zeta$
	be a map from $\ccX$ to $\ccH^+$. We contend that 
		\begin{equation}\label{eq:6142}
		\text{the good pair $(\ccX, \xi)$ resolves $(\ccN, \phi)$}\,.
	\end{equation}
		
	The demand~\ref{it:6149a} is clear and~\ref{it:6149b} follows from 
		\[
		\xi[\ccX]
		\subseteq
		\upsilon[\ccY]\cup \zeta[\ccZ]
		\overset{\text{\ref{it:6149b}}}{\subseteq}
		\alpha[\ccA']\cup \beta[\ccB']
		\overset{\eqref{eq:6111}}{=}
		\phi[\ccA]\cup \phi[\ccB]
		=
		\phi[\ccN]\,.
	\]
		For the verification of~\ref{it:6149c} we want to apply Lemma~\ref{lem:6032}
	to $\ccA\cup\ccY$, $\ccB\cup\ccZ$, and $x$ here in place of $\ccA$, $\ccB$, 
	and $e$ there. As the assumption that $(\ccA\cup\ccY)\cup\{x^+\}=\ccA'\cup\ccY$
	and $(\ccB\cup\ccZ)\cup\{x^+\}=\ccB'\cup\ccZ$ need to be forests of copies 
	are satisfied by our choice of the good pairs $(\ccY, \upsilon)$ and $(\ccZ, \zeta)$, 
	it remains to check that 
		\begin{equation}\label{eq:6614}
		V(\ccA\cup\ccY)\cap V(\ccB\cup\ccZ) \subseteq x\,.
	\end{equation}
	
	Towards the proof of this inclusion we observe that 
	$V(F_\star)\cap V(H)\subseteq V\bigl(\phi(F_\star)\bigr)$ holds for 
	every~$F_\star\in\ccA$, whence $V(\ccA)\cap V(H)\subseteq V(\phi[\ccA])$.
	Similarly one has 
		\[
		V(\ccY)\cap V(H)
		\subseteq 
		V(\upsilon[\ccY])
		\overset{\text{\ref{it:6149b}}}{\subseteq} 
		V(\alpha[\ccA'])
		\overset{\eqref{eq:6111}}{=}
		V(\phi[\ccA])
	\]
		and both statements together yield $V(\ccA\cup\ccY)\cap V(H)\subseteq V(\phi[\ccA])$. 
	Proceeding similarly with $\ccB\cup\ccZ$ and combining the results 
	we infer 
		\begin{equation}\label{eq:6615}
		\bigl(V(\ccA\cup\ccY)\cap V(\ccB\cup\ccZ)\bigr)\cap V(H) 
		\subseteq 
		V(\phi[\ccA])\cap V(\phi[\ccB])
		=
		x\,.
	\end{equation}
		Moreover, the definition of the partite amalgamation~\eqref{eq:6726} entails	
	$V(F_\star)\cap V(F_{\star\star})\subseteq V(H)$ for all $F_\star\in\ccA\cup\ccY$ 
	and $F_{\star\star}\in\ccB\cup\ccZ$ and for this reason
	we have 
		\[
		V(\ccA\cup\ccY)\cap V(\ccB\cup\ccZ)
		\subseteq 
		V(H)\,.
	\]
		Together with~\eqref{eq:6615} this establishes~\eqref{eq:6614} and thus 
	concludes the proof of~\eqref{eq:6142}.
	
	However, in the light of
		\[
		|\ccX|
		\le 
		|\ccY|+|\ccZ|
		\overset{\eqref{eq:6141}}{\le}
		\frac{|\ccA|-1}{g-1}+\frac{|\ccB|-1}{g-1}
		=
		\frac{|\ccN|-2}{g-1}
	\]
		this contradicts~\ref{it:6311a}. In other words, the case that $x$ is an edge 
	is impossible. 
	
	\smallskip
		
	{\it \hskip2em Second Case. $|x|\le 1$.}
		
	\smallskip
	
	We argue similar as in the first case. The pairs $(\ccA, \phi\upharpoonright \ccA)$ 
	and $(\ccB, \phi\upharpoonright \ccB)$ are good and the main point is that they 
	are resolved by certain good pairs $(\ccY, \upsilon)$ and $(\ccZ, \zeta)$
	such that 
		\begin{equation}\label{eq:6843}
		|\ccY|\le \frac{|\ccA|-1}{g-1}
		\quad \text{ and } \quad
		|\ccZ|\le \frac{|\ccB|-1}{g-1}\,,
	\end{equation}
	respectively. In fact, if $|\ccA|\ge 2$ we can argue exactly as in the first
	case in order to obtain such a good pair $(\ccY, \upsilon)$ satisfying the 
	stronger estimate $|\ccY|\le \frac{|\ccA|-2}{g-1}$. Moreover, in case
	$|\ccA|=1$ we can just take $\ccY=\upsilon=\vn$.
	The claim on $(\ccZ, \zeta)$ is proved in the same way. 
	
	Again we set $\ccX=\ccY\cup\ccZ$ and take a map $\xi\subseteq\upsilon\cup\zeta$
	from $\ccX$ to $\ccH^+$.
	We remark that~\eqref{eq:6614} is still valid in the present case and the only 
	change we need to make when concluding that $\ccN\cup\ccX$ is a forest of copies 
	is that this time we need to appeal to Lemma~\ref{lem:7032}. Apart from these
	small modifications, the proof that the good pair $(\ccX, \xi)$ 
	resolves $(\ccN, \phi)$ is still the same. Finally,~\eqref{eq:6843} implies
	$|\ccX|\le \frac{|\ccN|-2}{g-1}$, meaning that again we reach a contradiction
	to the choice of $(\ccN, \phi)$ in~\ref{it:6311a}. 
\end{proof}

After these preparations we can quickly establish the main result of this subsection. 

\begin{proof}[Proof of Theorem~\ref{thm:6643}]
	Let $\Phi$ be any linear Ramsey construction and let $\Xi$ be a partite 
	lemma applicable to $k$-partite $k$-uniform hypergraphs $F$ with $\gth(F)>n$
	that delivers linear systems of hypergraphs $(H, \ccH)$ 
	satisfying $\Gth(H, \ccH^+)>n$. We recall that the existence of such a partite 
	lemma is an immediate consequence of Theorem~\ref{thm:6653}.
	
	Let us attempt to perform the partite construction 
	$\PC(\Phi, \Xi)_r(F)	=(H, \ccH)$. By Lemma~\ref{lem:6156} picture zero is 
	$(g, n)$-certified and an iterative application of Lemma~\ref{lem:6229} 
	discloses that all further pictures are well-defined and, likewise, 
	$(g, n)$-certified. In particular, the final picture is 
	well-defined and the fact that it is $(g, n)$-certified implies that 
	it has the desired property.   
\end{proof}

Finally, we deduce the last statement announced in the introduction.   

\begin{proof}[Proof of Theorem~\ref{cor:19}]
	Given a linear $k$-uniform hypergraph $F$ and $r, n\in\NN$ we apply 
	Theorem~\ref{thm:1522} with $n'=\binom nk$ instead of $n$, thus obtaining
	some linear system $(H, \ccH)$. Without loss of generality we can assume 
	that $E(H)=E(\ccH)$, for deleting  
	edges from~$H$ that belong to none of the copies in $\ccH$ cannot influence
	whether the system~$(H, \ccH)$ satisfies the conclusion of Theorem~\ref{thm:1522}.
 
	In order to show that $H$ has the desired property we consider any set 
	$X\subseteq E(H)$ whose size it at most $n$. For every edge $e$ of $H$ 
	contained in $X$ we fix some copy $F_e\in\ccH$ such that $e\in E(F_e)$. 
	Since there are at most $\binom{|X|}{k}$ such edges, the set 
		\[
		\ccN^{-}=\bigl\{F_e\in \ccH\colon e\in E(H) \text{ and } e\subseteq X\bigr\} 
	\]
		has at most the size $n'$. In the special case that $|\ccN^{-}|\le 1$ 
	we can set $\ccN=\ccN^-$ and are done. Otherwise the conclusion of Theorem~\ref{thm:1522}
	yields a set $\ccX\subseteq\ccH$ such that $\ccN=\ccN^-\cup\ccX$ is a forest
	of copies with the desired property.  
\end{proof}

\subsection*{Acknowledgement} 
We would like to thank {\sc Joanna Polcyn} for drawing the beautiful pictures 
and {\sc Mathias Schacht} for devoting a great amount of his time 
to reading and studying this article with us while we have been writing it.

Moreover, we thank {\sc Mathias Schacht} for writing an initial draft of Section~\ref{sec:overview}. The current version owes a lot to his efforts. 
\printindex

\end{document}